\documentclass[10pt]{article}

\usepackage{amssymb,amsmath,amsthm,amsfonts}
\usepackage{graphicx}
\usepackage{epstopdf}
\usepackage{comment}
\usepackage{algorithm,algorithmic}
\graphicspath{{./pics/}}
\usepackage{verbatim}
\usepackage{color}
\allowdisplaybreaks

\numberwithin{equation}{section}
\numberwithin{algorithm}{section}
\newtheorem{theorem}{Theorem}[section]
\newtheorem{remark}{Remark}[section]

\newtheorem{lemma}{Lemma}[section]

\newtheorem{example}{Example}
\def\cT {\mathcal T}
\def\bold {\boldsymbol}
\def\Om {\Omega}
\def\J {\mathcal{J}}
\def\F {\mathcal F}
\def\to{\rightarrow}
\def\eps {\varepsilon}
\def\Atilde {\widetilde{\mathcal A}}

\newcommand{\dx}{\,{\rm d}x}
\newcommand{\dd}{\,{\rm d}}

\begin{document}
\title{Adaptive Reconstruction for Electrical Impedance Tomography with a Piecewise Constant Conductivity}
\author{
\and Bangti Jin\footnote{Department of Computer Science, University College London, Gower Street, London WC1E 6BT, UK (b.jin@ucl.ac.uk, bangti.jin@gmail.com)}
\and Yifeng Xu\footnote{Department of Mathematics and Scientific Computing Key Laboratory of Shanghai Universities, Shanghai Normal University, Shanghai 200234, China. (yfxu@shnu.edu.cn, mayfxu@gmail.com)}
}

\date{}

\maketitle
\begin{abstract}
In this work we propose and analyze a numerical method for electrical impedance tomography of recovering a piecewise
constant conductivity from boundary voltage measurements. It is based on standard Tikhonov regularization with
a Modica-Mortola penalty functional and adaptive mesh refinement using suitable \textit{a posteriori}
error estimators of residual type that involve the state, adjoint and variational inequality in the necessary
optimality condition and a separate marking strategy. We prove the convergence of the adaptive algorithm in the
following sense: the sequence of discrete solutions contains a subsequence convergent to a solution of the
continuous necessary optimality system. Several numerical examples are presented
to illustrate the convergence behavior of the algorithm.

\end{abstract}


\section{Introduction}\label{sect:intro}

Electrical impedance tomography (EIT) aims at recovering the electrical conductivity distribution of an object from
voltage measurements on the boundary. It has attracted much interest in medical imaging, geophysical prospecting,
nondestructive evaluation and pneumatic oil pipeline conveying etc. A large number of reconstruction algorithms have
been proposed; see, e.g., \cite{LechleiterRieder:2006,LechleiterHyvonen:2008,HintermullerLaurain2008,
KnudsenLassasMueller:2009,JinMaass:2010,JinKhanMaass:2012,HarrachUllrich:2013,ChowItoZou:2014,GehreJin:2014,MalonedosSantosHolder:2014,
LiuKolehmainen:2015,DunlopStuart:2015,AlbertiAmmari:2016,Klibanov:2017, KlibanovLiZhang:2019,ZhouHarrachSeo:2018,HyvonenMustonen:2018,XiaoLiuZhao:2018,
TanLvDong:2019} for a rather incomplete list. One prominent idea underpinning many imaging algorithms is regularization,
especially Tikhonov regularization \cite{ItoJin:2014}. In practice, they are customarily implemented using the
continuous piecewise linear finite element method (FEM) on quasi-uniform meshes, due to its flexibility
in handling spatially variable coefficients and general domain geometry. The convergence analysis of finite element
approximations was carried out in \cite{GehreJinLu:2014,Rondi:2016,Hinze:2018}.

In several practical applications, the physical process is accurately described by the complete electrode model (CEM)
\cite{ChengIsaacsonNewellGisser:1989,SomersaloCheneyIsaacson:1992}. It employs nonstandard boundary conditions to capture
characteristics of the experiment. In particular, around the electrode edges, the boundary condition changes
from the Neumann to Robin type, which, according to classical elliptic regularity theory \cite{Grisvard:1985},
induces weak solution singularity around the electrode edges; see, e.g., \cite{Pidcock:1995} for an early study. Further,
the low-regularity of the sought-for conductivity distribution, especially that enforced by a nonsmooth penalty, e.g.,
total variation, can also induce weak interior singularities of the solution. Thus, a (quasi-)uniform triangulation
of the domain can be inefficient for resolving these singularities, and the discretization errors around electrode
edges and internal interfaces can potentially compromise the reconstruction accuracy. These observations motivate the
use of an adaptive strategy to achieve the desired accuracy in order to enhance the overall computational efficiency.

For direct problems, the mathematical theory of AFEM, including \textit{a posteriori} error estimation, convergence
and computational complexity, has advanced greatly \cite{AinsworthOden:2000,NSV:2009,Ver:2013,
CFPP:2014}. A common adaptive FEM (AFEM) consists of the following successive loops:
\begin{equation}\label{afem_loop}
  \mbox{SOLVE}\rightarrow\mbox{ESTIMATE}\rightarrow\mbox{MARK}\rightarrow\mbox{REFINE}.
\end{equation}
The module \texttt{ESTIMATE} employs the given problem data and computed solutions to provide computable
quantities on the local errors, and distinguishes different adaptive algorithms.

In this work, we develop an adaptive EIT reconstruction algorithm with a piecewise constant conductivity. In practice, the
piecewise constant nature is commonly enforced by a total variation penalty. However, it is challenging for AFEM
treatment (see e.g., \cite{Bartels:2015} for image denoising). Thus, we take an indirect approach based on a
Modica-Mortola type functional:
\[
    \mathcal{F}_\varepsilon(\sigma) = \eps\int_\Om |\nabla \sigma|^2 \dx + \frac{1}{\eps}\int_{\Om} W(\sigma) \dx,
\]
where the constant $\eps>0$ is small and $W(s):\mathbb{R}\to\mathbb{R}$ is the double-well potential, i.e.,
\begin{equation}\label{eqn:double-well}
  W(s) = (s-c_0)^2(s-c_1)^2,
\end{equation}
with $c_0,c_1>0$ being two known values that the conductivity $\sigma$ can take. The functional $\mathcal{F}_{\varepsilon}$
$\Gamma$-converges to the total variation semi-norm \cite{Modica:1987,Modica:1977,Alberti:2000}. The corresponding regularized least-squares
formulation reads
\begin{equation}\label{eqn:tikh-MM-0}
        \inf_{\sigma\in \widetilde{\mathcal{A}}} \left\{\J_{\varepsilon}(\sigma) = \tfrac{1}{2} \|U(\sigma)-U^\delta\|^2 + \tfrac{\widetilde{\alpha}}{2}\mathcal{F}_\varepsilon(\sigma)\right\},
\end{equation}
where $\tilde \alpha>0$ is a regularization parameter; see Section \ref{sect:ps} for further details. In this
work, we propose \textit{a posteriori} error estimators and an adaptive reconstruction algorithm of the form \eqref{afem_loop}
for \eqref{eqn:tikh-MM-0} based on a separate marking using three error indicators in the module \texttt{MARK};
see Algorithm \ref{alg_afem_eit}. Further, we give a convergence analysis of the algorithm, in the
sense that the sequence of state, adjoint and conductivity generated by the adaptive algorithm contains a
convergent subsequence to a solution of the necessary optimality condition. The technical proof consists of
two steps: Step 1 shows the subsequential convergence to a solution of a limiting problem, and Step 2 proves
that the solution of the limiting problem satisfies the necessary optimality condition. The main technical
challenges in the convergence analysis include the nonlinearity of the forward model, the nonconvexity of
the double well potential and properly treating the variational inequality. The latter two are overcome by pointwise
convergence of discrete minimizers and Lebesgue's dominated convergence theorem, and AFEM analysis
techniques for elliptic obstacle problems, respectively. The adaptive algorithm and its convergence analysis
are the main contributions of this work.

Last, we situate this work in the existing literature. In recent years, several
adaptive techniques, including AFEM, have been applied to the numerical resolution of inverse problems.
In a series of works \cite{BeilinaClason:2006,BeilinaKlibanov:2010a,BeilinaKlibanov:2010b,BeilinaKlibanovKokurin:2010},
Beilina et al studied the AFEM in a dual weighted residual framework for parameter identification. Feng et al
\cite{FengYanLiu:2008} proposed a residual-based estimator for the state, adjoint and control by assuming convexity of
the cost functional and high regularity on the control. Li et al \cite{LiXieZou:2011} derived \textit{a posteriori}
error estimators for recovering the flux and proved their reliability; see \cite{XuZou:2015a} for
a plain convergence analysis. Clason et al \cite{ClasonKaltenbacher:2016} studied functional \textit{a posteriori} estimators
for convex regularized formulations. Recently, Jin et al \cite{JinXuZou:2016} proposed a first AFEM for Tikhonov
functional for EIT with an $H^1(\Omega)$ penalty, and also provided a convergence analysis. This work
extends the approach in \cite{JinXuZou:2016} to the case of a piecewise constant conductivity. There are a
number of major differences between this work and \cite{JinXuZou:2016}. First, the $H^1(\Omega)$ penalty in \cite{JinXuZou:2016}
facilitates deriving the \textit{a posteriori} estimator on the conductivity $\sigma$, by completing the squares
and suitable approximation argument, which is not directly available for the Mordica-Mortola functional $\mathcal{F}_\varepsilon$.
Second, we develop a sharper error indicator associated with the crucial variational inequality
than that in \cite{JinXuZou:2016}, by a novel constraint preserving interpolation operator \cite{ChenNochetto:2000};
see the proof of Theorem \ref{thm:gat_mc}, which represents the main technical novelty of this work. Third,
Algorithm \ref{alg_afem_eit} employs a separate marking for the estimators, instead of a collective marking in
\cite{JinXuZou:2016}, which automatically takes care of different scalings of the estimators.

The rest of this paper is organized as follows. In Section \ref{sect:ps}, we introduce the complete electrode model,
and the regularized least-squares formulation. In Section \ref{sect:alg}, we give the AFEM algorithm. In Section
\ref{sec:numer}, we present extensive numerical results to illustrate its convergence and efficiency. In Section \ref{sect:conv},
we present the lengthy technical convergence analysis. Throughout, $\langle\cdot,\cdot\rangle$ and $(\cdot,\cdot)$ denote the
inner product on the Euclidean space and $(L^2(\Omega))^d$, respectively, by $\|\cdot\|$ the Euclidean norm, and
occasionally abuse $\langle\cdot,\cdot\rangle$ for the duality pairing between the Hilbert space $\mathbb{H}$ and its dual
space. The superscript $\rm t$ denotes the transpose of a vector. The notation $c$ denotes a generic constant, which
may differ at each occurrence, but it is always independent of the mesh size and other quantities of interest.

\section{Regularized approach}\label{sect:ps}

This part describes the regularized approach for recovering piecewise constant conductivities.

\subsection{Complete electrode model (CEM)}

Let $\Omega$ be an open bounded domain in $\mathbb{R}^{d}$ $(d=2,3)$ with a polyhedral boundary $\partial\Omega$. We denote the
set of electrodes by $\{e_l\}_{l=1}^L$, which are line segments/planar surfaces on $\partial\Omega$ and satisfy
$\bar{e}_i\cap\bar{e}_k=\emptyset$ if $i\neq k$. The applied current on electrode $e_l$ is denoted by
$I_l$, and the vector $I=(I_1,\ldots,I_L)^\mathrm{t}\in\mathbb{R}^L$ satisfies $\sum_{l=1}^LI_l=0$, i.e., $I\in \mathbb{R}_\diamond^L
:=\{V\in \mathbb{R}^L: \sum_{l=1}^LV_l=0\}$. The electrode
voltage $U=(U_1,\ldots,U_L)^\mathrm{t}$ is normalized, i.e., $U\in\mathbb{R}_\diamond^L$. Then the CEM reads: given the
conductivity $\sigma$, positive contact impedances $\{z_l\}_{l=1}^L$ and input current $I\in\mathbb{R}_\diamond^L$, find
$(u,U)\in H^1(\Omega)\otimes\mathbb{R}_\diamond^L$ such that \cite{ChengIsaacsonNewellGisser:1989,SomersaloCheneyIsaacson:1992}
\begin{equation}\label{eqn:cem}
\left\{\begin{aligned}
\begin{array}{ll}
-\nabla\cdot(\sigma\nabla u)=0 & \mbox{ in }\Omega,\\[1ex]
u+z_l\sigma\frac{\partial u}{\partial n}=U_l& \mbox{ on } e_l, l=1,2,\ldots,L,\\[1ex]
\int_{e_l}\sigma\frac{\partial u}{\partial n}\mathrm{d}s =I_l& \mbox{ for } l=1,2,\ldots, L,\\ [1ex]
\sigma\frac{\partial u}{\partial n}=0&\mbox{ on } \partial\Omega\backslash\cup_{l=1}^Le_l.
\end{array}
\end{aligned}\right.
\end{equation}
The inverse problem is to recover the conductivity $\sigma$ from a noisy version $U^\delta$ of the electrode
voltage $U(\sigma^\dagger)$ (for the exact conductivity $\sigma^\dag$) corresponding to one or multiple input currents.

Below the conductivity $\sigma$ is assumed to be piecewise constant, i.e., in the admissible set
\[
    \mathcal{A}:=\{\sigma\in L^\infty(\Om):~\sigma=c_0+(c_1-c_0)\chi_{\Om_1}\},
\]
where the constants $c_1>c_0>0$ are known, $\overline{\Om}_1\subset\Om$ is an unknown open set with a
Lipschitz boundary and $\chi_{\Om_1}$ denotes its characteristic function. We denote by
$\mathbb{H}$ the space $H^1(\Omega)\otimes \mathbb{R}_\diamond^L$ with its norm given by
\begin{equation*}
  \|(u,U)\|_{\mathbb{H}}^2 = \|u\|_{H^1(\Omega)}^2 + \|U\|^2.
\end{equation*}

A convenient equivalent norm on the space $\mathbb{H}$ is given below.
\begin{lemma}\label{lem:normequiv}
On the space $\mathbb{H}$, the norm $\|\cdot\|_\mathbb{H}$ is equivalent to the norm
$\|\cdot\|_{\mathbb{H},*}$ defined by
\begin{equation*}
  \|(u,U)\|_{\mathbb{H},*}^2 = \|\nabla u\|_{L^2(\Omega)}^2 + \sum_{l=1}^L\|u-U_l\|_{L^2(e_l)}^2.
\end{equation*}
\end{lemma}

The weak formulation of the model \eqref{eqn:cem} reads \cite{SomersaloCheneyIsaacson:1992}: find $(u,U)\in \mathbb{H}$ such that
\begin{equation}\label{eqn:cemweakform}
  a(\sigma,(u,U),(v,V)) = \langle I,V\rangle \quad \forall (v,V)\in \mathbb{H},
\end{equation}
where the trilinear form $a(\sigma,(u,U),(v,V))$ on $\mathcal{A}\times\mathbb{H}\times\mathbb{H}$ is defined by
\begin{equation*}
   a(\sigma,(u,U),(v,V)) = (\sigma \nabla u ,\nabla v ) +\sum_{l=1}^Lz_l^{-1}(u-U_l,v-V_l)_{L^2(e_l)},
\end{equation*}
where $(\cdot,\cdot)_{L^2(e_l)}$ denotes the $L^2(e_l)$ inner product. For any $\sigma\in\mathcal{A}$,
$\{z_l\}_{l=1}^L$ and $I\in \Sigma_\diamond^L$, the existence and uniqueness of a solution
$(u,U)\in\mathbb{H}$ to \eqref{eqn:cemweakform} follows from Lemma \ref{lem:normequiv} and
Lax-Milgram theorem.

\subsection{Regularized reconstruction}
For numerical reconstruction with a piecewise constant conductivity, the total variation (TV) penalty is popular. The
conductivity $\sigma$ is assumed to be in the space $\mathrm{BV}(\Om)$ of bounded
variation \cite{AttouchButtazzoMichaille:2006,Evans:2015}, i.e.,
\[
    \mathrm{ BV} (\Om) = \left\{ v \in L^1(\Om): |v|_{\mathrm{TV}(\Om)}<\infty \right\},
\]
equipped with the norm $\|v\|_{\mathrm{BV}(\Om)}=\|v\|_{L^1(\Om)}+|v|_{\mathrm{TV}(\Om)}$, where
\[
    |v|_{\mathrm{TV}(\Om)}:=\sup\left\{\int_\Om v \nabla\cdot\bold{\phi}\dx:~\bold{\phi}\in (C_c^1(\Om))^d,~\|\bold{\phi}(x)\|\leq 1\right\}.
\]
Below we discuss only one dataset, since the case of multiple datasets is similar. Then Tikhonov regularization
leads to the following minimization problem:
\begin{equation}\label{eqn:tikh}
   \min_{\sigma\in\mathcal{A}}\left\{\J(\sigma) = \tfrac{1}{2} \|U(\sigma)-U^\delta\|^2 + \alpha|\sigma|_{{\rm TV}(\Om)}\right\},
\end{equation}
The scalar $\alpha>0$ is known as a regularization parameter. It has at least one minimizer \cite{Rondi:2008,GehreJinLu:2014}.

Since $\sigma$ is piecewise constant, by Lebesgue decomposition theorem \cite{AttouchButtazzoMichaille:2006}, the TV term $|\sigma|_{\rm TV(\Omega)}$ in \eqref{eqn:tikh} reduces to
$\int_{S_{\sigma}} |[\sigma]| \dd \mathcal{H}^{d-1} $, where $S_\sigma$ is the jump set, $[\sigma]=\sigma^+-\sigma^-$ denotes the jump
across $S_\sigma$ and $\mathcal{H}^{d-1}$ refers to the $(d-1)$-dimensional Hausdorff measure. The numerical approximation of \eqref{eqn:tikh}
requires simultaneously treating two sets of different Hausdorff dimensions (i.e., $\Om$ and $S_\sigma$), which is
very challenging. Thus, we replace the TV term $|\sigma|_{\rm TV(\Omega)}$ in \eqref{eqn:tikh} by a Modica--Mortola type functional \cite{Modica:1977}
\[
    \mathcal{F}_\varepsilon(z):=
    \left\{\begin{array}{ll}
        \varepsilon \|\nabla z\|^2_{L^2(\Om)}+\frac{1}{\varepsilon}\int_{\Om}W(z)\dx & \mbox{if}~z\in H^1(\Om),\\
        +\infty & \mbox{otherwise},
        \end{array}
        \right.
\]
where $\varepsilon$ is a small positive constant controlling the width of the transition interface,
and $W: \mathbb{R}\rightarrow\mathbb{R}$ is the double-well potential given in
\eqref{eqn:double-well}.  The functional $\mathcal{F}_\varepsilon$ was first proposed to model
phase transition of two immiscible fluids in \cite{Cahn:1958}. It is connected with the TV semi-norm as follows
\cite{Modica:1977,Modica:1987,Alberti:2000}; see \cite{Braides:2002} for an introduction to $\Gamma$-convergence.
\begin{theorem}\label{thm:G-conv}
With $c_W=\int_{c_0}^{c_1}\sqrt{W(s)}\dd s$, let
    \[
        \mathcal{F}(z):=\left\{\begin{array}{ll}
           2c_W|z|_{\rm TV(\Om)} & \mbox{if}~z\in \mathrm{BV}(\Om)\cap\mathcal{A}, \\
           +\infty  & \mbox{otherwise}.
        \end{array}
        \right.
    \]
Then $\mathcal{F}_\varepsilon$ $\Gamma$-converges to $\mathcal{F}$ in $L^1(\Om)$ as $\varepsilon\to 0^+$. Let $\{\eps_n\}_{n\geq1}$ and
$\{v_{n}\}_{n\geq 1}$ be given sequences such that $\eps_n\to 0^+$ and $\{\F_{\eps_n}(v_n)\}_{n\geq1}$ is bounded. Then $v_n$ is precompact in $L^1(\Om)$.
\end{theorem}

The proposed EIT reconstruction method reads
    \begin{equation}\label{eqn:tikh-MM}
        \inf_{\sigma\in \widetilde{\mathcal{A}}} \left\{\J_{\varepsilon}(\sigma) = \tfrac{1}{2} \|U(\sigma)-U^\delta\|^2 + \tfrac{\widetilde{\alpha}}{2}\mathcal{F}_\varepsilon(\sigma)\right\},
    \end{equation}
where $\widetilde{\alpha}=\alpha/c_W$,  and the admissible set $\widetilde{\mathcal{A}}$ is defined as
\[
    \widetilde{\mathcal{A}}:=\left\{{\sigma \in H^1(\Omega)}: c_0\leq \sigma(x)\leq c_1\mbox{ a.e. } x\in\Omega\right\}.
\]

Now we recall a useful continuity result of the forward map \cite[Lemma 2.2]{GehreJinLu:2014}, which
gives the continuity of the fidelity term in the functional $\mathcal{J}_\varepsilon$.
See also \cite{JinMaass:2010,DunlopStuart:2015} for related continuity results.
\begin{lemma}\label{lem:fm_cont}
    Let $\{\sigma_n\}_{n\geq1}\subset\widetilde{\mathcal{A}}$ satisfy $\sigma_n\to\sigma^\ast$ in $L^1(\Om)$. Then
    \begin{equation}\label{eqn:fm_cont}
        \lim_{n\to\infty}\|\left(u(\sigma_n)-u(\sigma^\ast),U(\sigma_n)- U(\sigma^\ast)\right)\|_{\mathbb{H}}=0.
    \end{equation}
\end{lemma}

Lemma \ref{lem:fm_cont} implies that $\{\mathcal{J}_\varepsilon\}_{\varepsilon>0}$ are continuous perturbations
of $\mathcal{J}$ in $L^1(\Om)$. Then the stability of $\Gamma$-convergence \cite[Proposition 1(ii)]{Alberti:2000}
\cite[Remark 1.4]{Braides:2002} and Theorem \ref{thm:G-conv} indicate that $\J_\eps$ $\Gamma$-converges to $\mathcal{J}$
with respect to $L^1(\Om)$, and $\mathcal{J}_{\varepsilon}$ can (approximately) recover piecewise constant conductivities.
Next we show the existence of a minimizer to $\J_\eps$.

\begin{theorem}\label{thm:tikh-MM}
    For each $\varepsilon>0$, there exists at least one minimizer to problem \eqref{eqn:tikh-MM}.
\end{theorem}
\begin{proof}
Since $\mathcal{J}_\varepsilon$ is nonnegative, there exists a minimizing sequence $\{\sigma_n\}_{n\geq1}\subset
\widetilde{\mathcal{A}}$ such that $\mathcal{J}_\varepsilon(\sigma_n)\to m_\eps:=\inf_{\sigma\in\widetilde{\mathcal{A}}}\mathcal{J}_\varepsilon(\sigma)$.
Thus, $\sup_n\|\nabla\sigma_n\|_{L^2(\Om)}<\infty$, which, along with $c_0\leq \sigma_n\leq c_1$, yields
$\|\sigma_n\|_{H^1(\Om)}\leq c$. Since $\widetilde{\mathcal{A}}$ is closed and convex, there exist a subsequence,
relabeled as $\{\sigma_n\}_{n\geq 1}$, and some $\sigma^\ast\in\widetilde{\mathcal{A}}$ such that
\begin{equation}\label{pf:min_cont01}
   \sigma_{n}\rightharpoonup\sigma^\ast\quad\mbox{weakly in}~H^1(\Om),\quad \sigma_n\to\sigma^\ast\quad\mbox{in}~L^1(\Om),\quad\sigma_n\to\sigma^\ast\quad\mbox{a.e. in}~\Om.
\end{equation}
Since $W(s)\in C^2[c_0,c_1]$, $\{W(\sigma_n)\}_{n\geq1}$ is uniformly bounded in $L^\infty(\Omega)$ and converges
to $W(\sigma^\ast)$ almost everywhere in $\Om$. By Lebesgue's dominated convergence theorem \cite[p. 28,
Theorem 1.19]{Evans:2015}, $\int_\Om W(\sigma_n)\dx\to\int_\Om W(\sigma^\ast)\dx.$ By Lemma \ref{lem:fm_cont}
and the weak lower semi-continuity of the $H^1(\Omega)$-seminorm, we obtain
\[
   \J_\varepsilon(\sigma^\ast)\leq \liminf_{n\to\infty}\J_{\varepsilon}(\sigma_n)\leq \lim_{n\to\infty}\J_{\varepsilon}(\sigma_n)=m_\eps.
\]
Thus $\sigma^\ast$ is a global minimizer of the functional $\mathcal{J}_\varepsilon $.
\end{proof}

To obtain the necessary optimality system of \eqref{eqn:tikh-MM}, we use the standard adjoint technique. The adjoint problem
for \eqref{eqn:cemweakform} reads: find $(p,P)\in\mathbb{H}$ such that
\begin{equation}\label{eqn:cem-adj}
  a(\sigma,(p,P),(v,V)) = \langle U(\sigma)-U^\delta,V\rangle \quad \forall (v,V)\in \mathbb{H}.
\end{equation}
By straightforward computation, the G\^{a}teaux derivative $\mathcal{J}_\varepsilon'(\sigma)[\mu]$ of
the functional $\mathcal{J}_{\varepsilon}$ at $\sigma\in\widetilde{\mathcal{A}}$ in the direction $\mu\in H^1(\Om)$ is given by
\begin{equation*}
  \mathcal{J}_{\varepsilon}'(\sigma)[\mu] = \widetilde{\alpha}\left[\eps(\nabla\sigma,\nabla\mu) + \tfrac{1}{2\eps}(W'(\sigma),\mu)\right] - (\mu \nabla u(\sigma),\nabla p(\sigma)).
\end{equation*}
Then the minimizer $\sigma^*$ to problem \eqref{eqn:tikh-MM} and the respective state
$(u^*,U^*)$ and adjoint $(p^*,P^*)$ satisfy the following necessary optimality system:
\begin{equation}\label{eqn:cem-optsys}
   \left\{\begin{aligned}
     &a(\sigma^*,(u^*,U^*),(v,V)) = \langle I,V\rangle \quad \forall (v,V)\in \mathbb{H},\\
     &a(\sigma^*,(p^*,P^*),(v,V)) = \langle U^*-U^\delta,V\rangle       \quad\forall(v,V)\in\mathbb{H},\\
     &        \widetilde{\alpha}\varepsilon (\nabla\sigma^{\ast},\nabla(\mu-\sigma^{\ast})) + \tfrac{\widetilde{\alpha}}{2\varepsilon}(W'(\sigma^\ast),\mu-\sigma^\ast) - ((\mu-\sigma^{\ast})\nabla u^{\ast},\nabla p^{\ast})\geq
            0\quad\forall \mu\in\widetilde{\mathcal{A}},
    \end{aligned}\right.
\end{equation}
where the variational inequality at the last line is due to the box constraint in the admissible set $\widetilde{\mathcal{A}}$.
The optimality system \eqref{eqn:cem-optsys} forms the basis of the adaptive algorithm and its convergence analysis.

\section{Adaptive algorithm}\label{sect:alg}

Now we develop an adaptive FEM for problem \eqref{eqn:tikh-MM}. Let $\cT_0$ be a shape regular triangulation of
$\overline{\Omega}$ into simplicial elements, each intersecting at most one electrode
surface $e_l$, and $\mathbb{T}$ be the set of all possible conforming triangulations of $\overline{\Omega}$
obtained from $\cT_0$ by the successive use of bisection. Then $\mathbb{T}$ is uniformly shape regular, i.e., the
shape-regularity of any mesh $\mathcal{T}\in\mathbb{T}$ is bounded by a constant depending only on $\cT_0$ \cite{NSV:2009,
Traxler:1997}. Over any $\cT\in\mathbb{T}$, we define a continuous piecewise linear finite element space
\begin{equation*}
   V_\cT = \left\{v\in C(\overline{\Omega}): v|_T\in P_1(T)\ \forall T\in\cT\right\},
\end{equation*}
where $P_1(T)$ consists of all linear functions on $T$. The space $V_\cT$ is used for approximating the state $u$ and adjoint
$p$, and the discrete admissible set $\widetilde{\mathcal{A}}_\mathcal{T}$ for the conductivity is given by
\begin{equation*}
   \widetilde{\mathcal{A}}_\cT:= V_\cT \cap {\widetilde{\mathcal{A}}}.
\end{equation*}

Given $\sigma_\cT\in \Atilde_\cT$, the discrete analogue of problem \eqref{eqn:cemweakform} is to find $(u_\cT,U_\cT)
\in \mathbb{H}_\cT \equiv V_\cT \otimes \mathbb{R}^L_\diamond$ such that
\begin{equation}\label{eqn:dispoly}
  a(\sigma_\cT,(u_\cT,U_\cT),(v_\cT,V)) = \langle I, V\rangle  \quad \forall (v_{\cT},V)\in \mathbb{H}_\cT.
\end{equation}
Then we approximate problem \eqref{eqn:tikh-MM} by minimizing the following functional over $\Atilde_\cT$:
\begin{align}\label{eqn:discopt}
  J_{\eps,\cT}(\sigma_{\cT})  = \tfrac{1}{2}\|U_{\cT}(\sigma_{\cT})-U^\delta\|^2 + \tfrac{\widetilde{\alpha}}{2}\mathcal{F}_\varepsilon(\sigma_\cT).
\end{align}

Then, similar to Theorem \ref{thm:tikh-MM}, there exists at least one minimizer $\sigma_\cT^*$ to \eqref{eqn:discopt}, and the minimizer $\sigma_{\cT}^{\ast}$
and the related state $(u^*_\cT,U^*_\cT)\in \mathbb{H}_\cT$ and adjoint  $(p^*_\cT,P^*_\cT)\in\mathbb{H}_{\cT}$ satisfy
\begin{equation}\label{eqn:cem-discoptsys}
    \left\{\begin{aligned}
      &a(\sigma_\cT^*,(u^*_\cT,U_\cT^*),(v,V)) = \langle I,V\rangle \quad\forall(v,V)\in \mathbb{H}_{\cT},\\
      &a(\sigma_\cT^*,(p^*_\cT,P^*_\cT),(v,V)) = \langle U^*_\cT-U^\delta,V\rangle\quad \forall(v,V)\in\mathbb{H}_{\cT},\\
      &\widetilde{\alpha}\varepsilon(\nabla\sigma_{\cT}^{\ast},\nabla(\mu-\sigma_{\cT}^{\ast})) + \frac{\widetilde{\alpha}}{2\varepsilon}(W'(\sigma_\cT^\ast),\mu-\sigma_\cT^\ast) -((\mu-\sigma_{\cT}^{\ast})\nabla u_{\cT}^{\ast},\nabla p_{\cT}^{\ast})\geq 0\quad\forall\mu\in\Atilde_{\cT}.
    \end{aligned}\right.
\end{equation}
Further, $(u_{\cT}^\ast,U_{\cT}^\ast)$ and $(p_{\cT}^\ast,P_{\cT}^\ast)$ depend continuously on the problem data, i.e.,
\begin{equation}\label{stab-discpolyadj}
    \|(u_{\cT}^\ast,U_{\cT}^\ast)\|_{\mathbb{H}}+\|(p_{\cT}^\ast,P_{\cT}^\ast)\|_{\mathbb{H}}\leq c(\|I\|+\|U^\delta\|),
\end{equation}
{where the constant $c$ can be made independent of $\alpha$ and $\eps$.}

To describe the error estimators, we first recall some useful notation.
The collection of all faces (respectively all interior faces) in $\cT\in\mathbb{T}$ is denoted by $\mathcal{F}_{\cT}$ (respectively
$\mathcal{F}_{\cT}^i$) and its restriction to the electrode $\bar{e}_{l}$ and $\partial\Omega\backslash\cup_{l=1}^Le_l$ by $\mathcal{F}_{
\cT}^{l}$ and $\mathcal{F}_{\cT}^{c}$, respectively. A face/edge $F$ has a fixed normal unit vector $\boldsymbol{n}_{F}$ in $\overline{\Omega}$ with
$\boldsymbol{n}_{F}=\boldsymbol{n}$ on $\partial\Omega$.  The diameter of any $T\in\cT$ and $F\in\mathcal{F}_{\cT}$
is denoted by $h_{T}:=|T|^{1/d}$ and $h_{F}:=|F|^{1/(d-1)}$, respectively.
For the solution $(\sigma^{\ast}_{\cT},(u _{\cT}^{\ast},U_{\cT}^{\ast}),
(p_{\cT}^{\ast},P_{\cT}^{\ast}))$ to problem \eqref{eqn:cem-discoptsys}, we define
two element residuals for each element $T\in\cT$ and two face residuals for each face
$F\in\mathcal{F}_{\cT}$ by
\begin{equation*}
 \begin{aligned}
    R_{T,1}(\sigma_{\cT}^{\ast},u^{\ast}_{\cT}) & = \nabla\cdot(\sigma_{\cT}^{\ast}\nabla u^{\ast}_{\cT}),\\
    R_{T,2}(\sigma_{\cT}^\ast,u^{\ast}_{\cT},p^{\ast}_{\cT})& =\tfrac{\widetilde{\alpha}}{2\varepsilon}W'(\sigma_\cT^\ast)-\nabla u^{\ast}_{\cT}\cdot\nabla p^{\ast}_{\cT},\\
    J_{F,1}(\sigma_{\cT}^{\ast},u^{\ast}_{\cT},U_{\cT}^{\ast}) &=
    \left\{\begin{array}{lll}
                        [\sigma_{\cT}^{\ast}\nabla u_{\cT}^{\ast}\cdot\boldsymbol{n}_{F}]\quad&
                        \mbox{for} ~~F\in\mathcal{F}_{\cT}^i,\\ [1ex]
                        \sigma_{\cT}^{\ast}\nabla u_{\cT}^{\ast}\cdot\boldsymbol{n}+(u_{\cT}^{\ast}-U_{\cT,l}^{\ast})/z_{l}\quad&
                        \mbox{for} ~~ F\in\mathcal{F}_{\cT}^l,\\ [1ex]
                        \sigma_{\cT}^{\ast}\nabla u_{\cT}^{\ast}\cdot\boldsymbol{n}\quad&
                        \mbox{for} ~~F\in\mathcal{F}_{\cT}^{c},
    \end{array}\right.\\
    J_{F,2}(\sigma^{\ast}_{\cT}) &= \left\{\begin{array}{lll}
        \widetilde{\alpha}\varepsilon[\nabla\sigma_{\cT}^{\ast}\cdot\boldsymbol{n}_{F}]\quad&
                        \mbox{for} ~~F\in\mathcal{F}_{\cT}^i,\\ [1ex]
                        \widetilde{\alpha}\varepsilon\nabla\sigma_{\cT}^{\ast}\cdot\boldsymbol{n}\quad&
                        \mbox{for} ~~F\in \mathcal{F}_\cT^l \cup \mathcal{F}_\cT^c,
        \end{array}\right.
\end{aligned}
\end{equation*}
where $[\cdot]$ denotes the jump across interior face $F$. Then for any element
$T\in \cT$, we define the following three error estimators
\begin{align*}
    \eta_{\cT,1}^{2}(\sigma_{\cT}^{\ast},u_{\cT}^{\ast},U_{\cT}^{\ast},T)
    & :=h_{T}^{2}\|R_{T,1}(\sigma^{\ast}_{\cT},u^{\ast}_{\cT})\|_{L^{2}(T)}^{2}
    +\sum_{F\subset\partial T}h_{F}\|J_{F,1}(\sigma^{\ast}_{\cT},u^{\ast}_{\cT},U^{\ast}_{\cT})\|_{L^{2}(F)}^{2},\\
    \eta_{\cT,2}^{2}(\sigma_{\cT}^{\ast},p_{\cT}^{\ast},P_{\cT}^{\ast},T)
    &:=h_{T}^{2}\|R_{T,1}(\sigma^{\ast}_{\cT},p^{\ast}_{\cT})\|_{L^{2}(T)}^{2}
    +\sum_{F\subset\partial T}h_{F}\|J_{F,1}(\sigma^{\ast}_{\cT},p^{\ast}_{\cT},P^{\ast}_{\cT})\|_{L^{2}(F)}^{2},\\
    \eta_{\cT,3}^{q}(\sigma_{\cT}^{\ast},u_{\cT}^{\ast},p_{\cT}^{\ast},T)
    &:=h_{T}^{q}\|R_{T,2}(\sigma_{\cT}^\ast,u_{\cT}^{\ast},p_{\cT}^{\ast})\|^{q}_{L^{q}(T)}+\sum_{F\subset\partial T}h_{F}\|J_{F,2}(\sigma_{\cT}^{\ast})\|^{q}_{L^{q}(F)}
\end{align*}
with $q=d/(d-1)$. The estimator $\eta_{\cT,1}(\sigma_{\cT}^{\ast},u_{\cT}^{\ast},U_{\cT}^{\ast},T)$ is identical
with the standard residual error indicator for the direct problem: find $(\tilde u,\tilde U)\in \mathbb{H}$ such that
\begin{equation*}
  a(\sigma_\cT^*,(\tilde u,\tilde U),(v,V)) = \langle I,V\rangle, \quad \forall (v,V)\in \mathbb{H}.
\end{equation*}
It differs from the direct problem in \eqref{eqn:cem-optsys} by replacing the conductivity $\sigma^*$ with $\sigma^*_\cT$
instead, and is a perturbation of the latter case. The perturbation is vanishingly small in the event of the conjectured
(subsequential) convergence $\sigma^*_\cT\to \sigma^*$. The estimator $\eta_{\cT,2}(\sigma_{\cT}^{\ast},p_{\cT}^{\ast},
P_{\cT}^{\ast},T)$ admits a similar interpretation. These two estimators are essentially identical with that for the $H^1
(\Omega)$ penalty in \cite{JinXuZou:2016}, and we refer to \cite[Section 3.3]{JinXuZou:2016} for a detailed
heuristic derivation. The estimator $\eta_{\cT,3}(\sigma_{\cT}^{\ast},u_{\cT}^{\ast},p_{\cT}^{\ast},T)$ is related to
the variational inequality in the necessary optimality condition \eqref{eqn:cem-optsys}, and roughly provides a quantitative
measure how well it is satisfied. The estimator (including the exponent $q$) is motivated by the convergence analysis; see
the proof of Theorem \ref{thm:gat_mc} and Remark \ref{rmk:gat_mc} below. It represents the main new ingredient for problem
\eqref{eqn:tikh-MM}, and differs from that for the $H^1(\Omega)$ penalty in \cite{JinXuZou:2016}.
\begin{remark}
The estimator $\eta_{k,3}$ improves that in \cite{JinXuZou:2016}, i.e.,
\begin{equation*}
   \eta_{\cT,3}^2(\sigma_{\cT}^{\ast},u_{\cT}^{\ast},p_{\cT}^{\ast},T)
    :=h_{T}^4\|R_{T,2}(\sigma_{\cT}^\ast,u_{\cT}^{\ast},p_{\cT}^{\ast})\|^2_{L^2(T)}+\sum_{F\subset\partial T}h_{F}^2\|J_{F,2}(\sigma_{\cT}^{\ast})\|^2_{L^2(F)},
\end{equation*}
in terms of the exponents on $h_T$ and $h_F$. This improvement is achieved by a novel
constraint preserving interpolation operator defined in \eqref{eqn:cn_int_def} below.
\end{remark}

Now we can formulate an adaptive algorithm for \eqref{eqn:tikh-MM}; see Algorithm \ref{alg_afem_eit}.
Below we indicate the dependence on the mesh $\cT_k$ by the subscript $k$, e.g.,
$\mathcal{J}_{\varepsilon,k}$ for $\mathcal{J}_{\varepsilon,\cT_k}$.
\begin{algorithm}
\caption{AFEM for EIT with a piecewise constant conductivity.}\label{alg_afem_eit}
\begin{algorithmic}[1]
  \STATE Specify an initial mesh $\cT_{0}$, and set the maximum number $K$ of refinements.
  \FOR {$k=0:K-1$}
  \STATE {(\texttt{SOLVE})} Solve problem \eqref{eqn:dispoly}-\eqref{eqn:discopt} over $\cT_{k}$ for
  $(\sigma_k^{\ast},(u_{k}^{\ast},U_{k}^{\ast}))\in\Atilde_{k}\times\mathbb{H}_{k}$ and \eqref{eqn:cem-discoptsys}
  for $(p_k^\ast,P_k^\ast)\in \mathbb{H}_k$.
  \STATE {(\texttt{ESTIMATE})} {Compute error indicators $\eta_{k,1}^{2}(\sigma_{k}^{\ast},u_{k}^{\ast},U_{k}^{\ast})$,
    $\eta_{k,2}^{2}(\sigma_{k}^{\ast},p_{k}^{\ast},P_{k}^{\ast})$ and $\eta_{k,3}^{q}(\sigma_{k}^{\ast},u_{k}^{\ast},p_{k}^{\ast})$.}
  \STATE{(\texttt{MARK})} Mark three subsets $\mathcal{M}_{k}^i\subseteq\cT_{k}$ ($i=1,2,3$) such that each $\mathcal{M}_k^i$ contains at least one element $\widetilde{T}_k^i\in\cT_{k}$ ($i=1,2,3$) with the largest error indicator:
    \begin{equation}\label{eqn:marking}
        \eta_{k,i}(\widetilde{T}_k^i)=\max_{T\in\cT_{k}}\eta_{k,i}.
    \end{equation}
    Then $\mathcal{M}_{k}:=\mathcal{M}_{k}^1\cup\mathcal{M}_{k}^2\cup\mathcal{M}_{k}^3$.
   \STATE {(\texttt{REFINE})} Refine each element $T$ in $\mathcal{M}_{k}$ by bisection to get $\cT_{k+1}$.
   \STATE Check the stopping criterion.
   \ENDFOR
   \STATE Output $(\sigma_k^*,(u_k^*,U_k^*),(p_k^*,P_k^*))$.
\end{algorithmic}
\end{algorithm}

The \texttt{MARK} module selects a collection of elements in the mesh $\mathcal{T}_k$. The condition
\eqref{eqn:marking} covers several commonly used marking strategies, e.g., maximum, equidistribution,
modified equidistribution, and D\"{o}rfler's strategy \cite[pp. 962]{Siebert:2011}. Compared with a
collective marking in AFEM in \cite{JinXuZou:2016}, Algorithm \ref{alg_afem_eit} employs a separate
marking to select more elements for refinement in each loop, which leads to fewer iterations of the
adaptive process. The error estimators may also be used for coarsening, which is relevant if the
recovered inclusions change dramatically during the iteration. However, the convergence analysis below
does not carry over to coarsening, and it will not be further explored.

Last, we give the main theoretical result: for each fixed $\eps>0$, the sequence of discrete
solutions $\{\sigma_k^\ast, (u_k^\ast, U_k^\ast), (p_k^\ast, P_k^\ast)\}_{k\geq0}$ generated by Algorithm
\ref{alg_afem_eit} contains a subsequence converging in $H^1(\Om)\times\mathbb{H}\times\mathbb{H}$ to a
solution of system \eqref{eqn:cem-optsys}. The proof is lengthy and technical,
and thus deferred to Section \ref{sect:conv}.
\begin{theorem}\label{thm:conv_alg}
The sequence of discrete solutions $\{\sigma_{k}^\ast,(u_{k}^\ast,U_{k}^\ast),(p_{k}^\ast,P_{k}^\ast)\}_{k\geq0}$
by Algorithm \ref{alg_afem_eit} contains a subsequence $\{\sigma_{k_j}^\ast,(u_{k_j}^\ast,U_{k_j}^\ast),(p_{k_j}^\ast,
P_{k_j}^\ast)\}_{j\geq0}$ convergent to a solution $(\sigma^\ast,(u^\ast,U^\ast),(p^\ast,P^\ast))$ of system
\eqref{eqn:cem-optsys}:
\[
   \|\sigma^\ast_{k_j}-\sigma^\ast\|_{H^1(\Omega)},~\|(u^\ast_{k_j}-u^\ast,U^\ast_{k_j}-U^\ast)\|_{\mathbb{H}},~
   \|(p^\ast_{k_j}-p^\ast,P^\ast_{k_j}-P^\ast)\|_{\mathbb{H}}\rightarrow 0\quad\mbox{as}~j\rightarrow\infty.
\]
\end{theorem}

\section{Numerical experiments and discussions}\label{sec:numer}

Now we present numerical results to illustrate Algorithm \ref{alg_afem_eit} on a square domain
$\Omega=(-1,1)^2$. There are sixteen electrodes $\{e_l\}_{l=1}^L$ (with $L=16$) evenly distributed
along $\partial\Omega$, each of length $1/4$. The contact impedances $\{z_l\}_{l=1}^L$ are all set
to unit. We take ten sinusoidal input currents, and for each voltage $U(\sigma^\dag)\in\mathbb{R}^L_\diamond$, generate the noisy
data $U^\delta$ by
\begin{equation}\label{eqn:noisydata}
  U^\delta_l = U_l(\sigma^\dag) + \epsilon \max_l|U_l(\sigma^\dag)|\xi_l,\ \ l=1,\ldots, L,
\end{equation}
where $\epsilon$ is the (relative) noise level, and $\{\xi_l\}_{l=1}^L$ follow the standard normal
distribution. Note that $\epsilon=\text{1e-2}$ refers to a relatively high noise level for EIT.
The exact data $U(\sigma^\dag)$ is computed using a much finer uniform mesh, to avoid
the most obvious form of ``inverse crime''.

In the experiments, we fix $K$ (the number of refinements) at $15$, $q$ (exponent in
$\eta_{k,3}^q$) at $2$, and $\varepsilon$ (the functional $\mathcal{F}_\varepsilon$)
at $\text{1e-2}$. The marking strategy \eqref{eqn:marking} in the module \texttt{MARK}
selects a minimal refinement set $\mathcal{M}_k{:=\cup_{i=1}^3\mathcal{M}_{k}^i}\subseteq\cT_k$ such that
\begin{equation*}
   \begin{aligned}
   \eta_{k,1}^2(\sigma_k^{\ast},u_k^{\ast},U_k^{\ast},\mathcal{M}_{k}^1)
        &\geq\theta
        \eta_{k,1}^2(\sigma_k^{\ast},u_k^{\ast},U_k^{\ast}),\quad
        \eta_{k,2}^2(\sigma_k^{\ast},p_k^{\ast},P_k^{\ast},\mathcal{M}_{k}^2)\geq\theta
        \eta_{k,2}^2(\sigma_k^{\ast},p_k^{\ast},P_k^{\ast}),\\
    &\eta_{k,3}^2(\sigma_k^{\ast},u_k^{\ast},p_k^{\ast},\mathcal{M}_{k}^3)\geq\theta
        \eta_{k,3}^2(\sigma_k^{\ast},u_k^{\ast},p_k^{\ast}),
        \end{aligned}
\end{equation*}
with a threshold $\theta=0.7$. The refinement is performed with one popular refinement strategy, i.e.,
newest vertex bisection \cite{Mitchell:1989}. Specifically, it connects the midpoint $x_T$, as a newest
vertex, of a reference edge $F$ of an element $T\in\cT_k$ to the opposite node of $F$, and employs
two edges opposite to the midpoint $x_T$ as reference edges of the two newly created triangles in
$\cT_{k+1}$. Problem \eqref{eqn:dispoly}-\eqref{eqn:discopt} is solved by a Newton type method; see
Appendix \ref{app:newton} for the detail. The conductivity on $\cT_0$ is initialized to $\sigma_0=c_0$,
and then for $k=1,2,\ldots$, $\sigma_{k-1}^*$ (defined on $\cT_{k-1}$) is interpolated to $\cT_k$ to warm
start the optimization. The regularization parameter $\widetilde\alpha$ in \eqref{eqn:tikh-MM} is determined
in a trial-and-error manner. All computations are performed using \texttt{MATLAB} 2018a on a personal
laptop with 8.00 GB RAM and 2.5 GHz CPU.

\begin{figure}[hbt!]
  \centering\setlength{\tabcolsep}{0em}
  \begin{tabular}{ccccc}
      \includegraphics[trim = .5cm 0cm 1cm 0cm, clip=true,width=0.2\textwidth]{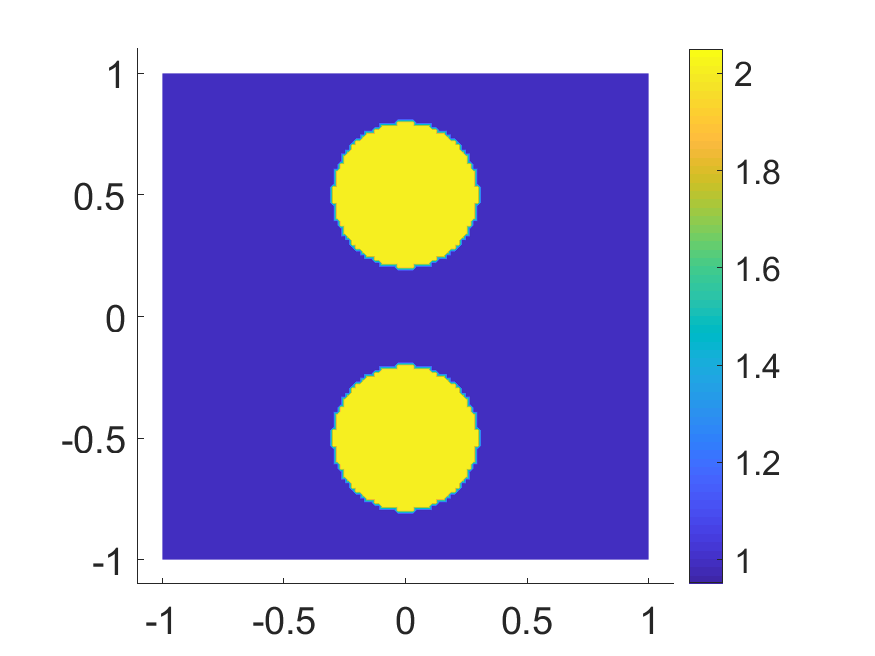}
    & \includegraphics[trim = .5cm 0cm 1cm 0cm, clip=true,width=.2\textwidth]{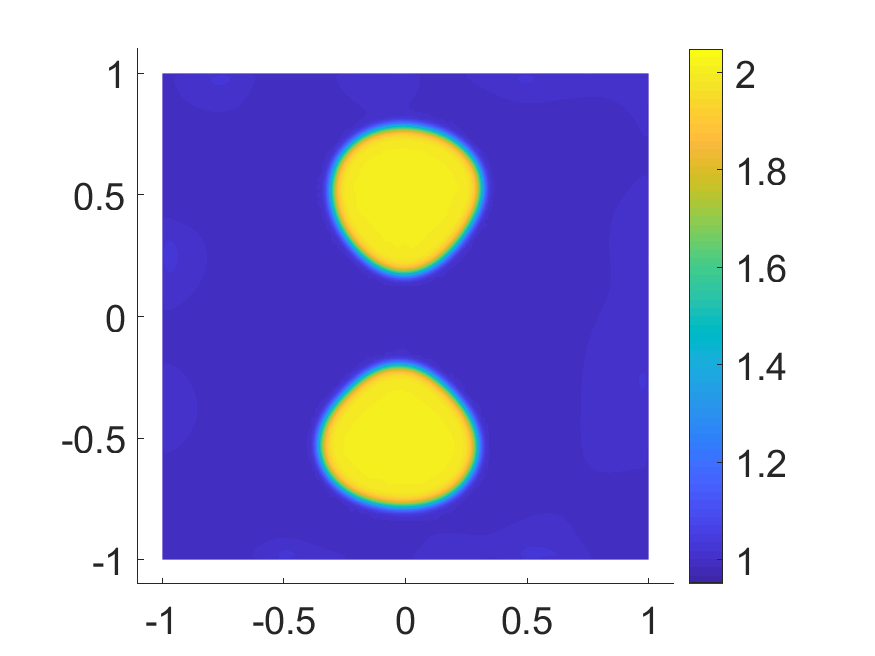}
    & \includegraphics[trim = .5cm 0cm 1cm 0cm, clip=true,width=.2\textwidth]{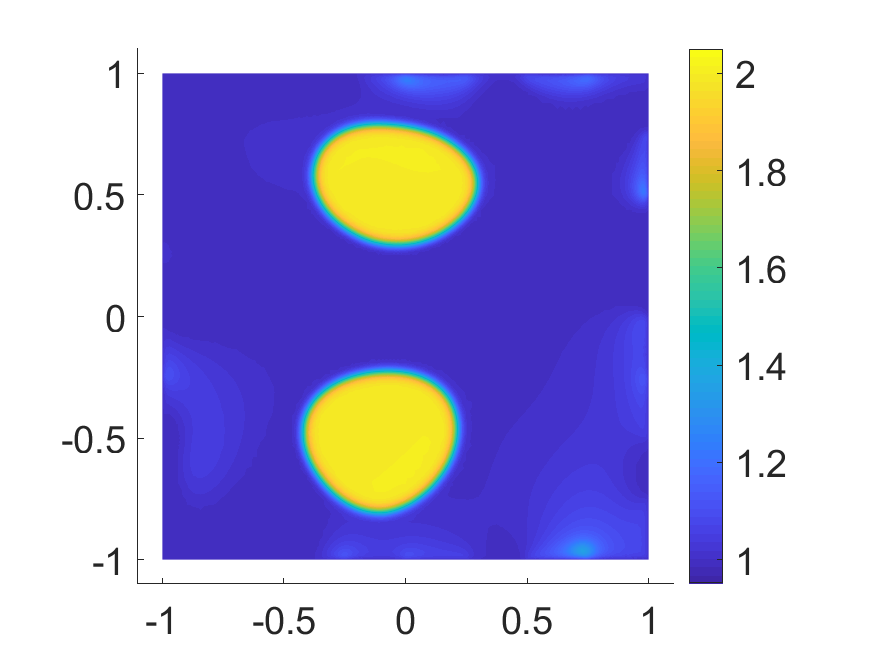}
    & \includegraphics[trim = .5cm 0cm 1cm 0cm, clip=true,width=.2\textwidth]{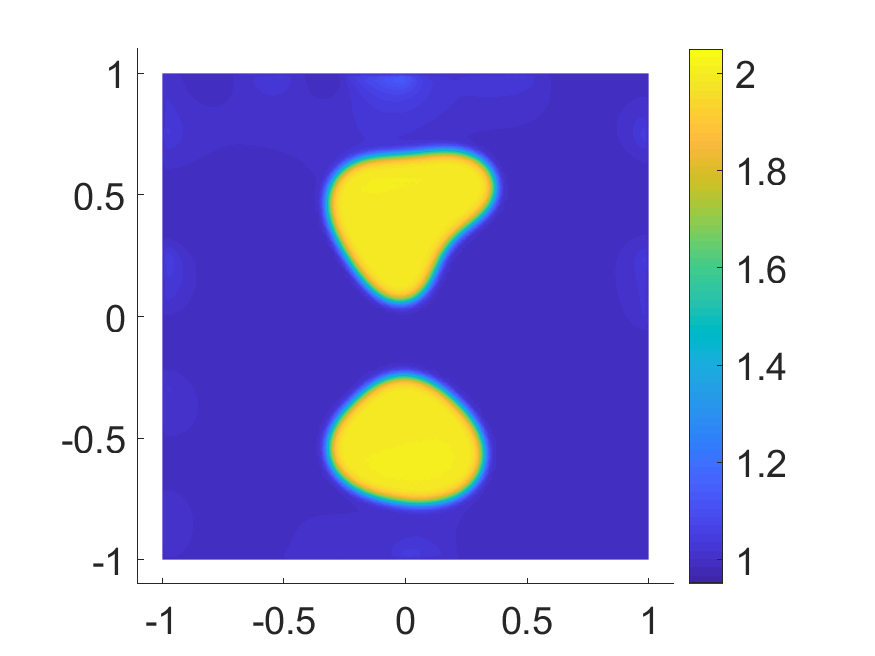}
    & \includegraphics[trim = .5cm 0cm 1cm 0cm, clip=true,width=.2\textwidth]{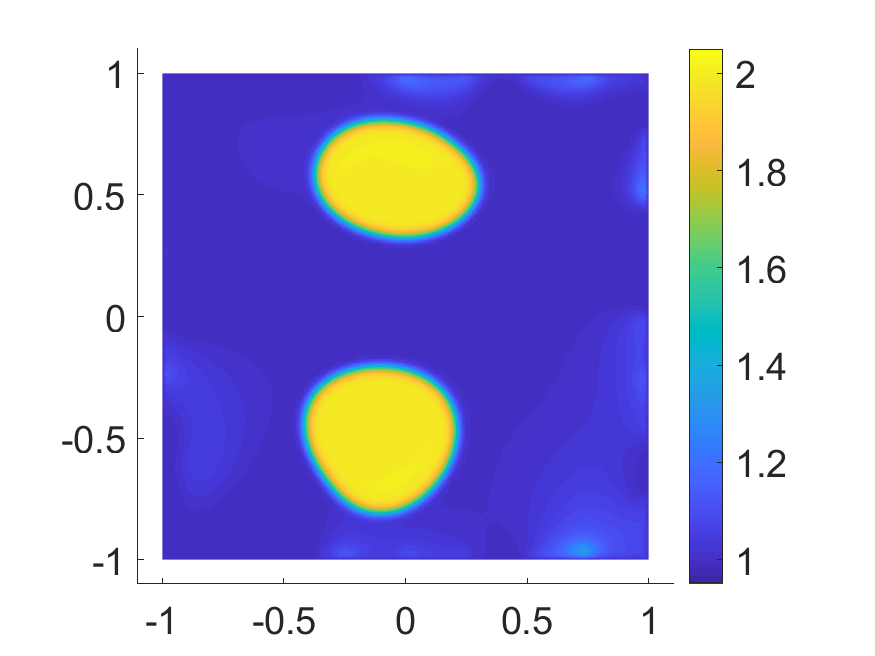}\\
    (a) true conductivity & (b) adaptive & (c) adaptive & (d) uniform & (e) uniform
  \end{tabular}
  \caption{The final recoveries by the adaptive and uniform refinements for Example
  \ref{exam2}(i). The results in (b) and (d) are for $\epsilon=\text{1e-3}$ and
  $\tilde \alpha = \text{2e-2}$, and  (c) and (e) for $\epsilon=\text{1e-2}$
  and $\tilde \alpha=\text{3e-2}$. d.o.f. in (b), (c), (d) and (e) are $15830$, $18770$, $16641$
  and 16641, respectively.}\label{fig:exam2i-recon}
\end{figure}

The first set of examples are concerned with two inclusions.
\begin{example}\label{exam2}
The background conductivity $\sigma_0(x)=1$.
\begin{itemize}
  \item[$\rm(i)$] The true conductivity $\sigma^\dag$ is given by $ \sigma_0(x)+\chi_{B_1}(x) +
  \chi_{B_2}(x)$, with $B_1$ and $B_2$ denote two circles centered at $(0,0.5)$ and $(0,-0.5)$,
  respectively, both with a radius 0.3.
  \item[$\rm(ii)$] The true conductivity $\sigma^\dag$ is given by $ \sigma_0(x)+1 + 1.2e^{-\frac{25(x_1^2+(x_2-0.5)^2)}{2}} + 1.2e^{-\frac{25(x_1^2+(x_2+0.5)^2)}{2}}$, i.e., two Gaussian bumps centered at $(0,0.5)$ and $(0,-0.5)$.
  \item[$\rm(iii)$] The true conductivity $\sigma^\dag$ is given by $ \sigma_0(x)+5\chi_{B_1}(x) +
  5\chi_{B_2}(x)$, with $B_1$ and $B_2$ denote two circles centered at $(0,0.5)$ and $(0,-0.5)$,
  respectively, both with a radius 0.3.
\end{itemize}
\end{example}

The numerical results for Example \ref{exam2}(i) with $\epsilon=\text{1e-3}$ and $\epsilon=\text{1e-2}$ are
shown in Figs. \ref{fig:exam2i-recon}--\ref{fig:exam2-efficiency}, where d.o.f. denotes the degree of
freedom of the mesh. It is observed from Fig. \ref{fig:exam2i-recon} that with both uniform and
adaptive refinements, the final recoveries have comparable accuracy and capture well the inclusion locations.

\begin{figure}[hbt!]
 \centering
 \setlength{\tabcolsep}{0pt}
 \begin{tabular}{cccccccc}
 \includegraphics[trim = {2.5cm 1.5cm 2.5cm 1.2cm}, clip, width=.2\textwidth]{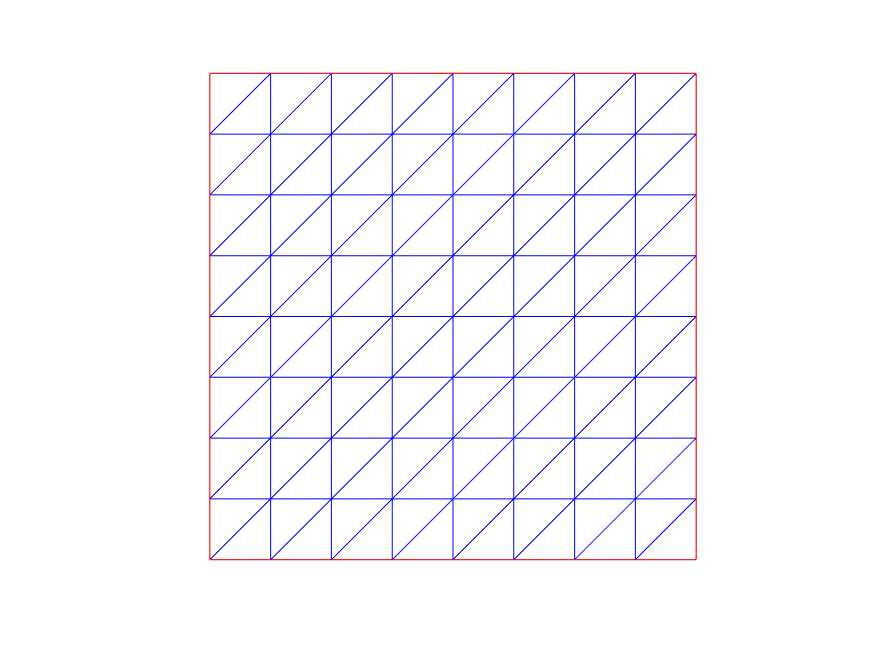}&
 \includegraphics[trim = {2.5cm 1.5cm 2.5cm 1.2cm}, clip, width=.2\textwidth]{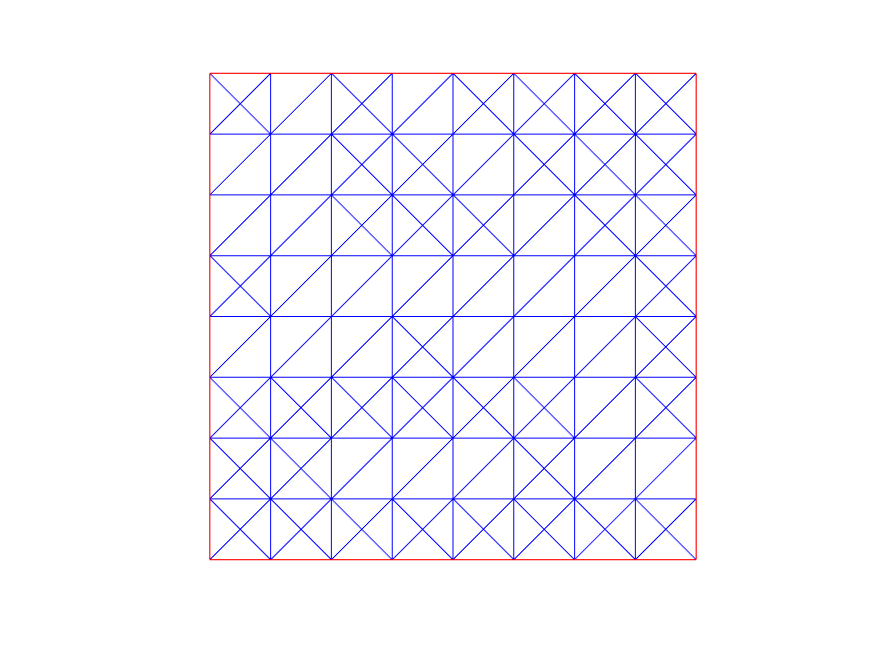}&
 \includegraphics[trim = {2.5cm 1.5cm 2.5cm 1.2cm}, clip, width=.2\textwidth]{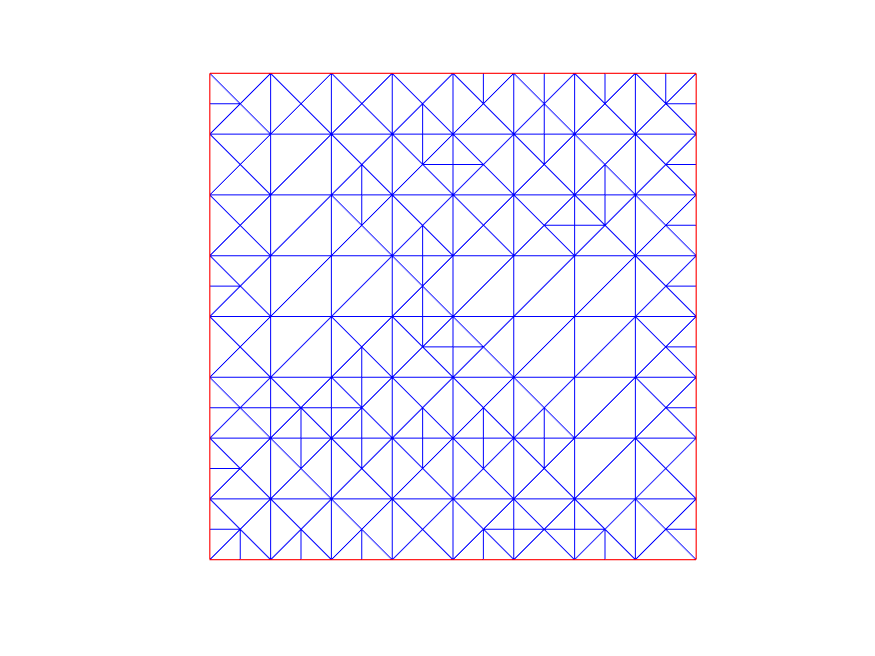}&
 \includegraphics[trim = {2.5cm 1.5cm 2.5cm 1.2cm}, clip, width=.2\textwidth]{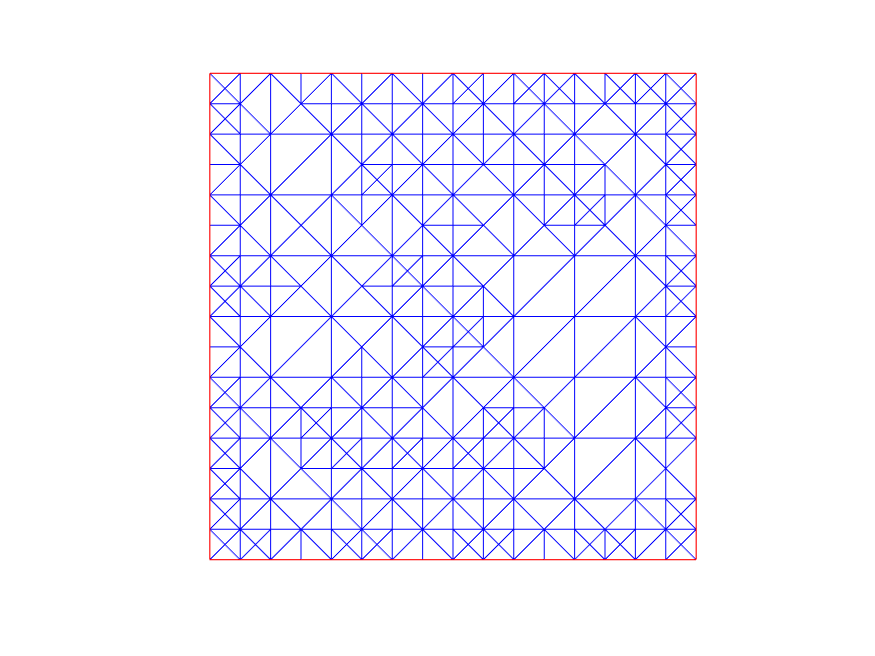}&
 \includegraphics[trim = {2.5cm 1.5cm 2.5cm 1.2cm}, clip, width=.2\textwidth]{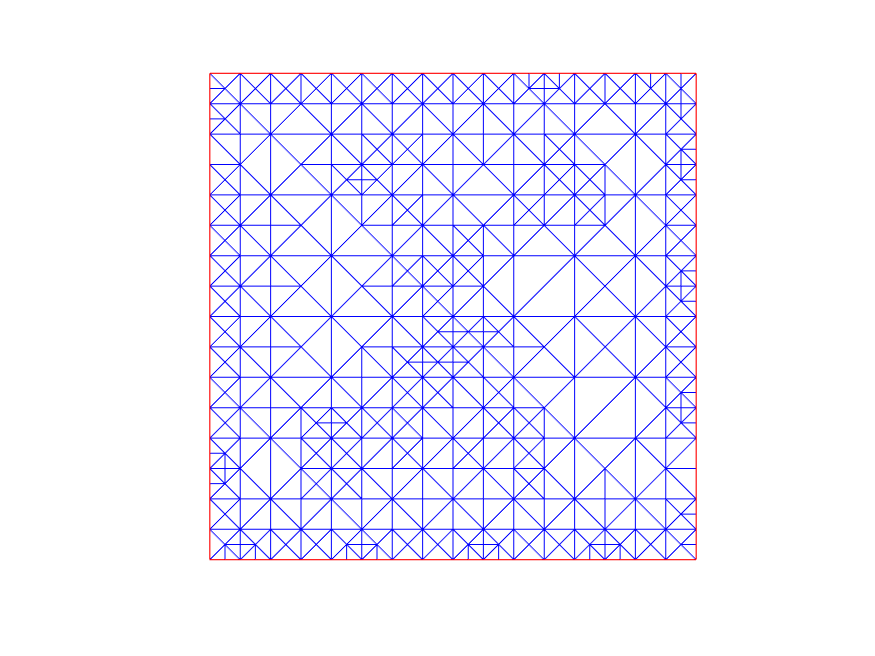}\\
 \includegraphics[trim = 1cm 0.5cm 0.5cm 0cm, width=.2\textwidth]{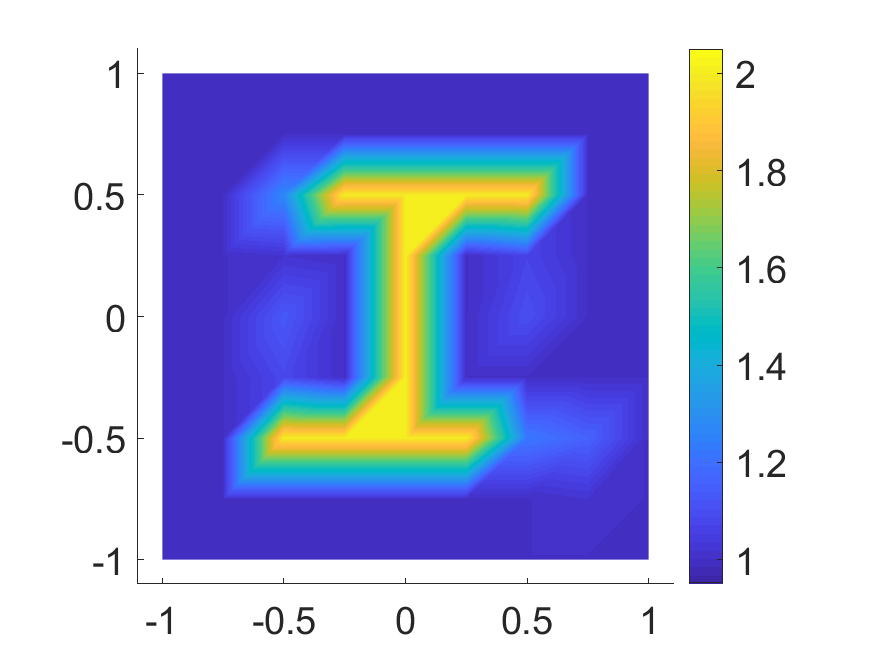}&
 \includegraphics[trim = 1cm 0.5cm 0.5cm 0cm, width=.2\textwidth]{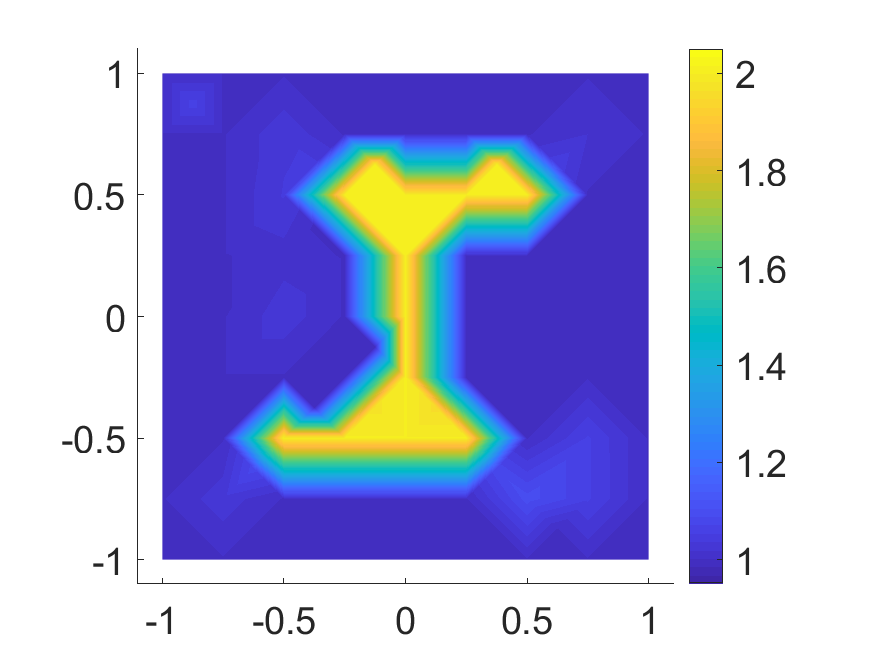}&
 \includegraphics[trim = 1cm 0.5cm 0.5cm 0cm, width=.2\textwidth]{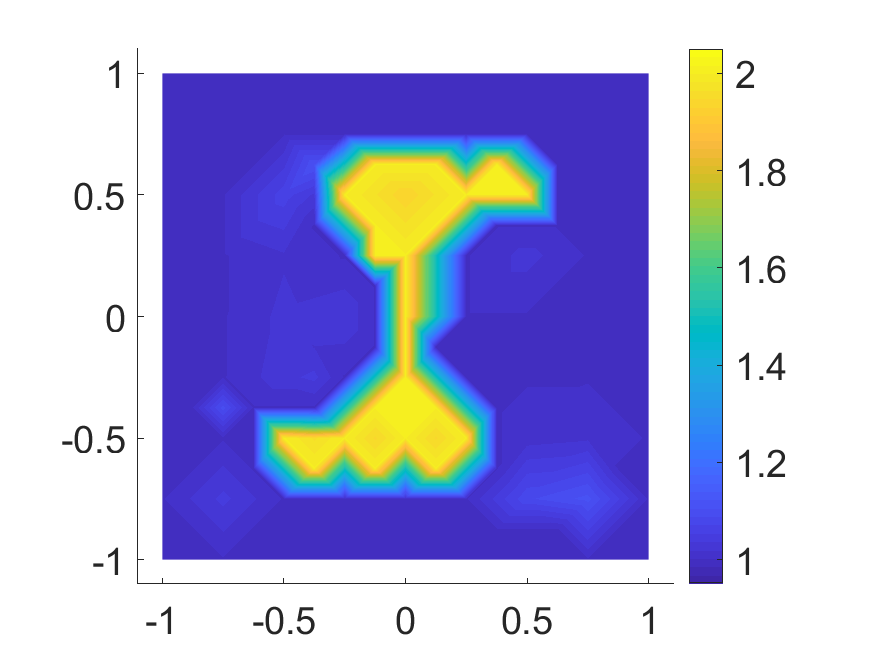}&
 \includegraphics[trim = 1cm 0.5cm 0.5cm 0cm, width=.2\textwidth]{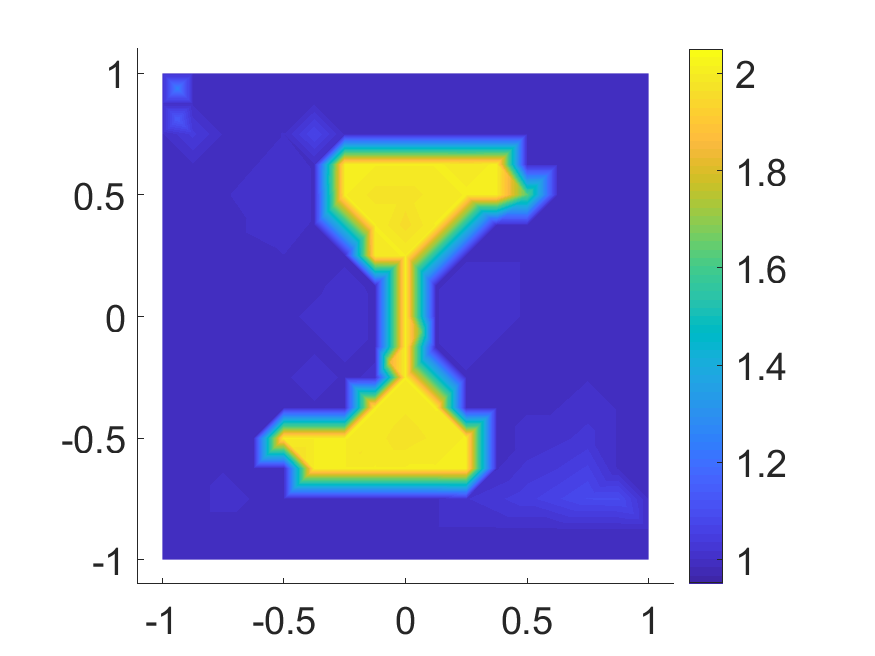}&
 \includegraphics[trim = 1cm 0.5cm 0.5cm 0cm, width=.2\textwidth]{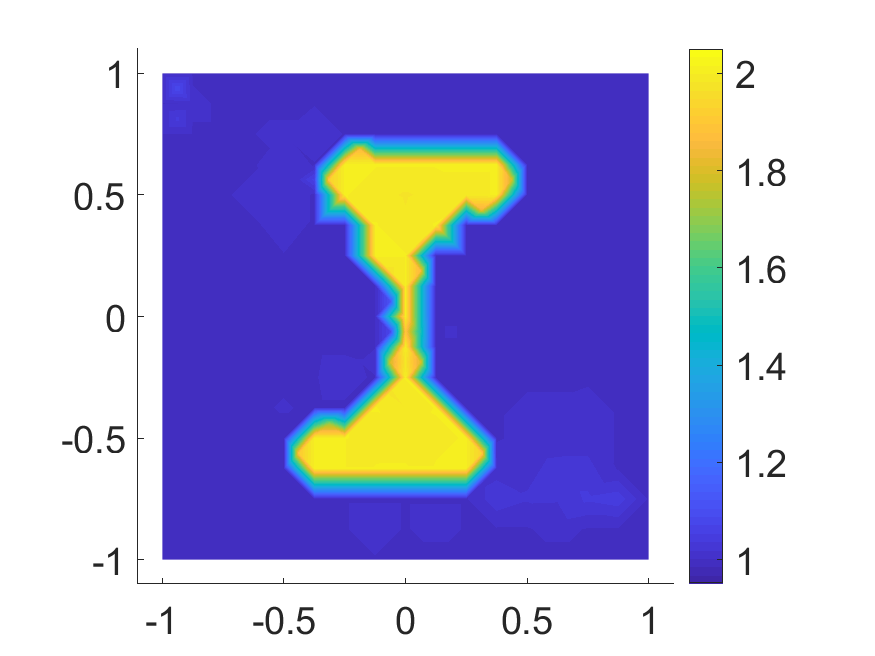}\\
 81    &     119   &      173  &       277   &      390   \\
 \includegraphics[trim = {2.5cm 1.5cm 2.5cm 1.2cm}, clip, width=.2\textwidth]{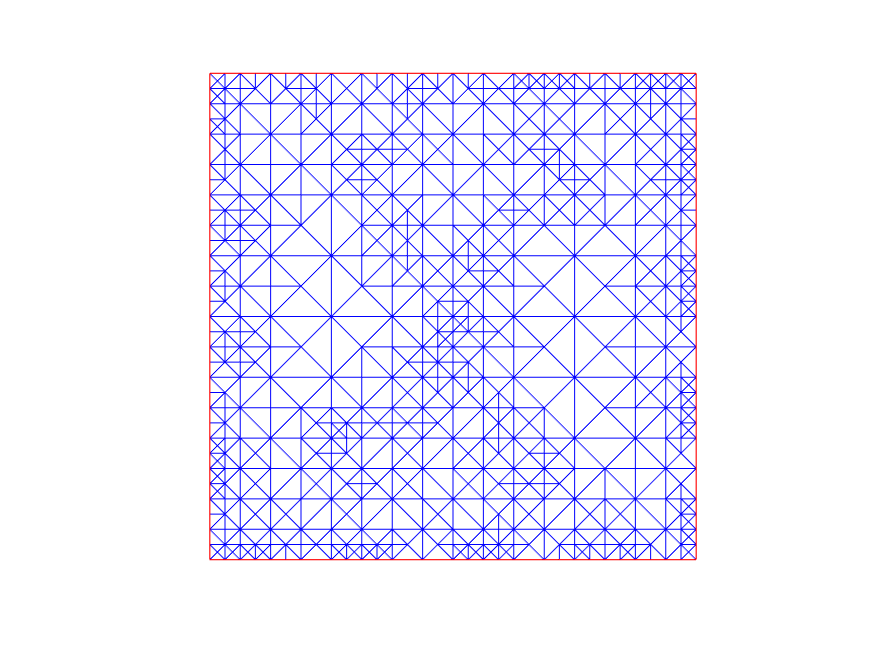}&
 \includegraphics[trim = {2.5cm 1.5cm 2.5cm 1.2cm}, clip, width=.2\textwidth]{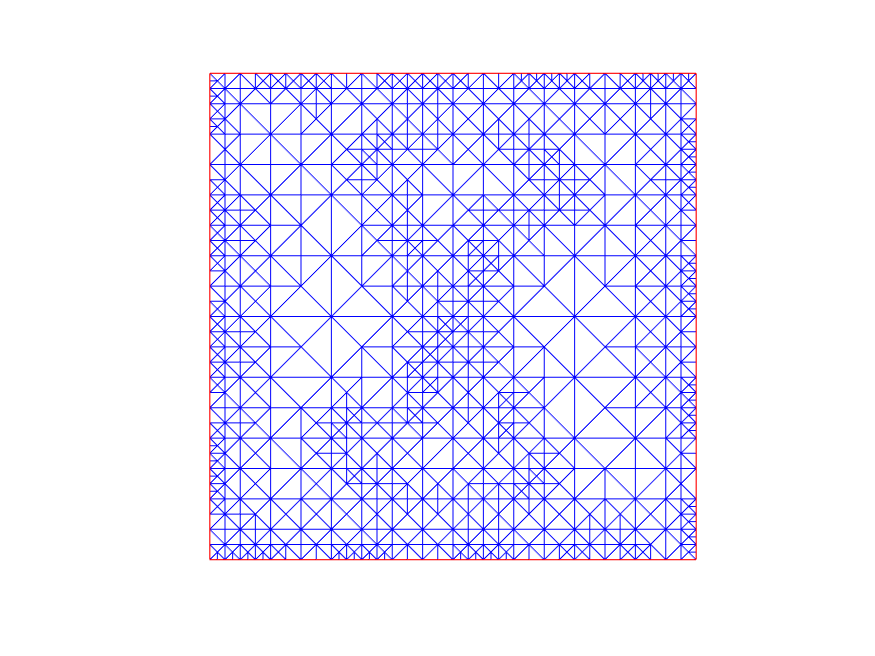}&
 \includegraphics[trim = {2.5cm 1.5cm 2.5cm 1.2cm}, clip, width=.2\textwidth]{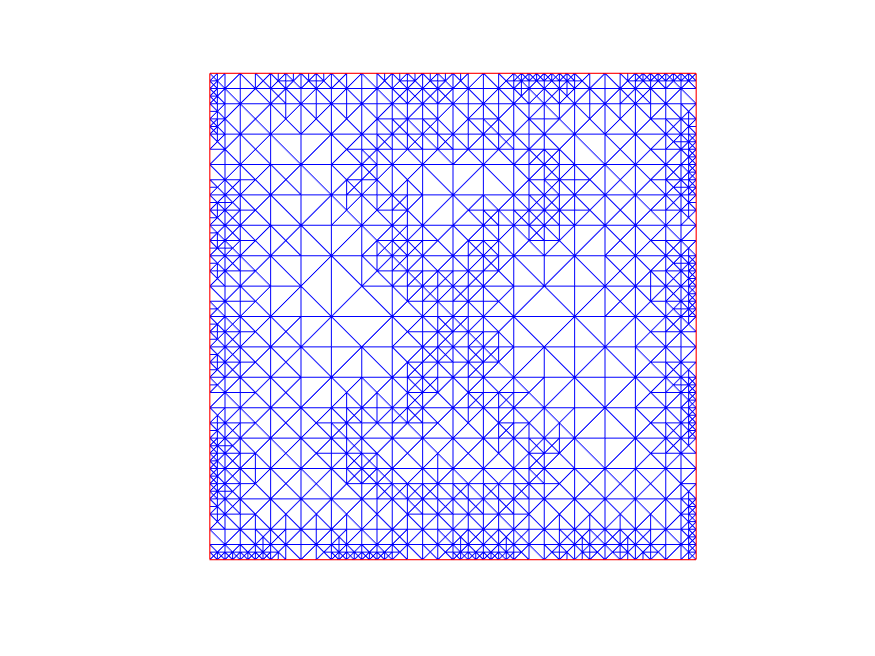}&
 \includegraphics[trim = {2.5cm 1.5cm 2.5cm 1.2cm}, clip, width=.2\textwidth]{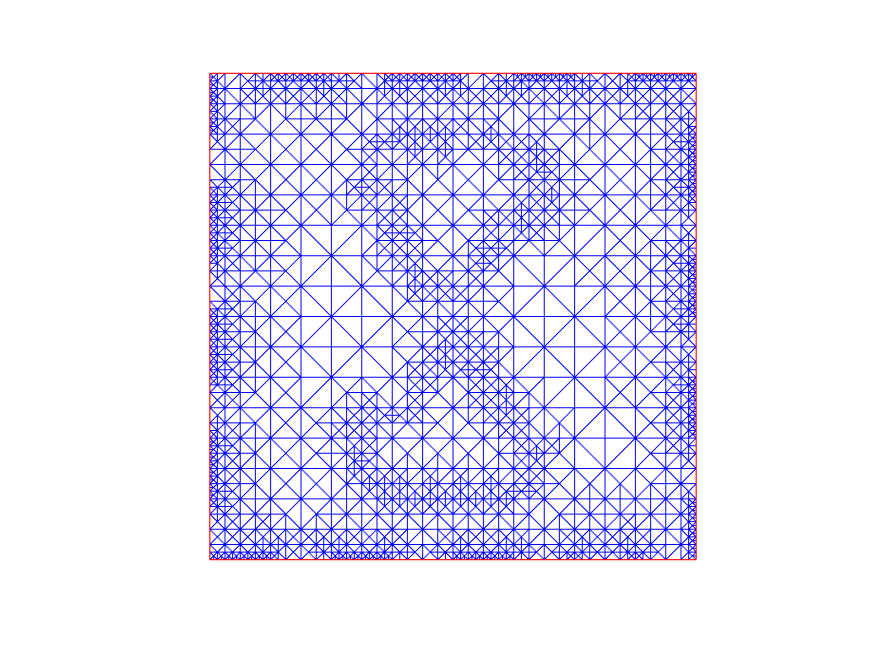}&
 \includegraphics[trim = {2.5cm 1.5cm 2.5cm 1.2cm}, clip, width=.2\textwidth]{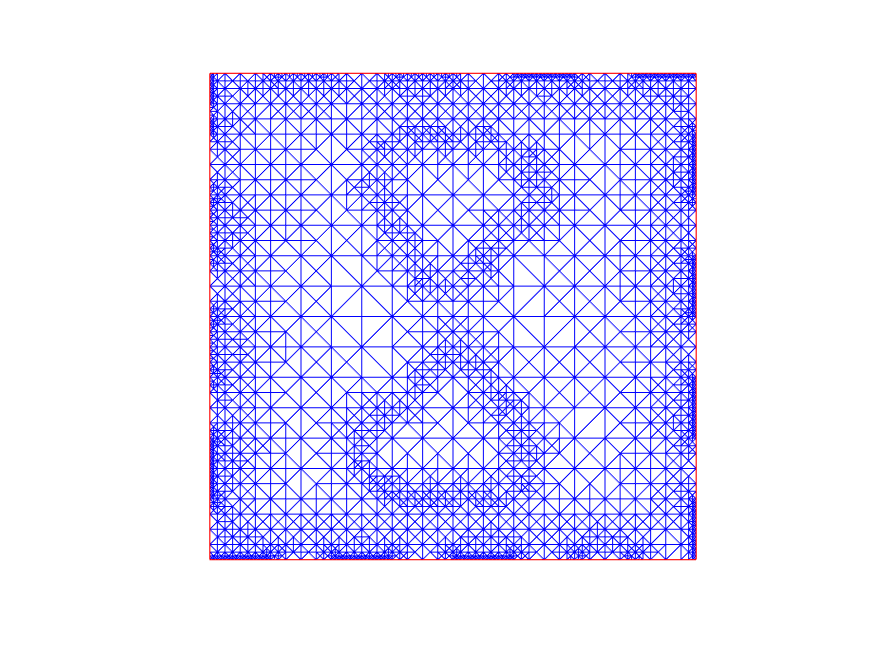}\\
  \includegraphics[trim = 1cm 0.5cm 0.5cm 0cm, width=.2\textwidth]{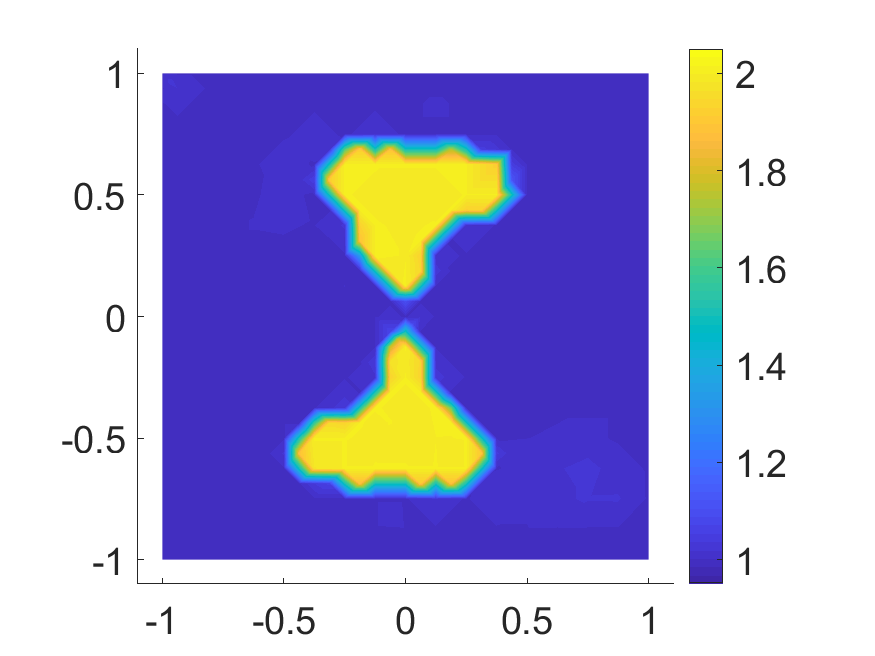}&
 \includegraphics[trim = 1cm 0.5cm 0.5cm 0cm, width=.2\textwidth]{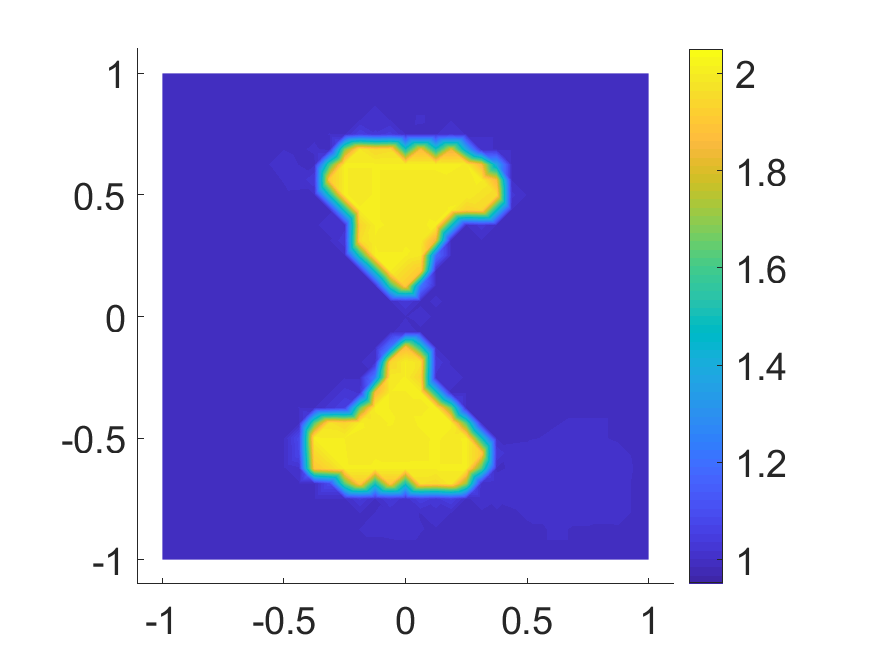}&
 \includegraphics[trim = 1cm 0.5cm 0.5cm 0cm, width=.2\textwidth]{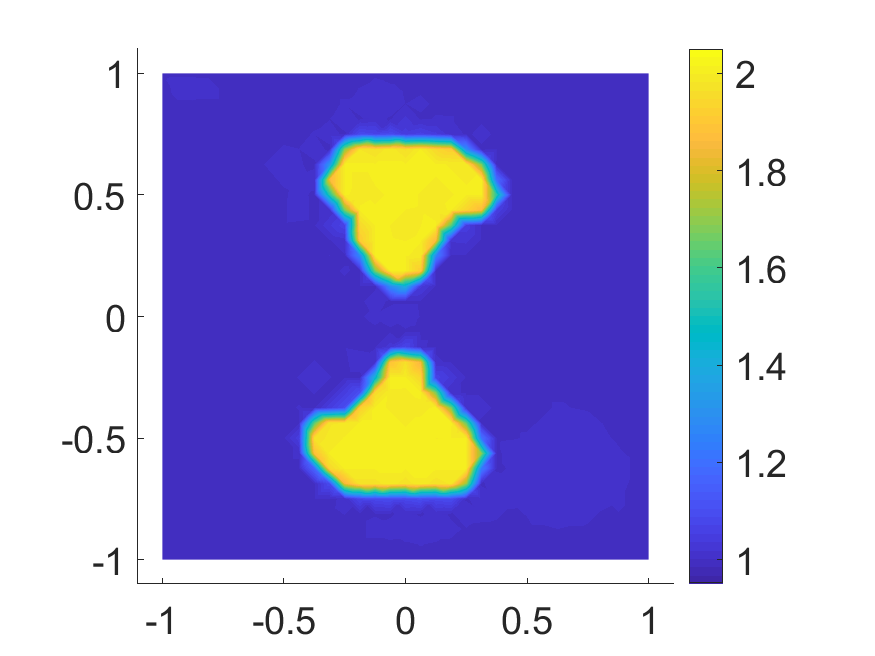}&
 \includegraphics[trim = 1cm 0.5cm 0.5cm 0cm, width=.2\textwidth]{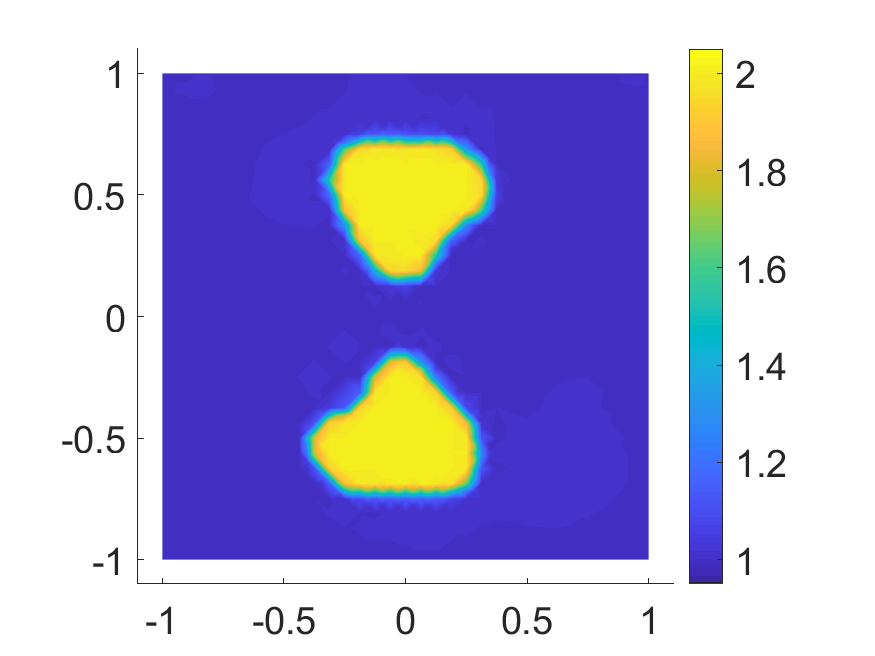}&
 \includegraphics[trim = 1cm 0.5cm 0.5cm 0cm, width=.2\textwidth]{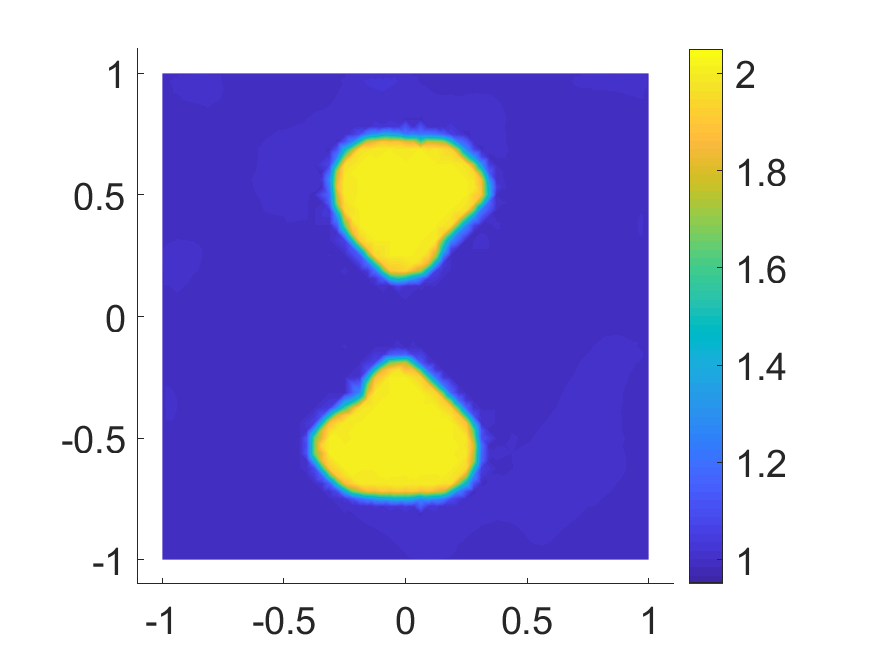}\\
    605   &    812  &    1266 &  1714  &  2630  \\
  \includegraphics[trim = {2.5cm 1.5cm 2.5cm 1.2cm}, clip, width=.2\textwidth]{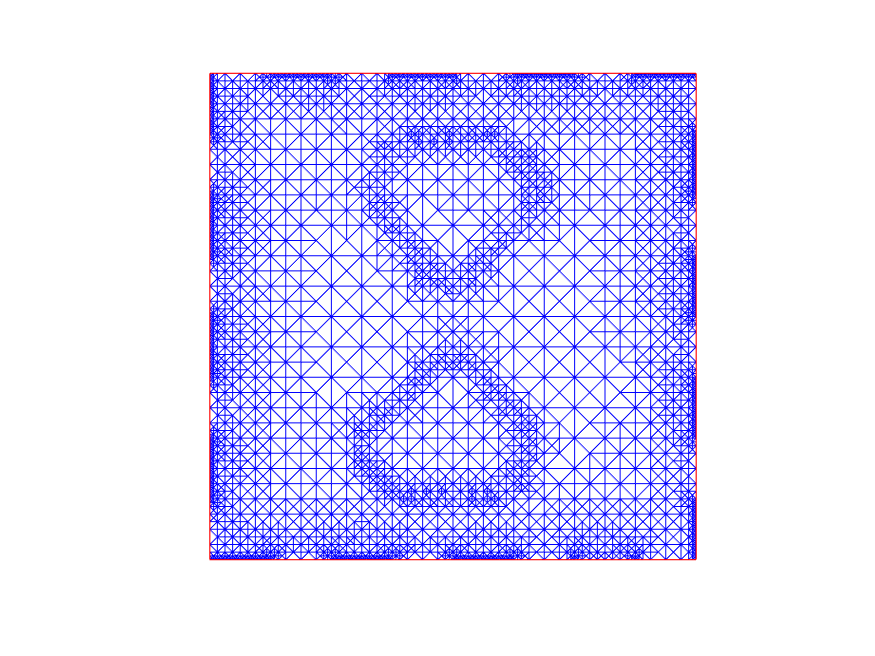}&
 \includegraphics[trim = {2.5cm 1.5cm 2.5cm 1.2cm}, clip, width=.2\textwidth]{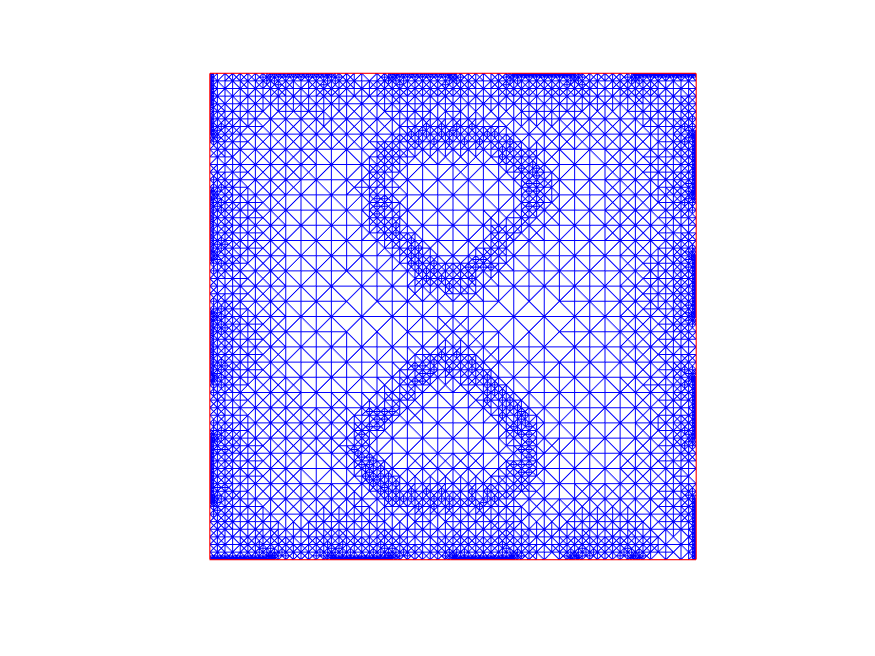}&
 \includegraphics[trim = {2.5cm 1.5cm 2.5cm 1.2cm}, clip, width=.2\textwidth]{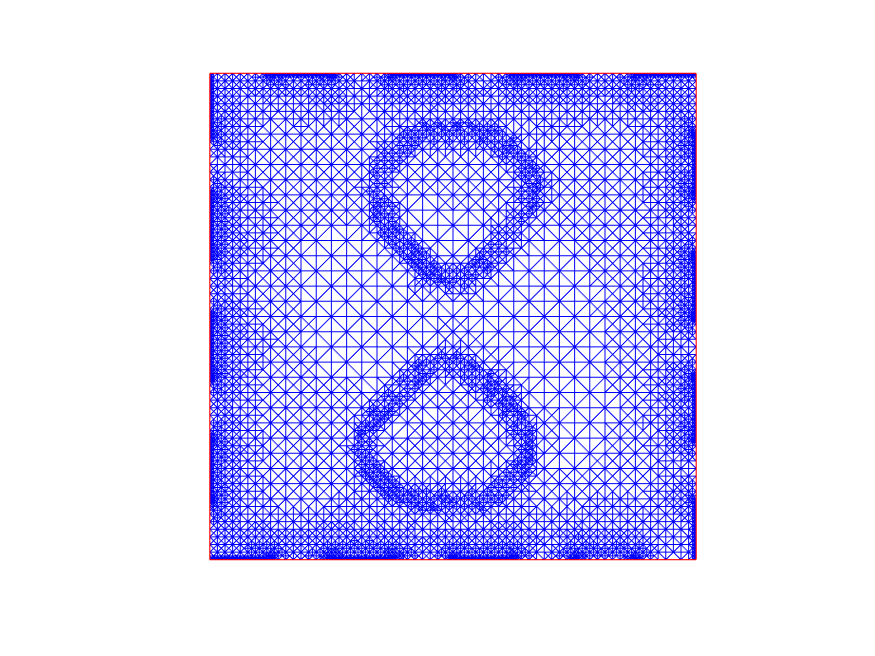}&
 \includegraphics[trim = {2.5cm 1.5cm 2.5cm 1.2cm}, clip, width=.2\textwidth]{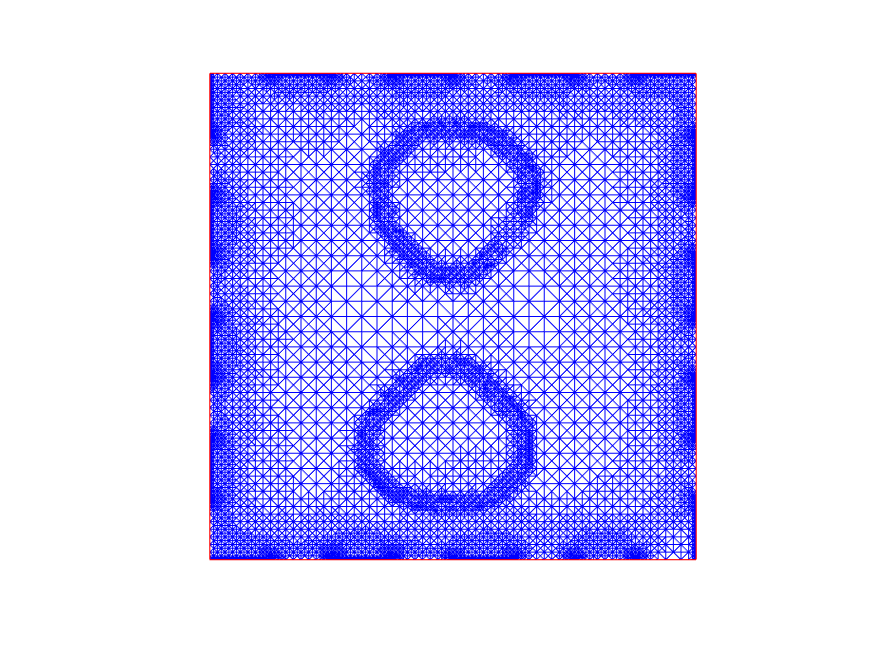}&
 \includegraphics[trim = {2.5cm 1.5cm 2.5cm 1.2cm}, clip, width=.2\textwidth]{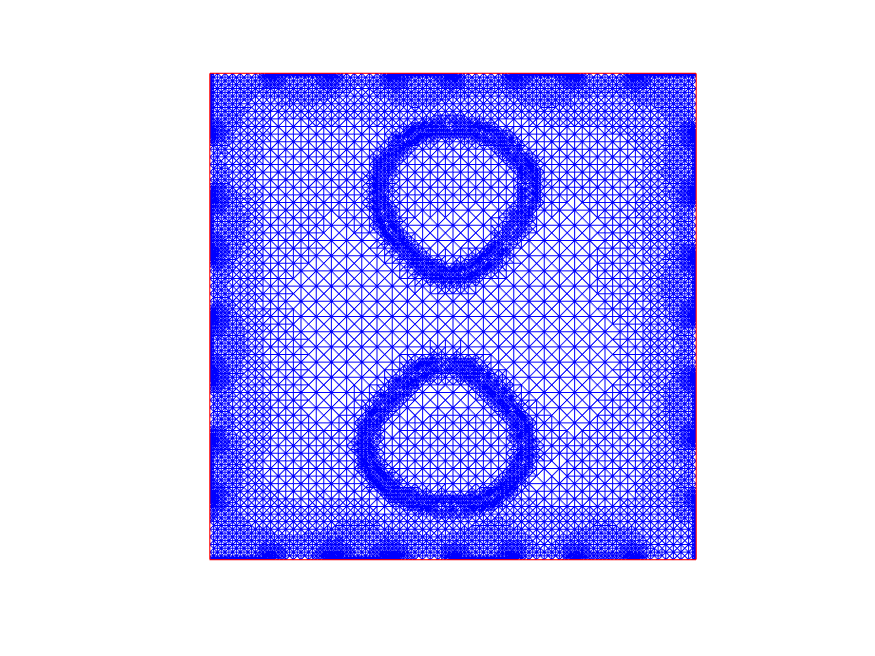}\\
 \includegraphics[trim = 1cm 0.5cm 0.5cm 0cm, width=.2\textwidth]{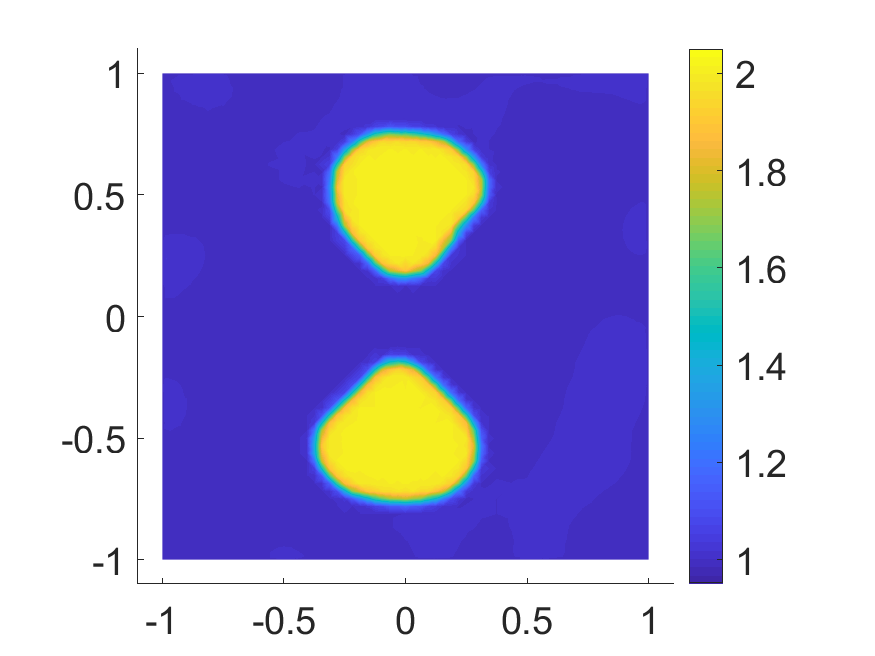}&
 \includegraphics[trim = 1cm 0.5cm 0.5cm 0cm, width=.2\textwidth]{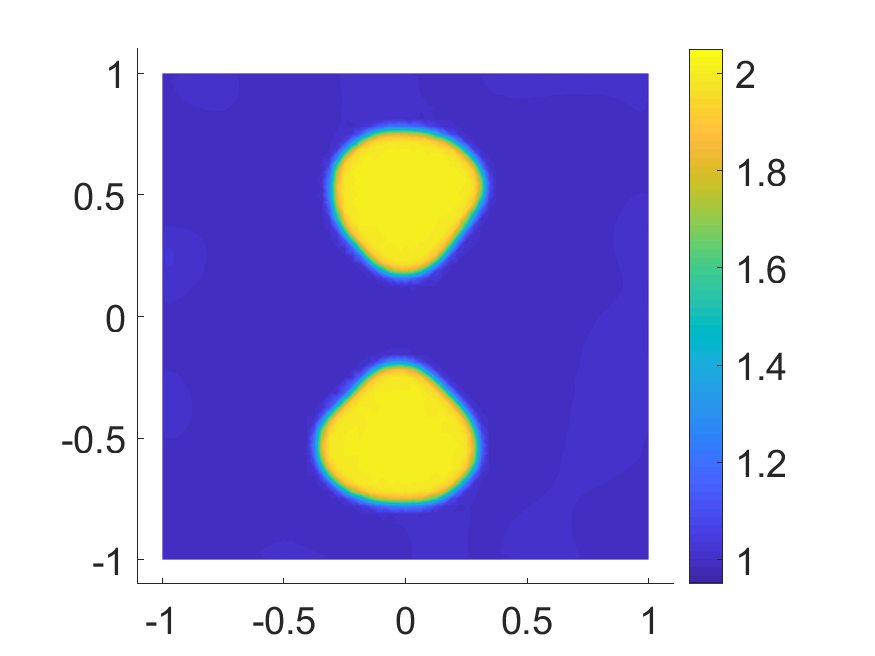}&
 \includegraphics[trim = 1cm 0.5cm 0.5cm 0cm, width=.2\textwidth]{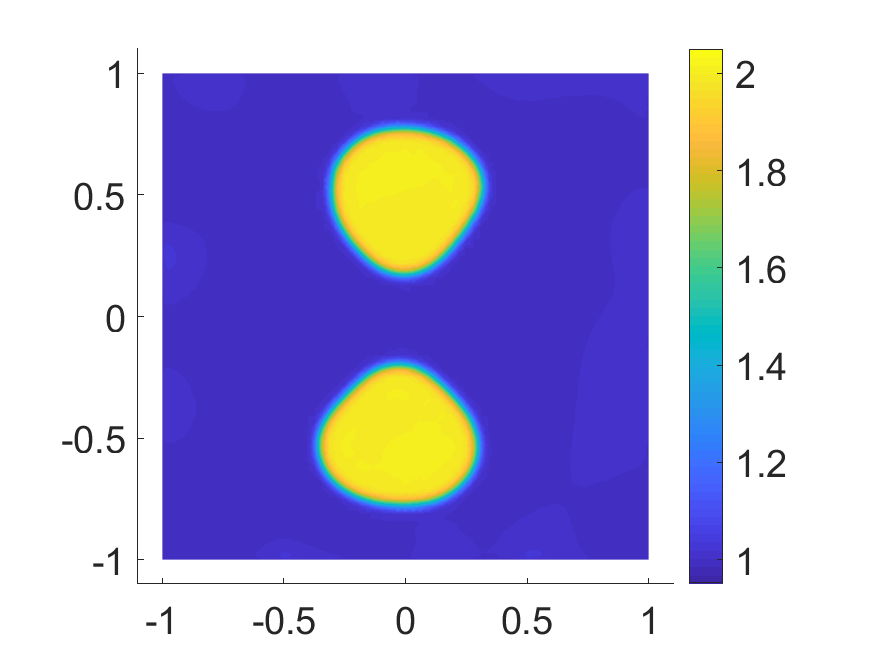}&
 \includegraphics[trim = 1cm 0.5cm 0.5cm 0cm, width=.2\textwidth]{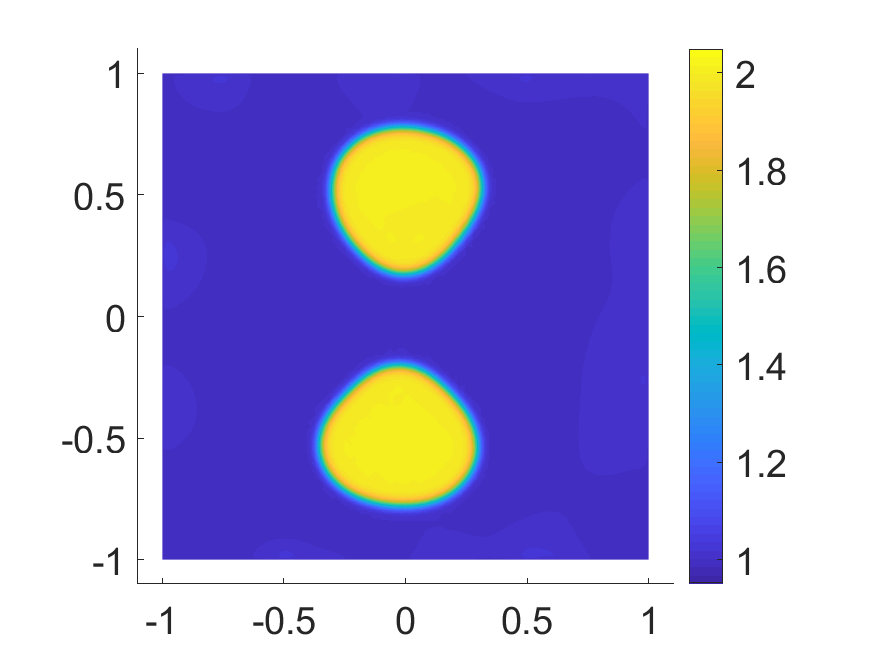}&
 \includegraphics[trim = 1cm 0.5cm 0.5cm 0cm, width=.2\textwidth]{ex2_1e3ad_it15}\\
 3526 &  5482 &  7549 &  11450 &  15830
 \end{tabular}
 \caption{The meshes $\mathcal{T}_k$ and recoveries $\sigma_k$ during the adaptive refinement, for Example \ref{exam2}(i) with
 $\epsilon=\text{1e-3}$ and $\tilde\alpha=\text{2e-2}$. The numbers refer to d.o.f.} \label{fig:exam2i-recon-iter1e3}
\end{figure}

Next we examine the adaptive refinement process more closely. In Figs. \ref{fig:exam2i-recon-iter1e3} and \ref{fig:exam2i-recon-iter1e2},
we show the meshes $\mathcal{T}_k$ during the iteration and the corresponding recoveries $\sigma_k$ for Example \ref{exam2}(i)
at two noise levels $\epsilon = \text{1e-3}$ and $\epsilon=\text{1e-2}$, respectively. On the coarse mesh $\cT_0$, the recovery
has very large errors and can only identify one component and thus fails to correctly identify the number of inclusions, due to
the severe under-resolution of both state and conductivity. Nonetheless, Algorithm \ref{alg_afem_eit} can correctly recover the
two components with reasonable accuracy after several adaptive loops, and accordingly, the support of the recovery is gradually
refined with its accuracy improving steadily. In particular, the inclusion locations stabilize after several loops, and thus
coarsening of the mesh seems unnecessary. Throughout, the refinement occurs mainly in the regions around the electrode edges
and internal interface, which is clearly observed for both noise levels. This is attributed to the separable marking
strategy, which allows detecting different sources of singularities simultaneously. In Fig. \ref{fig:exam2i_err-ind},
we display the evolution of the error indicators for Example \ref{exam2}(i) with $\epsilon=\text{1e-3}$. The estimators
play different roles: $\eta_{k,1}^2$ and $\eta_{k,2}^2$ indicate the electrode edges during first iterations and
then also internal interface, whereas throughout $\eta_{k,3}^2$ concentrates on the internal interface. Thus,
$\eta_{k,1}^2$ and $\eta_{k,2}^2$ are most effective for resolving the state and adjoint, whereas $\eta_{k,3}^2$ is effective
for detecting internal jumps of the conductivity. The magnitude of $\eta_{k,2}^2$ is much smaller than
$\eta_{k,1}^2$, since the boundary data $U^\delta-U(\sigma_k)$ for the adjoint is much smaller than the input current $I$ for the state.
Thus, a simple collective marking strategy (i.e., $\eta_k^2 =\eta_{k,1}^2 +\eta_{k,2}^2 + \eta_{k,3}^2$) may
miss the correct singularity, due to their drastically different scalings. In contrast, the separate marking in \eqref{eqn:marking} can
take care of the scaling automatically.

\begin{figure}[hbt!]
 \centering
 \setlength{\tabcolsep}{0pt}
 \begin{tabular}{cccccccc}
 \includegraphics[trim = {2.5cm 1.5cm 2.5cm 1.2cm}, clip, width=.2\textwidth]{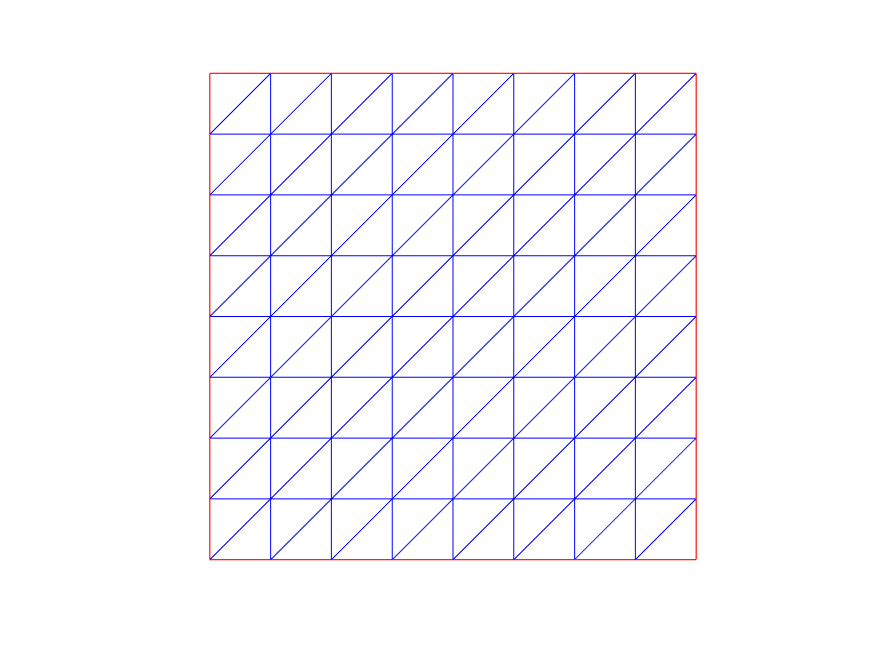}&
 \includegraphics[trim = {2.5cm 1.5cm 2.5cm 1.2cm}, clip, width=.2\textwidth]{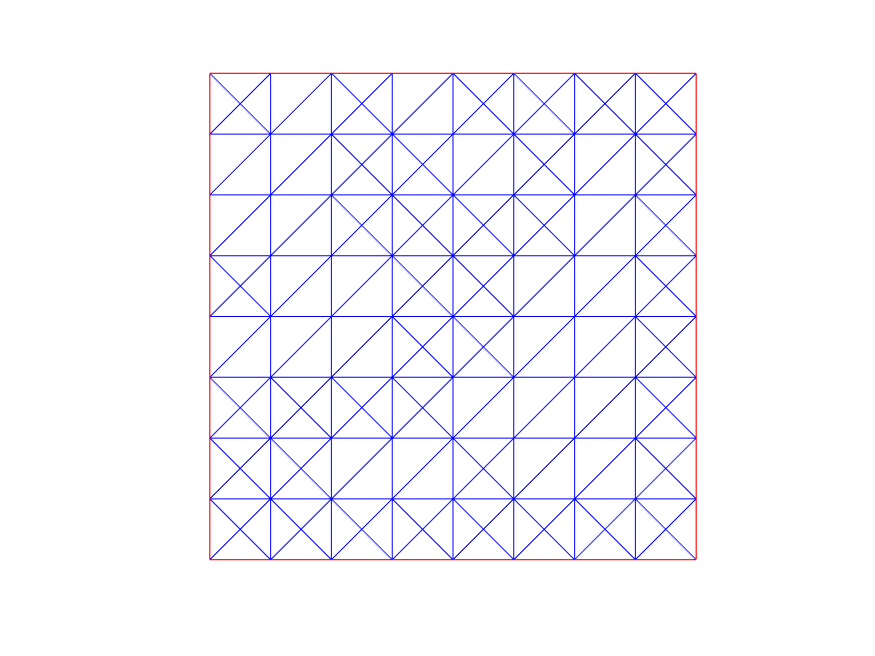}&
 \includegraphics[trim = {2.5cm 1.5cm 2.5cm 1.2cm}, clip, width=.2\textwidth]{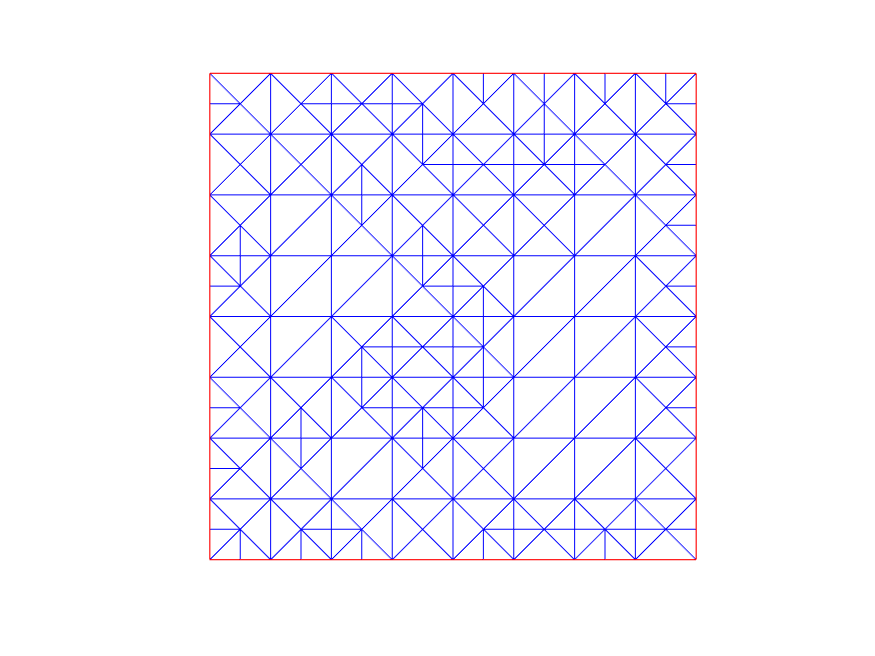}&
 \includegraphics[trim = {2.5cm 1.5cm 2.5cm 1.2cm}, clip, width=.2\textwidth]{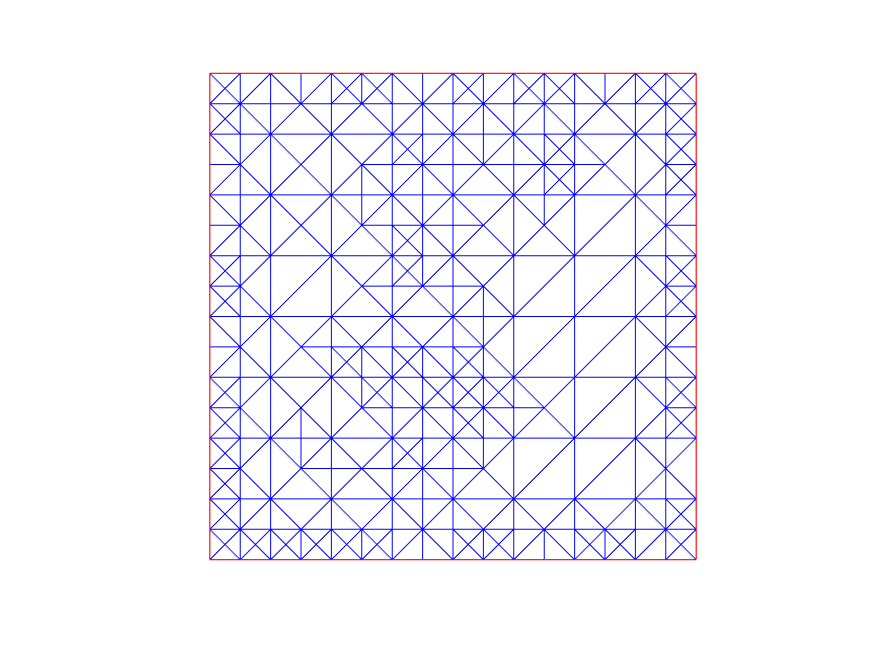}&
 \includegraphics[trim = {2.5cm 1.5cm 2.5cm 1.2cm}, clip, width=.2\textwidth]{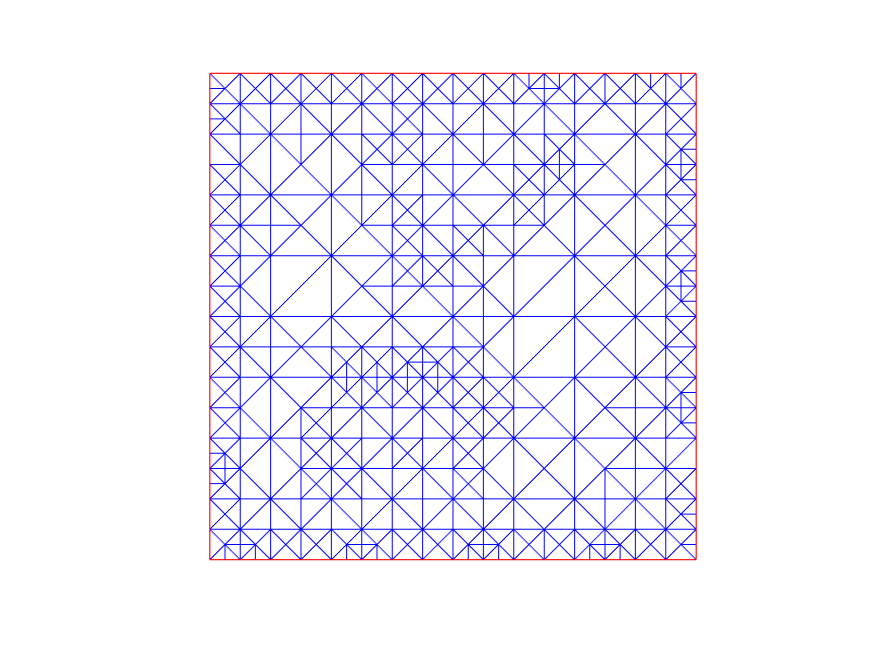}\\
 \includegraphics[trim = 1cm 0.5cm 0.5cm 0cm, width=.2\textwidth]{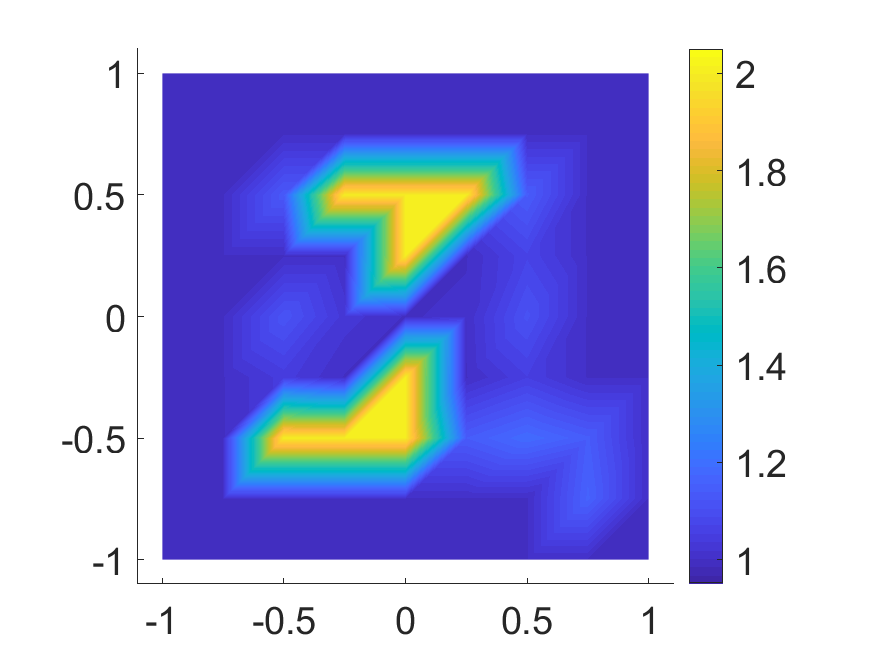}&
 \includegraphics[trim = 1cm 0.5cm 0.5cm 0cm, width=.2\textwidth]{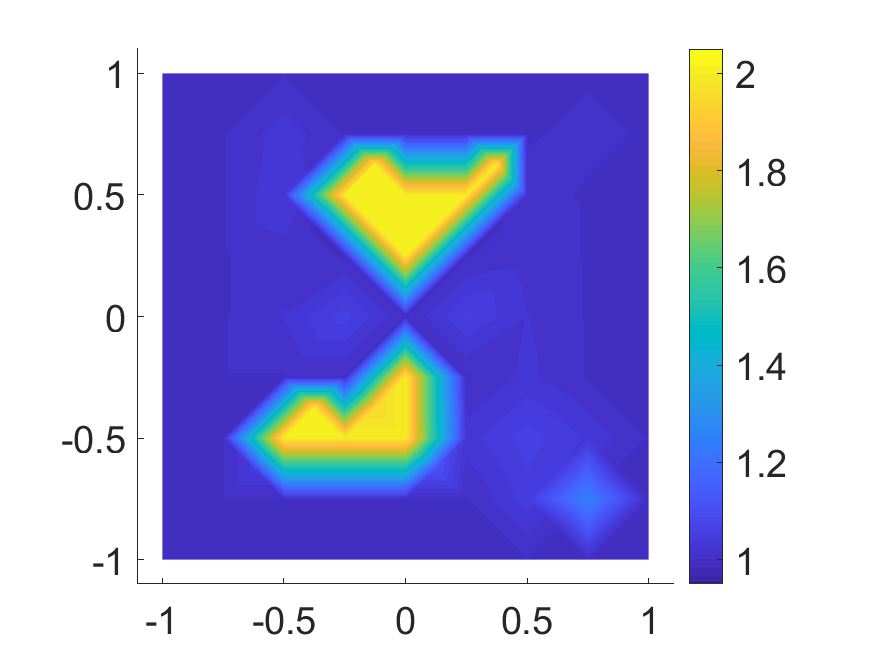}&
 \includegraphics[trim = 1cm 0.5cm 0.5cm 0cm, width=.2\textwidth]{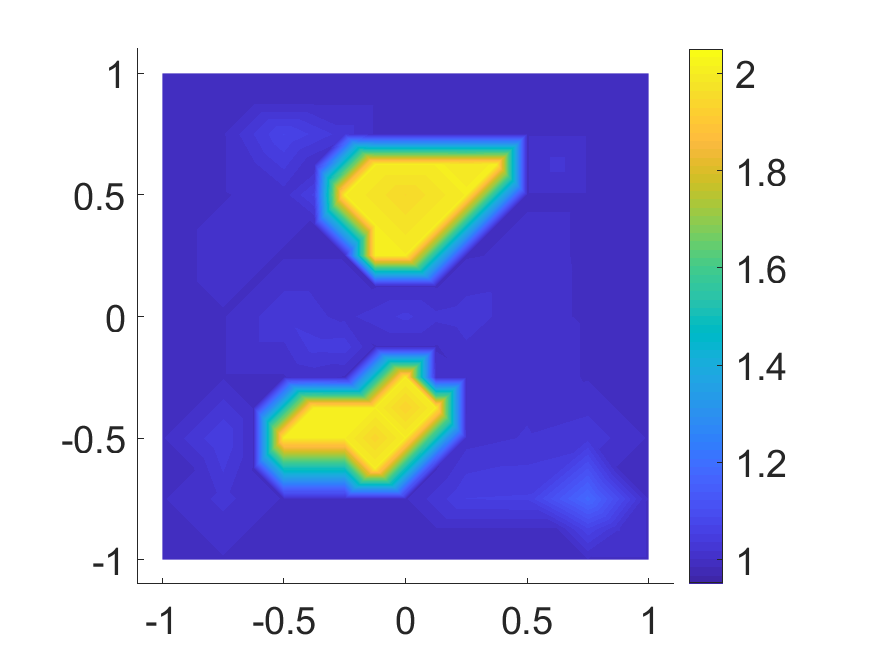}&
 \includegraphics[trim = 1cm 0.5cm 0.5cm 0cm, width=.2\textwidth]{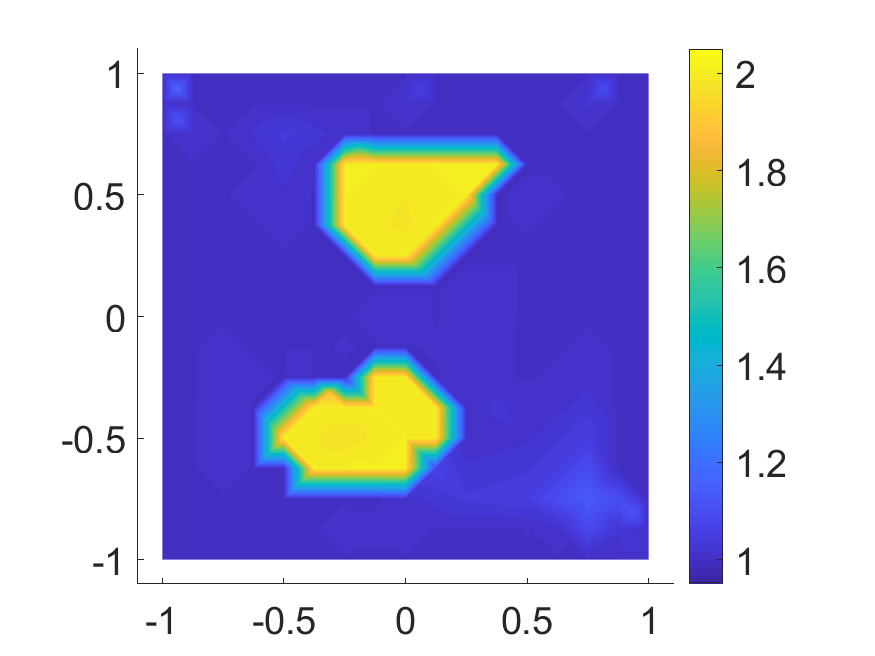}&
 \includegraphics[trim = 1cm 0.5cm 0.5cm 0cm, width=.2\textwidth]{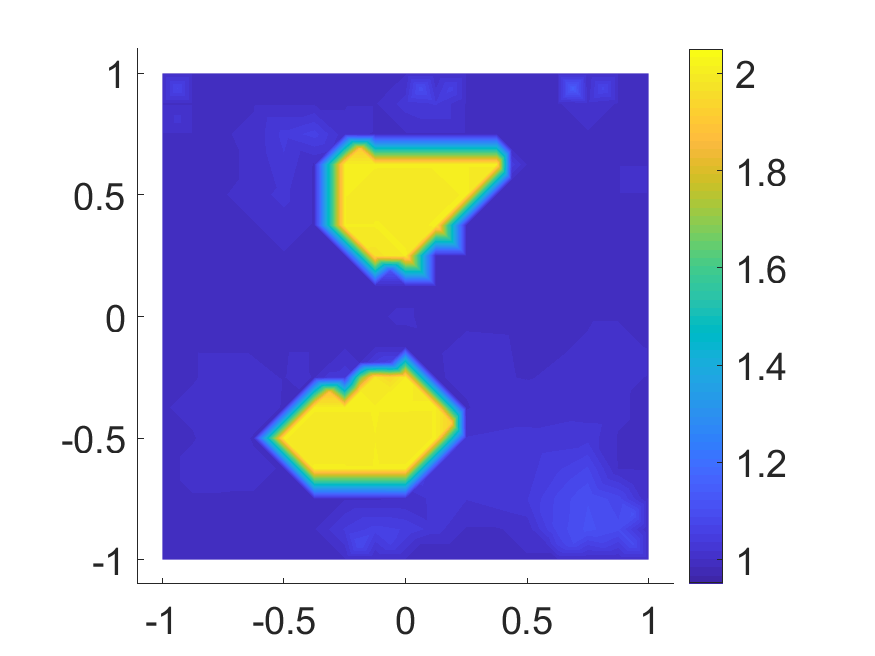}\\
 81    &   120   &   176  &   276   &   378\\
 \includegraphics[trim = {2.5cm 1.5cm 2.5cm 1.2cm}, clip, width=.2\textwidth]{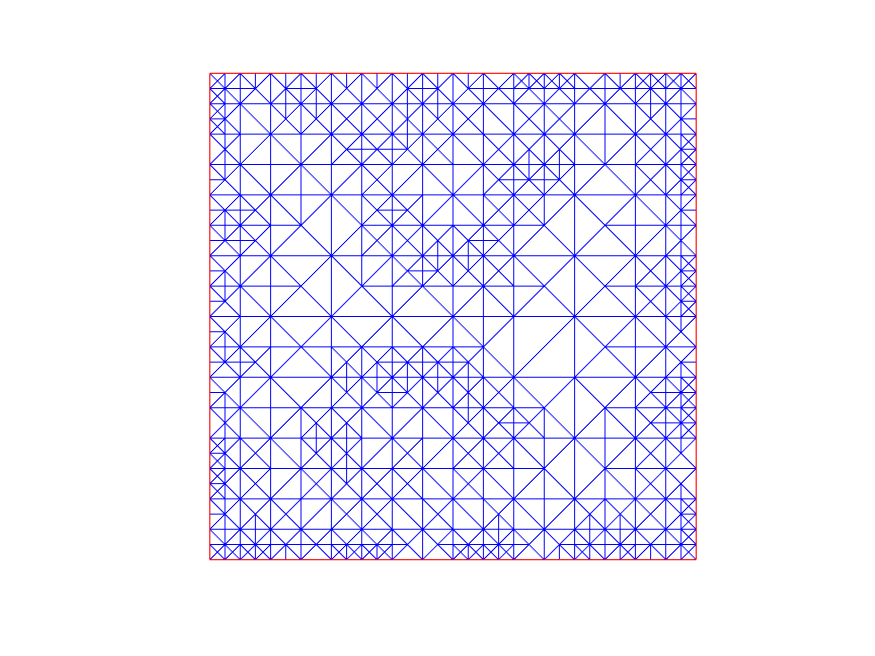}&
 \includegraphics[trim = {2.5cm 1.5cm 2.5cm 1.2cm}, clip, width=.2\textwidth]{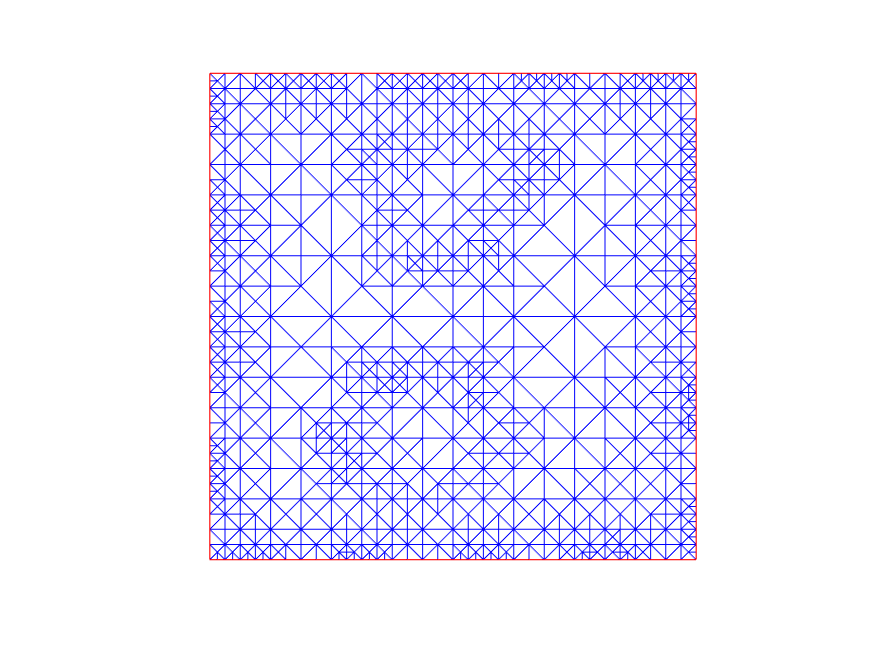}&
 \includegraphics[trim = {2.5cm 1.5cm 2.5cm 1.2cm}, clip, width=.2\textwidth]{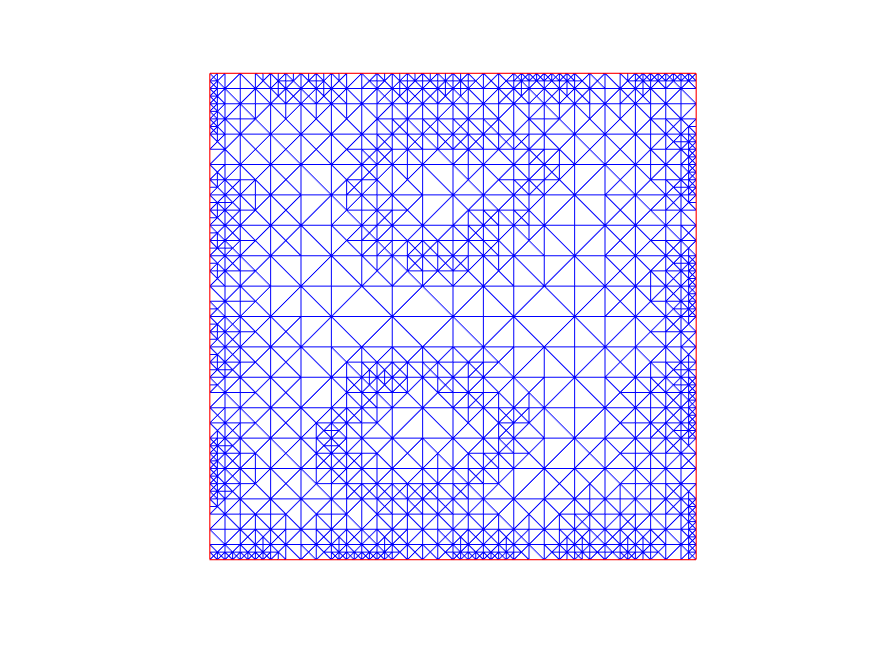}&
 \includegraphics[trim = {2.5cm 1.5cm 2.5cm 1.2cm}, clip, width=.2\textwidth]{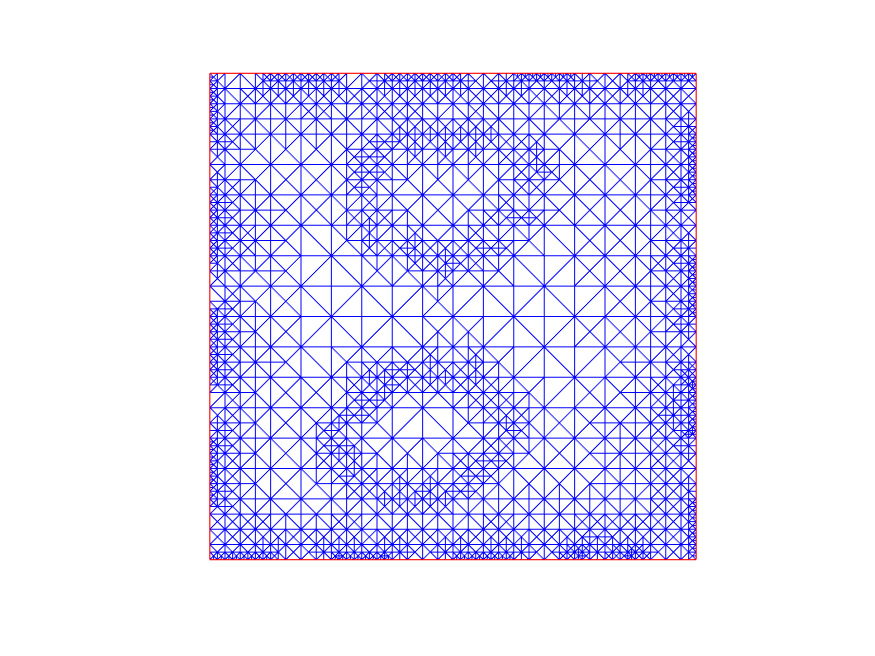}&
 \includegraphics[trim = {2.5cm 1.5cm 2.5cm 1.2cm}, clip, width=.2\textwidth]{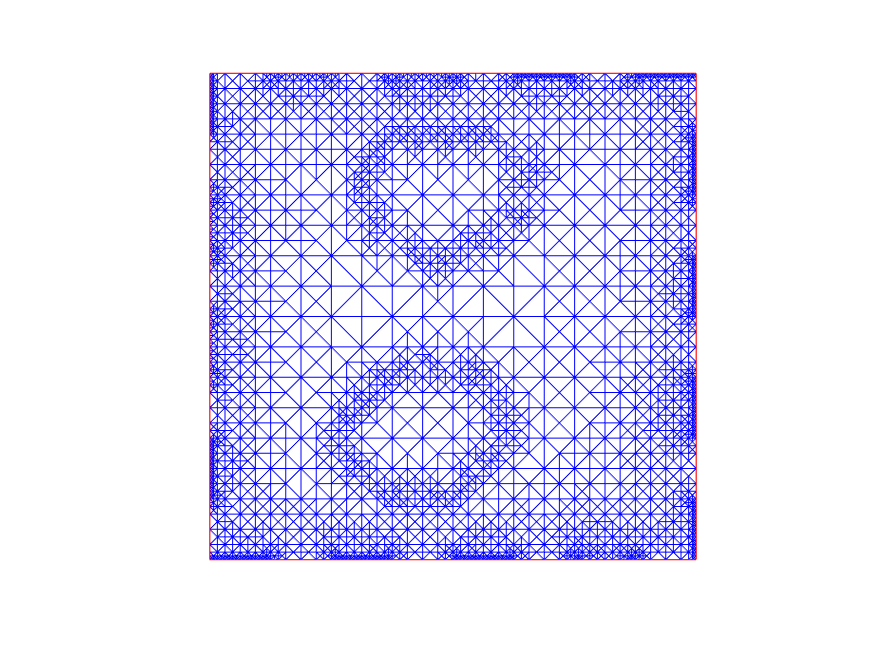}\\
  \includegraphics[trim = 1cm 0.5cm 0.5cm 0cm, width=.2\textwidth]{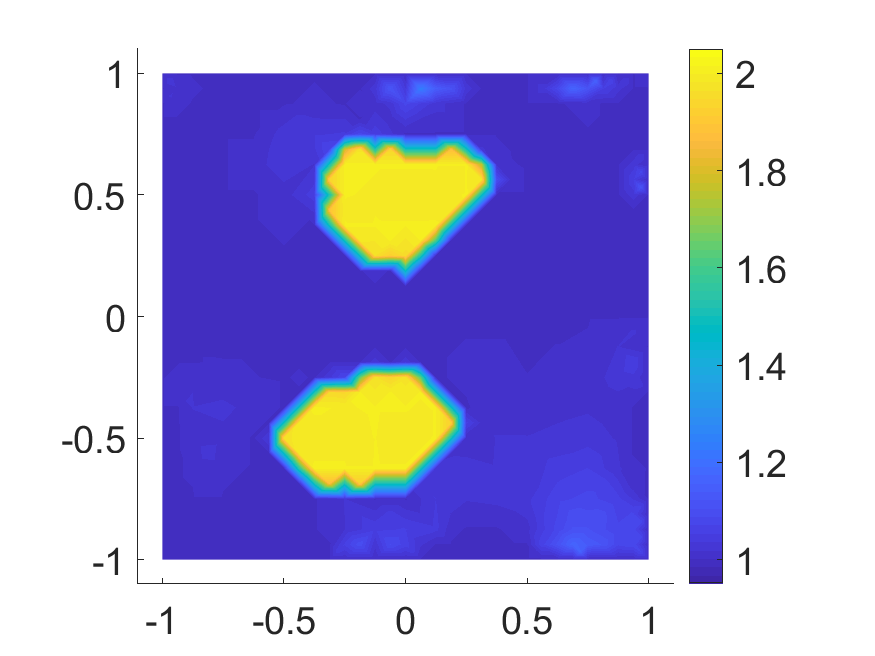}&
 \includegraphics[trim = 1cm 0.5cm 0.5cm 0cm, width=.2\textwidth]{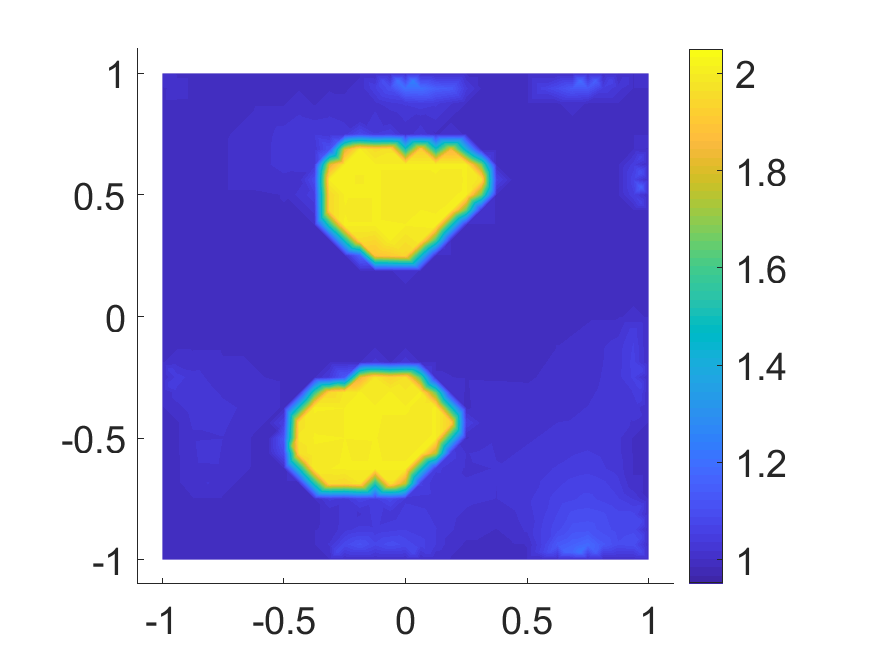}&
 \includegraphics[trim = 1cm 0.5cm 0.5cm 0cm, width=.2\textwidth]{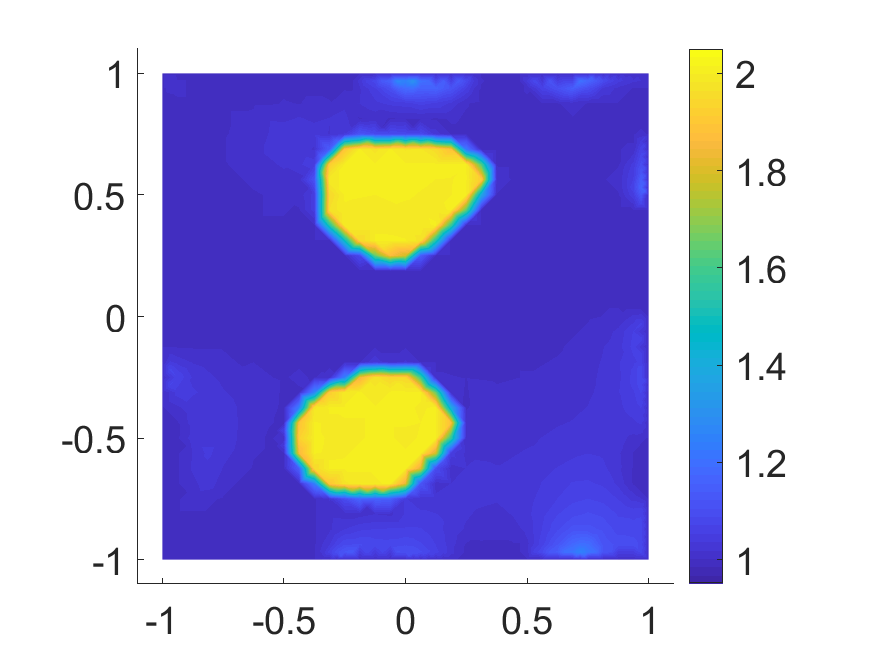}&
 \includegraphics[trim = 1cm 0.5cm 0.5cm 0cm, width=.2\textwidth]{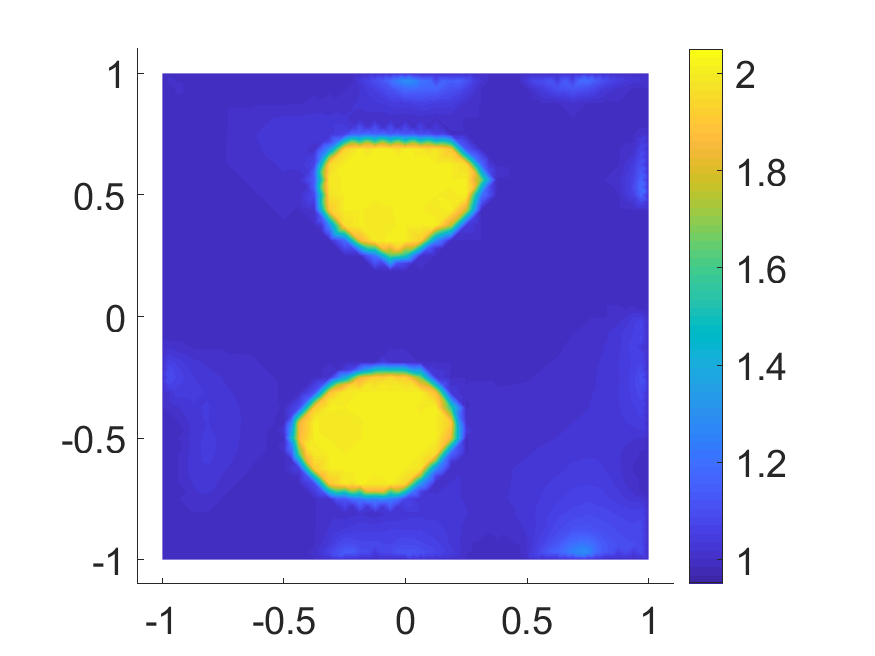}&
 \includegraphics[trim = 1cm 0.5cm 0.5cm 0cm, width=.2\textwidth]{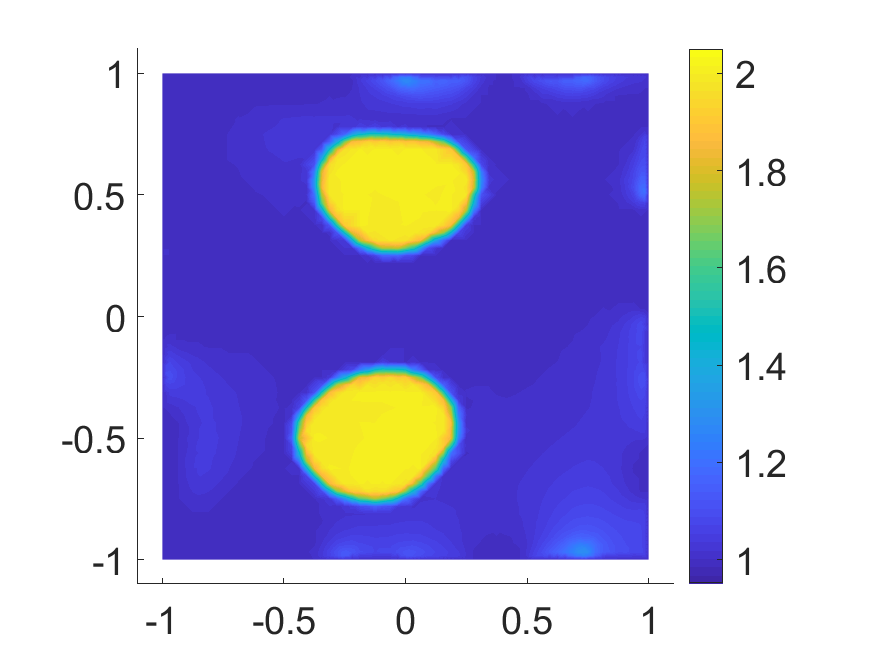}\\
   595    &   816   &       1324  &       1862  &       2884\\
  \includegraphics[trim = {2.5cm 1.5cm 2.5cm 1.2cm}, clip, width=.2\textwidth]{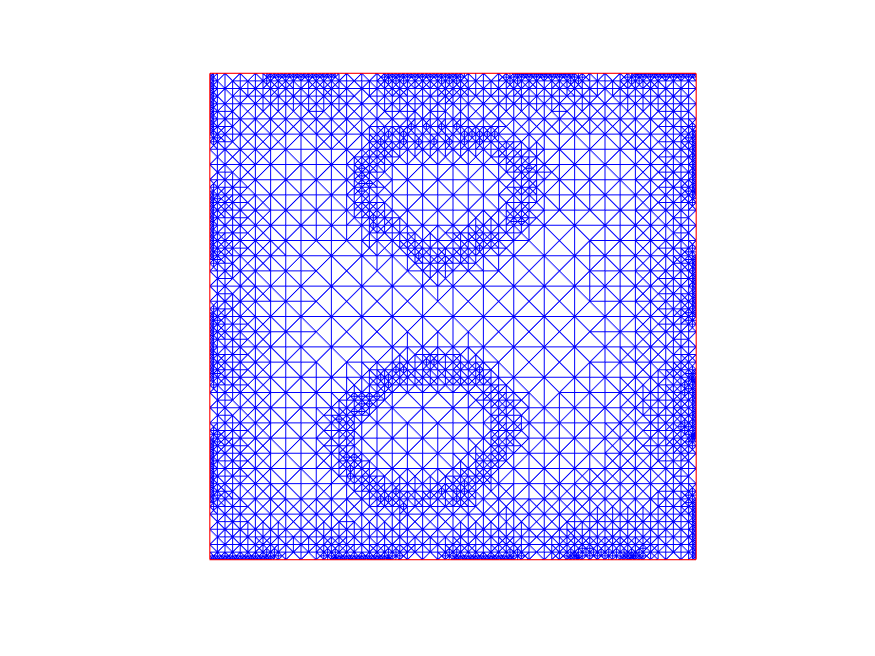}&
 \includegraphics[trim = {2.5cm 1.5cm 2.5cm 1.2cm}, clip, width=.2\textwidth]{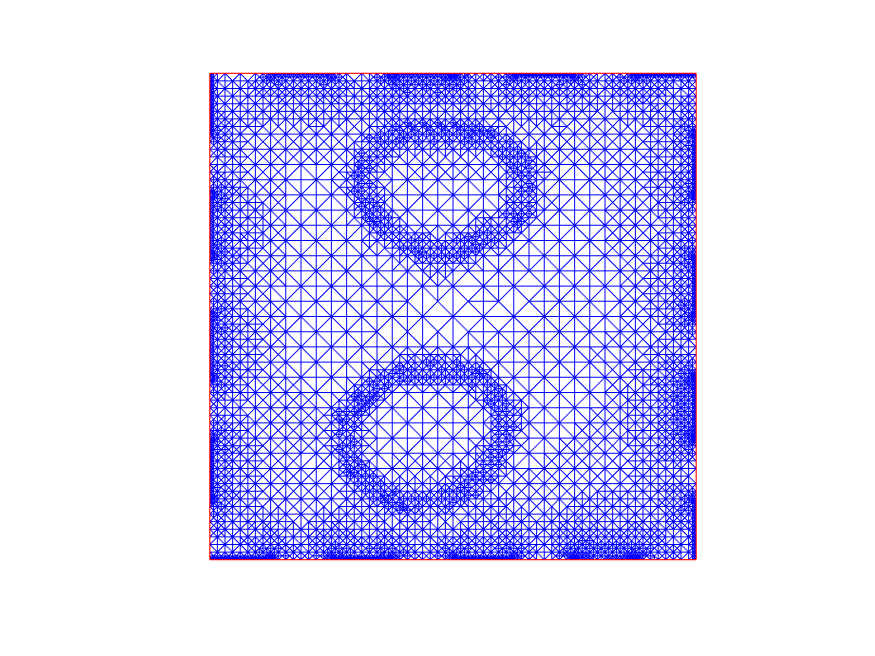}&
 \includegraphics[trim = {2.5cm 1.5cm 2.5cm 1.2cm}, clip, width=.2\textwidth]{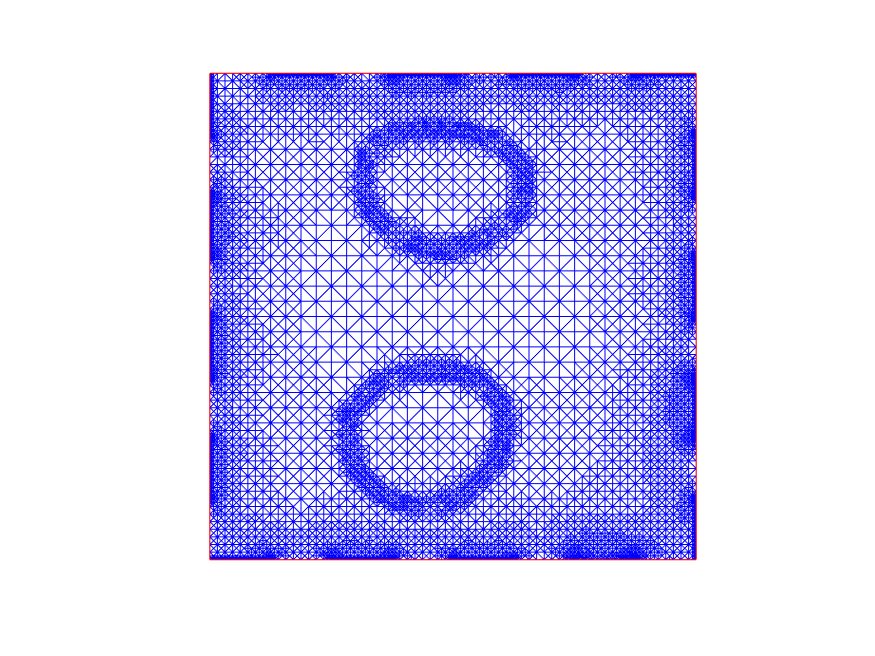}&
 \includegraphics[trim = {2.5cm 1.5cm 2.5cm 1.2cm}, clip, width=.2\textwidth]{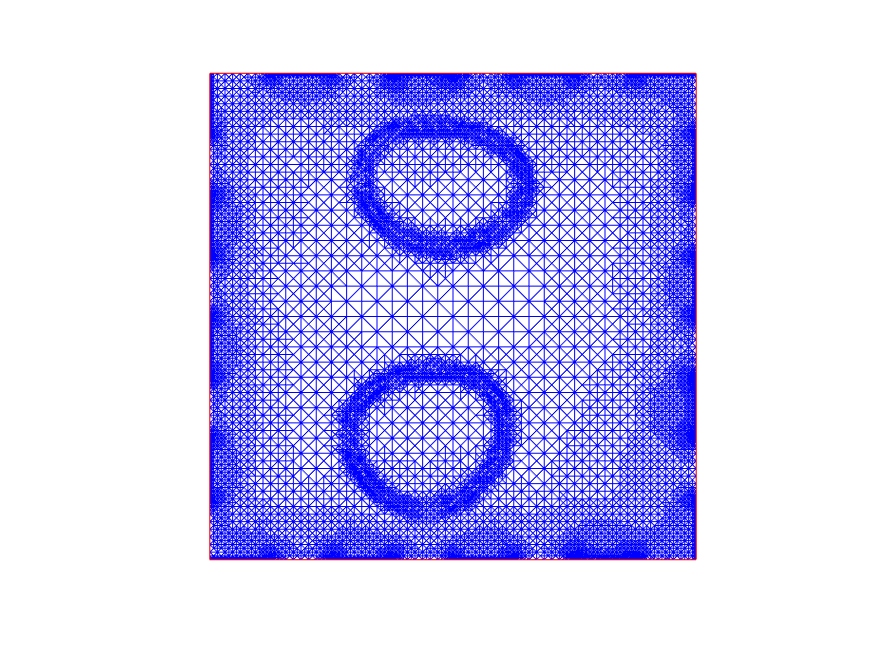}&
 \includegraphics[trim = {2.5cm 1.5cm 2.5cm 1.2cm}, clip, width=.2\textwidth]{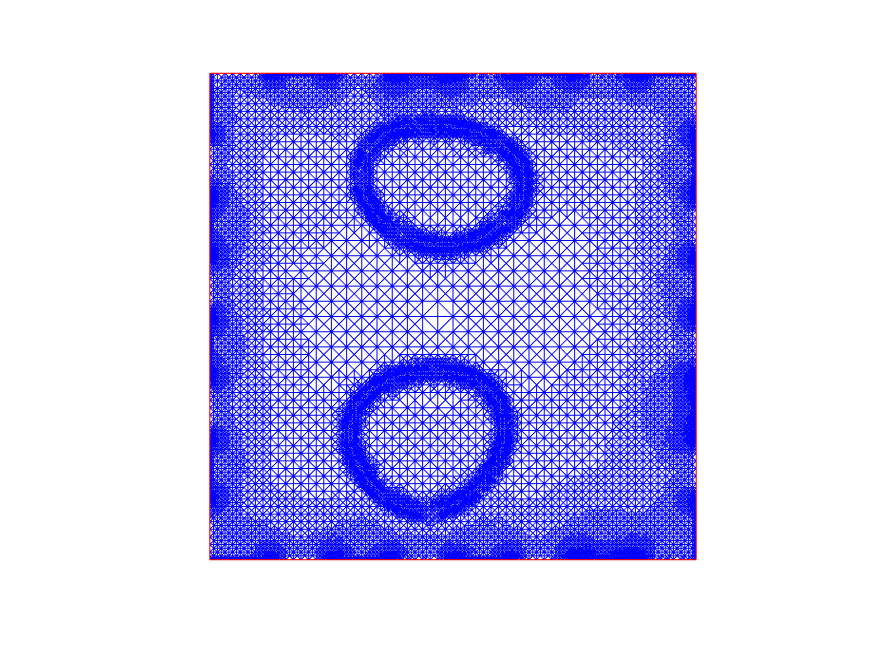}\\
 \includegraphics[trim = 1cm 0.5cm 0.5cm 0cm, width=.2\textwidth]{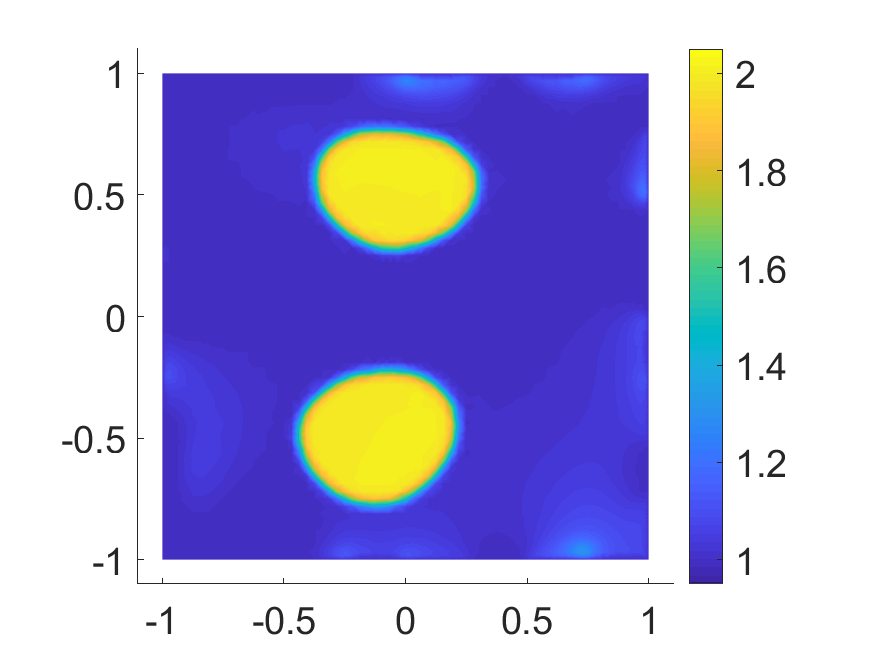}&
 \includegraphics[trim = 1cm 0.5cm 0.5cm 0cm, width=.2\textwidth]{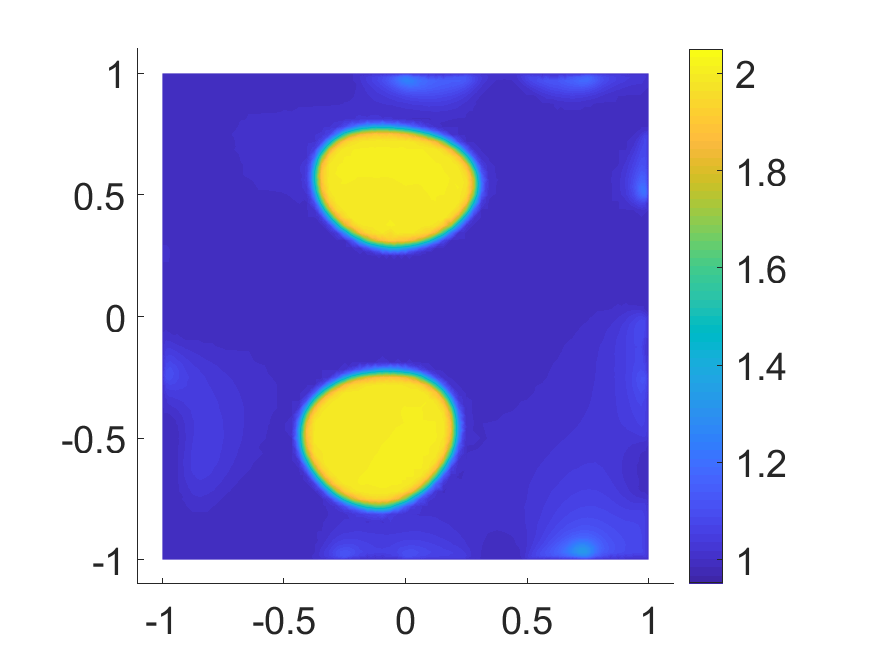}&
 \includegraphics[trim = 1cm 0.5cm 0.5cm 0cm, width=.2\textwidth]{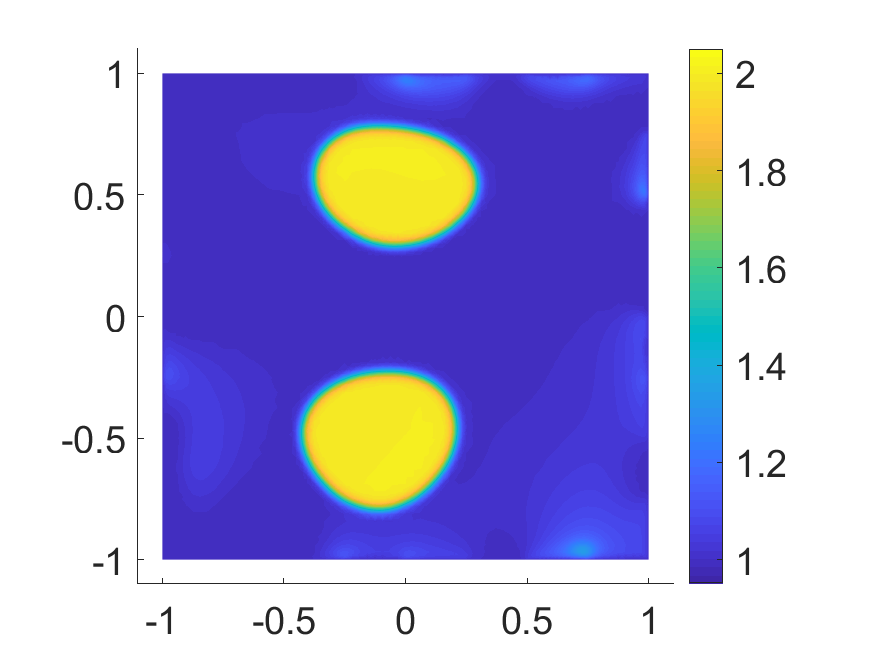}&
 \includegraphics[trim = 1cm 0.5cm 0.5cm 0cm, width=.2\textwidth]{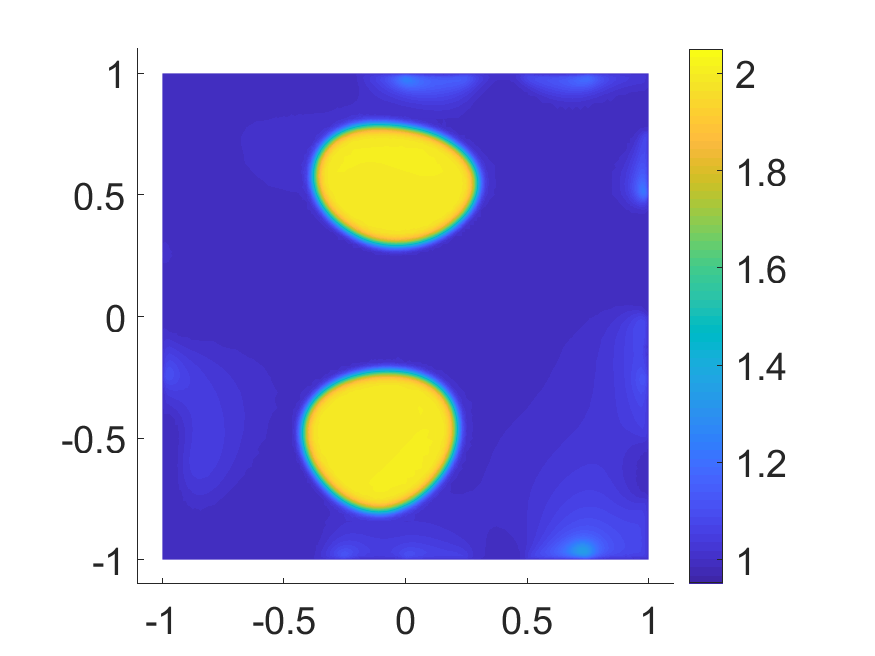}&
 \includegraphics[trim = 1cm 0.5cm 0.5cm 0cm, width=.2\textwidth]{ex2_1e2ad_it15}\\
  4004   &       6250   &        8690    &      13374   &      18770
 \end{tabular}
 \caption{The meshes $\mathcal{T}_k$ and recoveries $\sigma_k$ during the adaptive
 refinement, for Example \ref{exam2}(i) with $\epsilon=\text{1e-2}$ and $\tilde\alpha=\text{3e-2}$. The numbers
 refer to d.o.f.} \label{fig:exam2i-recon-iter1e2}
\end{figure}

\begin{figure}[hbt!]
  \centering
  \begin{tabular}{ccccc}
    \includegraphics[trim = {.2cm 0cm 0cm 0cm}, clip, width=.18\textwidth]{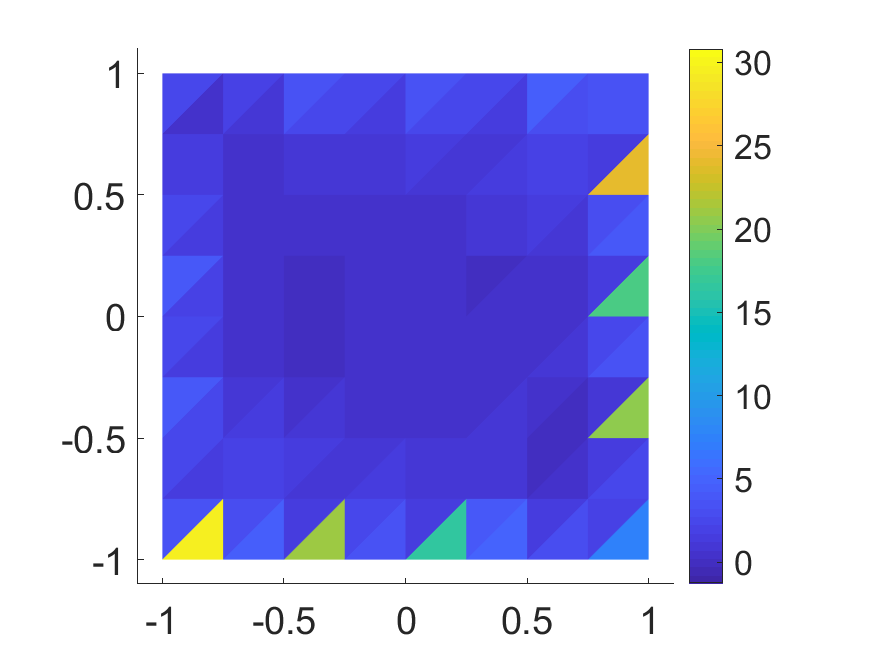} &
    \includegraphics[trim = {.2cm 0cm 0cm 0cm}, clip, width=.18\textwidth]{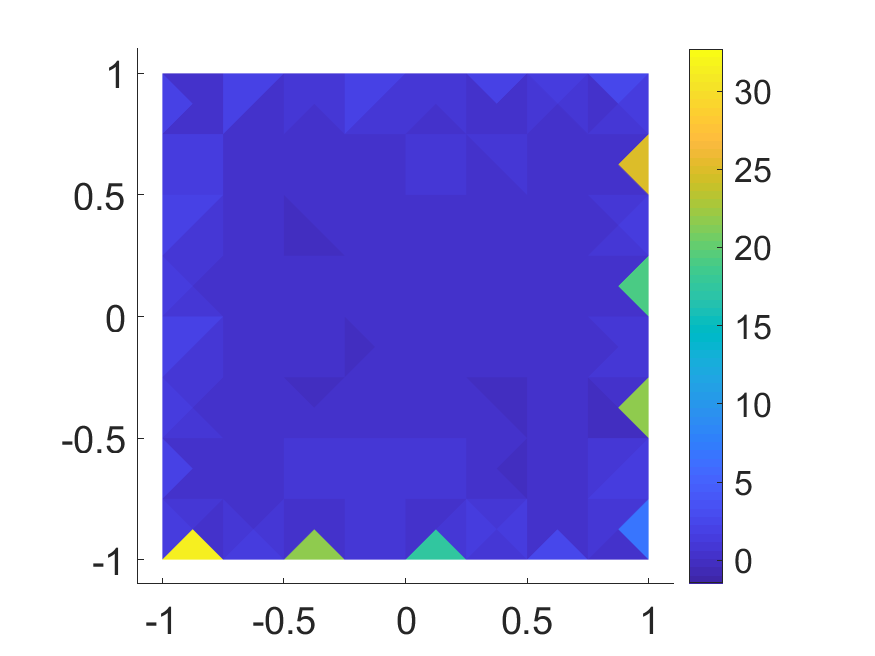} &
    \includegraphics[trim = {.2cm 0cm 0cm 0cm}, clip, width=.18\textwidth]{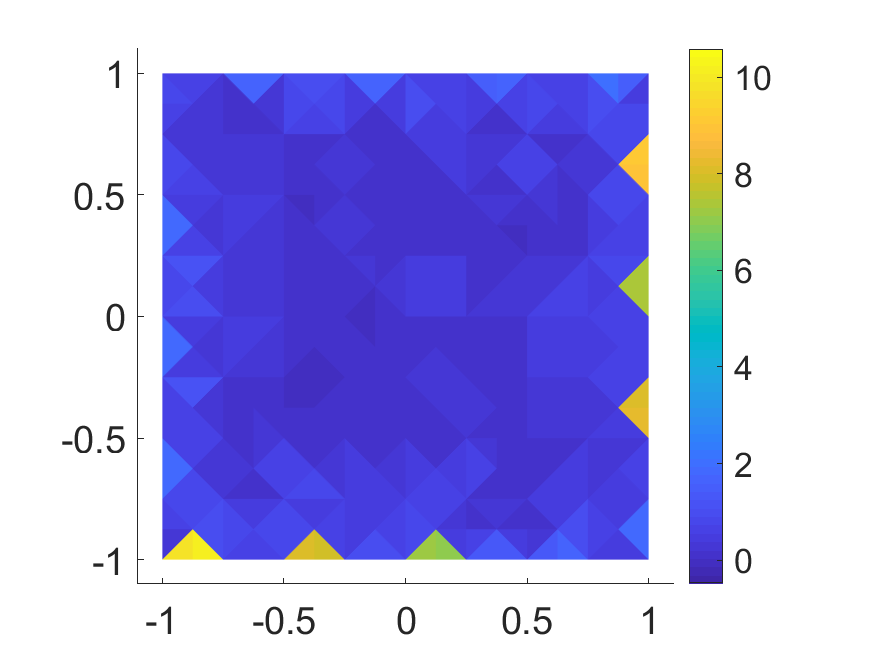} &
    \includegraphics[trim = {.2cm 0cm 0cm 0cm}, clip, width=.18\textwidth]{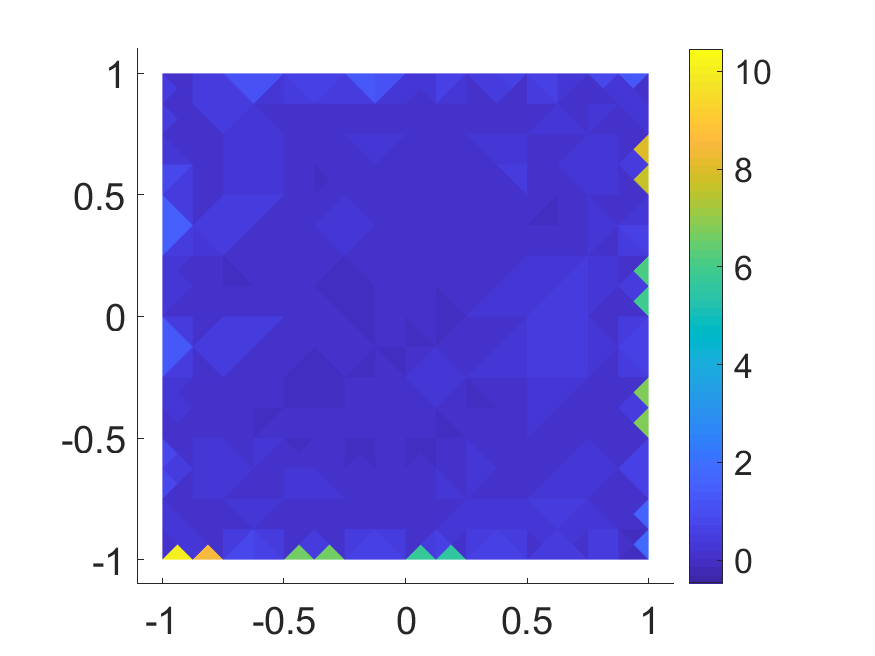} &
    \includegraphics[trim = {.2cm 0cm 0cm 0cm}, clip, width=.18\textwidth]{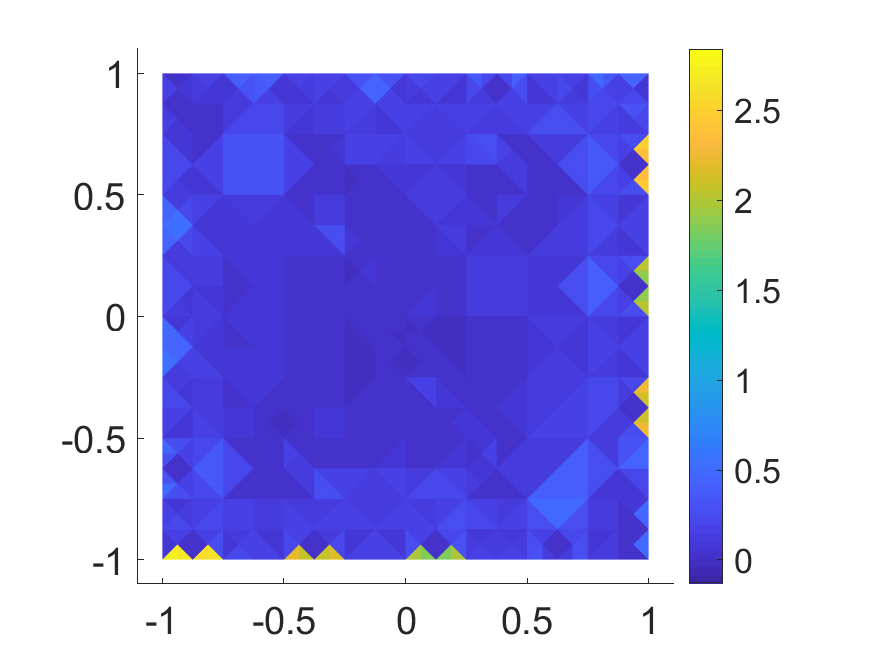}\\
    \includegraphics[trim = {.2cm 0cm 0cm 0cm}, clip, width=.18\textwidth]{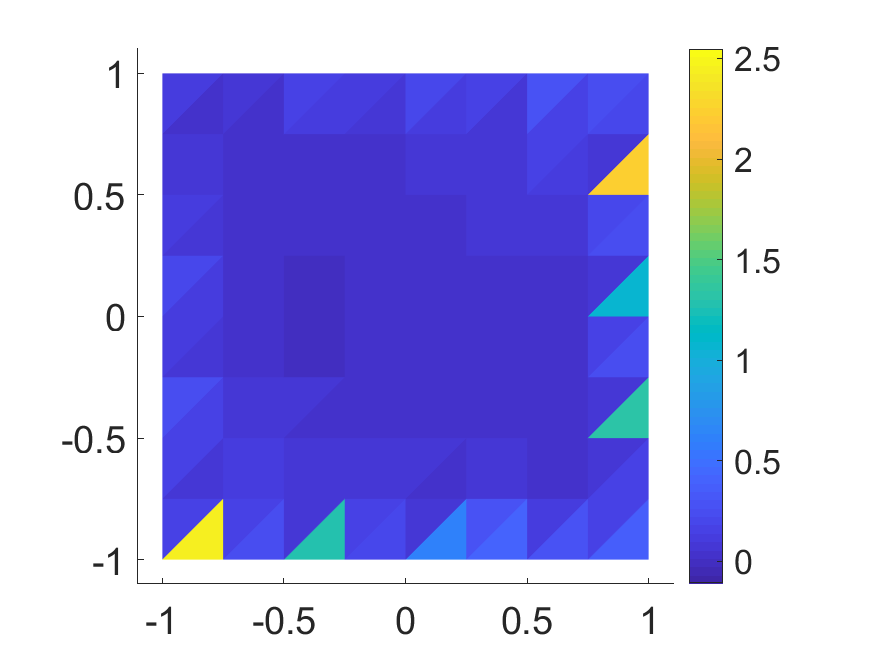} &
    \includegraphics[trim = {.2cm 0cm 0cm 0cm}, clip, width=.18\textwidth]{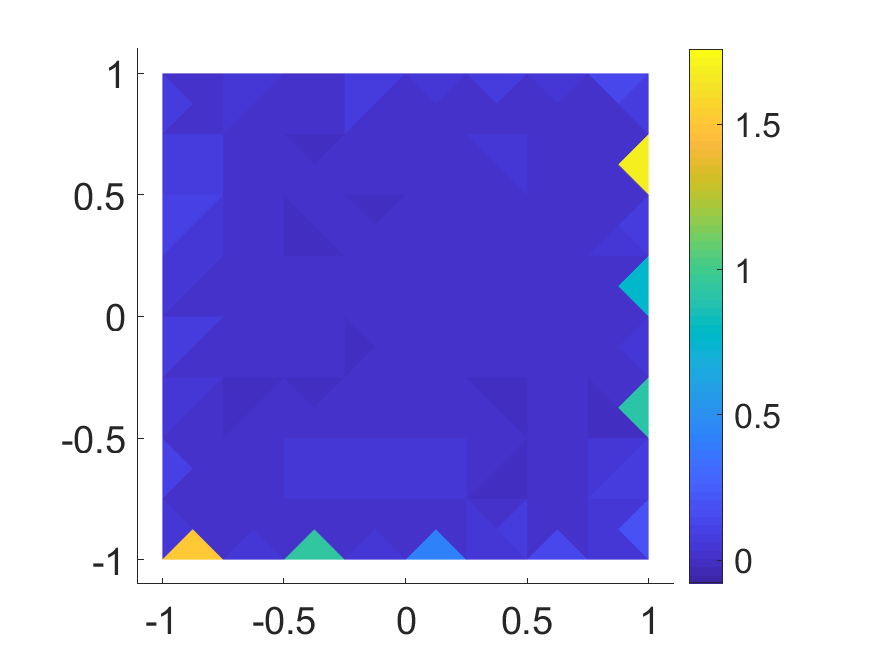} &
    \includegraphics[trim = {.2cm 0cm 0cm 0cm}, clip, width=.18\textwidth]{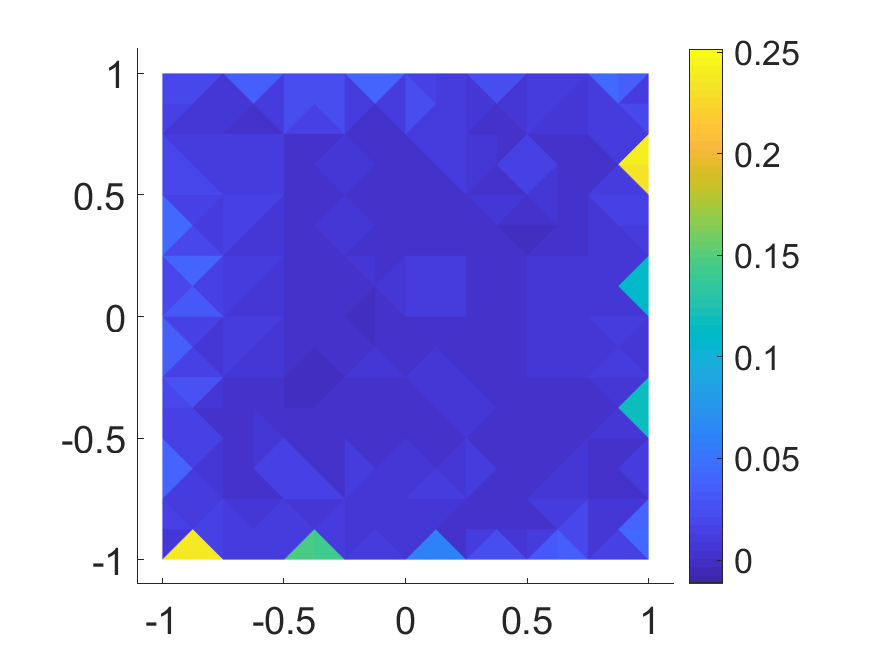} &
    \includegraphics[trim = {.2cm 0cm 0cm 0cm}, clip, width=.18\textwidth]{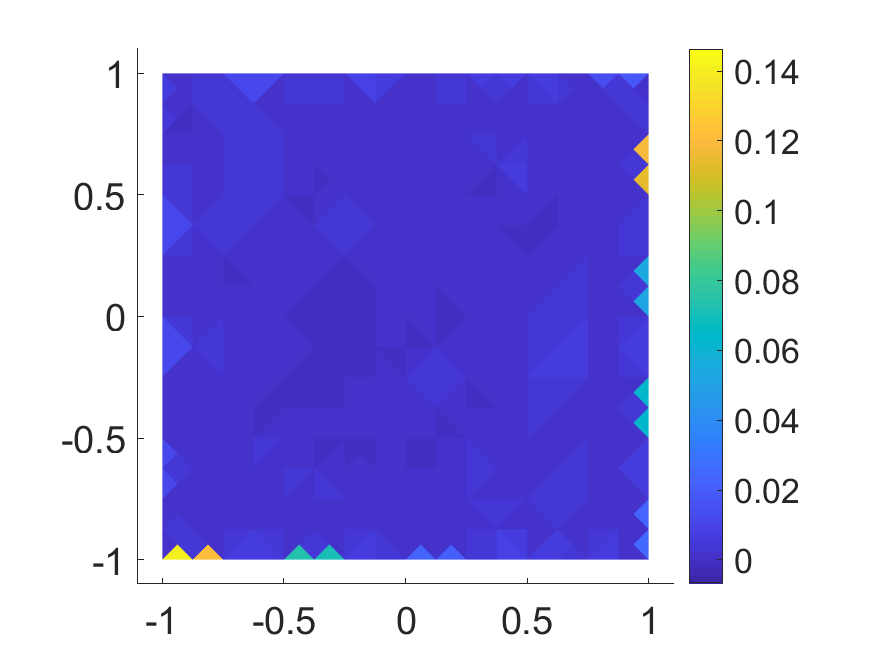} &
    \includegraphics[trim = {.2cm 0cm 0cm 0cm}, clip, width=.18\textwidth]{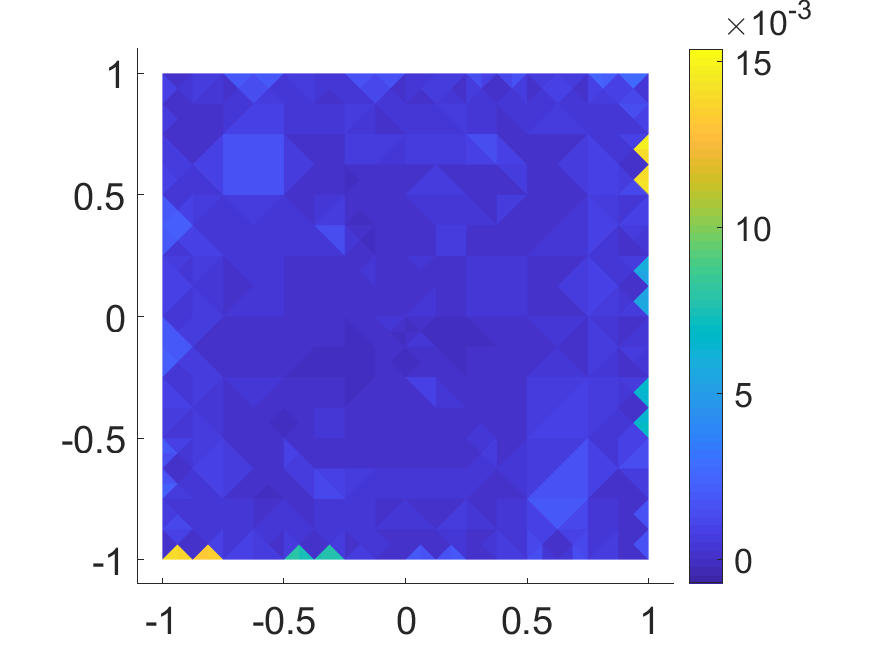}\\
    \includegraphics[trim = {.2cm 0cm 0cm 0cm}, clip, width=.18\textwidth]{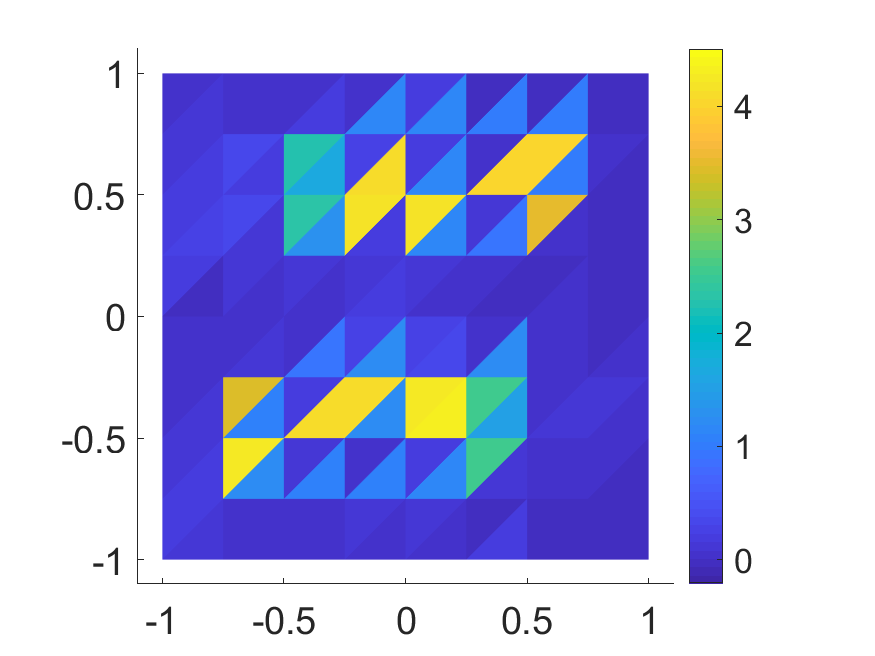} &
    \includegraphics[trim = {.2cm 0cm 0cm 0cm}, clip, width=.18\textwidth]{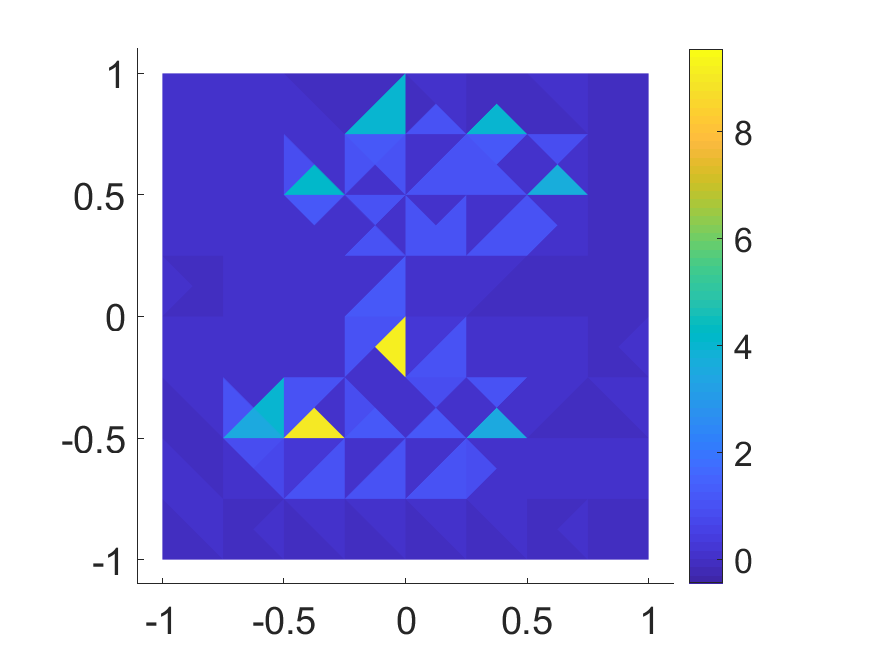} &
    \includegraphics[trim = {.2cm 0cm 0cm 0cm}, clip, width=.18\textwidth]{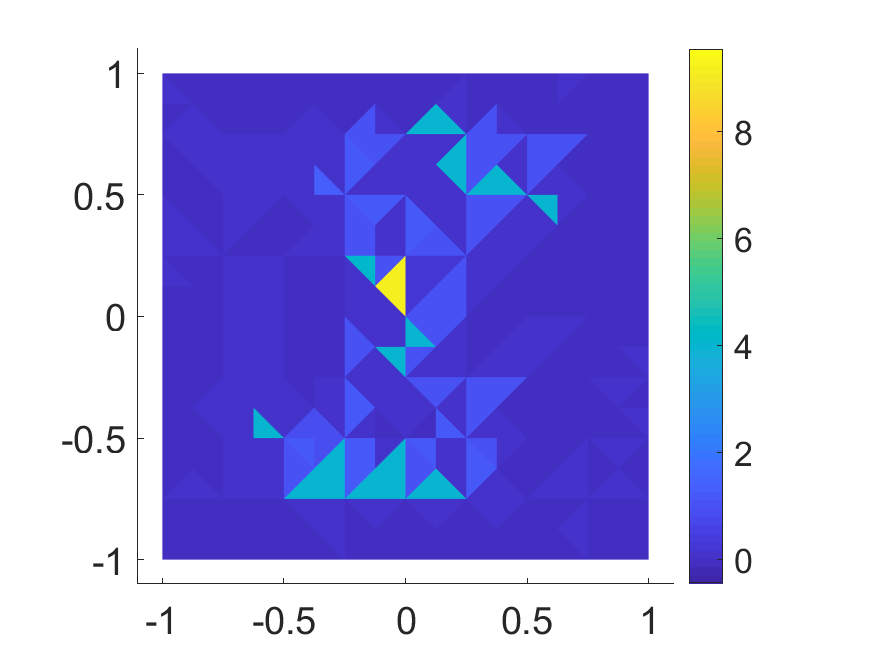} &
    \includegraphics[trim = {.2cm 0cm 0cm 0cm}, clip, width=.18\textwidth]{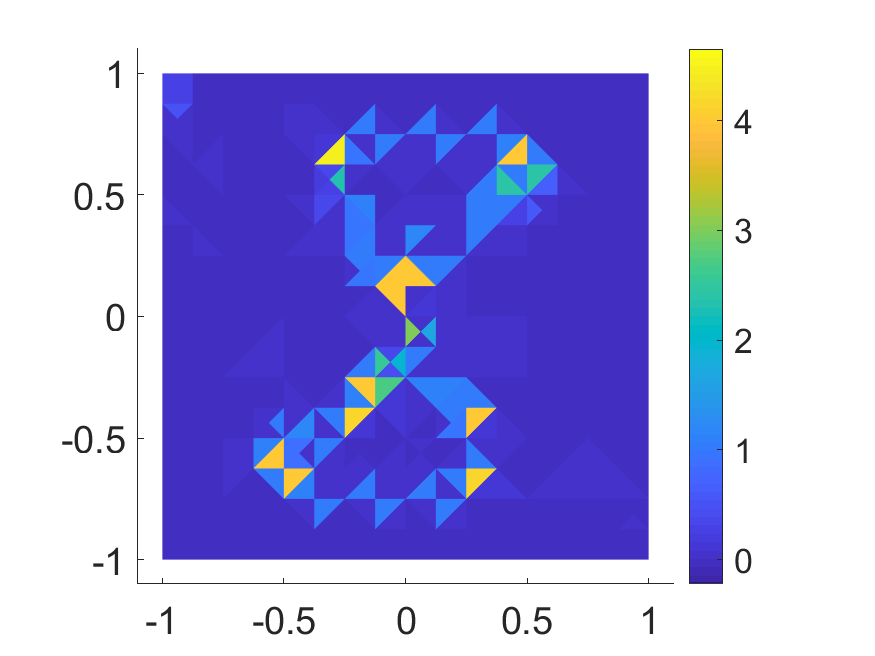} &
    \includegraphics[trim = {.2cm 0cm 0cm 0cm}, clip, width=.18\textwidth]{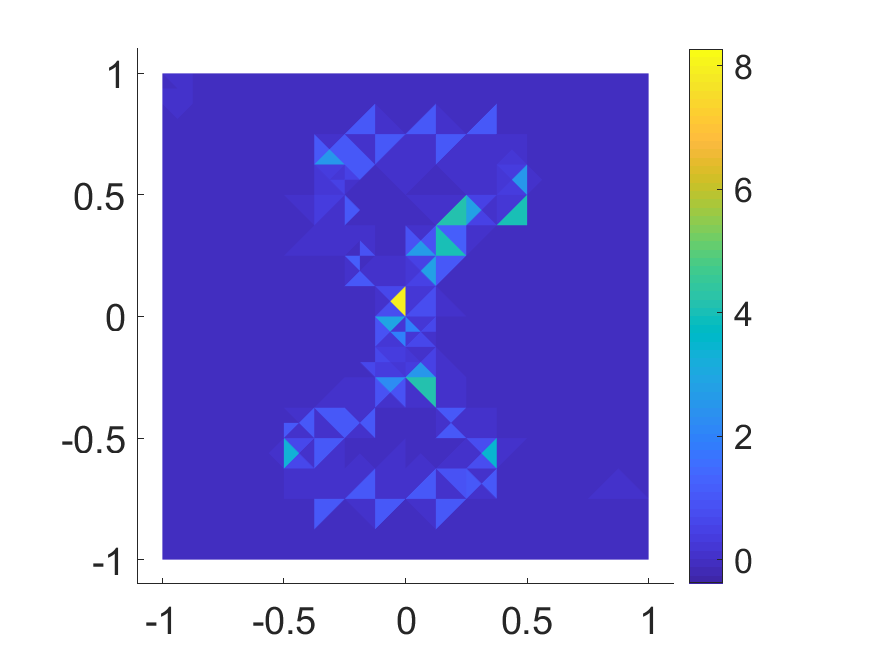}\\
     $k=0$ & $k=1$ & $k=2$ & $k=3$ & $k=4$ \\
    \includegraphics[trim = {.2cm 0cm 0cm 0cm}, clip, width=.18\textwidth]{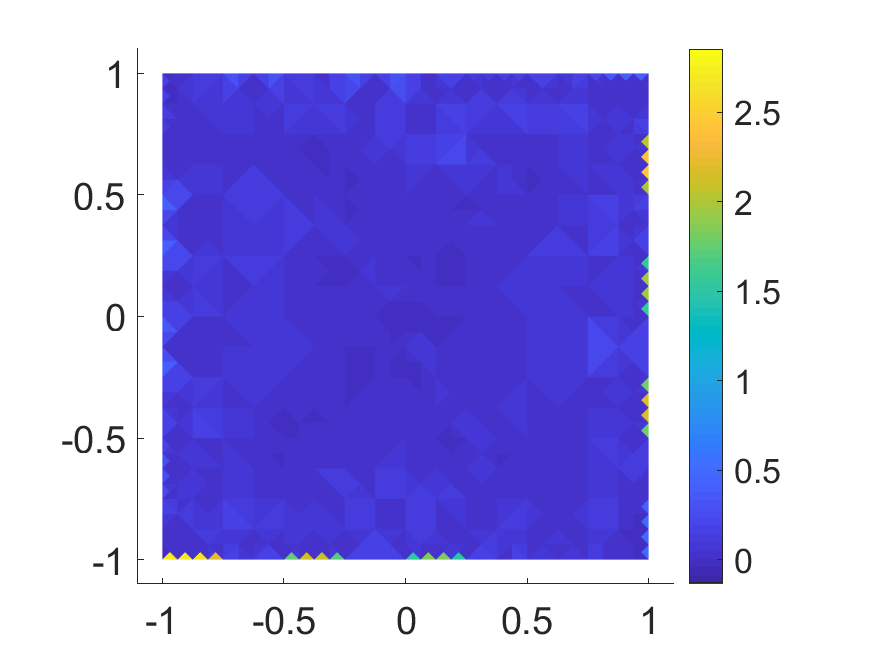} &
    \includegraphics[trim = {.2cm 0cm 0cm 0cm}, clip, width=.18\textwidth]{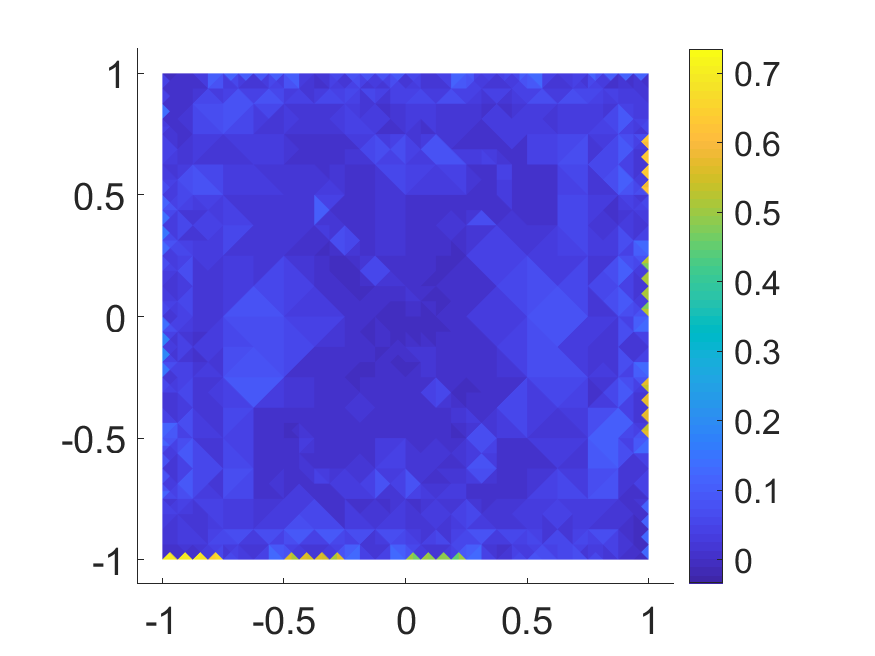} &
    \includegraphics[trim = {.2cm 0cm 0cm 0cm}, clip, width=.18\textwidth]{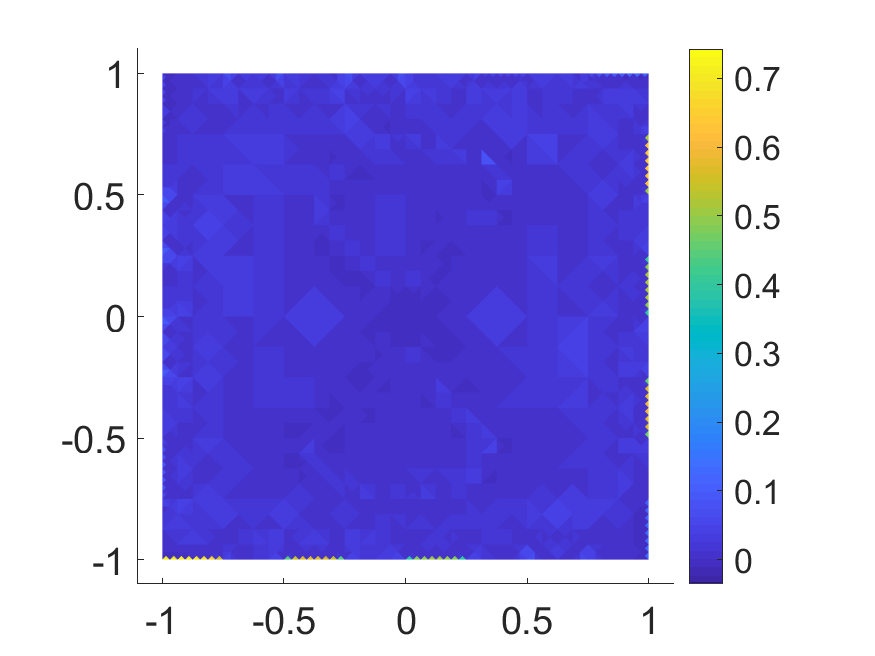} &
    \includegraphics[trim = {.2cm 0cm 0cm 0cm}, clip, width=.18\textwidth]{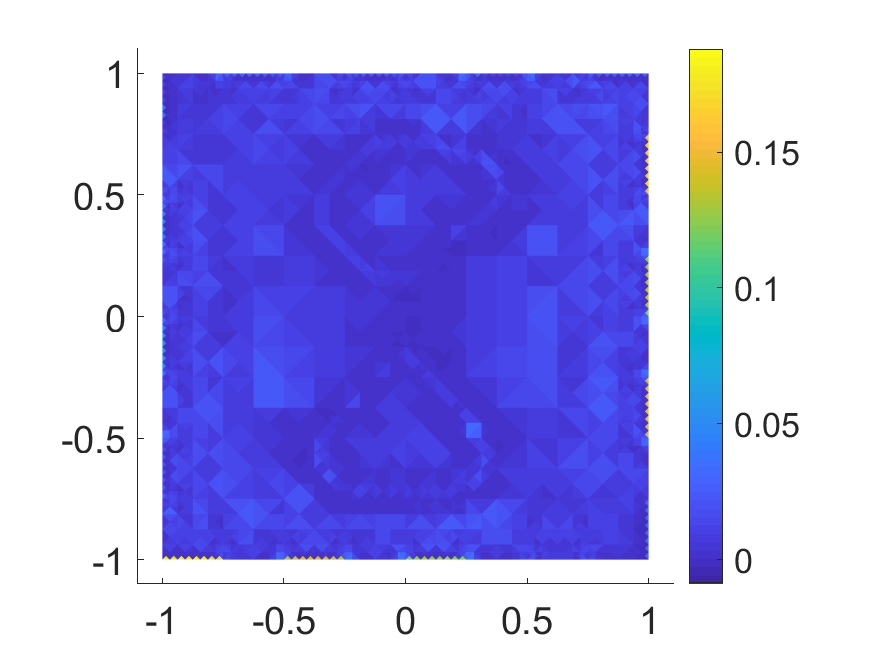} &
    \includegraphics[trim = {.2cm 0cm 0cm 0cm}, clip, width=.18\textwidth]{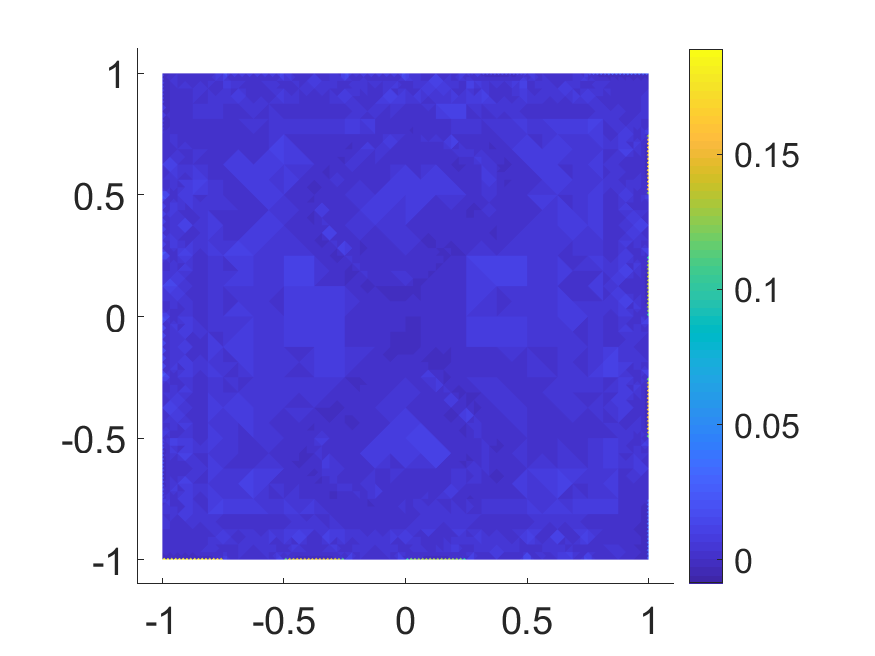}\\
    \includegraphics[trim = {.2cm 0cm 0cm 0cm}, clip, width=.18\textwidth]{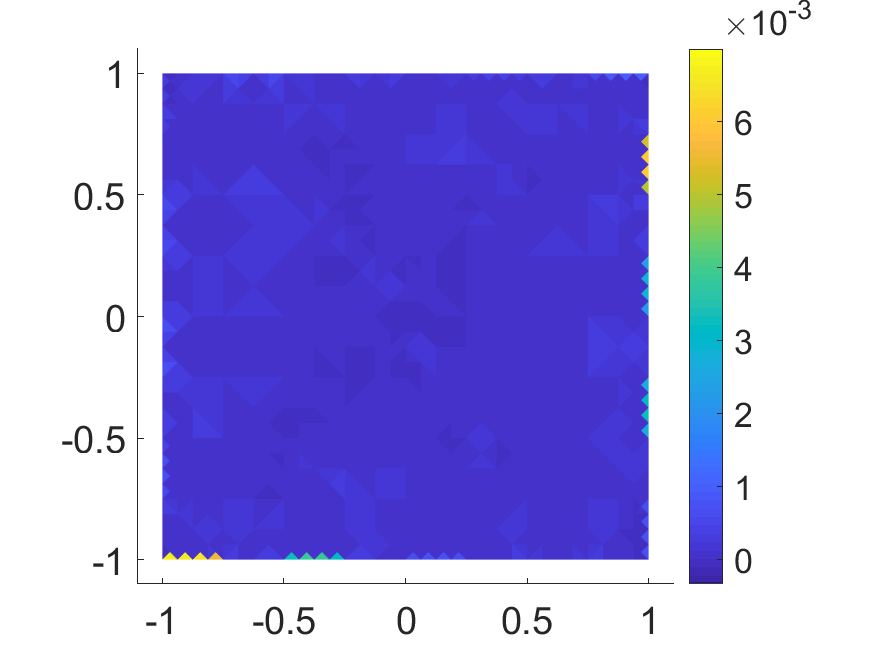} &
    \includegraphics[trim = {.2cm 0cm 0cm 0cm}, clip, width=.18\textwidth]{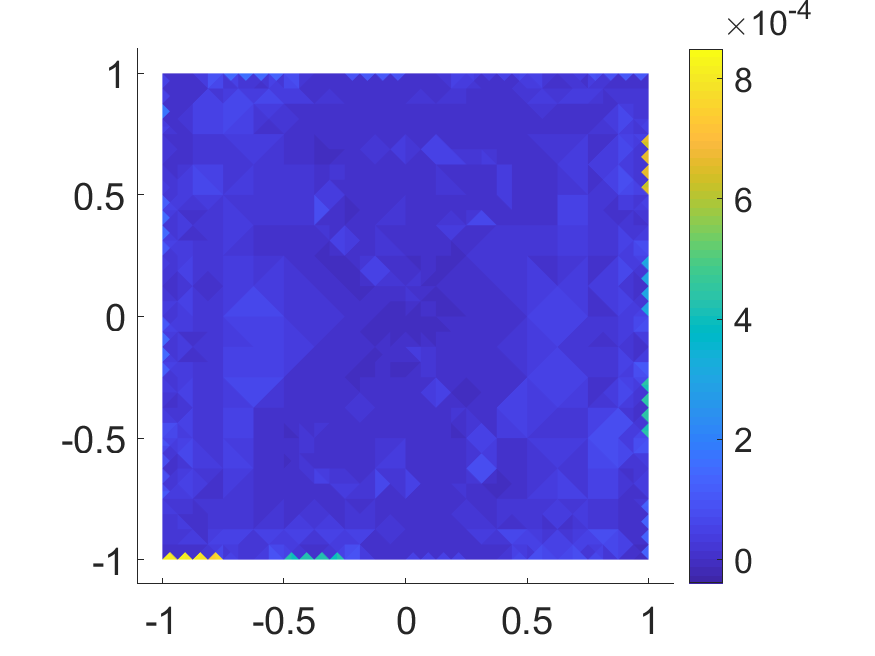} &
    \includegraphics[trim = {.2cm 0cm 0cm 0cm}, clip, width=.18\textwidth]{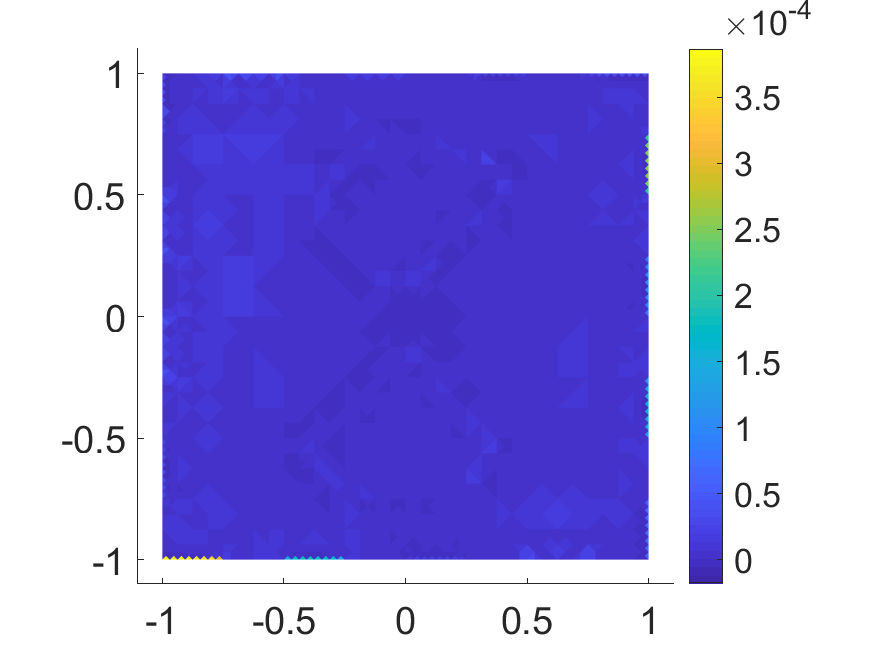} &
    \includegraphics[trim = {.2cm 0cm 0cm 0cm}, clip, width=.18\textwidth]{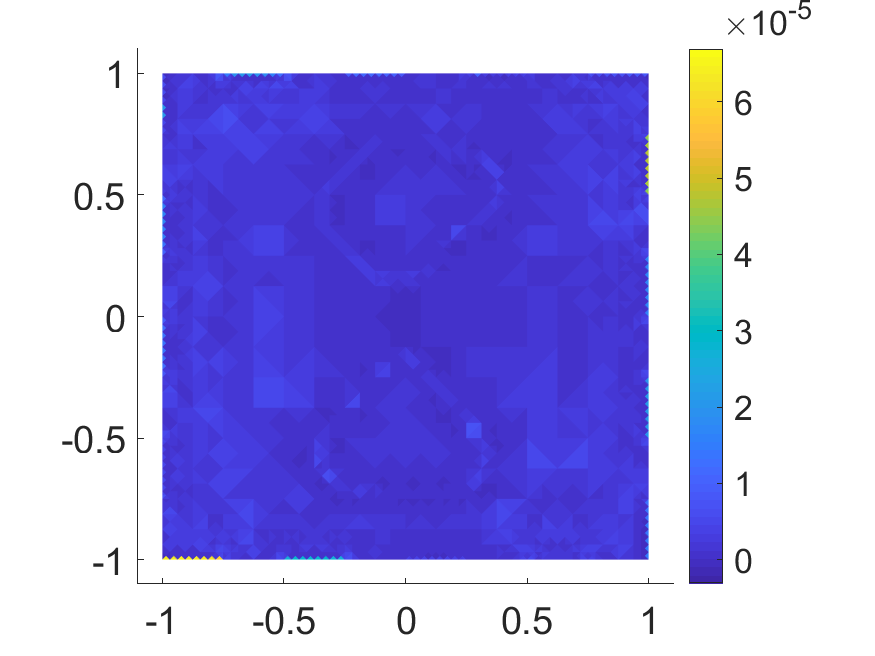} &
    \includegraphics[trim = {.2cm 0cm 0cm 0cm}, clip, width=.18\textwidth]{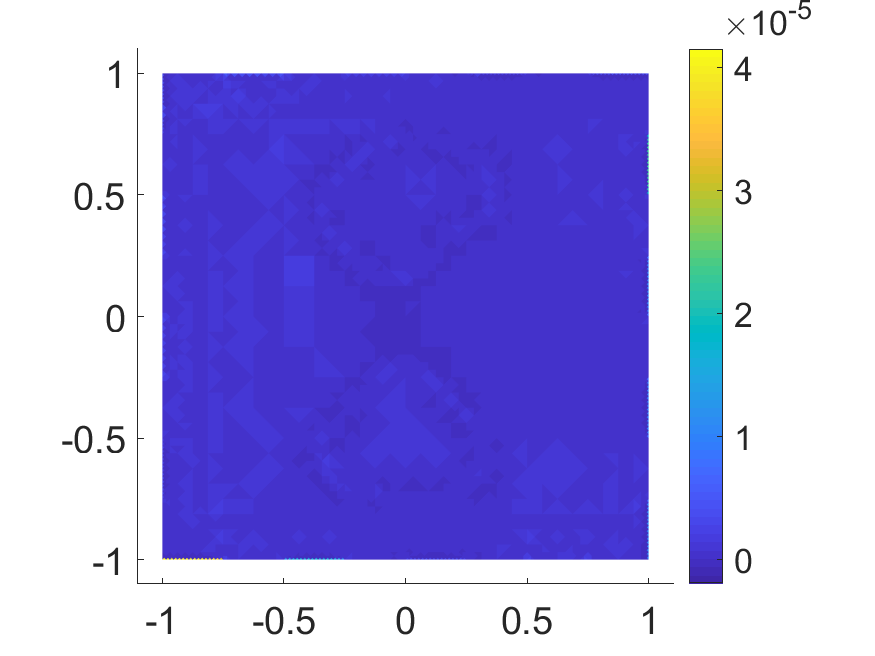}\\
    \includegraphics[trim = {.2cm 0cm 0cm 0cm}, clip, width=.18\textwidth]{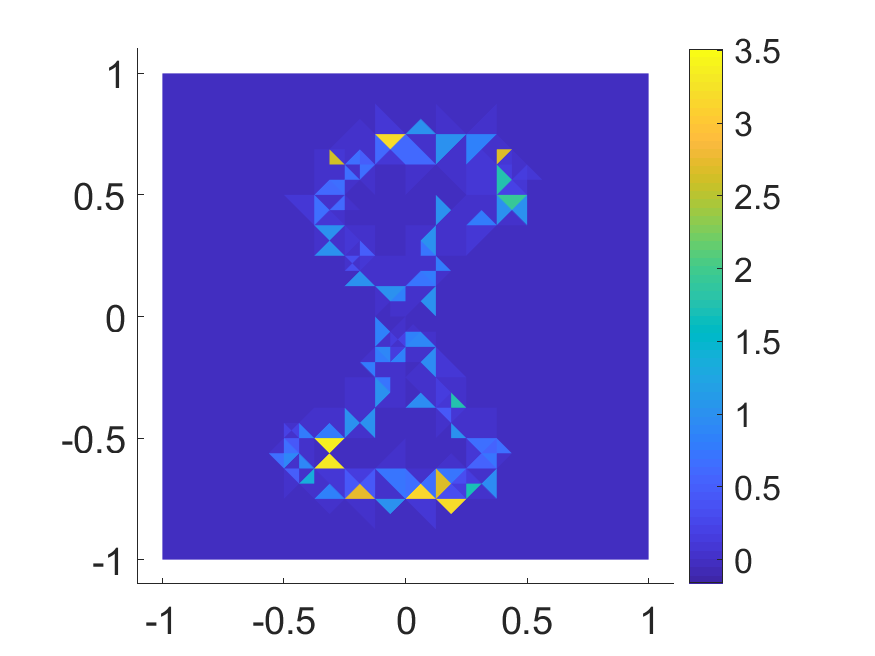} &
    \includegraphics[trim = {.2cm 0cm 0cm 0cm}, clip, width=.18\textwidth]{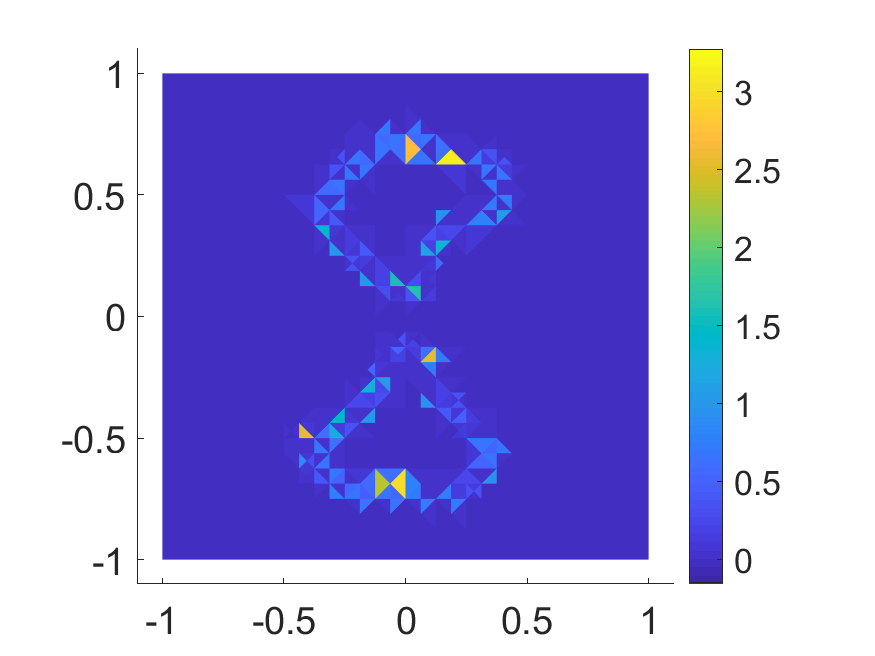} &
    \includegraphics[trim = {.2cm 0cm 0cm 0cm}, clip, width=.18\textwidth]{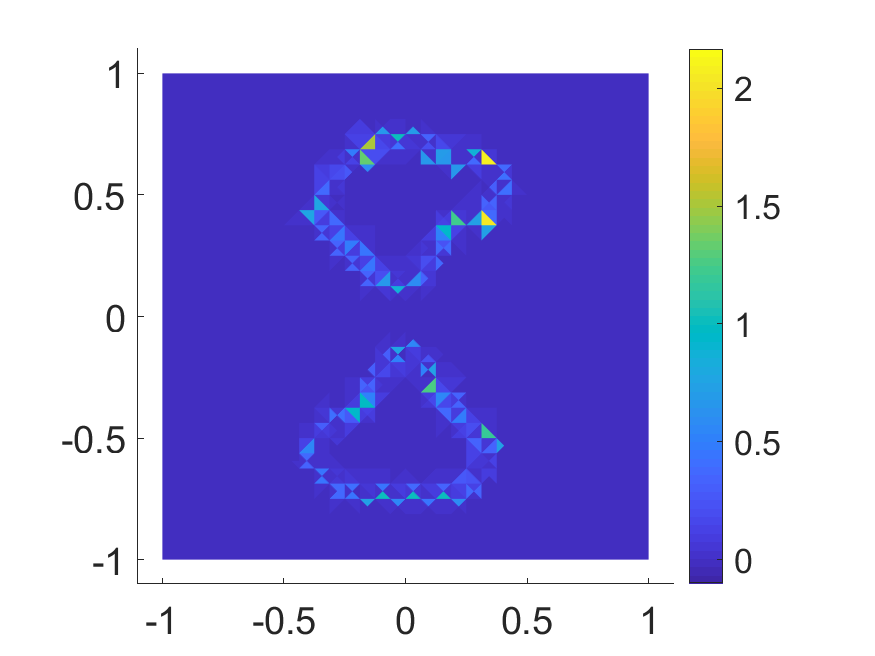} &
    \includegraphics[trim = {.2cm 0cm 0cm 0cm}, clip, width=.18\textwidth]{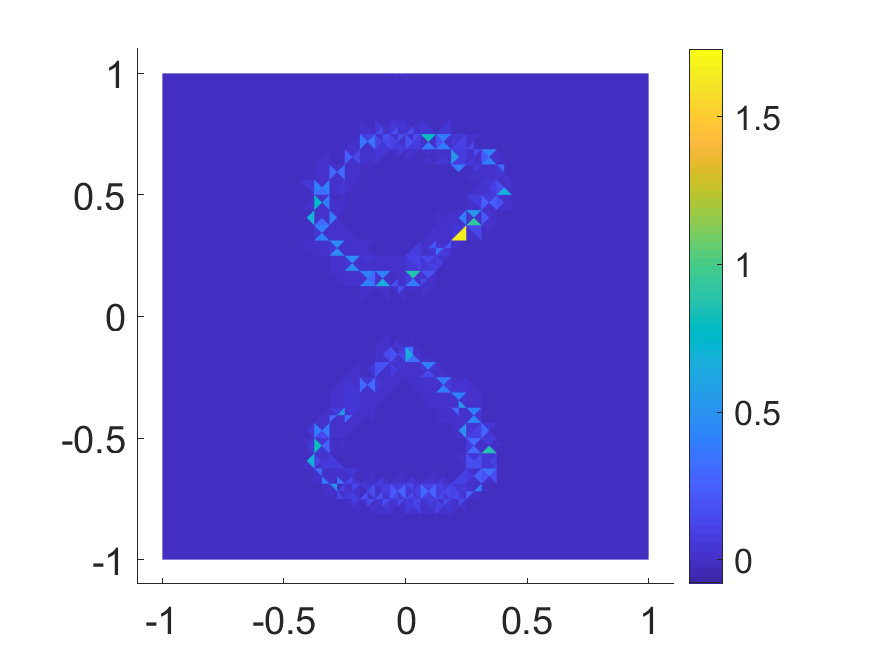} &
    \includegraphics[trim = {.2cm 0cm 0cm 0cm}, clip, width=.18\textwidth]{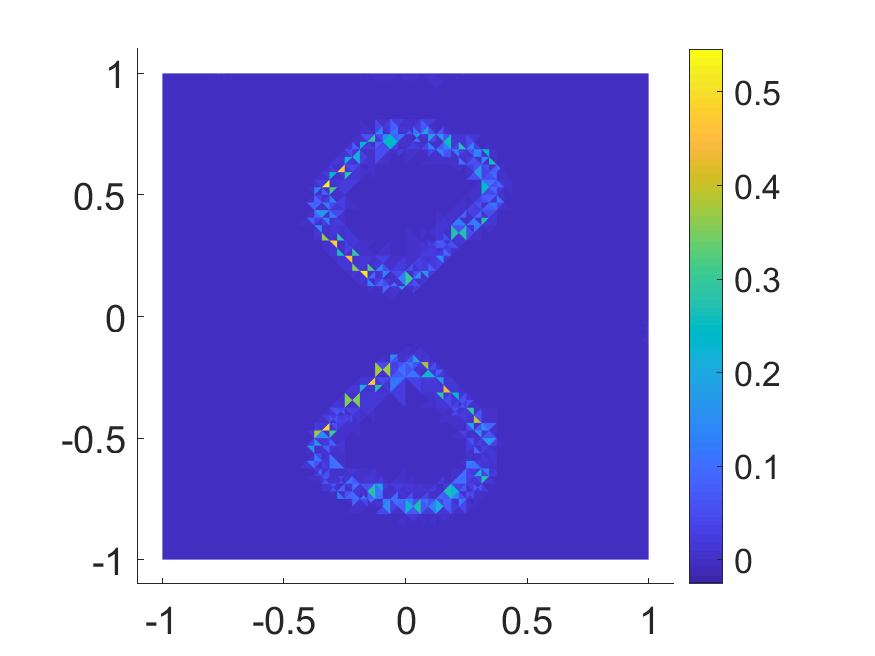}\\
    $k=5$ & $k=6$ & $k=7$ & $k=8$ & $k=9$
  \end{tabular}
  \caption{The evolution of the three error indicators $\eta_{k,i}^2$ for $k=0,1,\ldots,9$, $i=1$ (top), $i=2$ (middle) and $i=3$ (bottom), for
  Example \ref{exam2}(i) with $\epsilon=\text{1e-3}$  and $\tilde\alpha=\text{2e-2}$.}\label{fig:exam2i_err-ind}
\end{figure}

\begin{figure}[hbt!]
  \centering
  \begin{tabular}{cc}
    \includegraphics[trim = .5cm 0cm 0cm 0cm, clip=true,width=.23\textwidth]{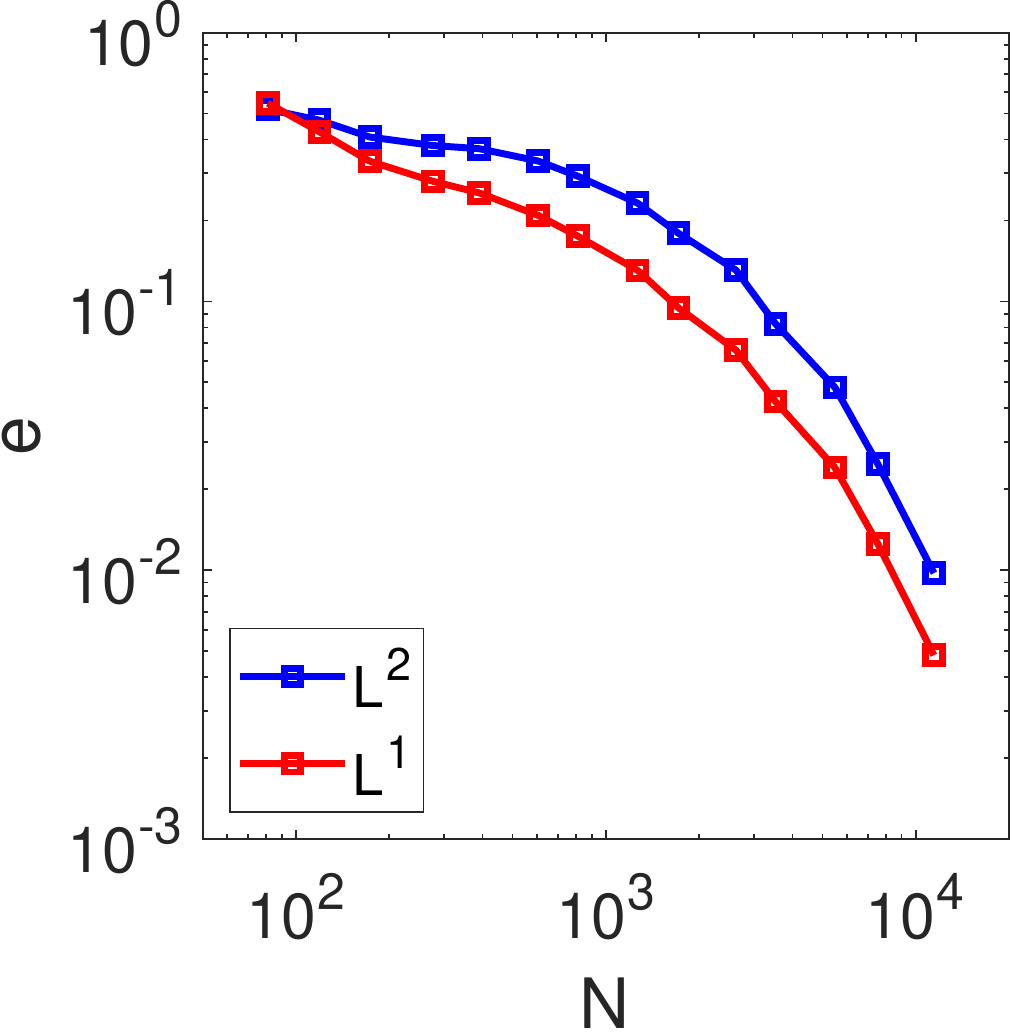}
     \includegraphics[trim = .5cm 0cm 0cm 0cm, clip=true,width=.23\textwidth]{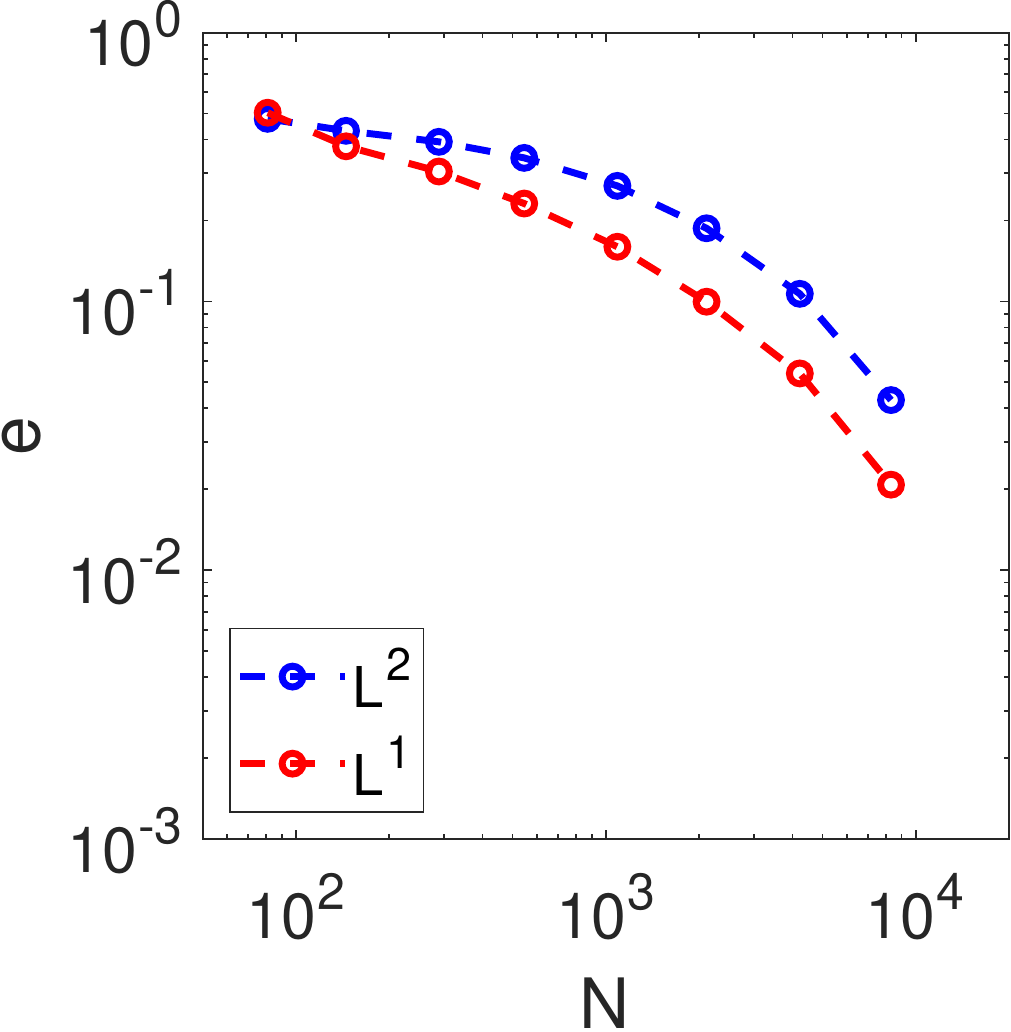}&
    \includegraphics[trim = .5cm 0cm 0cm 0cm, clip=true,width=.23\textwidth]{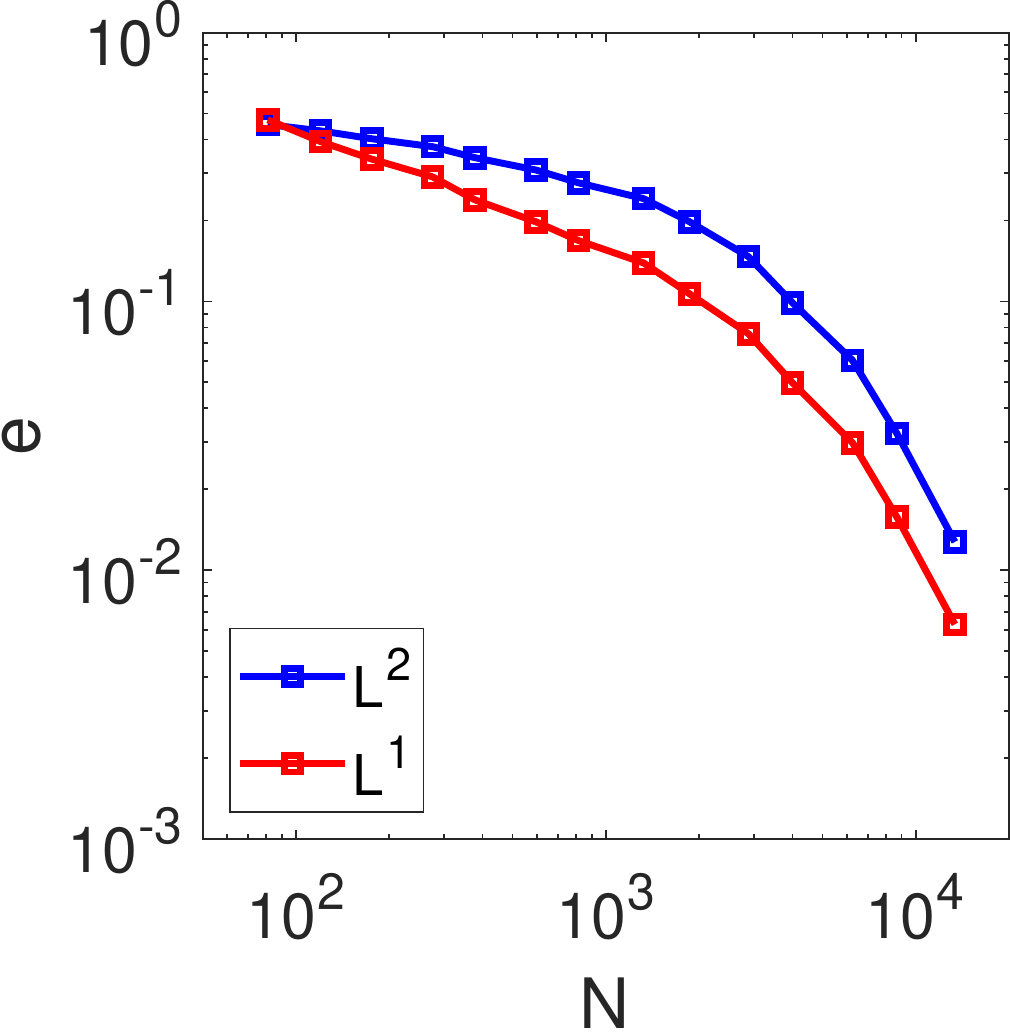}
     \includegraphics[trim = .5cm 0cm 0cm 0cm, clip=true,width=.23\textwidth]{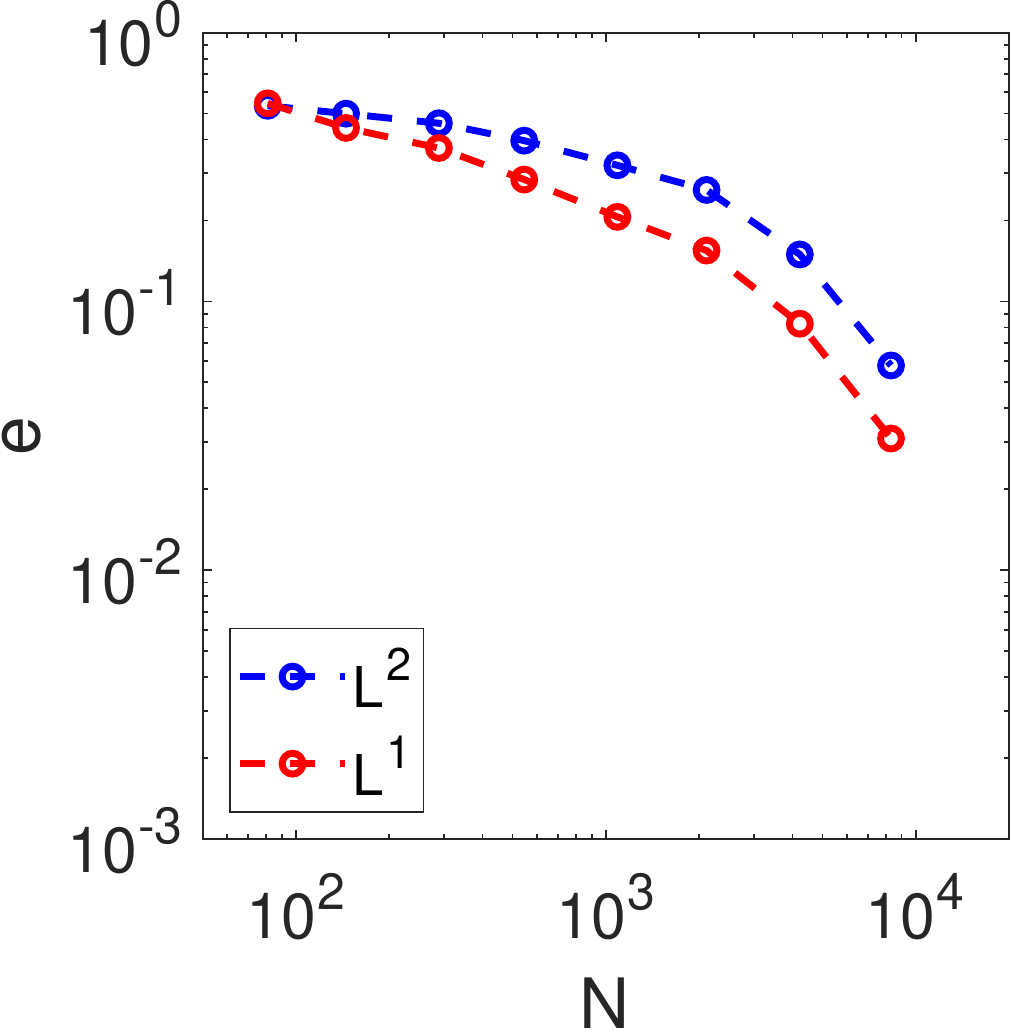}\\
    (a) $\epsilon=\text{1e-3}$, $\tilde\alpha=\text{2e-2}$ & (b) $\epsilon=\text{1e-2}$, $\tilde\alpha=\text{3e-2}$
  \end{tabular}
  \caption{The $L^2(\Omega)$ and $L^1(\Omega)$ errors versus d.o.f. $N$ of the mesh, for Example \ref{exam2}(i),
  using the adaptive (solid) and uniform (dashed) refinement.}\label{fig:exam2-efficiency}
\end{figure}

In Fig. \ref{fig:exam2-efficiency}, we plot the $L^2(\Omega)$ and $L^1(\Omega)$ errors of the recoveries
versus d.o.f. $N$, where the recovery on the corresponding finest mesh is taken as the reference (since
the recoveries by the adaptive and uniform refinements are slightly different; see Fig. \ref{fig:exam2i-recon}).
Due to the discontinuity of the sought-for conductivity, the $L^1(\Omega)$ norm is especially suitable for
measuring the convergence. The convergence of the algorithm is clearly observed for both adaptive and uniform
refinements. Further, with a fixed d.o.f., AFEM gives more accurate results than the uniform one in both
error metrics. These observations show the computational efficiency of the adaptive algorithm.

Examples \ref{exam2}(ii) and (iii) are variations of Example \ref{exam2}(i), and the results are presented in
Figs. \ref{fig:exam2ii-recon}--\ref{fig:exam2iii-recon-iter-1e3}. The proposed approach assumes a piecewise
constant conductivity with known lower and upper bounds. Example \ref{exam2}(ii) does not fulfill the assumption,
since the true conductivity $\sigma^\dag$ is not piecewise constant. Thus the algorithm can only produce a piecewise
constant approximation to the exact one. Nonetheless, the inclusion support is reasonably identified. When the
noise level $\epsilon$ increases from $\text{1e-3}$ to $\text{1e-2}$, the reconstruction accuracy deteriorates
significantly; see Fig. \ref{fig:exam2ii-recon}. Example \ref{exam2}(iii) involves high contrast inclusions,
which are well known to be numerically more challenging. This is clearly observed in Fig. \ref{fig:exam2iii-recon},
where the recovery accuracy is inferior, especially for the noise level $\epsilon=\text{1e-2}$. However,
the adaptive refinement procedure works well similarly as the preceding examples: the refinement occurs
mainly around the electrode edges and inclusion interface; see Figs. \ref{fig:exam2ii-recon-iter-1e3}
and \ref{fig:exam2iii-recon-iter-1e3} for the details.

\begin{figure}[hbt!]
  \centering\setlength{\tabcolsep}{0em}
  \begin{tabular}{ccccc}
      \includegraphics[trim = .5cm 0cm 1cm 0cm, clip=true,width=0.2\textwidth]{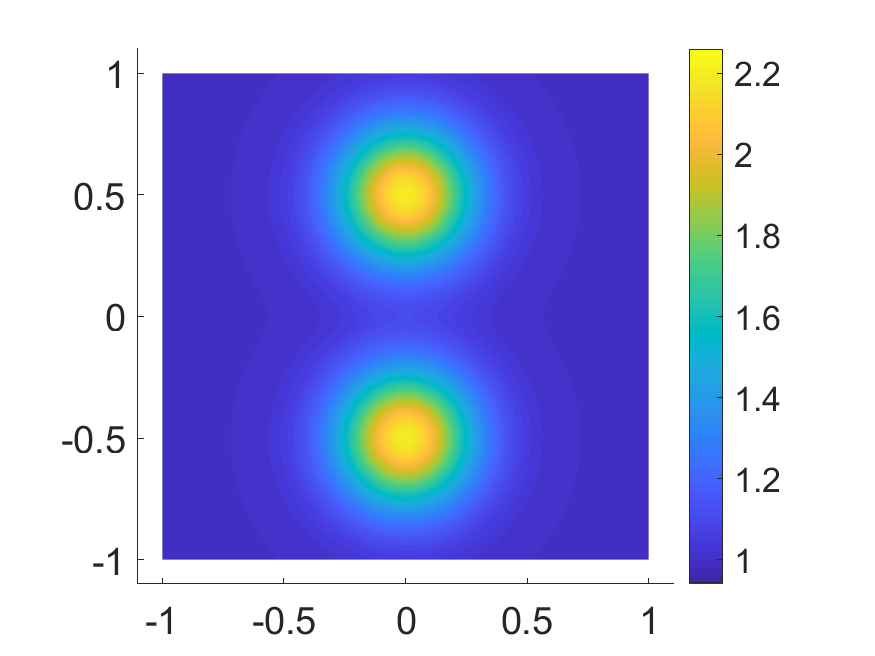}
    & \includegraphics[trim = .5cm 0cm 1cm 0cm, clip=true,width=.2\textwidth]{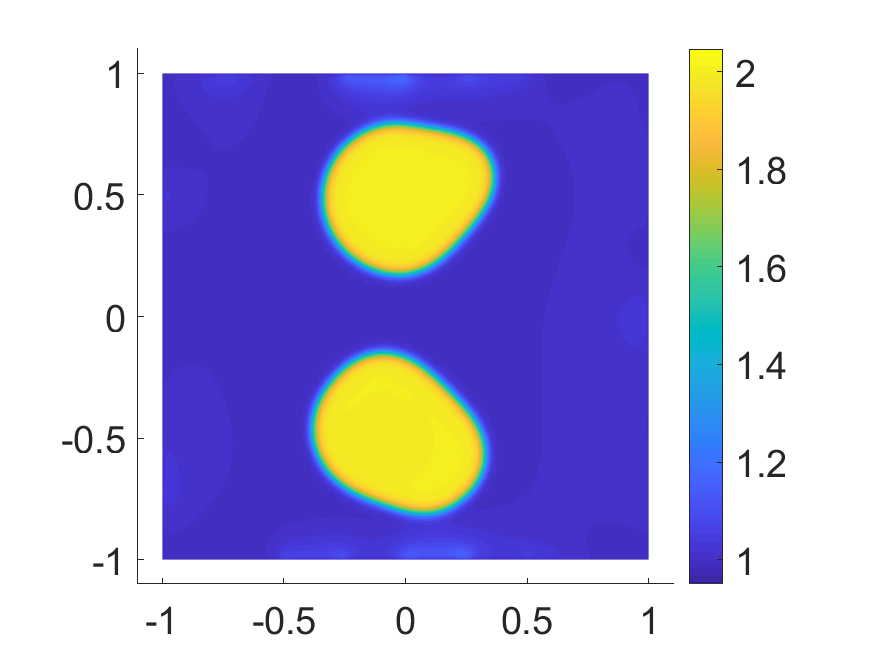}
    & \includegraphics[trim = .5cm 0cm 1cm 0cm, clip=true,width=.2\textwidth]{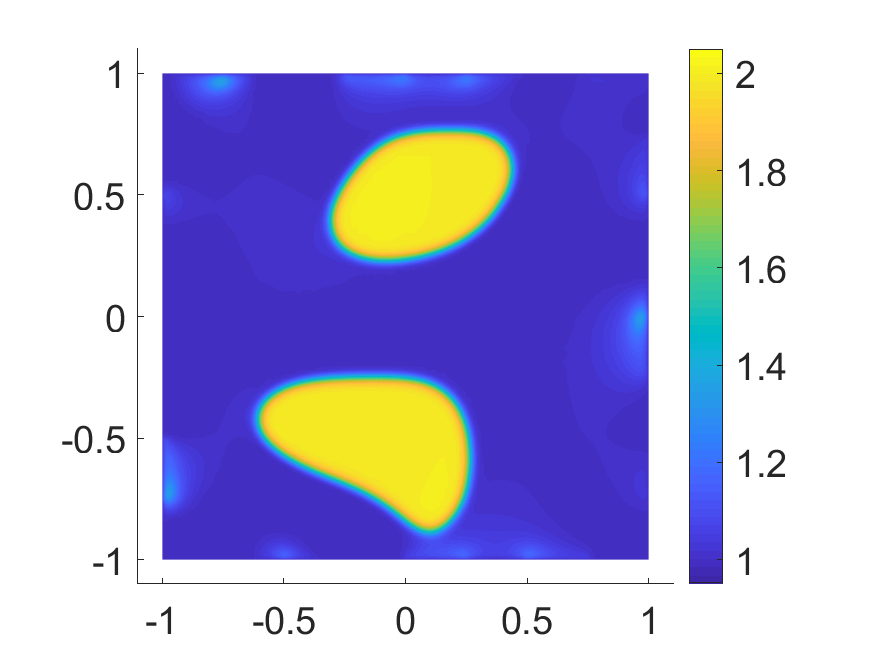}
    & \includegraphics[trim = .5cm 0cm 1cm 0cm, clip=true,width=.2\textwidth]{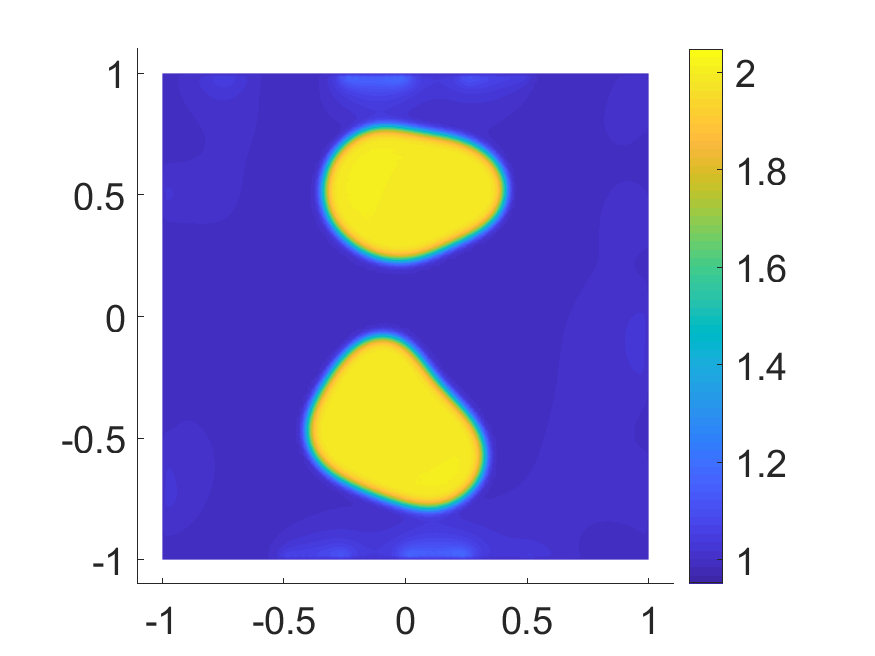}
    & \includegraphics[trim = .5cm 0cm 1cm 0cm, clip=true,width=.2\textwidth]{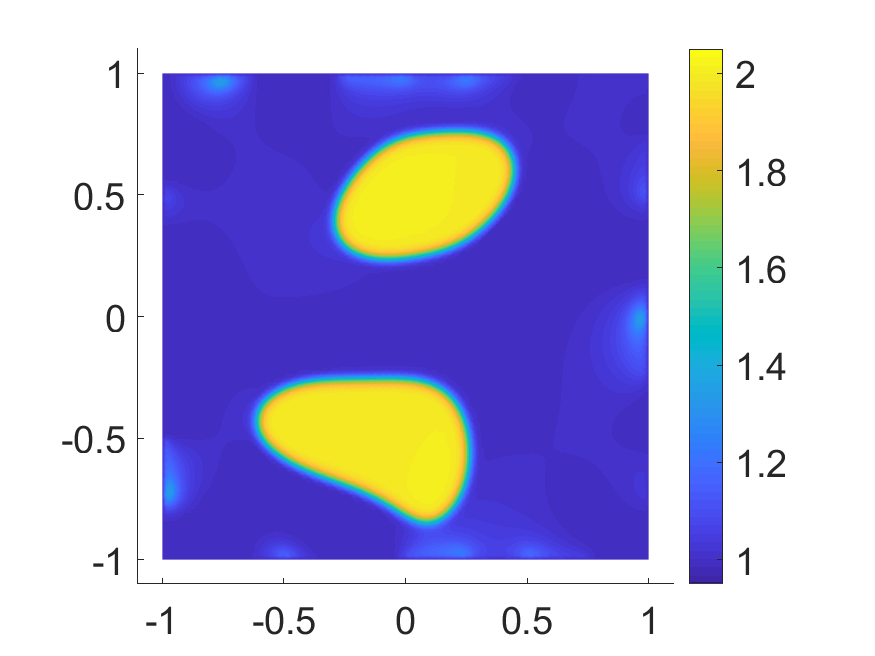}\\
    (a) true conductivity & (b) adaptive & (c) adaptive & (d) uniform & (e) uniform
  \end{tabular}
  \caption{The final recoveries by the adaptive and uniform refinements for Example
  \ref{exam2}(ii). The results in (b) and (d) are for $\epsilon=\text{1e-3}$ and $\tilde\alpha=\text{2e-2}$,
  and (c) and (e) for $\epsilon=\text{1e-2}$ and $\tilde\alpha=\text{5e-2}$. The d.o.f. of (b), (c), (d) and (e) are 17736, 20524,
  16641 and 16641.}\label{fig:exam2ii-recon}
\end{figure}

\begin{figure}[hbt!]
 \centering
 \setlength{\tabcolsep}{0pt}
 \begin{tabular}{cccccccc}
 \includegraphics[trim = {2.5cm 1.5cm 2.5cm 1.2cm}, clip, width=.2\textwidth]{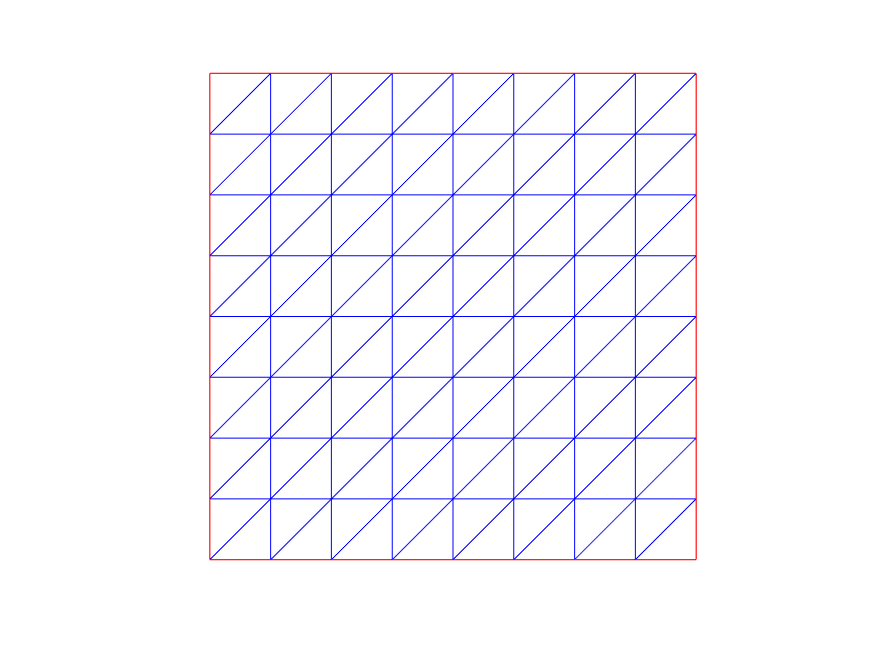}&
 \includegraphics[trim = {2.5cm 1.5cm 2.5cm 1.2cm}, clip, width=.2\textwidth]{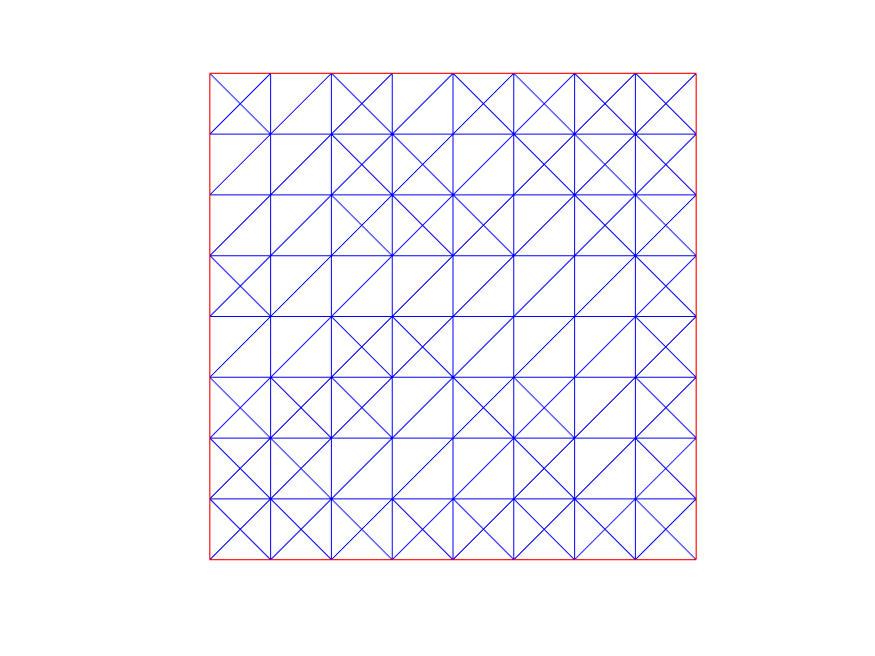}&
 \includegraphics[trim = {2.5cm 1.5cm 2.5cm 1.2cm}, clip, width=.2\textwidth]{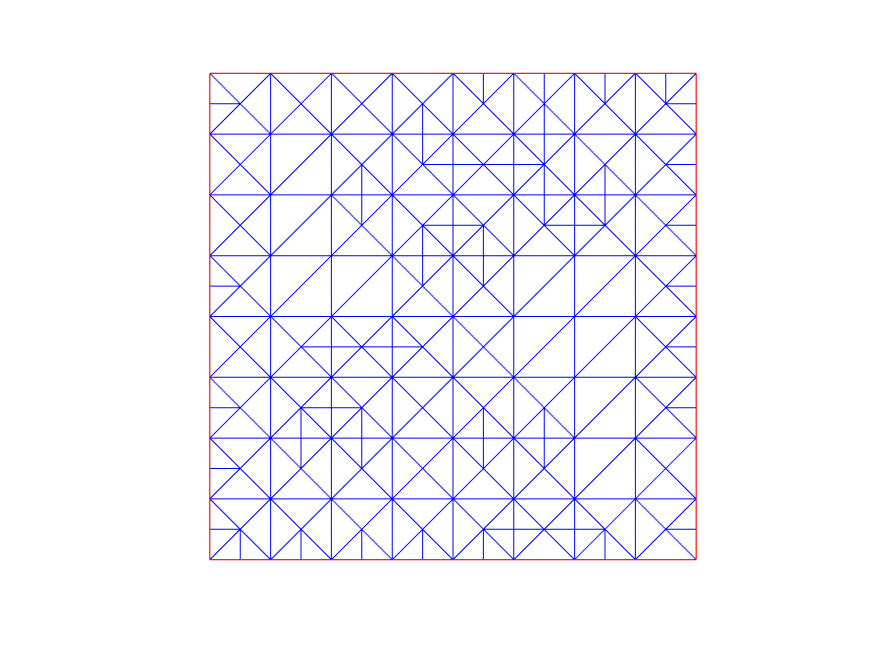}&
 \includegraphics[trim = {2.5cm 1.5cm 2.5cm 1.2cm}, clip, width=.2\textwidth]{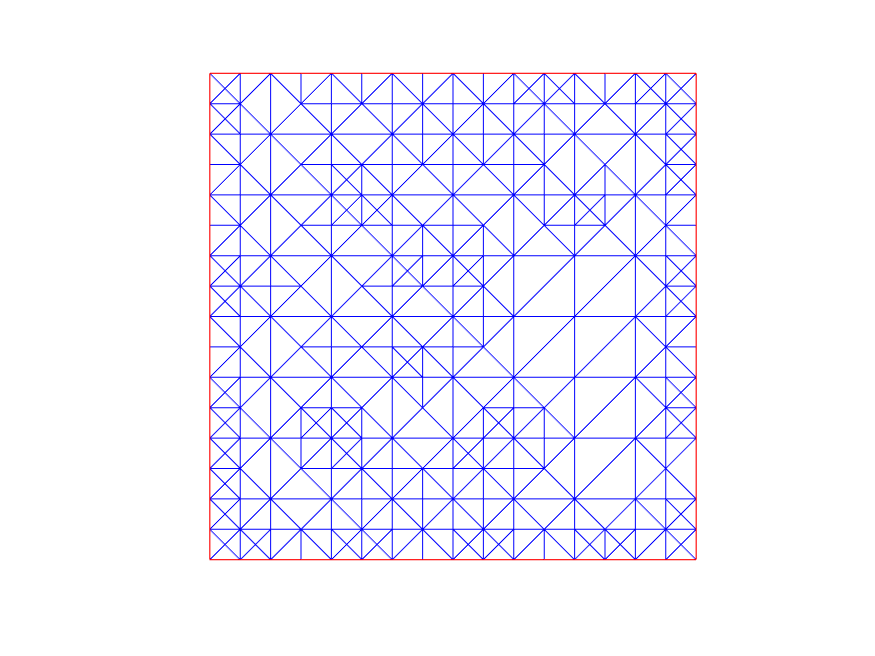}&
 \includegraphics[trim = {2.5cm 1.5cm 2.5cm 1.2cm}, clip, width=.2\textwidth]{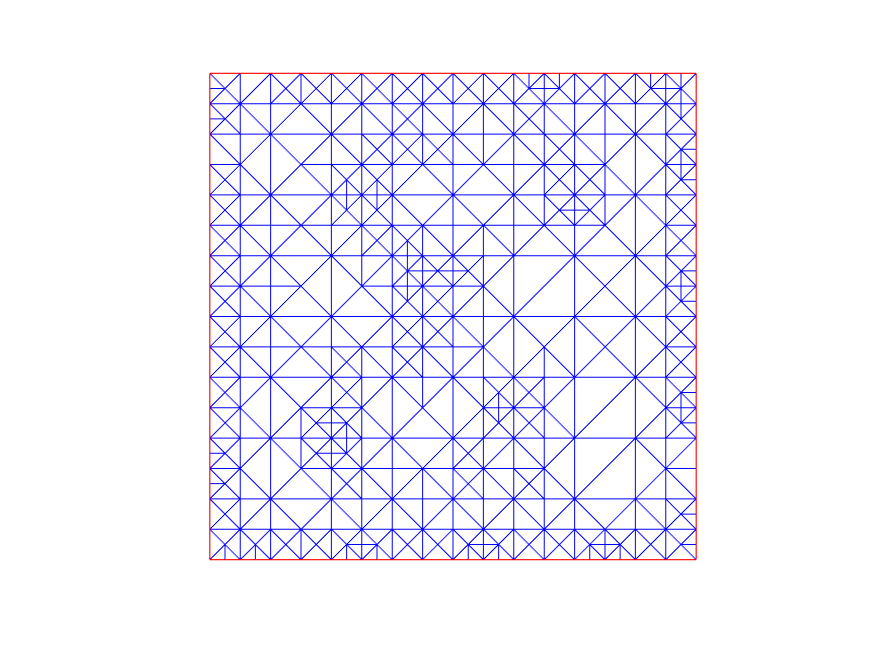}\\
 \includegraphics[trim = 1cm 0.5cm 0.5cm 0cm, width=.2\textwidth]{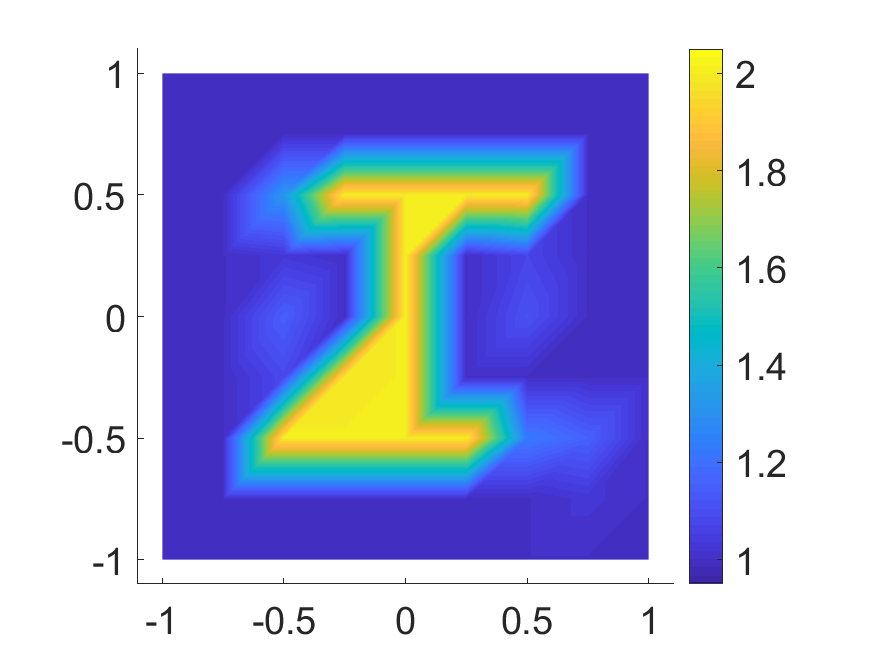}&
 \includegraphics[trim = 1cm 0.5cm 0.5cm 0cm, width=.2\textwidth]{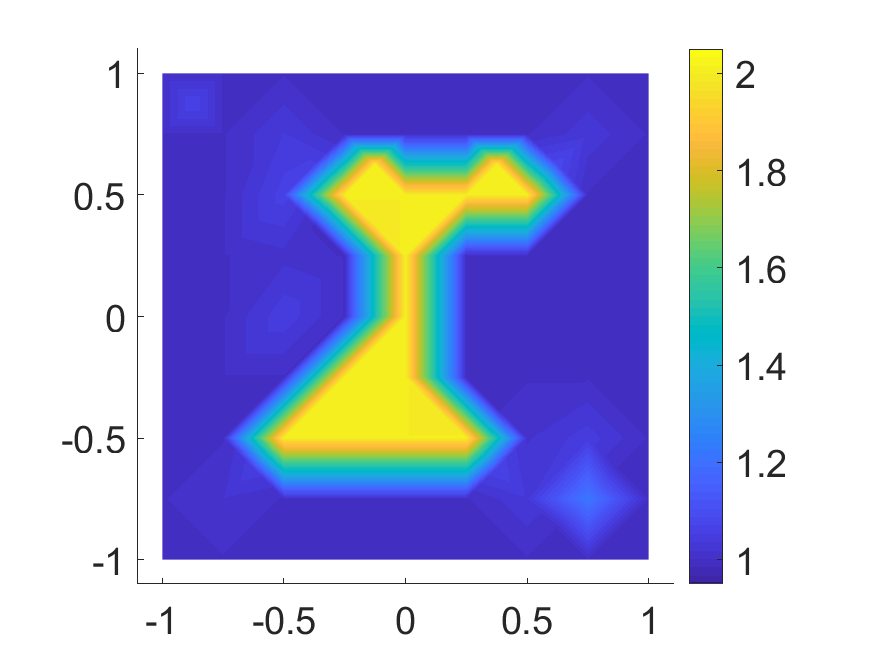}&
 \includegraphics[trim = 1cm 0.5cm 0.5cm 0cm, width=.2\textwidth]{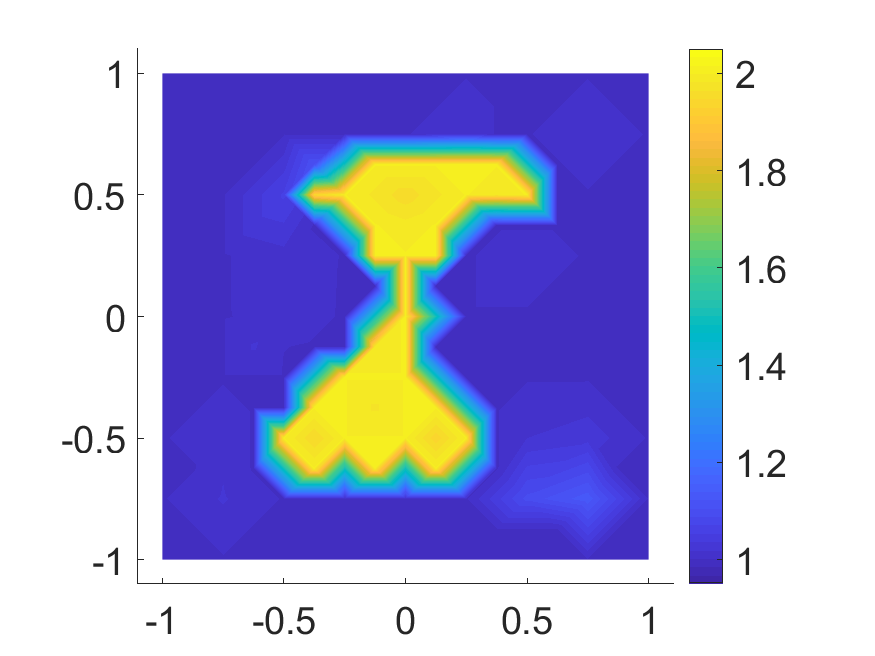}&
 \includegraphics[trim = 1cm 0.5cm 0.5cm 0cm, width=.2\textwidth]{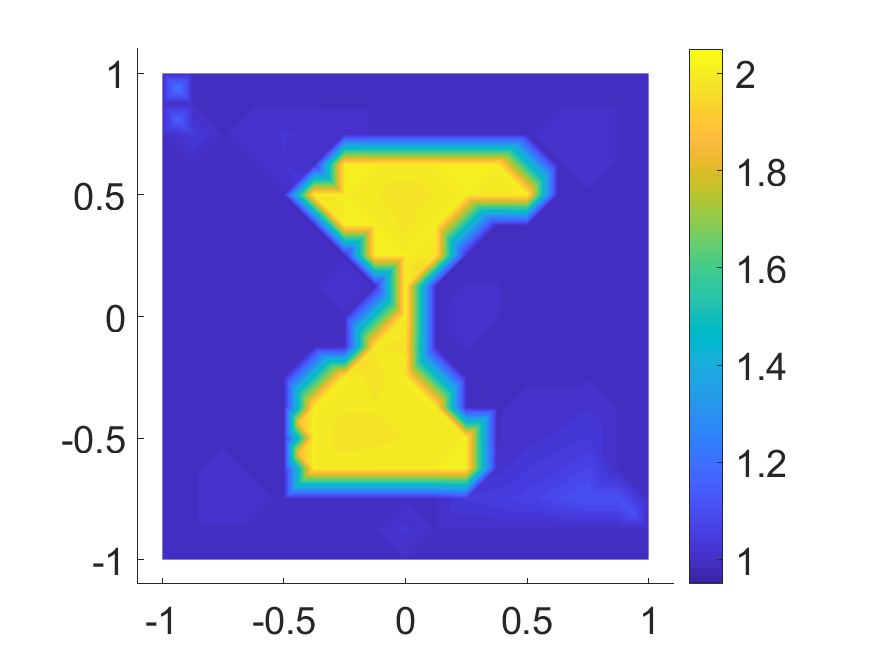}&
 \includegraphics[trim = 1cm 0.5cm 0.5cm 0cm, width=.2\textwidth]{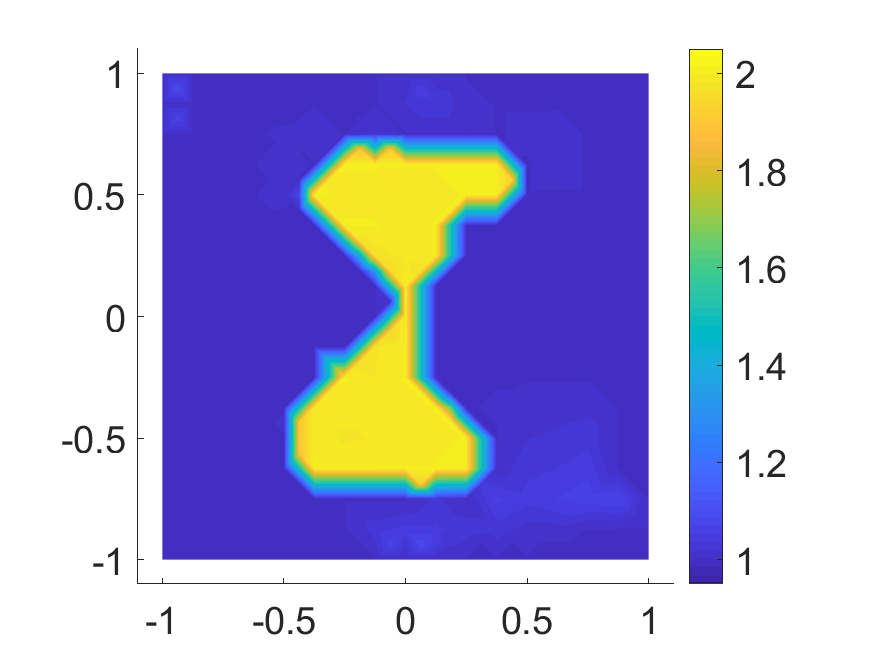}\\
  81    &     120  &  177  &  276  & 389   \\
 \includegraphics[trim = {2.5cm 1.5cm 2.5cm 1.2cm}, clip, width=.2\textwidth]{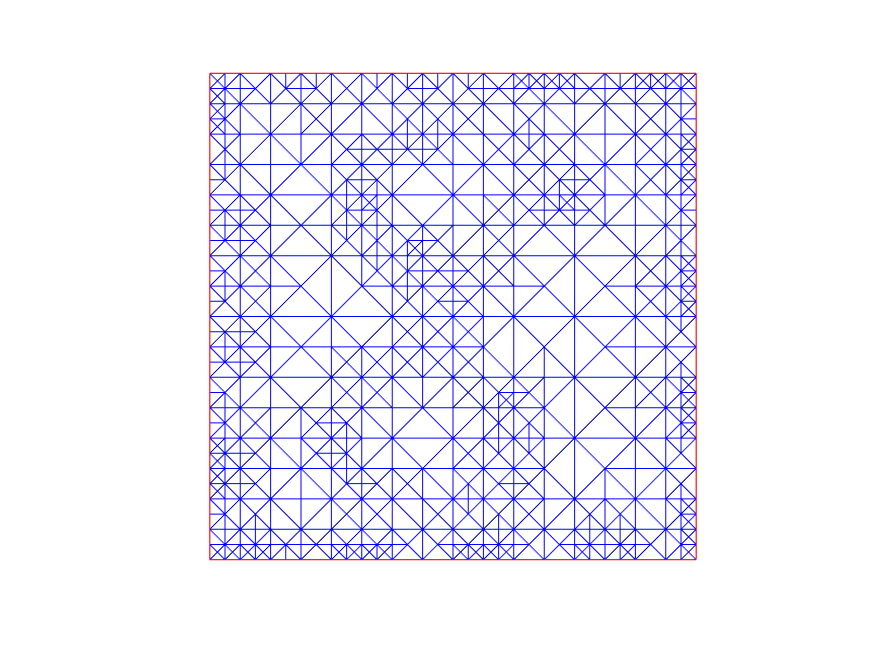}&
 \includegraphics[trim = {2.5cm 1.5cm 2.5cm 1.2cm}, clip, width=.2\textwidth]{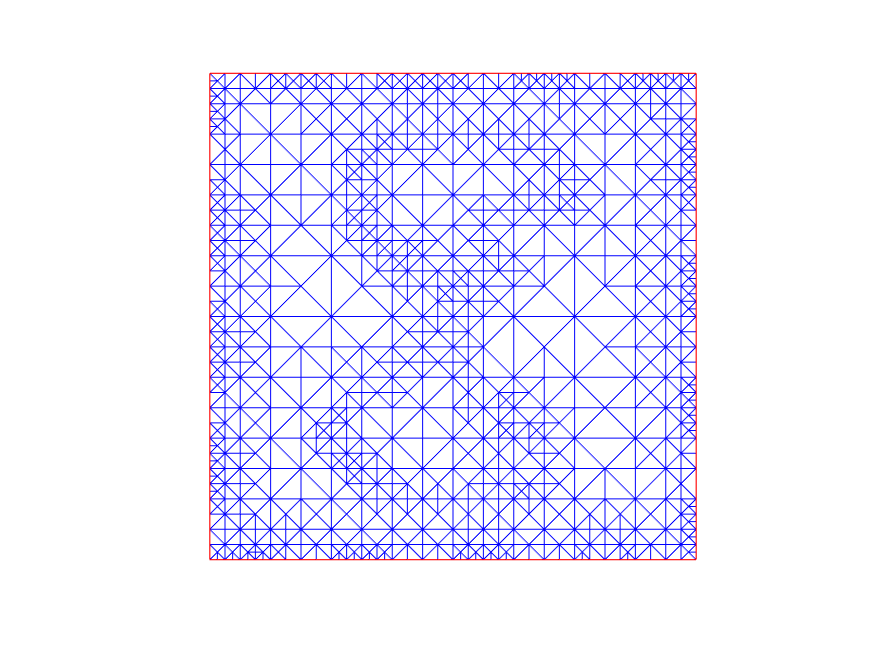}&
 \includegraphics[trim = {2.5cm 1.5cm 2.5cm 1.2cm}, clip, width=.2\textwidth]{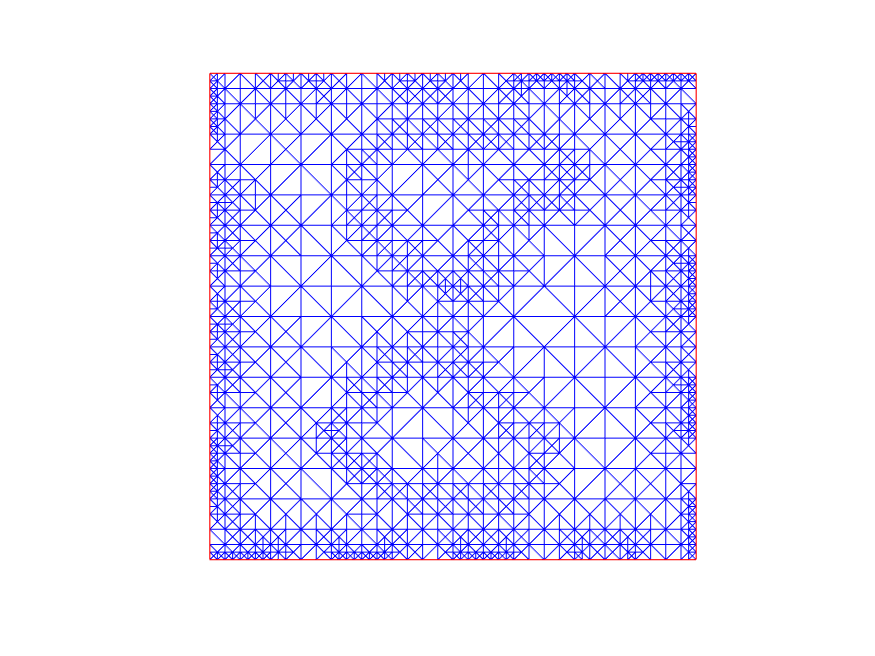}&
 \includegraphics[trim = {2.5cm 1.5cm 2.5cm 1.2cm}, clip, width=.2\textwidth]{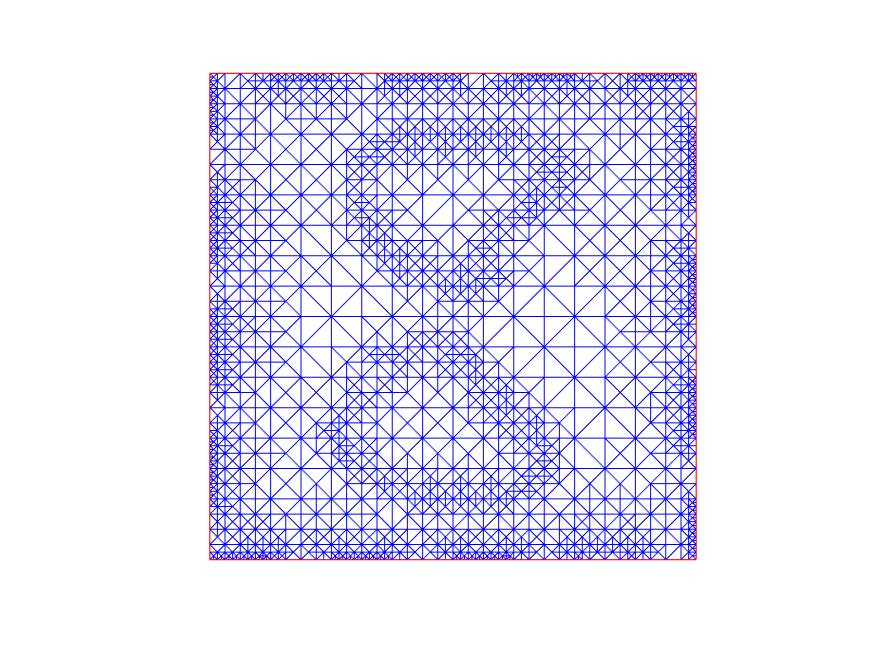}&
 \includegraphics[trim = {2.5cm 1.5cm 2.5cm 1.2cm}, clip, width=.2\textwidth]{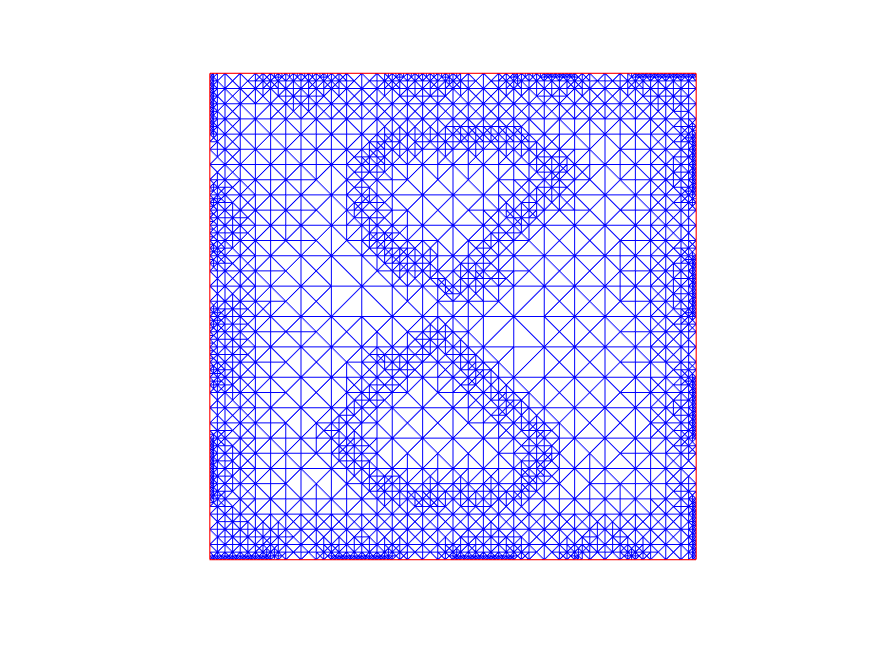}\\
  \includegraphics[trim = 1cm 0.5cm 0.5cm 0cm, width=.2\textwidth]{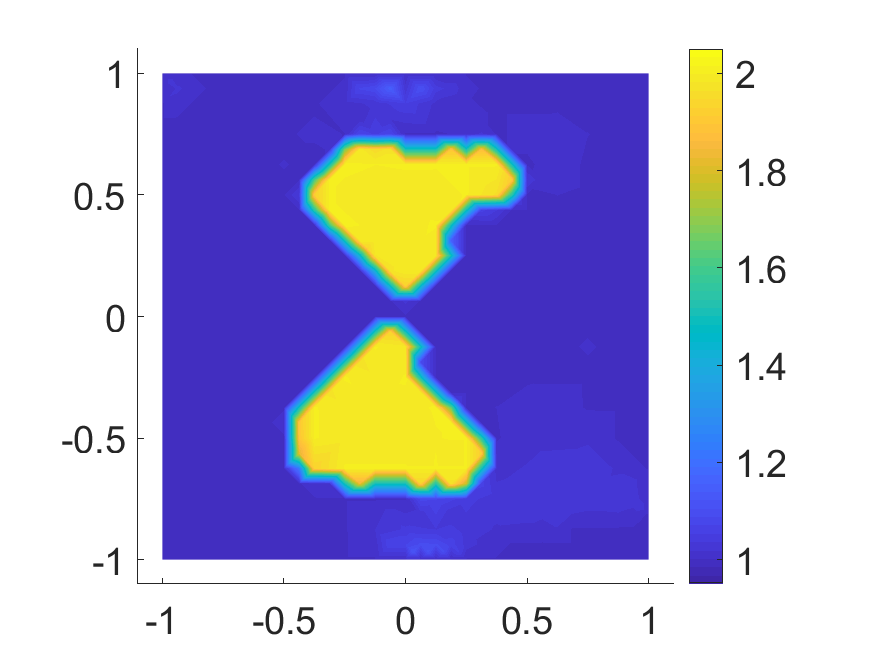}&
 \includegraphics[trim = 1cm 0.5cm 0.5cm 0cm, width=.2\textwidth]{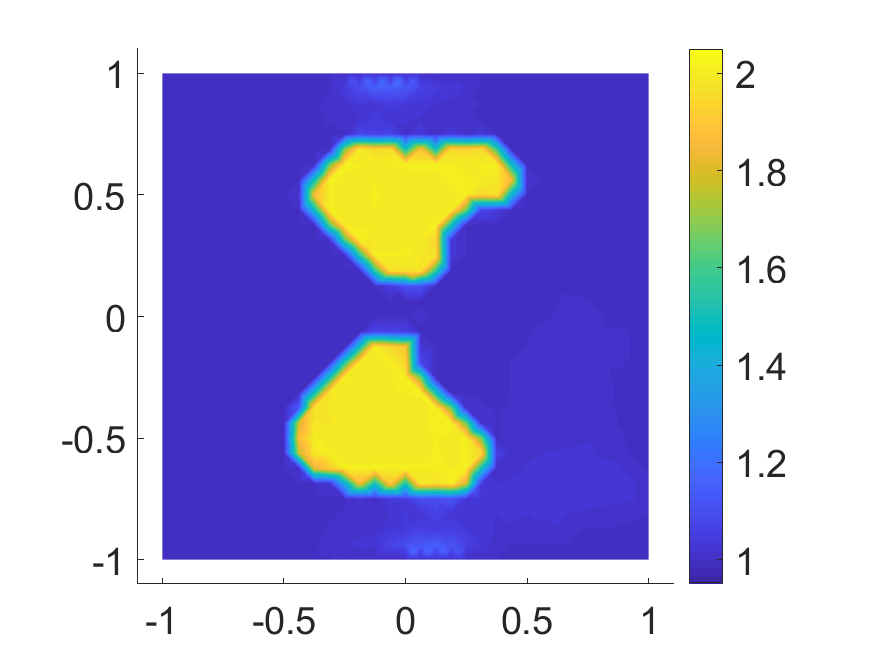}&
 \includegraphics[trim = 1cm 0.5cm 0.5cm 0cm, width=.2\textwidth]{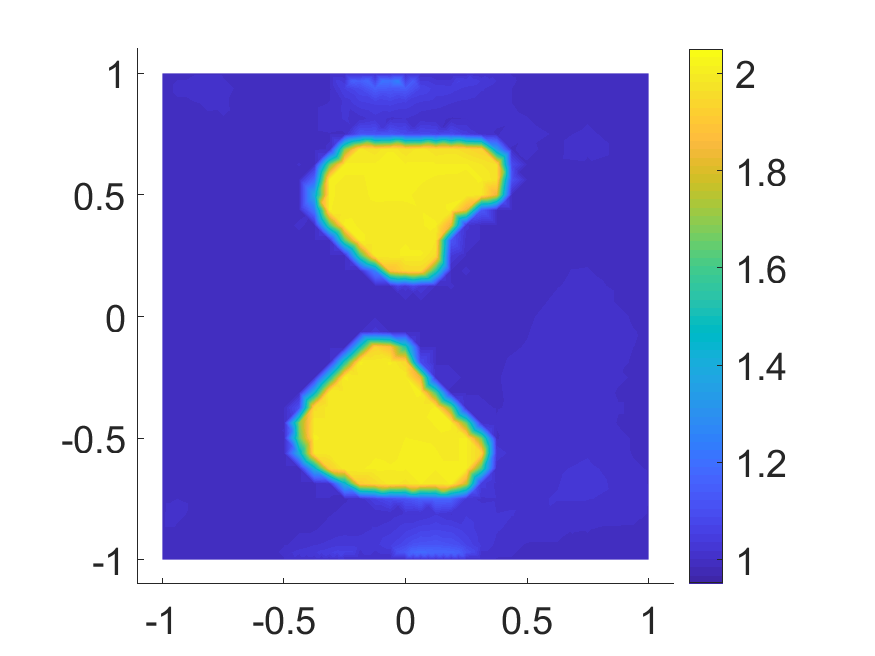}&
 \includegraphics[trim = 1cm 0.5cm 0.5cm 0cm, width=.2\textwidth]{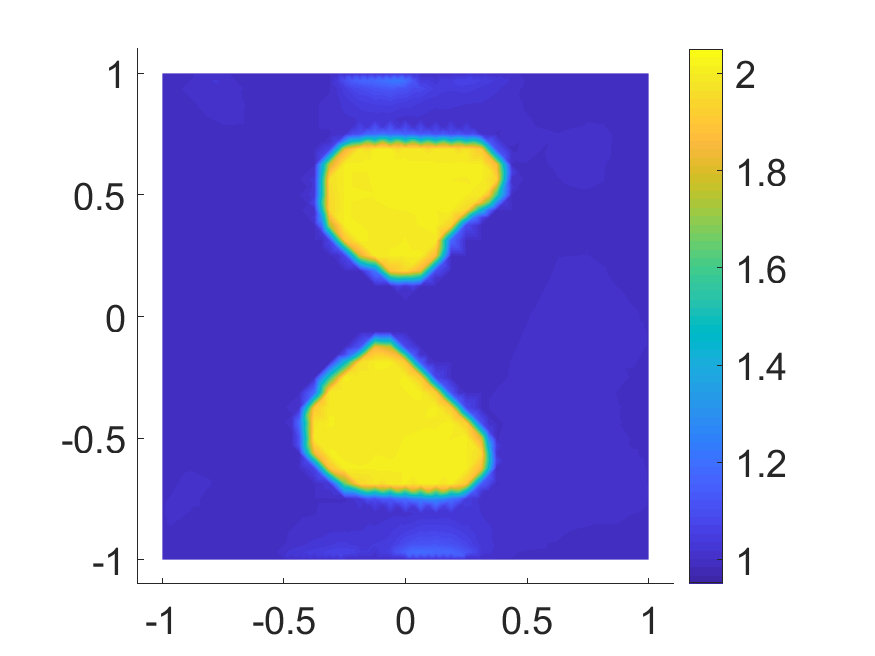}&
 \includegraphics[trim = 1cm 0.5cm 0.5cm 0cm, width=.2\textwidth]{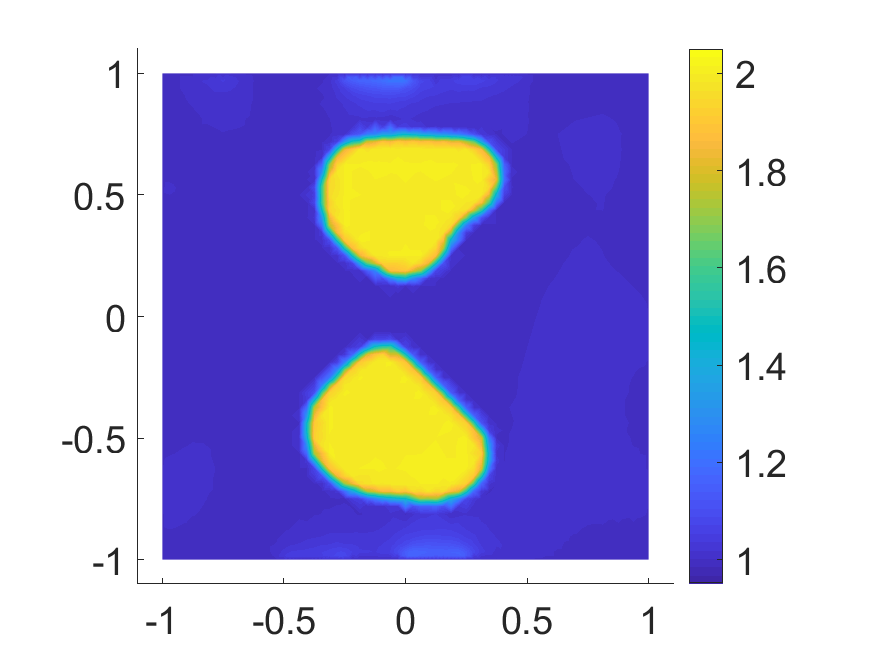}\\
  606     &     821     &        1308  &      1792    &       2764\\
  \includegraphics[trim = {2.5cm 1.5cm 2.5cm 1.2cm}, clip, width=.2\textwidth]{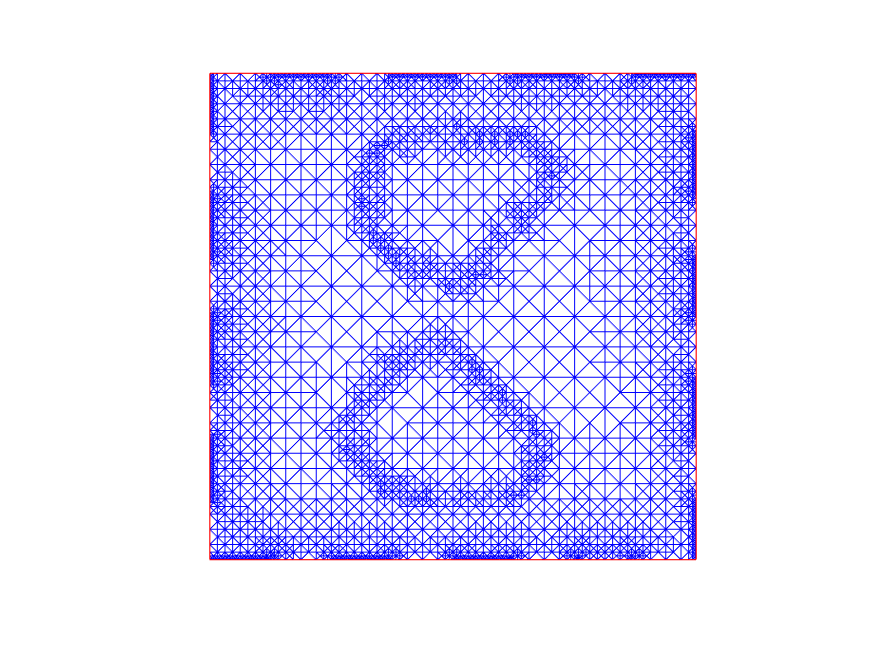}&
 \includegraphics[trim = {2.5cm 1.5cm 2.5cm 1.2cm}, clip, width=.2\textwidth]{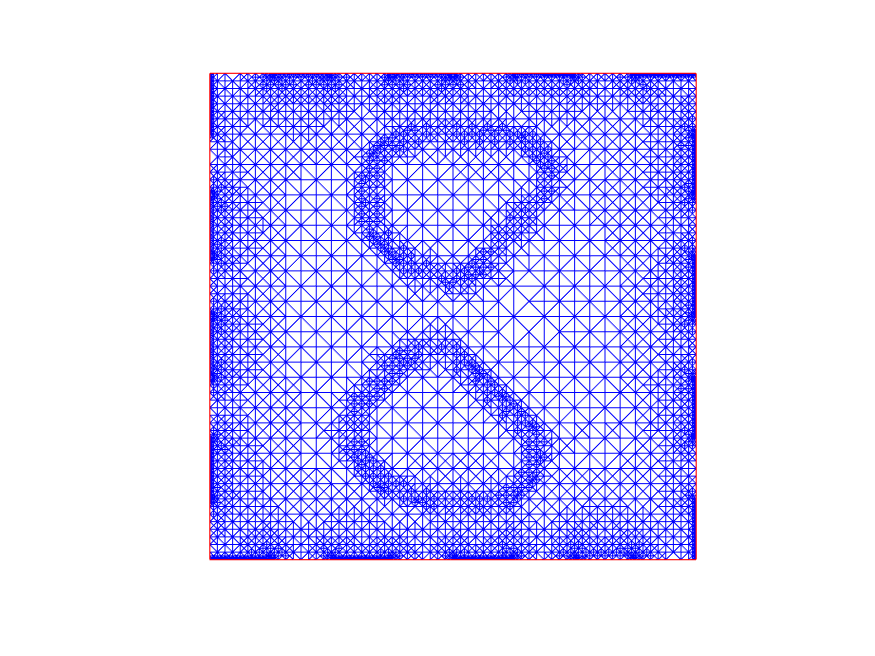}&
 \includegraphics[trim = {2.5cm 1.5cm 2.5cm 1.2cm}, clip, width=.2\textwidth]{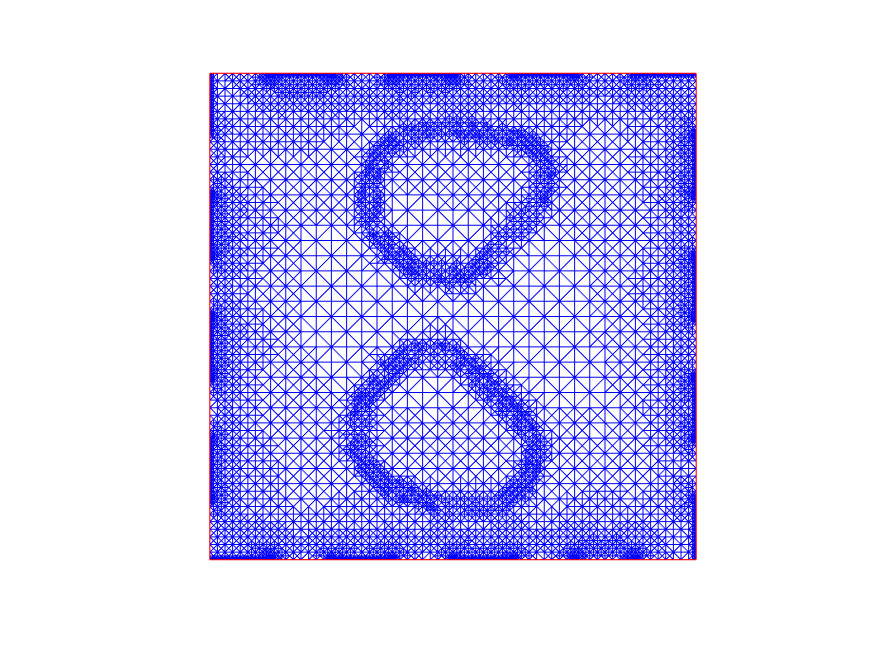}&
 \includegraphics[trim = {2.5cm 1.5cm 2.5cm 1.2cm}, clip, width=.2\textwidth]{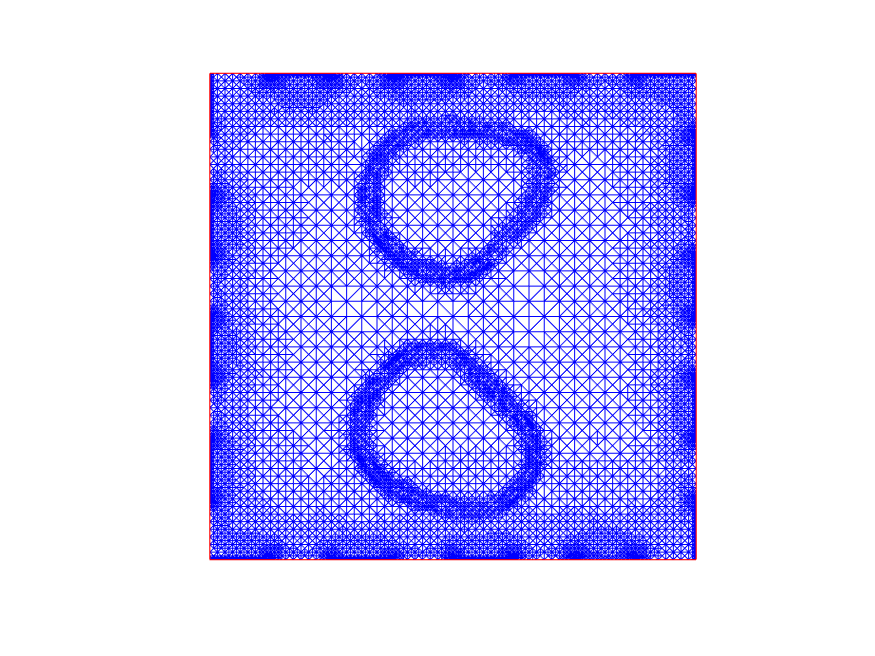}&
 \includegraphics[trim = {2.5cm 1.5cm 2.5cm 1.2cm}, clip, width=.2\textwidth]{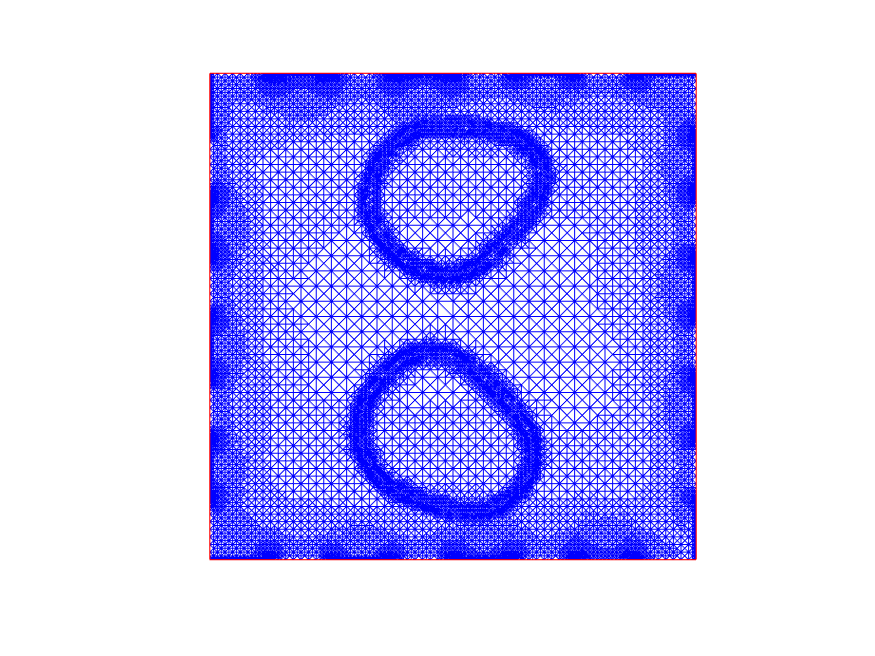}\\
 \includegraphics[trim = 1cm 0.5cm 0.5cm 0cm, width=.2\textwidth]{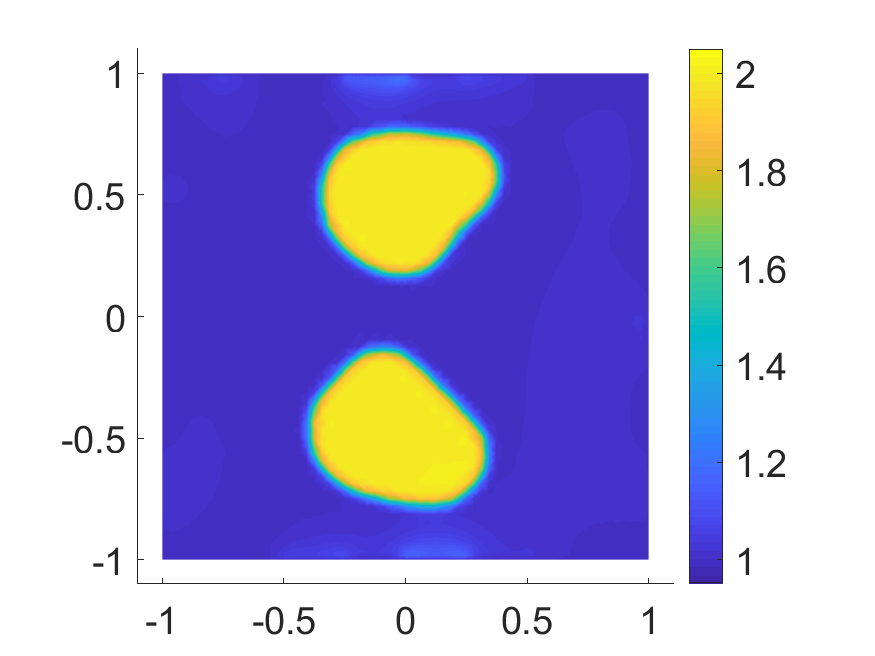}&
 \includegraphics[trim = 1cm 0.5cm 0.5cm 0cm, width=.2\textwidth]{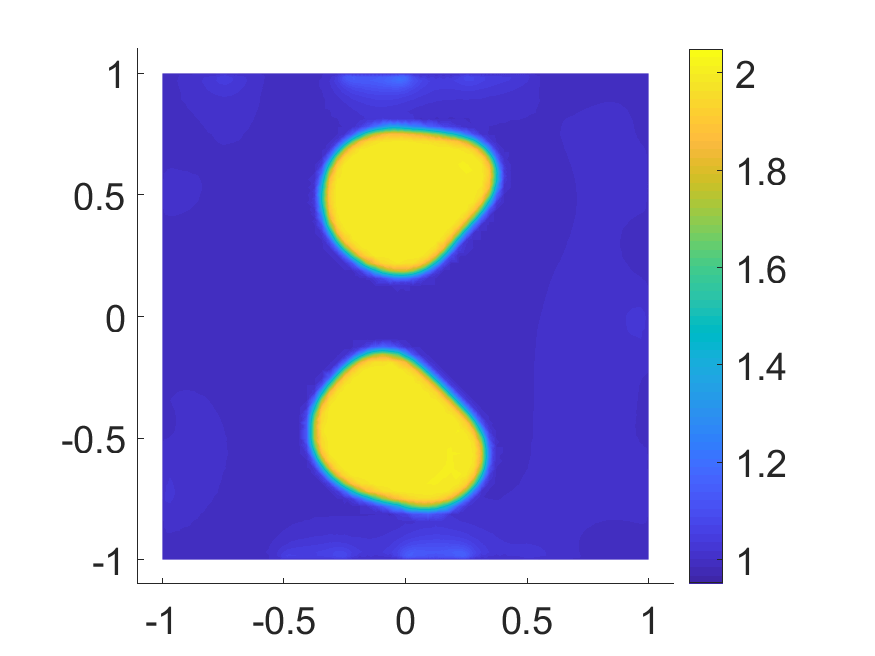}&
 \includegraphics[trim = 1cm 0.5cm 0.5cm 0cm, width=.2\textwidth]{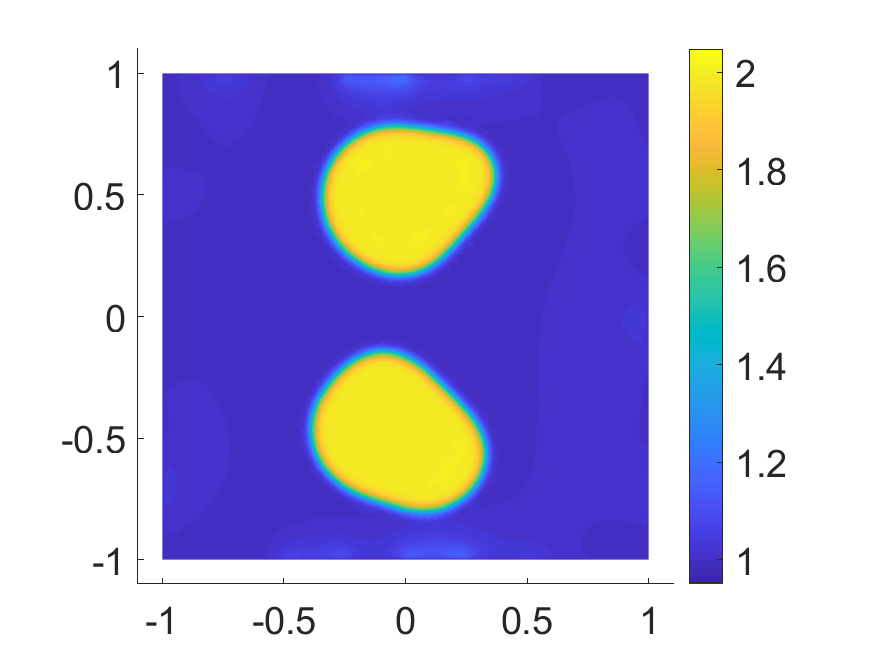}&
 \includegraphics[trim = 1cm 0.5cm 0.5cm 0cm, width=.2\textwidth]{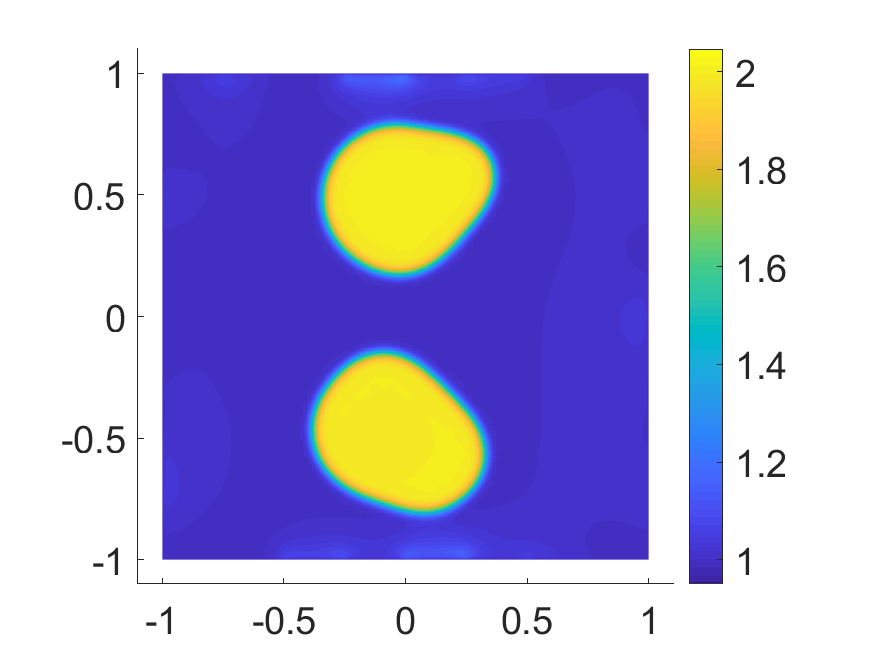}&
 \includegraphics[trim = 1cm 0.5cm 0.5cm 0cm, width=.2\textwidth]{ex4_1e3ad_it15}\\
 3816     &      5934       &        8299        &      12661       &       17736
 \end{tabular}
 \caption{The meshes $\mathcal{T}_k$ and recovered conductivities $\sigma_k$ during the adaptive iteration for Example \ref{exam2}(ii) with
 $\epsilon=\text{1e-3}$ and $\tilde\alpha=\text{2e-2}.$
 The number under each figure refers to d.o.f.} \label{fig:exam2ii-recon-iter-1e3}
\end{figure}

\begin{figure}[hbt!]
  \centering\setlength{\tabcolsep}{0em}
  \begin{tabular}{ccccc}
      \includegraphics[trim = .5cm 0cm 1cm 0cm, clip=true,width=.2\textwidth]{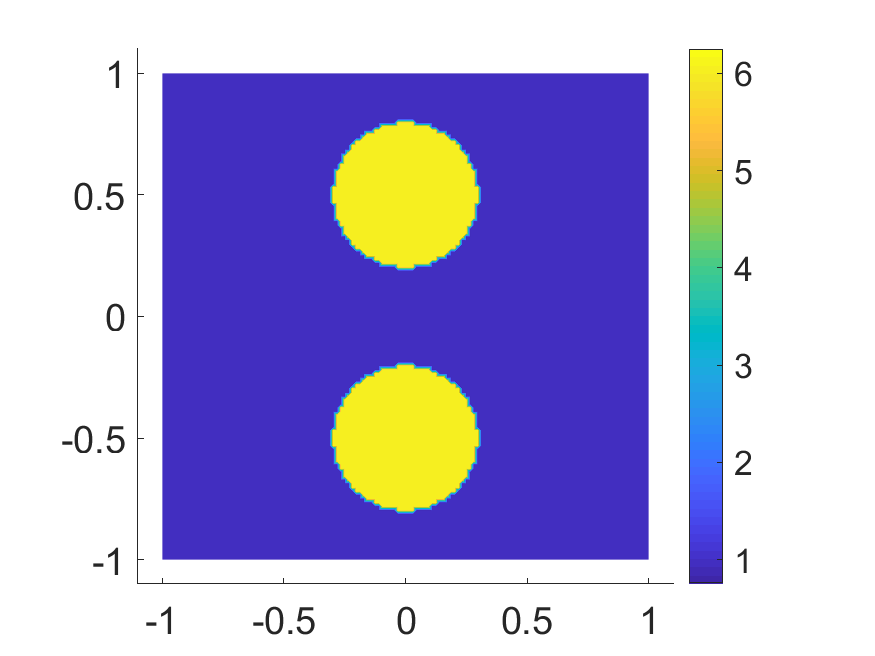}
    & \includegraphics[trim = .5cm 0cm 1cm 0cm, clip=true,width=.2\textwidth]{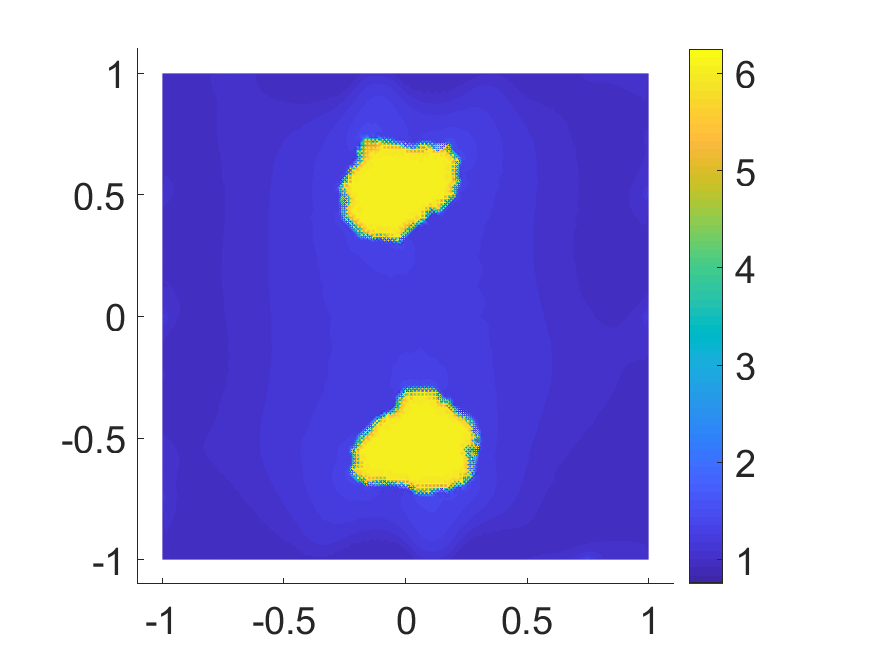}
    & \includegraphics[trim = .5cm 0cm 1cm 0cm, clip=true,width=.2\textwidth]{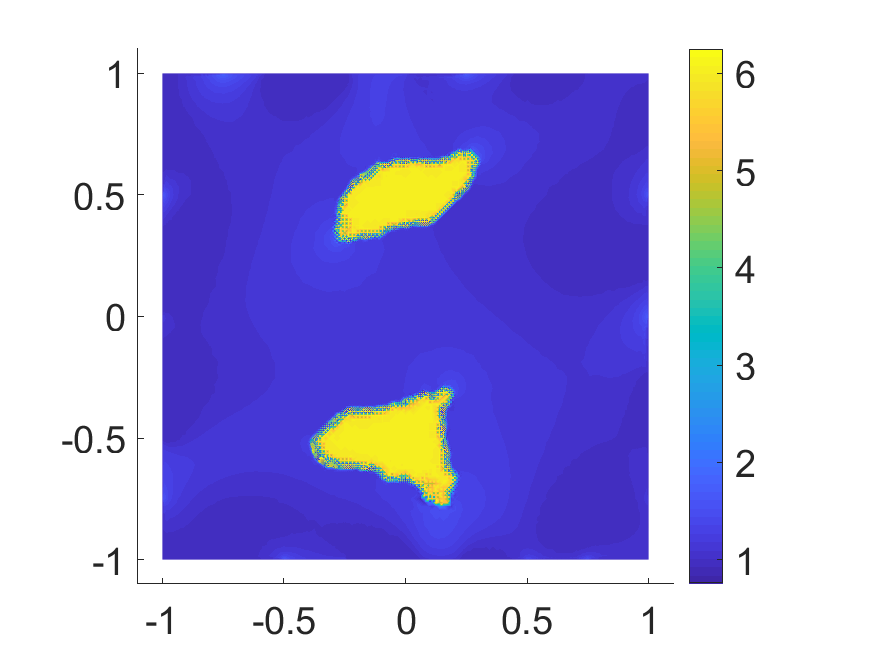}
    & \includegraphics[trim = .5cm 0cm 1cm 0cm, clip=true,width=.2\textwidth]{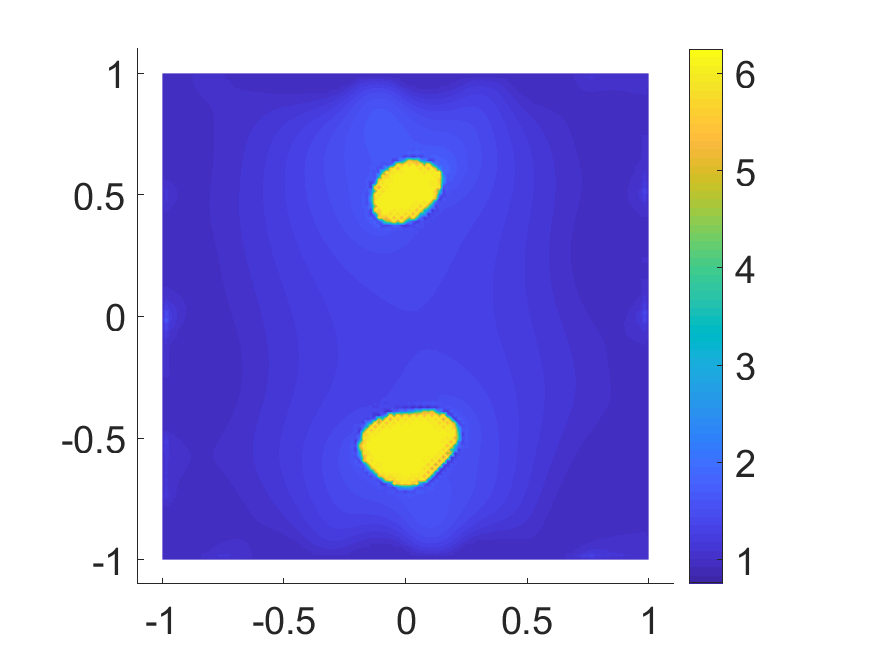}
    & \includegraphics[trim = .5cm 0cm 1cm 0cm, clip=true,width=.2\textwidth]{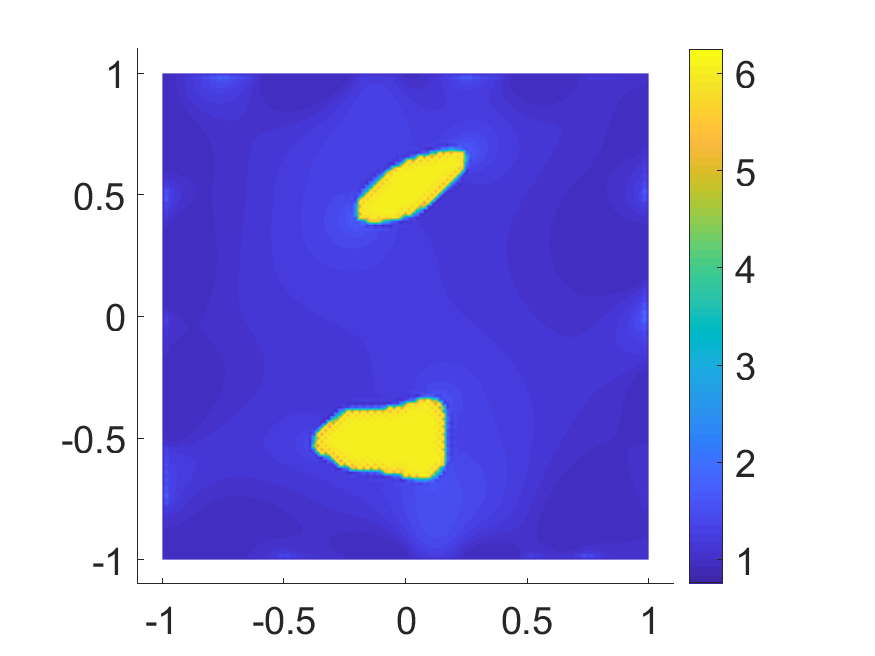}\\
    (a) true conductivity & (b) adaptive & (c) adaptive & (d) uniform & (e) uniform
  \end{tabular}
  \caption{The final recoveries by the adaptive and uniform refinements for Example
  \ref{exam2}(iii). The results in (b) and (d) are for $\epsilon=\text{1e-3}$ and
  $\tilde\alpha=\text{1e-4}$, and (c) and (e) for $\epsilon=\text{1e-2}$ and $\tilde\alpha=\text{2e-4}$.
  The d.o.f. for (b), (c), (d) and (e) is  14620, 19355, 16641 and 16641 respectively.}\label{fig:exam2iii-recon}
\end{figure}

\begin{figure}[hbt!]
 \centering
 \setlength{\tabcolsep}{0pt}
 \begin{tabular}{cccccccc}
 \includegraphics[trim = {2.5cm 1.5cm 2.5cm 1.2cm}, clip, width=.2\textwidth]{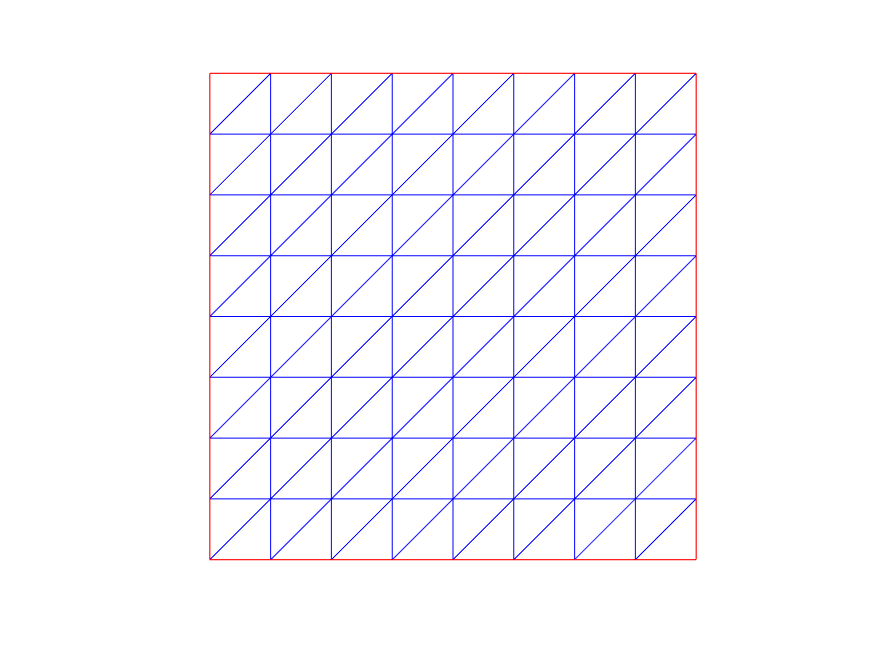}&
 \includegraphics[trim = {2.5cm 1.5cm 2.5cm 1.2cm}, clip, width=.2\textwidth]{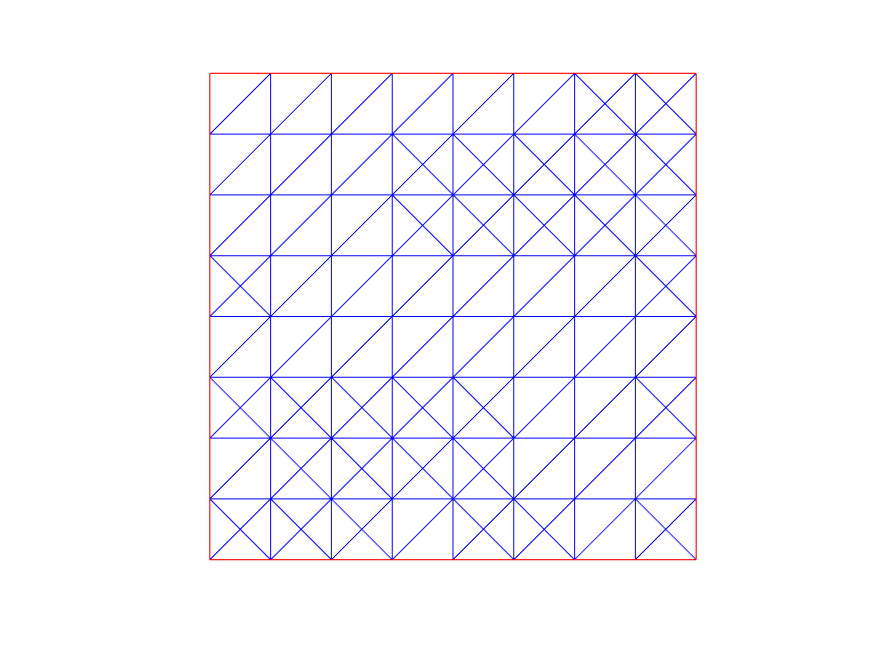}&
 \includegraphics[trim = {2.5cm 1.5cm 2.5cm 1.2cm}, clip, width=.2\textwidth]{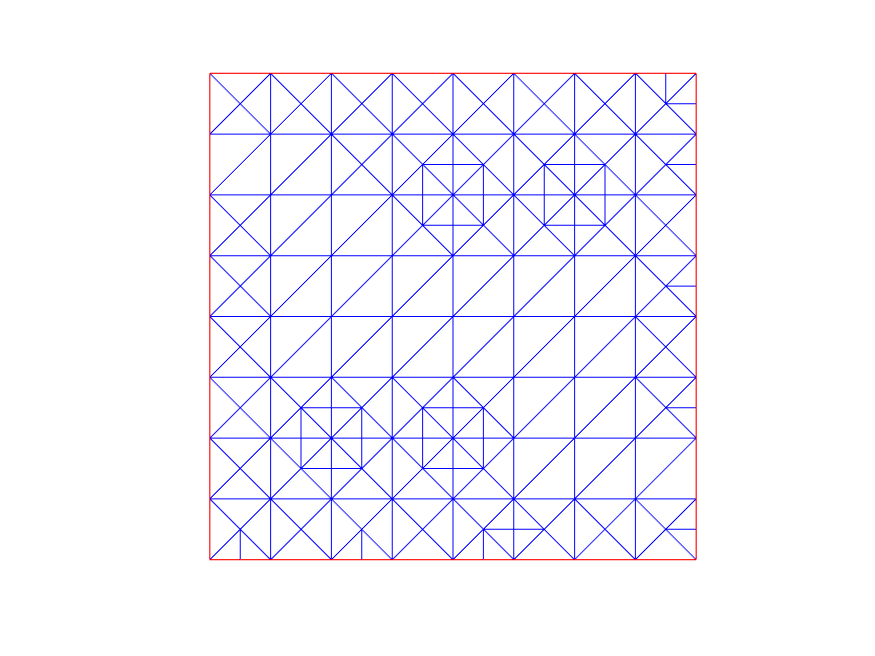}&
 \includegraphics[trim = {2.5cm 1.5cm 2.5cm 1.2cm}, clip, width=.2\textwidth]{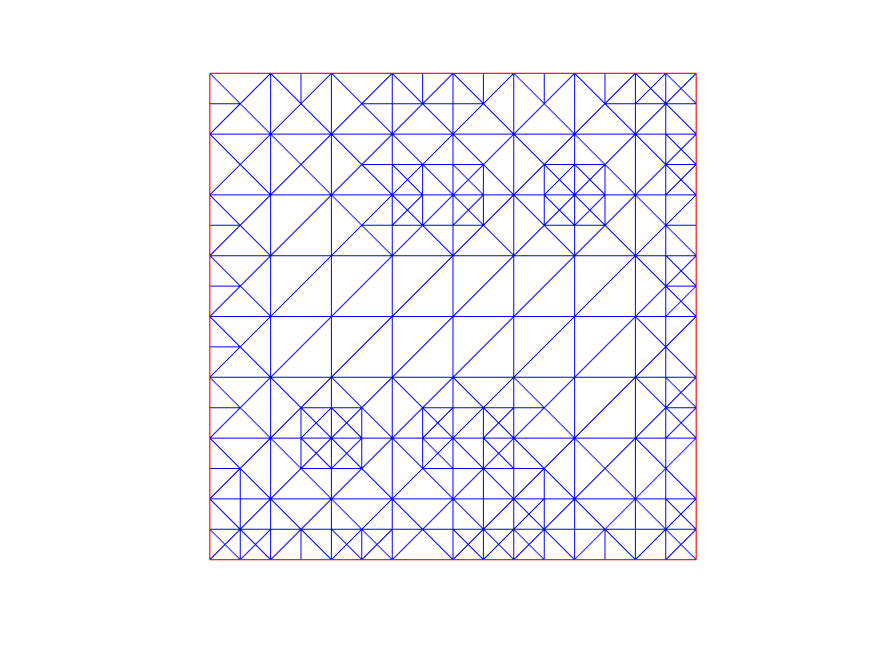}&
 \includegraphics[trim = {2.5cm 1.5cm 2.5cm 1.2cm}, clip, width=.2\textwidth]{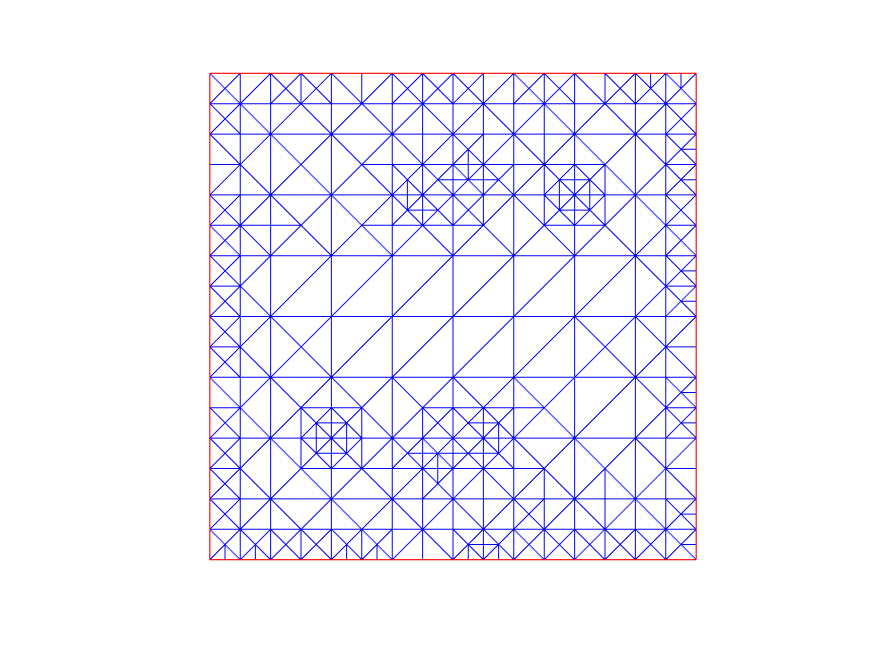}\\
 \includegraphics[trim = 1cm 0.5cm 0.5cm 0cm, width=.2\textwidth]{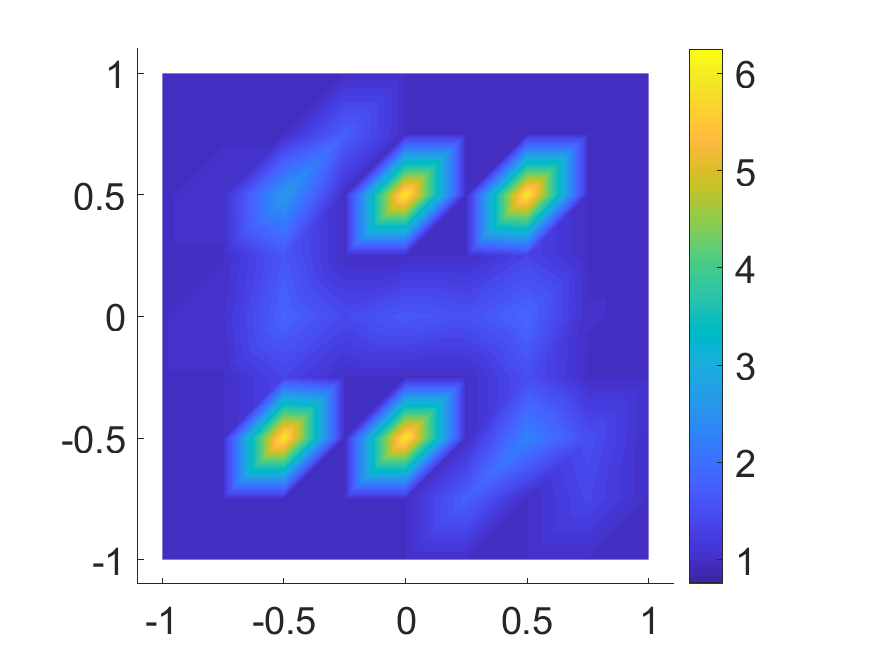}&
 \includegraphics[trim = 1cm 0.5cm 0.5cm 0cm, width=.2\textwidth]{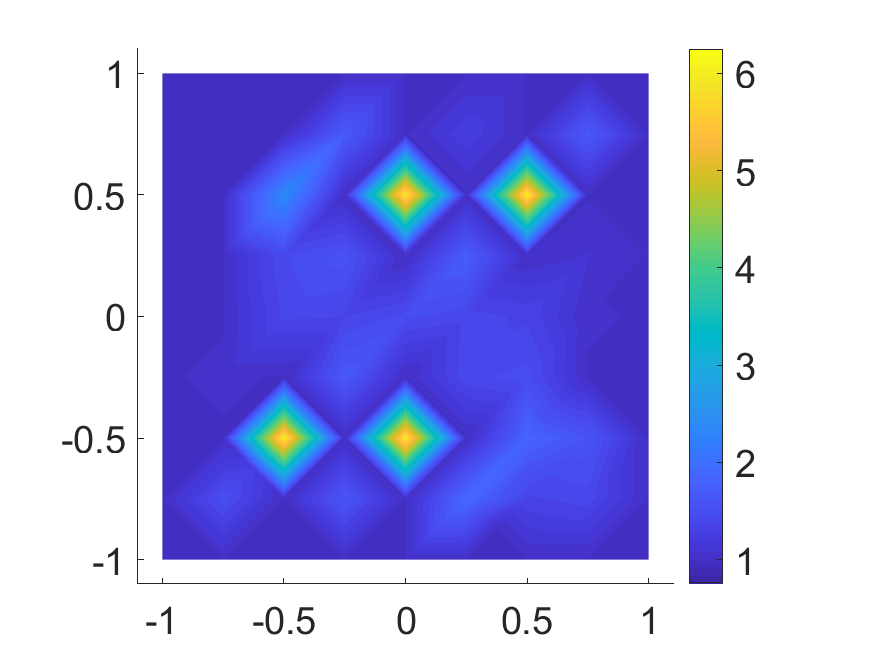}&
 \includegraphics[trim = 1cm 0.5cm 0.5cm 0cm, width=.2\textwidth]{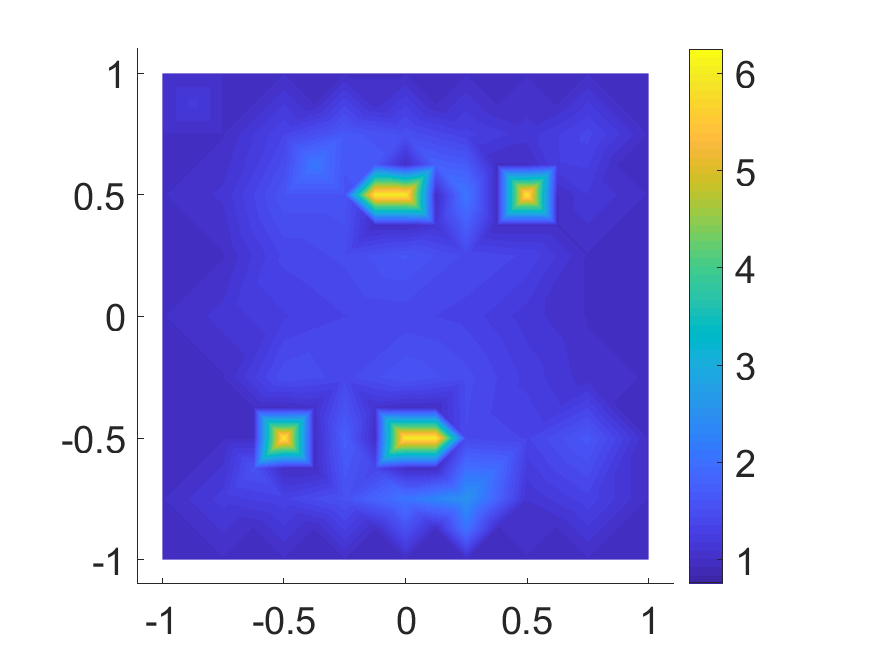}&
 \includegraphics[trim = 1cm 0.5cm 0.5cm 0cm, width=.2\textwidth]{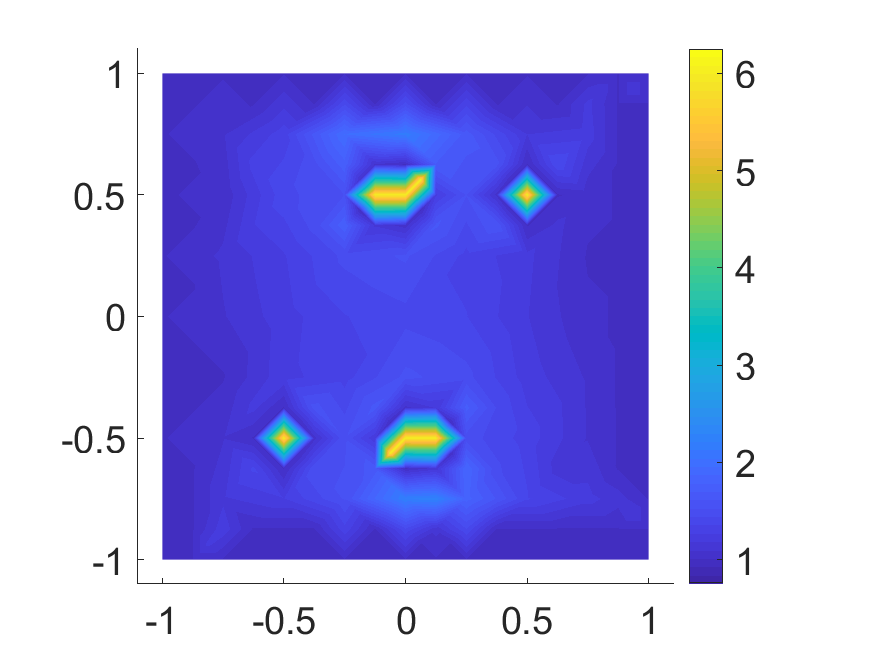}&
 \includegraphics[trim = 1cm 0.5cm 0.5cm 0cm, width=.2\textwidth]{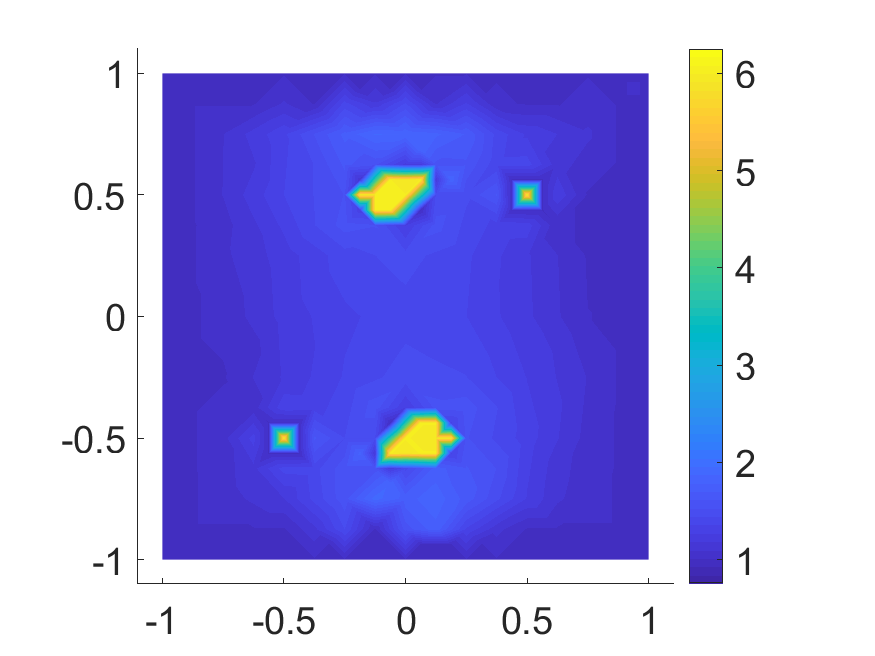}\\
    81     &     111    &    150    &     231      &   331\\
 \includegraphics[trim = {2.5cm 1.5cm 2.5cm 1.2cm}, clip, width=.2\textwidth]{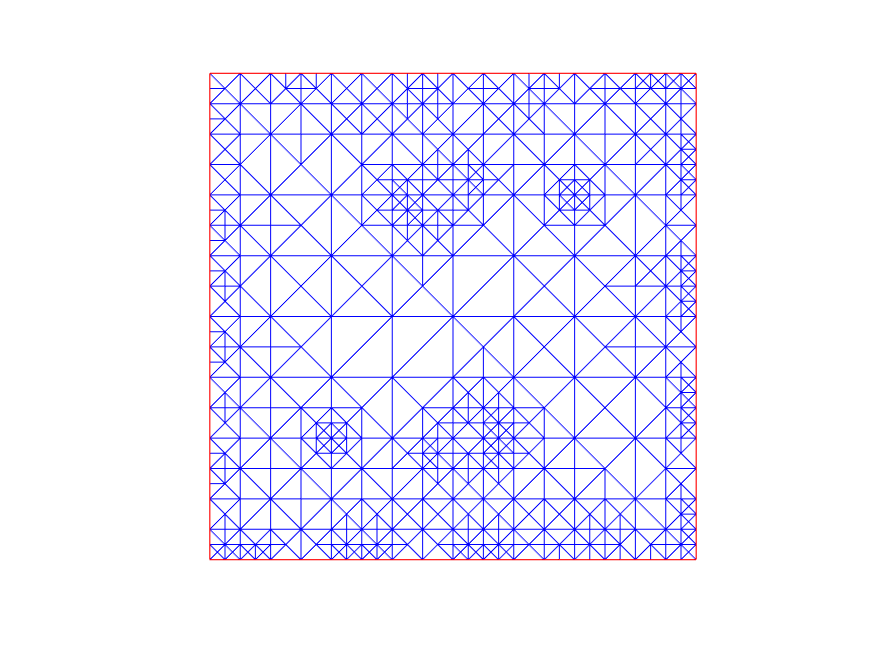}&
 \includegraphics[trim = {2.5cm 1.5cm 2.5cm 1.2cm}, clip, width=.2\textwidth]{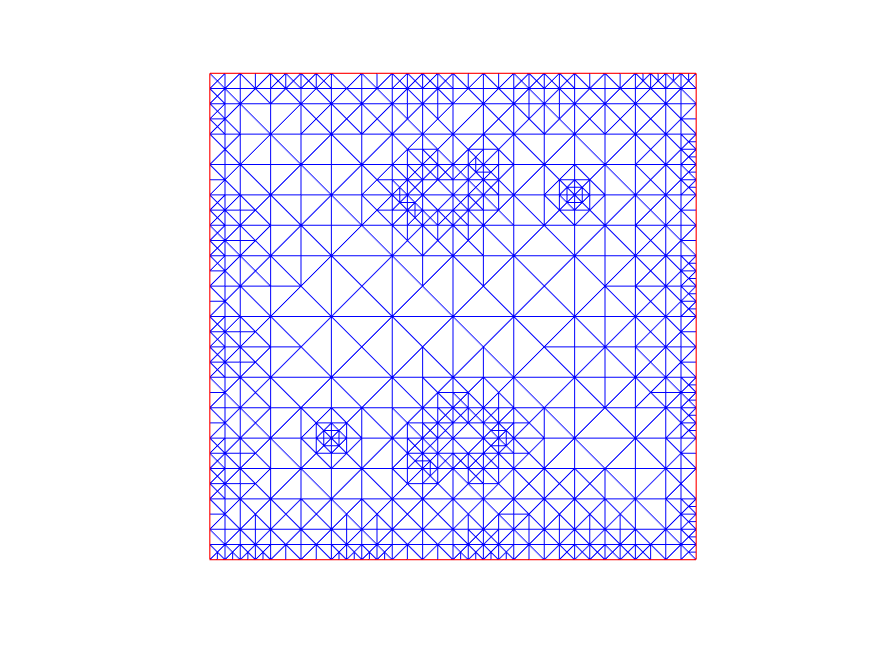}&
 \includegraphics[trim = {2.5cm 1.5cm 2.5cm 1.2cm}, clip, width=.2\textwidth]{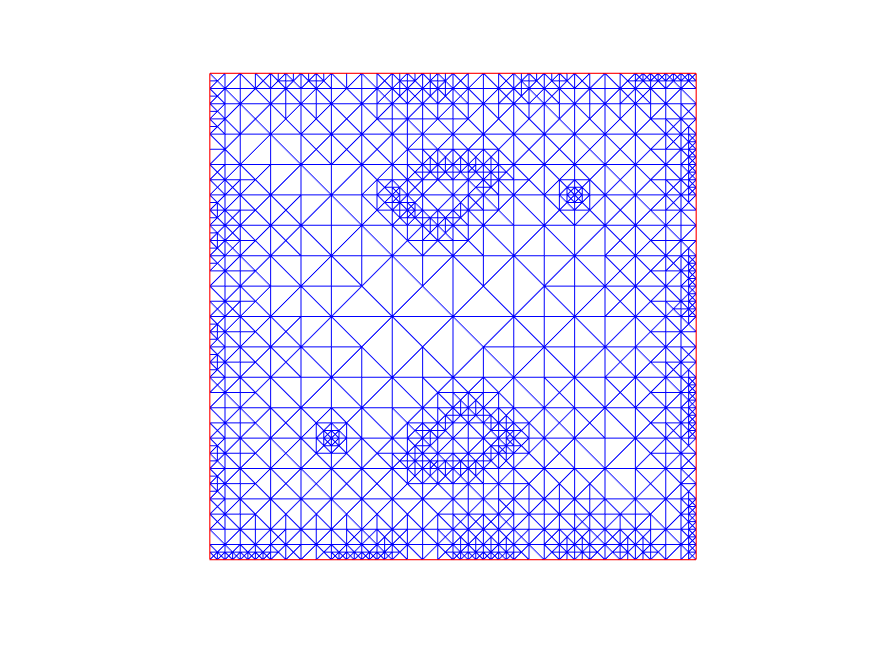}&
 \includegraphics[trim = {2.5cm 1.5cm 2.5cm 1.2cm}, clip, width=.2\textwidth]{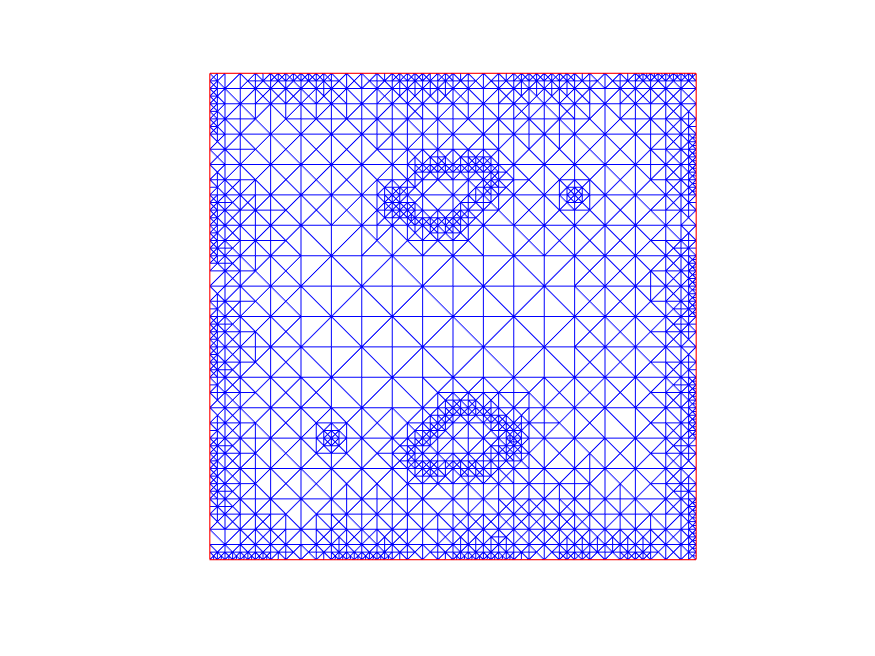}&
 \includegraphics[trim = {2.5cm 1.5cm 2.5cm 1.2cm}, clip, width=.2\textwidth]{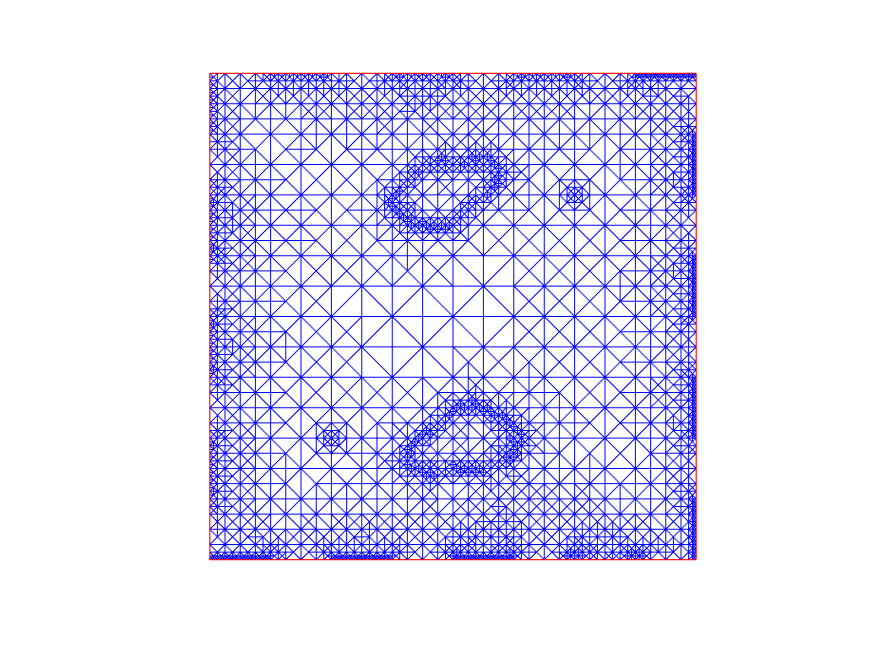}\\
 \includegraphics[trim = 1cm 0.5cm 0.5cm 0cm, width=.2\textwidth]{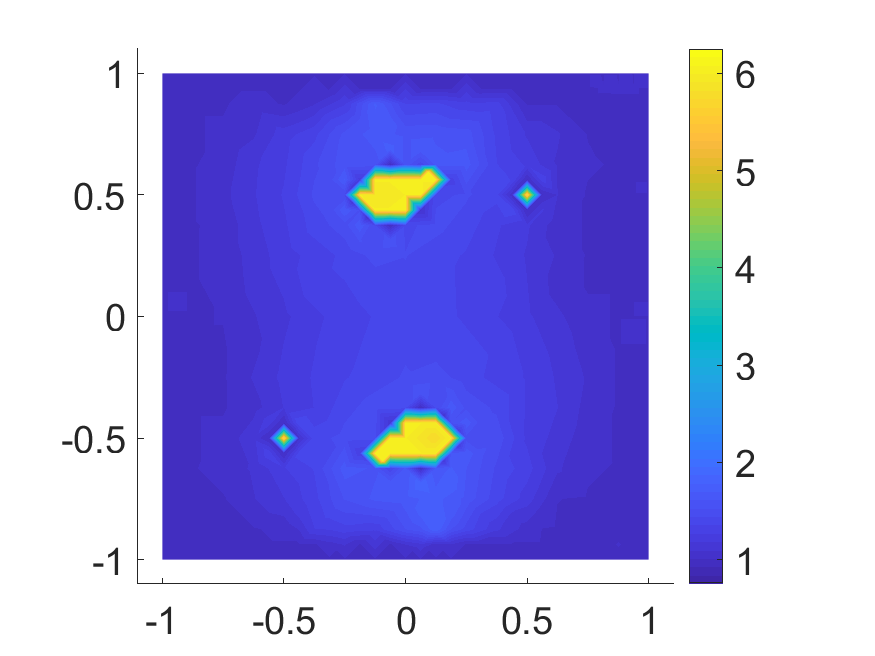}&
 \includegraphics[trim = 1cm 0.5cm 0.5cm 0cm, width=.2\textwidth]{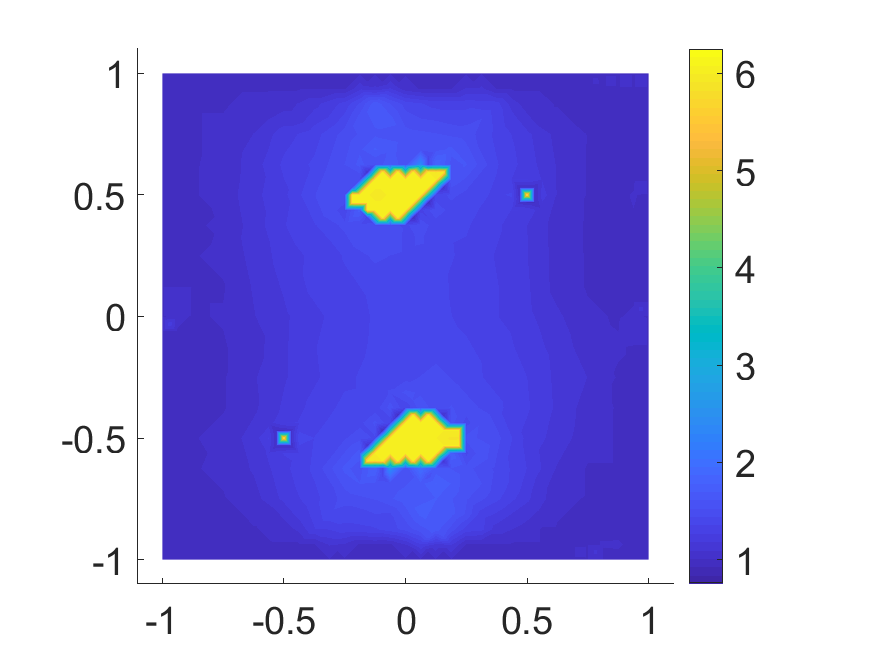}&
 \includegraphics[trim = 1cm 0.5cm 0.5cm 0cm, width=.2\textwidth]{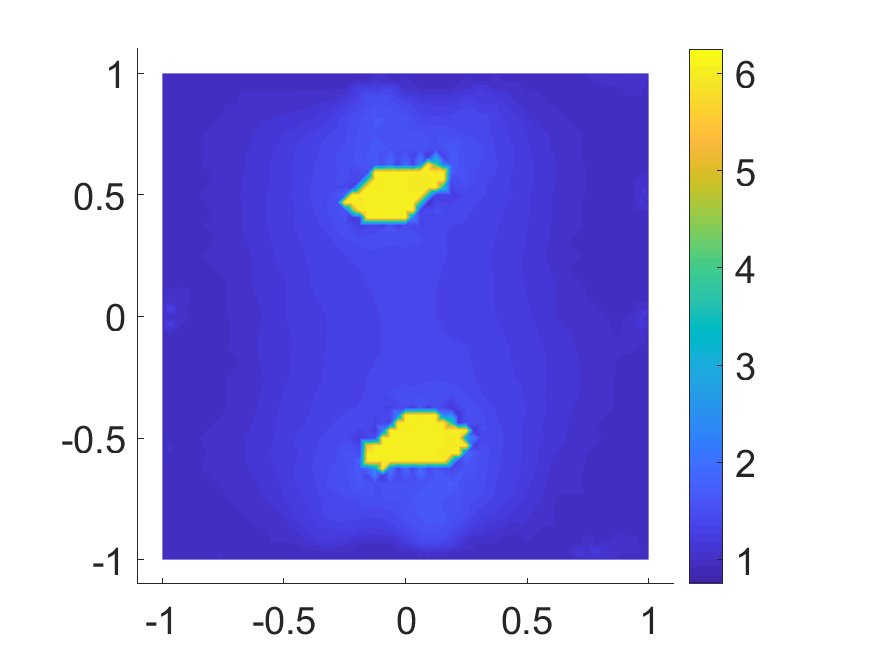}&
 \includegraphics[trim = 1cm 0.5cm 0.5cm 0cm, width=.2\textwidth]{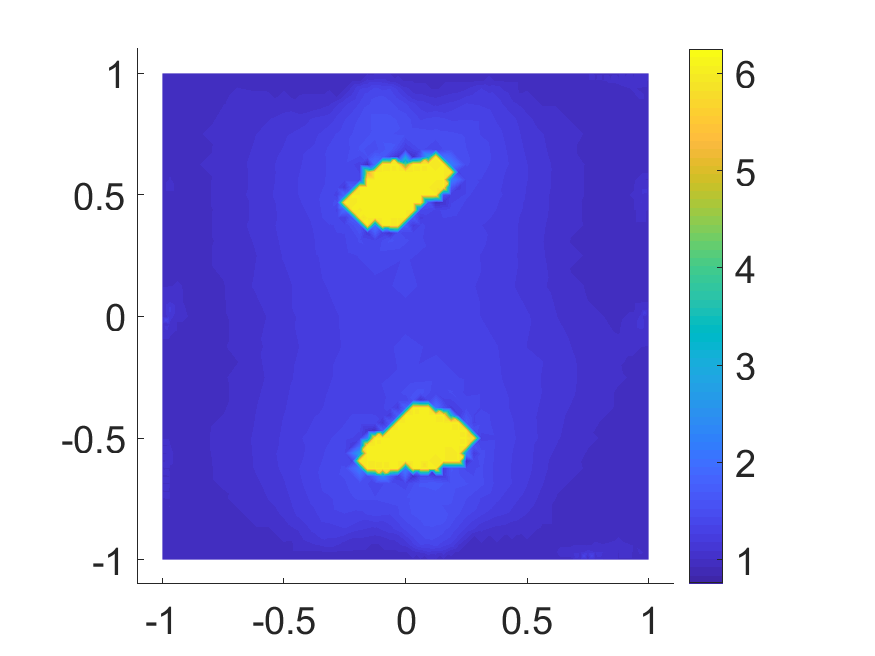}&
 \includegraphics[trim = 1cm 0.5cm 0.5cm 0cm, width=.2\textwidth]{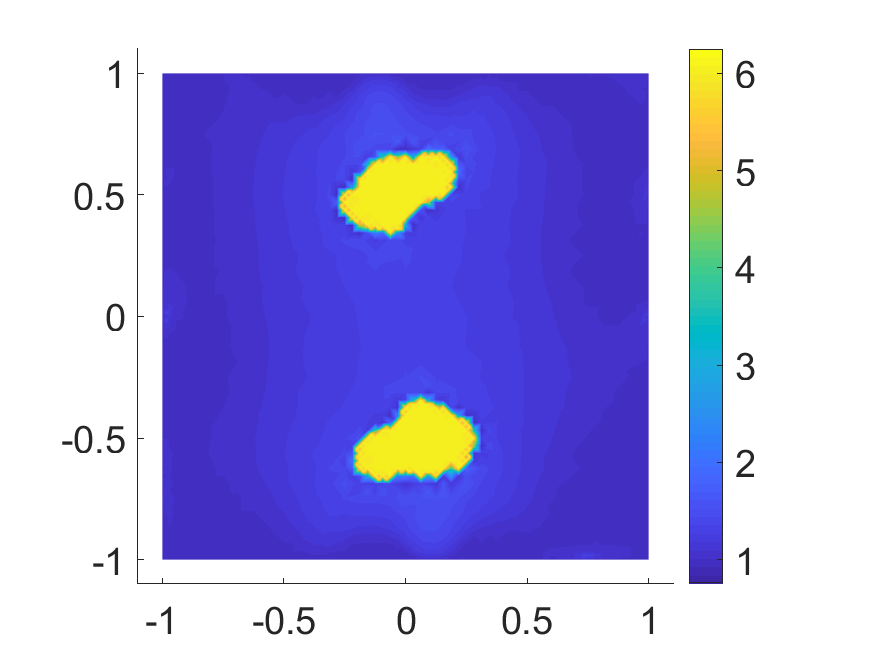}\\
    535    &   775    &     1215   &        1735    &    2594\\
 \includegraphics[trim = {2.5cm 1.5cm 2.5cm 1.2cm}, clip, width=.2\textwidth]{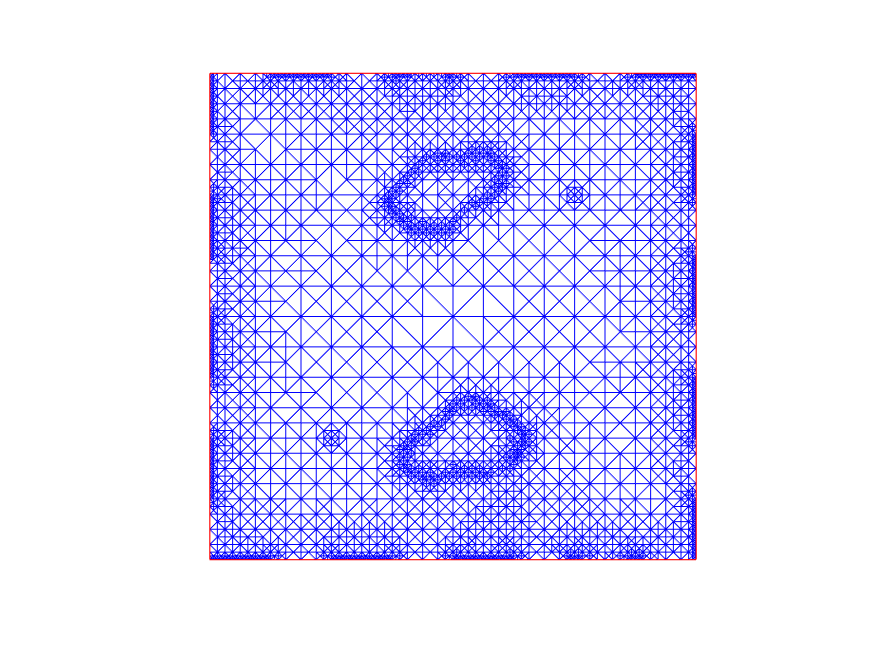}&
 \includegraphics[trim = {2.5cm 1.5cm 2.5cm 1.2cm}, clip, width=.2\textwidth]{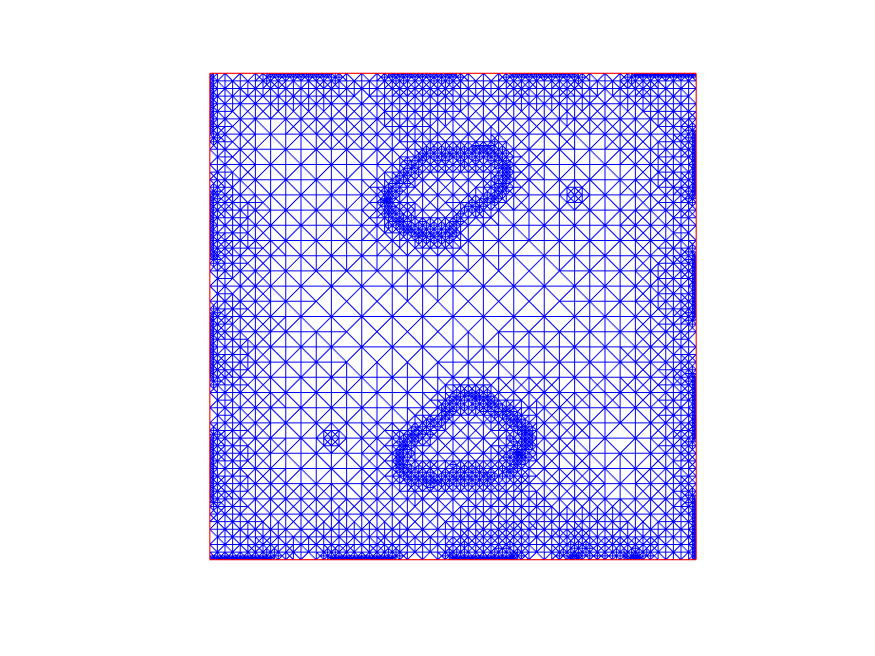}&
 \includegraphics[trim = {2.5cm 1.5cm 2.5cm 1.2cm}, clip, width=.2\textwidth]{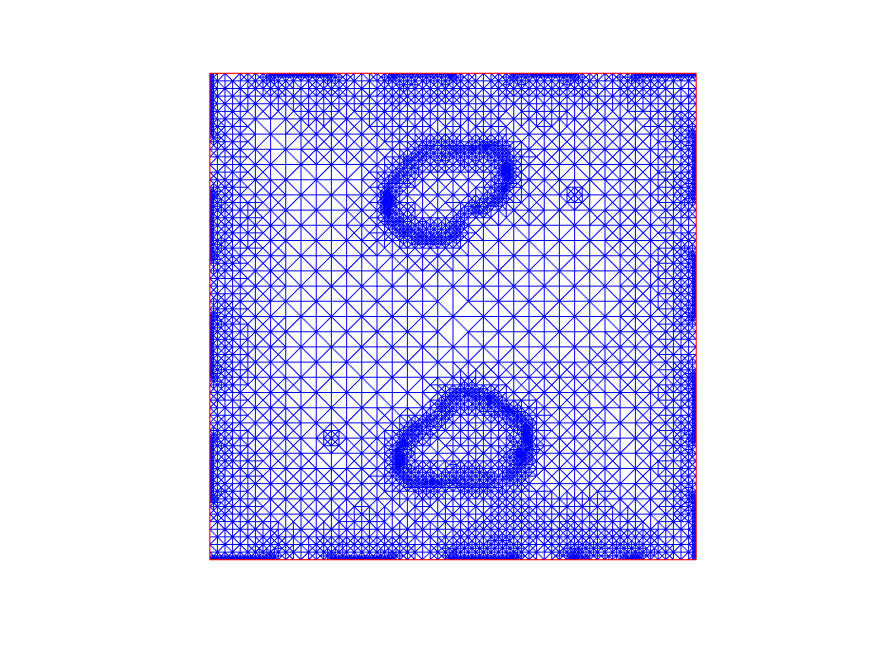}&
 \includegraphics[trim = {2.5cm 1.5cm 2.5cm 1.2cm}, clip, width=.2\textwidth]{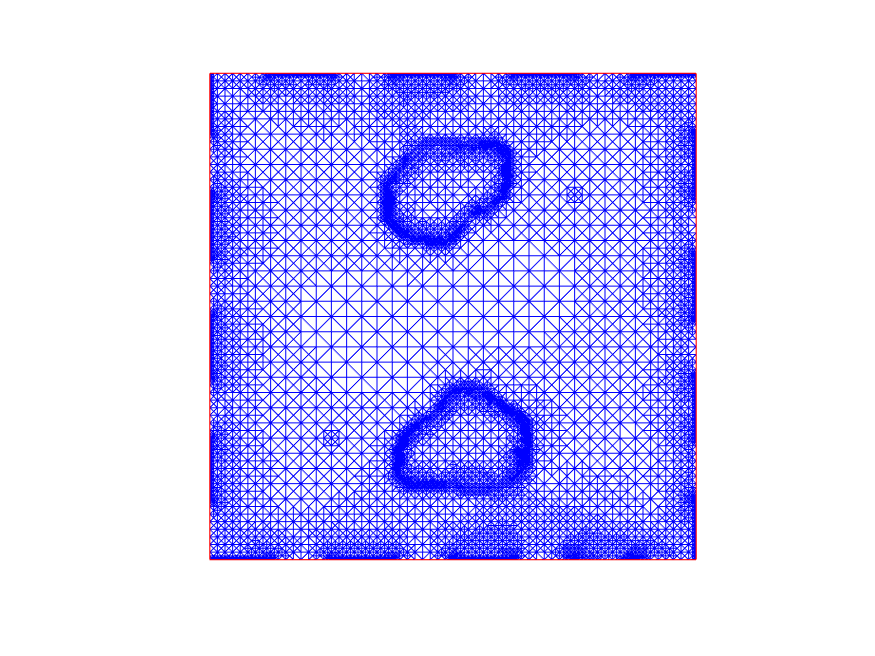}&
 \includegraphics[trim = {2.5cm 1.5cm 2.5cm 1.2cm}, clip, width=.2\textwidth]{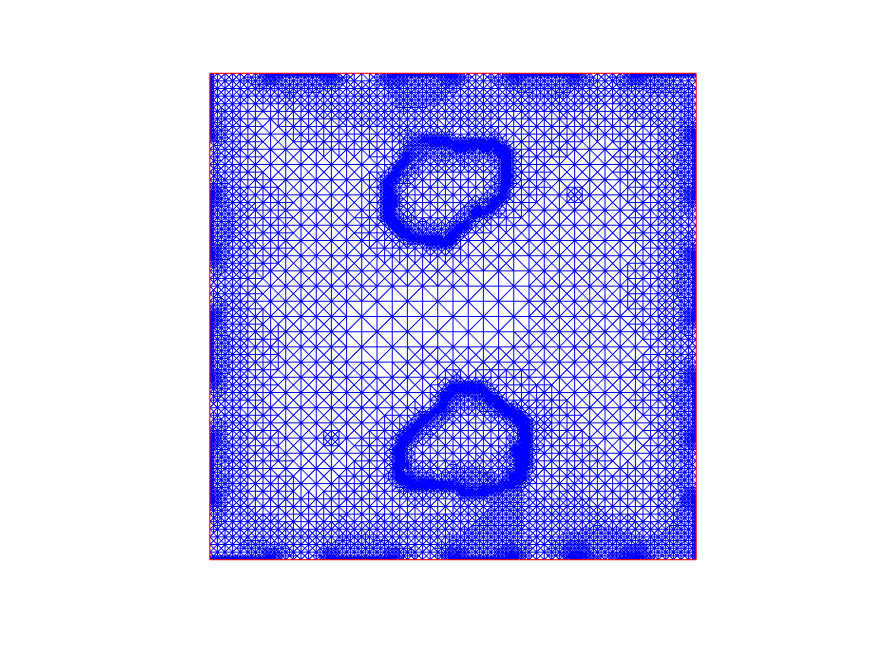}\\
 \includegraphics[trim = 1cm 0.5cm 0.5cm 0cm, width=.2\textwidth]{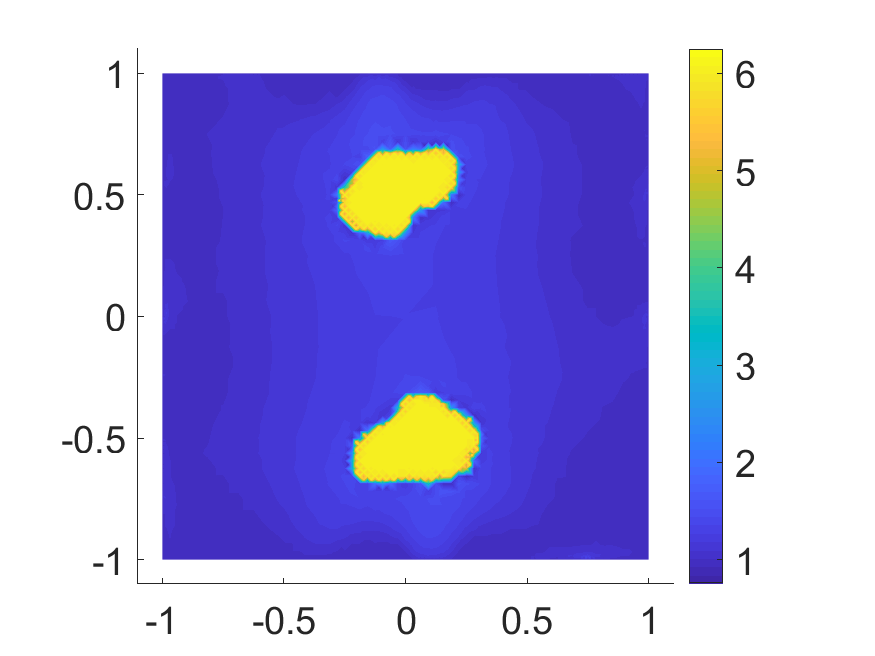}&
 \includegraphics[trim = 1cm 0.5cm 0.5cm 0cm, width=.2\textwidth]{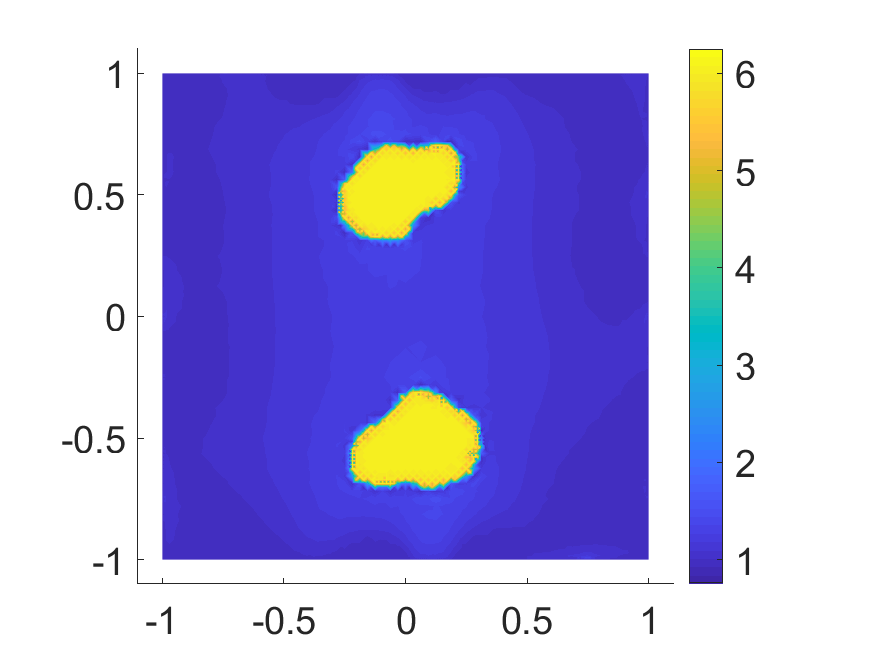}&
 \includegraphics[trim = 1cm 0.5cm 0.5cm 0cm, width=.2\textwidth]{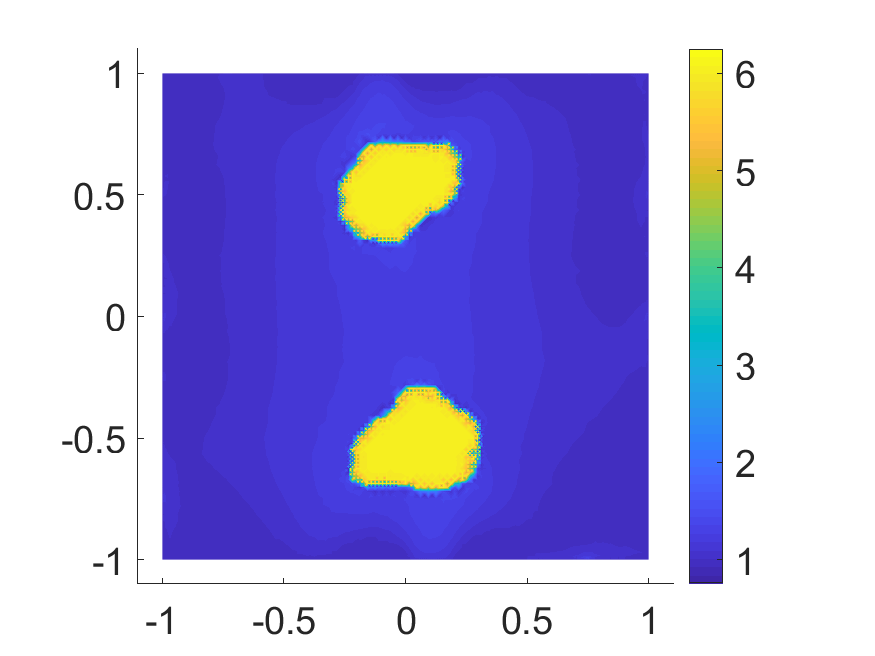}&
 \includegraphics[trim = 1cm 0.5cm 0.5cm 0cm, width=.2\textwidth]{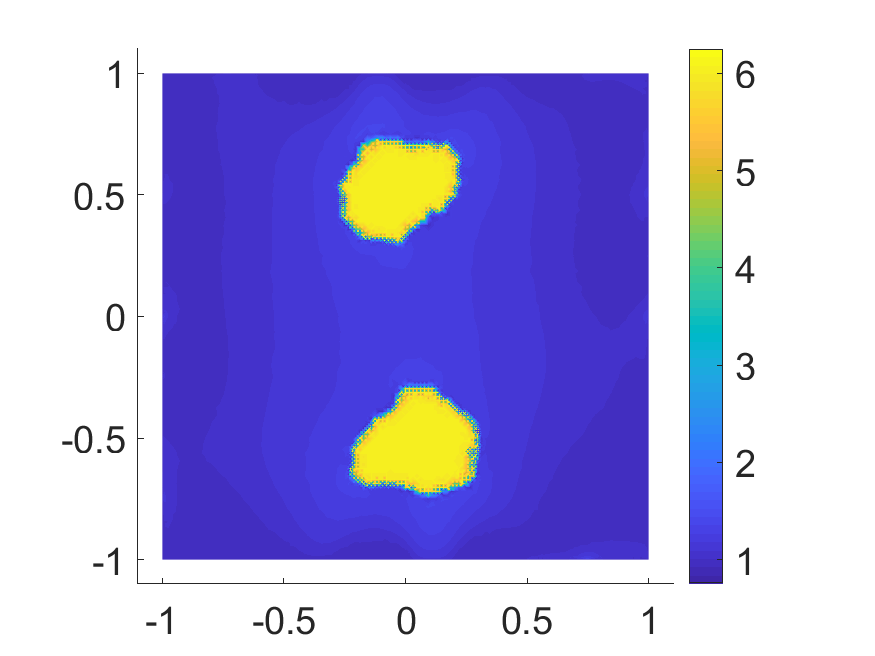}&
 \includegraphics[trim = 1cm 0.5cm 0.5cm 0cm, width=.2\textwidth]{ex5_1e3ad_it15}\\
 3521    &      5085     &       7080    &       10360    &       14620
 \end{tabular}
 \caption{The meshes $\mathcal{T}_k$ and recovered conductivities $\sigma_k$ during the adaptive iteration
 for Example \ref{exam2}(iii) with $\epsilon=\text{1e-3}$ and $\tilde\alpha = \text{1e-4}$. The number under each figure
 refers to d.o.f.} \label{fig:exam2iii-recon-iter-1e3}
\end{figure}

Now we consider one more challenging example with four inclusions.
\begin{example}\label{exam3}
The true conductivity $\sigma^\dag$ is given by $\sigma_0(x)+\sum_{i=1}^4\chi_{B_i}(x)$,
with the circles $B_i$ centered at $(0.6,\pm0.6)$ and $(-0.6,\pm0.6)$,
respectively, all with a radius 0.2, and the background conductivity $\sigma_0(x)=1$.
\end{example}

The numerical results for Example \ref{exam3} are given in Figs. \ref{fig:exam3-recon}--\ref{fig:exam3-efficiency}.
The results are in excellent agreement with the observations from Example \ref{exam2}: The algorithm converges
steadily as the adaptive iteration proceeds, and with a low noise level, it can accurately recover all
four inclusions, showing clearly the efficiency of the adaptive approach. The refinement is mainly around
the electrode edge and interval interface.

\begin{figure}[hbt!]
  \centering\setlength{\tabcolsep}{0em}
  \begin{tabular}{ccccc}
    \includegraphics[trim = .5cm 0cm 1cm 0cm,clip=true,width=.2\textwidth]{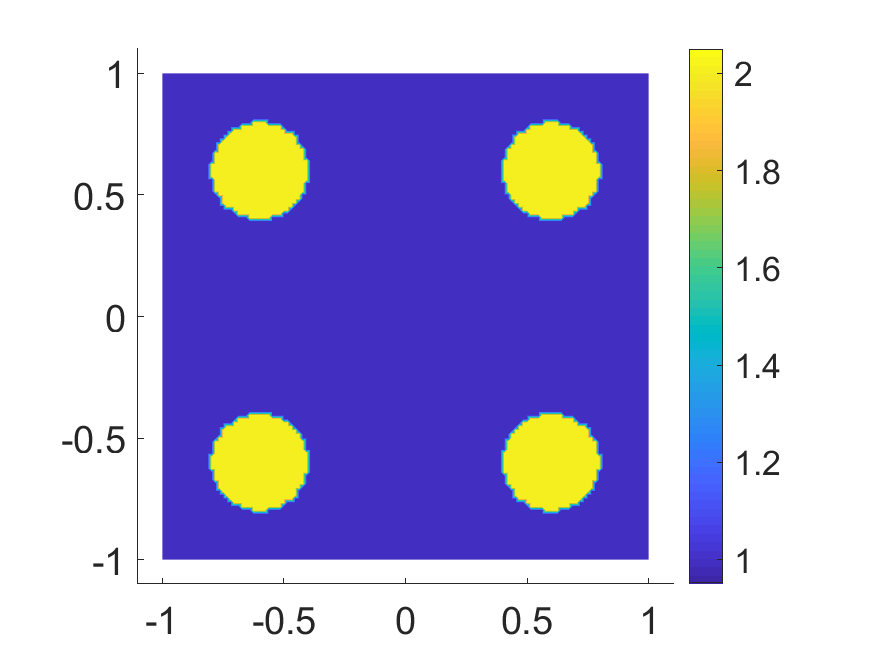}
    & \includegraphics[trim = .5cm 0cm 1cm 0cm, clip=true,width=.2\textwidth]{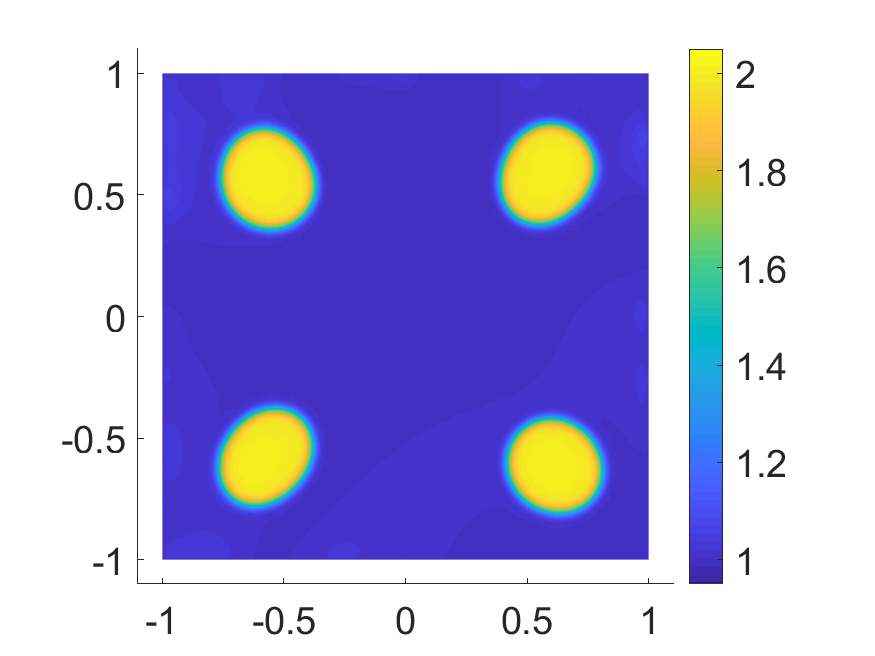}
    & \includegraphics[trim = .5cm 0cm 1cm 0cm, clip=true,width=.2\textwidth]{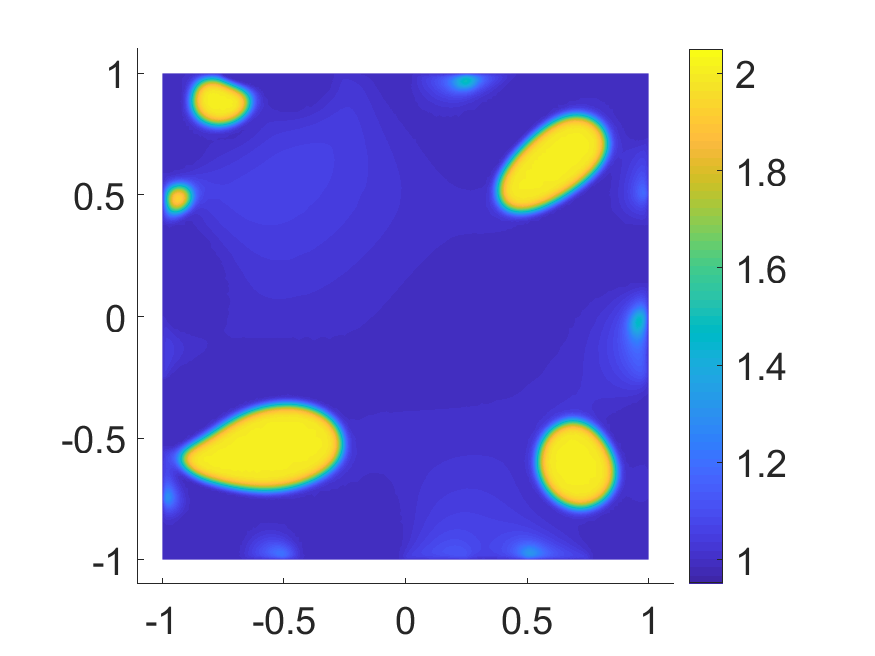}
    & \includegraphics[trim = .5cm 0cm 1cm 0cm, clip=true,width=.2\textwidth]{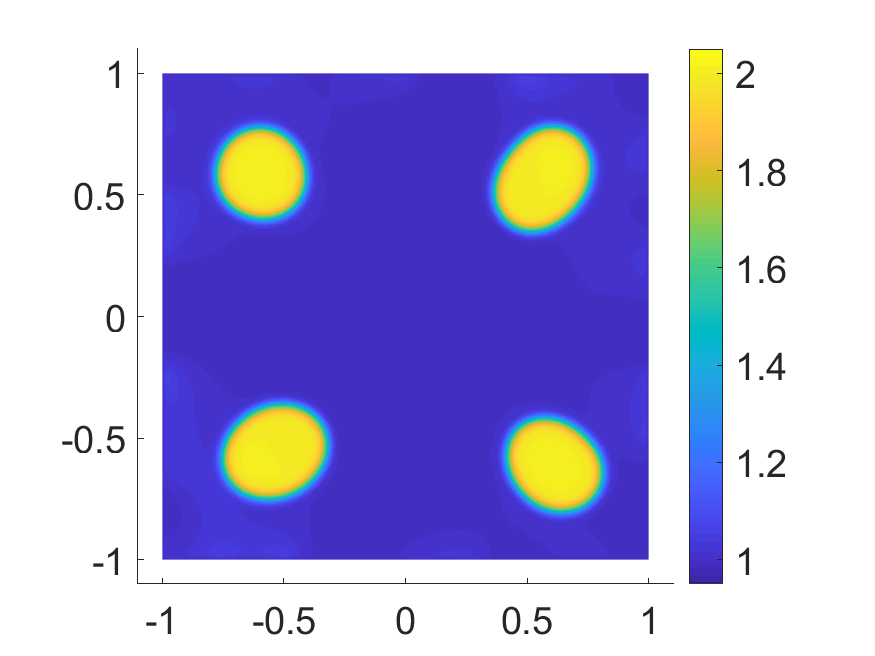}
    & \includegraphics[trim = .5cm 0cm 1cm 0cm, clip=true,width=.2\textwidth]{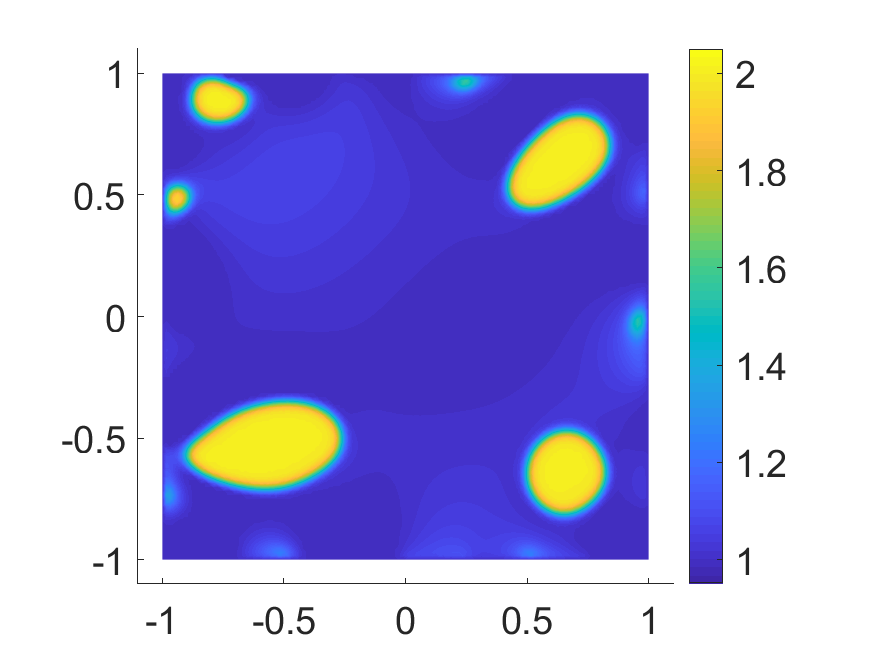}\\
    (a) exact & (b) adaptive & (c) adaptive & (d) uniform & (e) uniform
  \end{tabular}
  \caption{The final recoveries by the adaptive and uniform refinements for Example
  \ref{exam3}. The results in (b) and (d) are for $\epsilon=\text{1e-3}$ and $\tilde\alpha=\text{2e-2}$,
  and (c) and (e) for $\epsilon=\text{1e-2}$ and $\tilde\alpha=\text{3e-2}$. The d.o.f. of (b),
  (c), (d) and (e) is $18008$, $21120$ and $16641$ and $16641$, respectively.}\label{fig:exam3-recon}
\end{figure}

\begin{figure}[hbt!]
 \centering
 \setlength{\tabcolsep}{0pt}
 \begin{tabular}{cccccccc}
 \includegraphics[trim = {2.5cm 1.5cm 2.5cm 1.2cm}, clip, width=.2\textwidth]{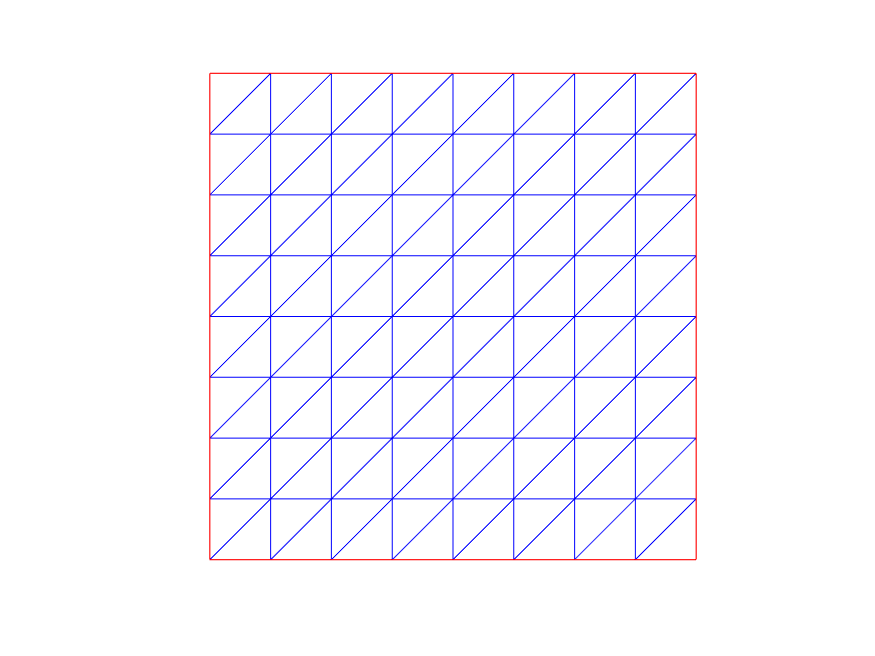}&
 \includegraphics[trim = {2.5cm 1.5cm 2.5cm 1.2cm}, clip, width=.2\textwidth]{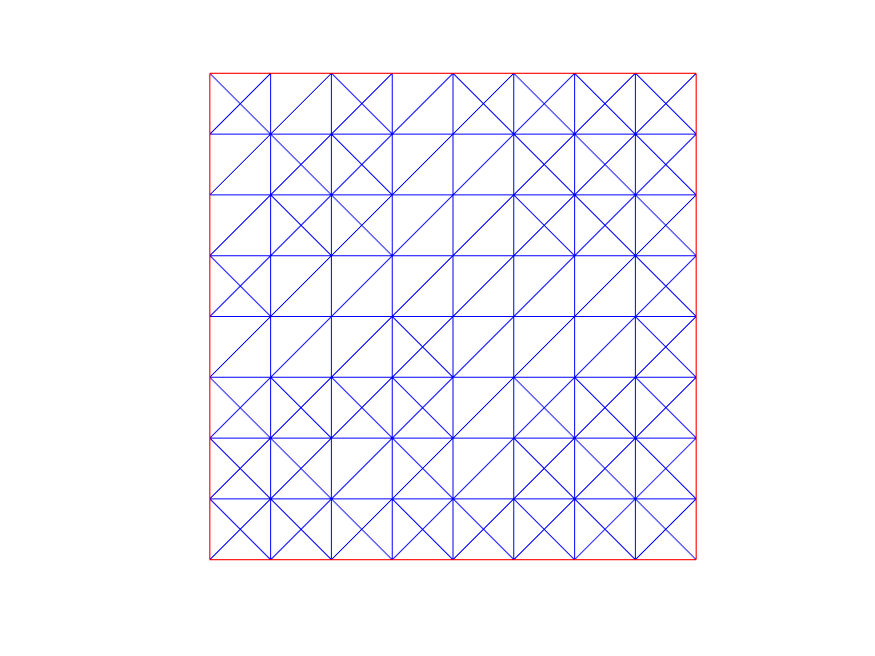}&
 \includegraphics[trim = {2.5cm 1.5cm 2.5cm 1.2cm}, clip, width=.2\textwidth]{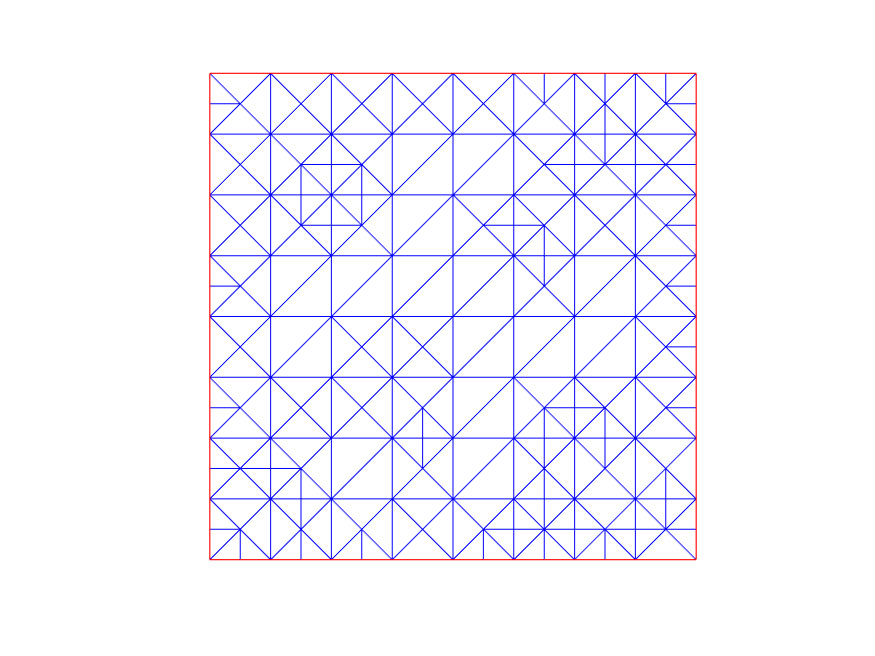}&
 \includegraphics[trim = {2.5cm 1.5cm 2.5cm 1.2cm}, clip, width=.2\textwidth]{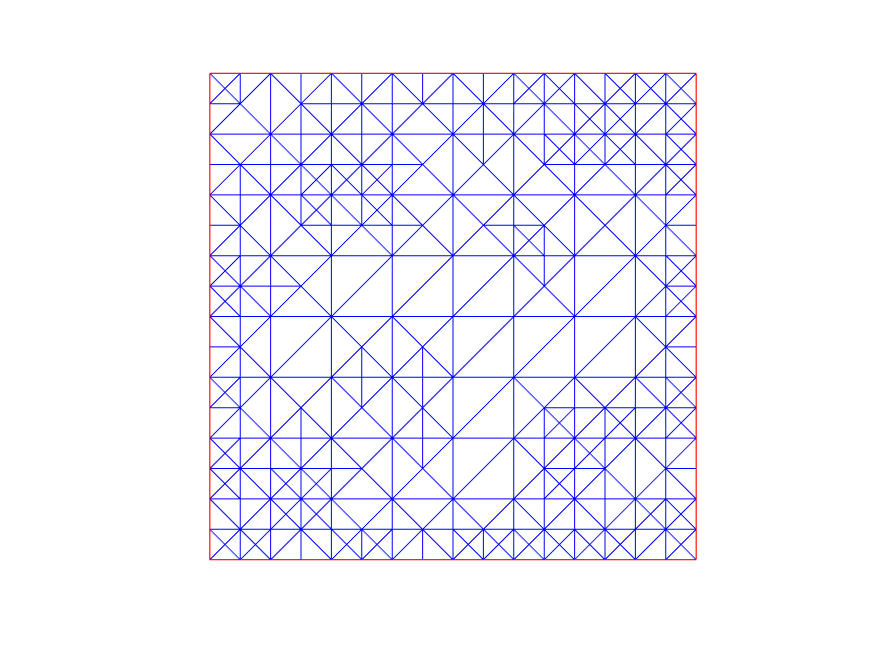}&
 \includegraphics[trim = {2.5cm 1.5cm 2.5cm 1.2cm}, clip, width=.2\textwidth]{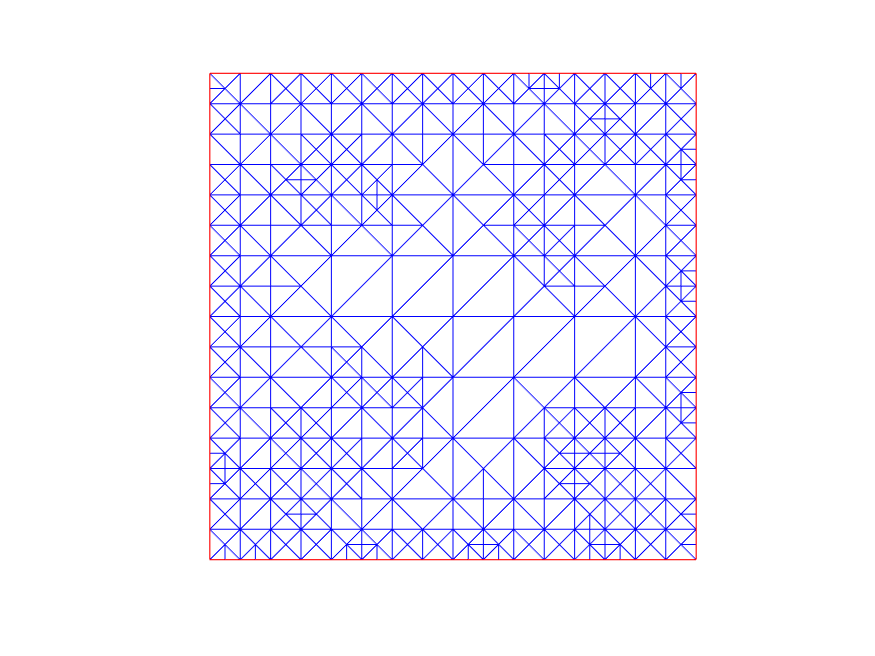}\\
 \includegraphics[trim = 1cm 0.5cm 0.5cm 0cm, width=.2\textwidth]{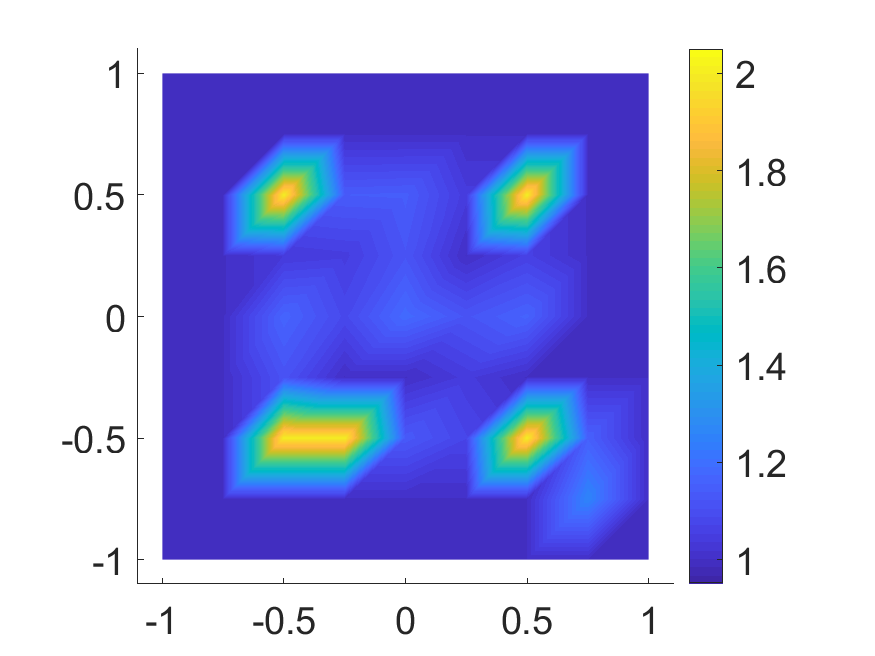}&
 \includegraphics[trim = 1cm 0.5cm 0.5cm 0cm, width=.2\textwidth]{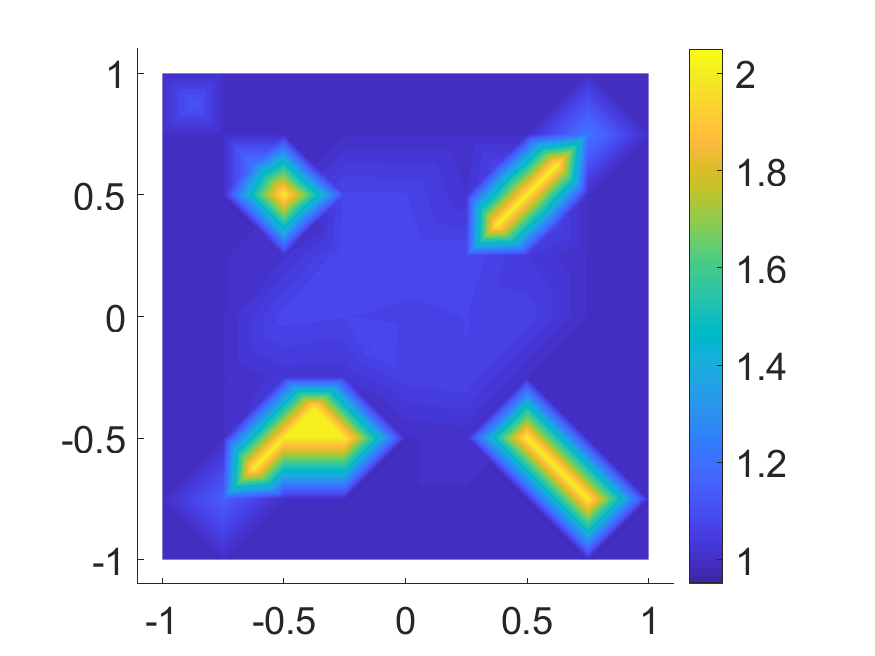}&
 \includegraphics[trim = 1cm 0.5cm 0.5cm 0cm, width=.2\textwidth]{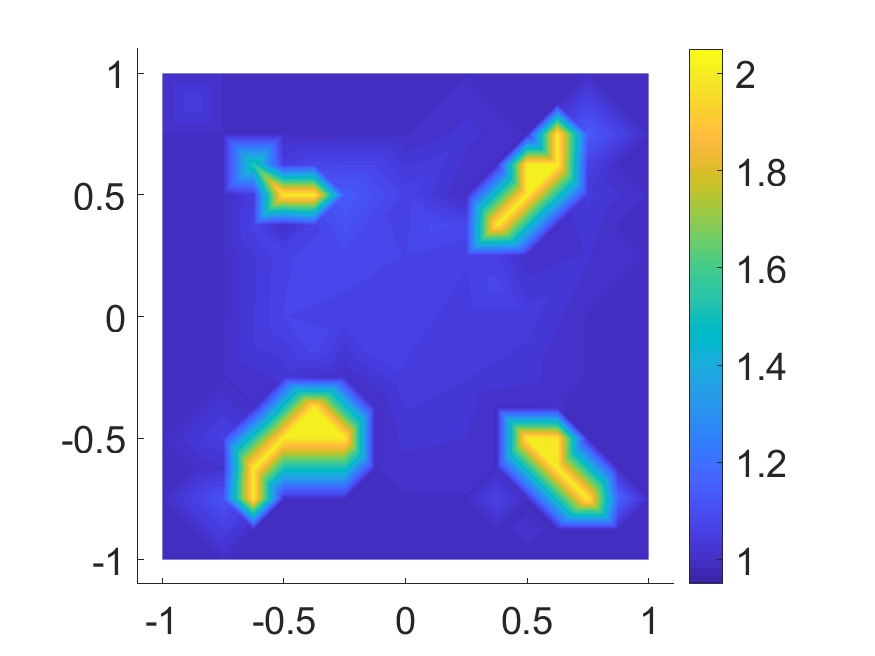}&
 \includegraphics[trim = 1cm 0.5cm 0.5cm 0cm, width=.2\textwidth]{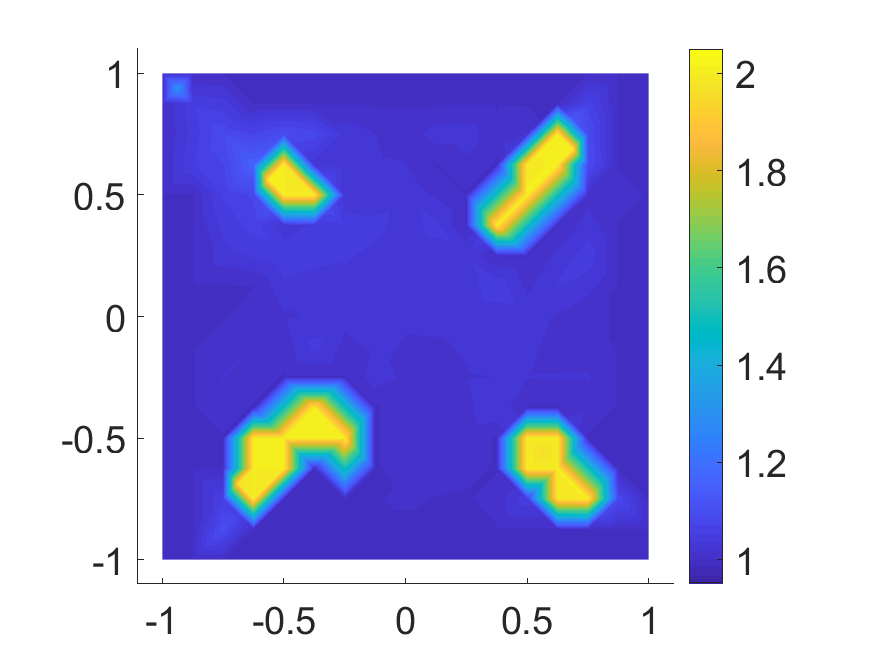}&
 \includegraphics[trim = 1cm 0.5cm 0.5cm 0cm, width=.2\textwidth]{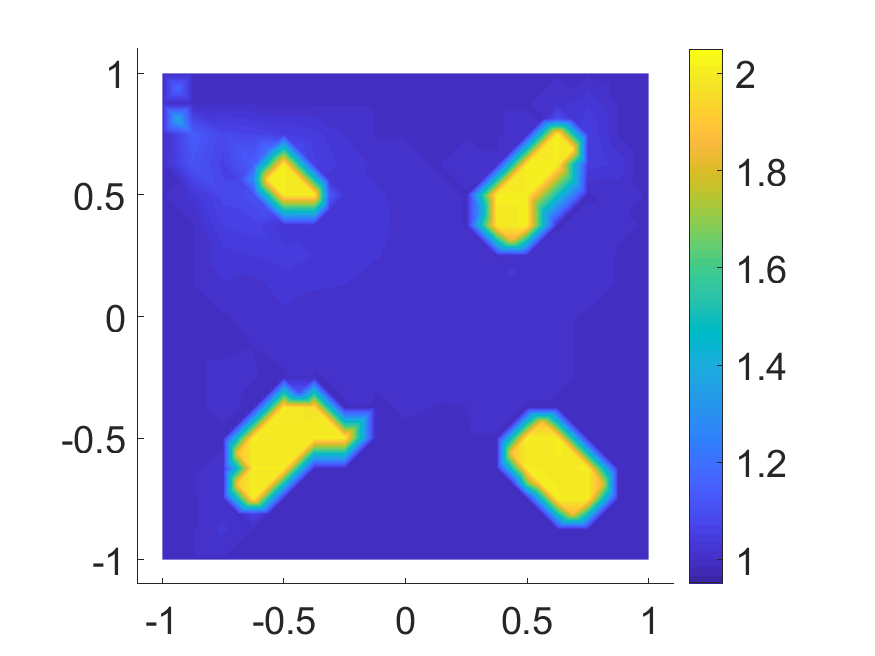}\\
  81  &  122 &  171 &  274  &  389  \\
 \includegraphics[trim = {2.5cm 1.5cm 2.5cm 1.2cm}, clip, width=.2\textwidth]{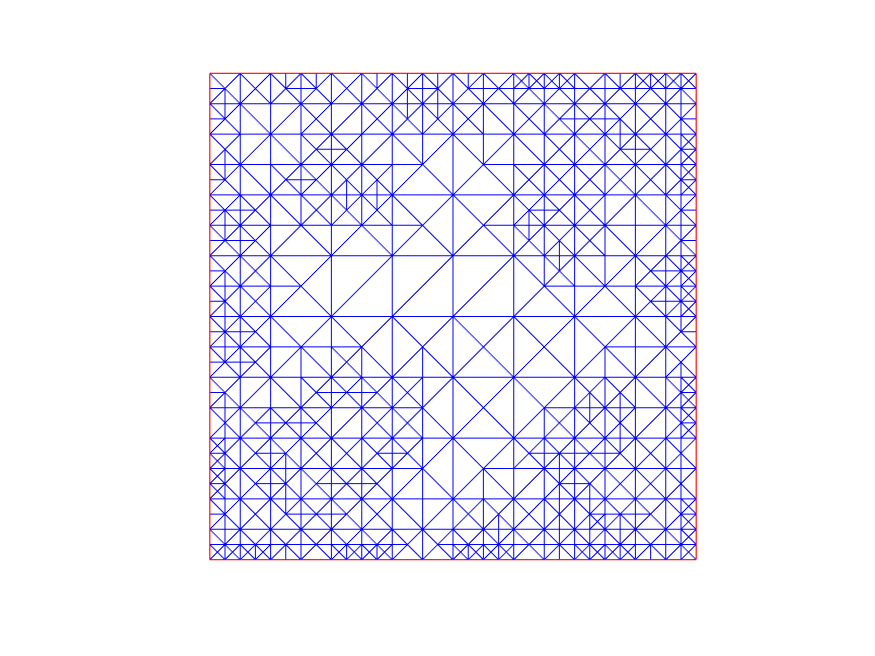}&
 \includegraphics[trim = {2.5cm 1.5cm 2.5cm 1.2cm}, clip, width=.2\textwidth]{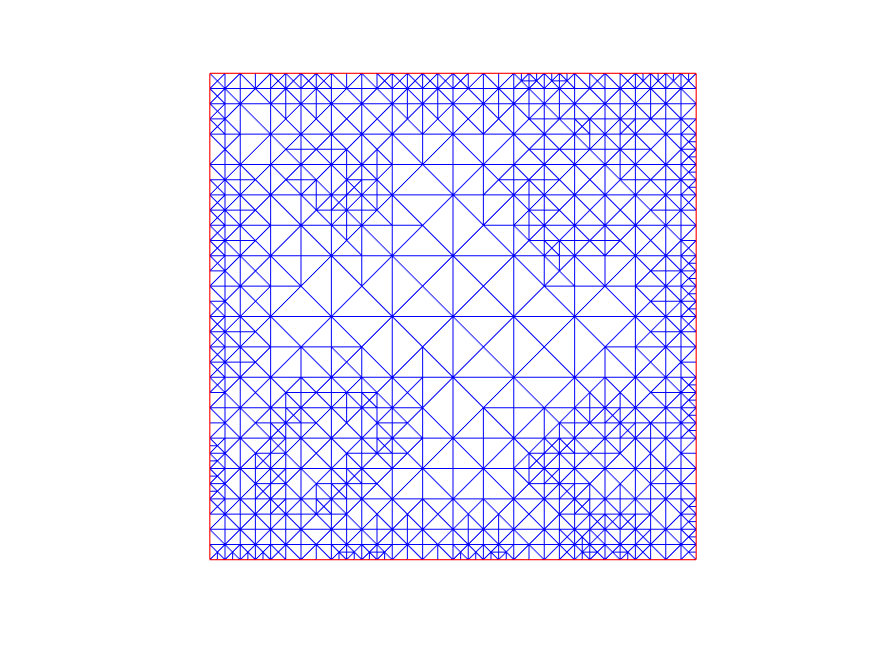}&
 \includegraphics[trim = {2.5cm 1.5cm 2.5cm 1.2cm}, clip, width=.2\textwidth]{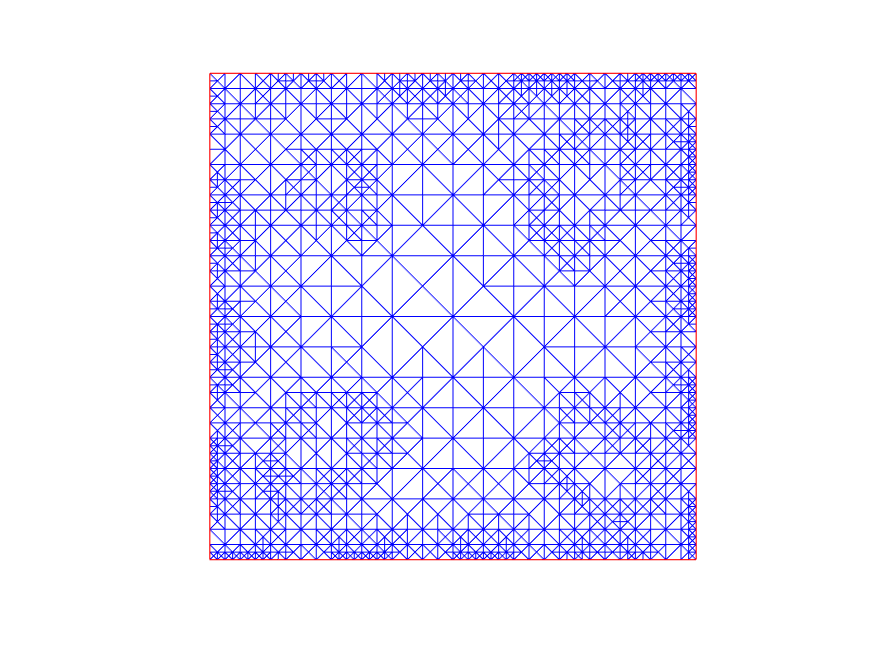}&
 \includegraphics[trim = {2.5cm 1.5cm 2.5cm 1.2cm}, clip, width=.2\textwidth]{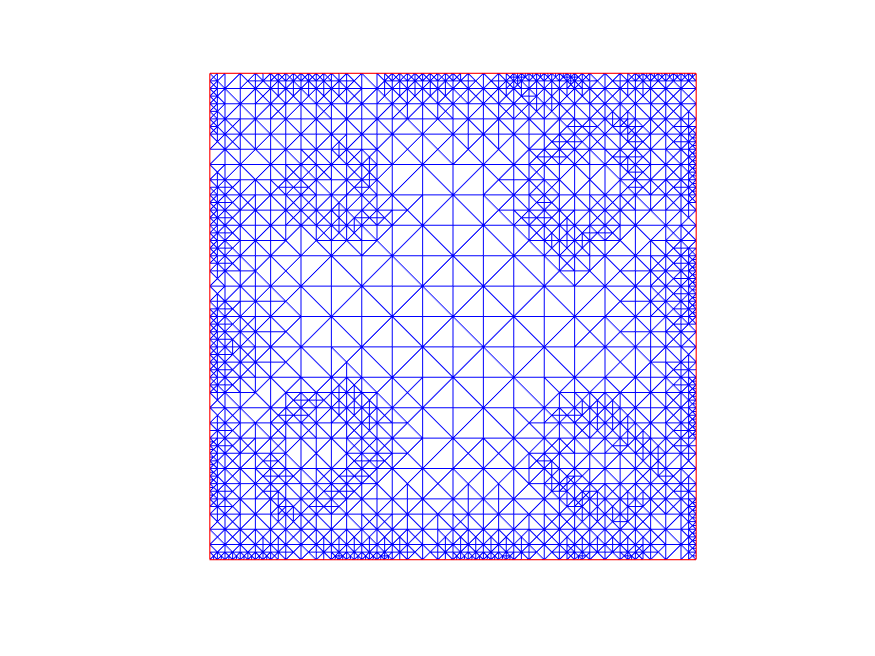}&
 \includegraphics[trim = {2.5cm 1.5cm 2.5cm 1.2cm}, clip, width=.2\textwidth]{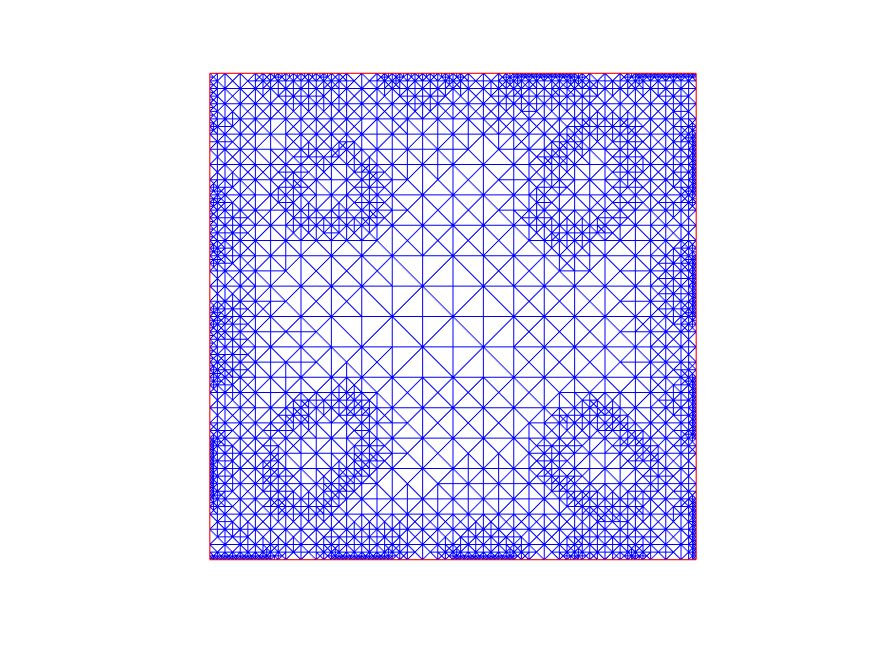}\\
 \includegraphics[trim = 1cm 0.5cm 0.5cm 0cm, width=.2\textwidth]{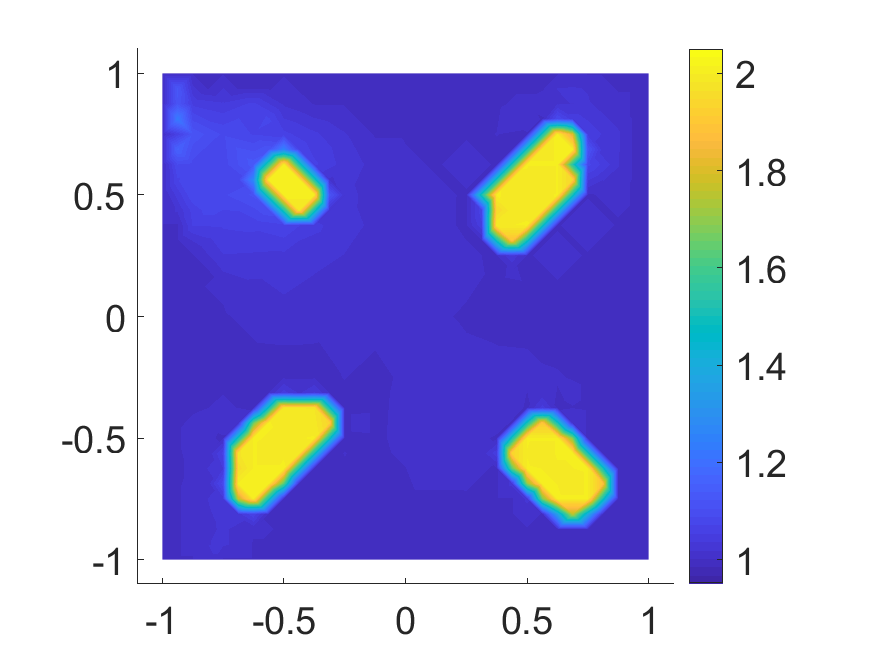}&
 \includegraphics[trim = 1cm 0.5cm 0.5cm 0cm, width=.2\textwidth]{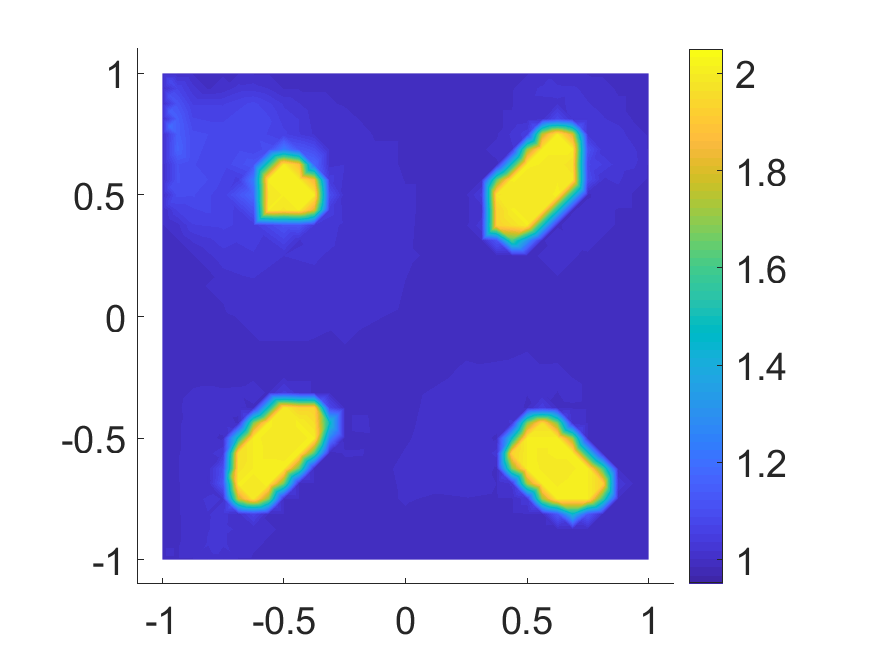}&
 \includegraphics[trim = 1cm 0.5cm 0.5cm 0cm, width=.2\textwidth]{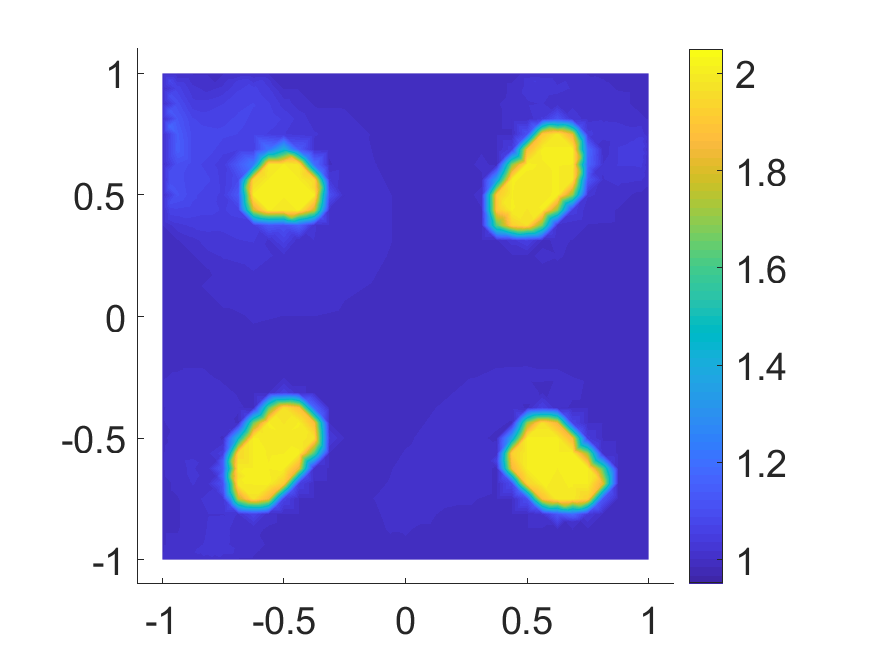}&
 \includegraphics[trim = 1cm 0.5cm 0.5cm 0cm, width=.2\textwidth]{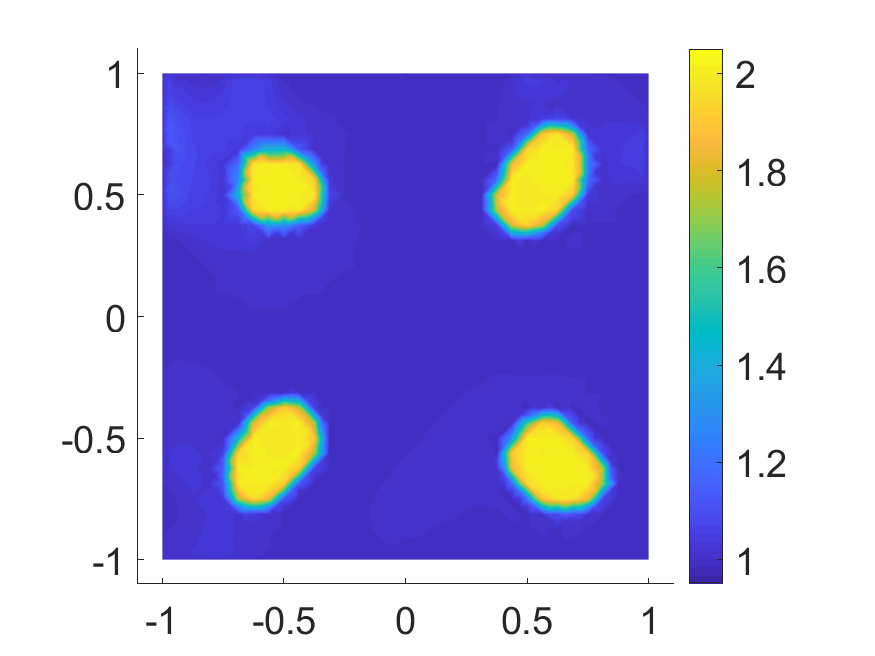}&
 \includegraphics[trim = 1cm 0.5cm 0.5cm 0cm, width=.2\textwidth]{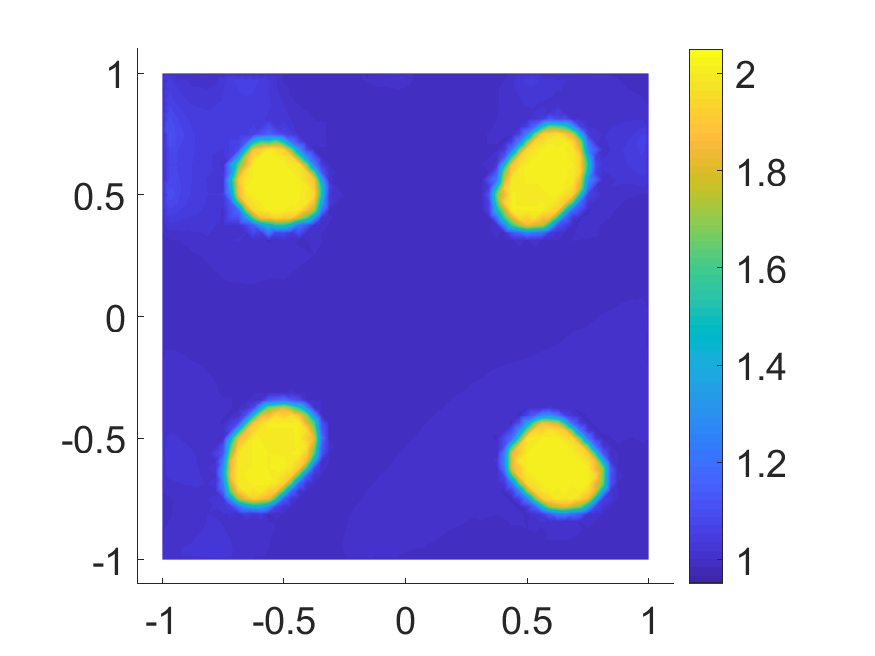}\\
    599  &  844  &  1339 &  1894  &  2904  \\
 \includegraphics[trim = {2.5cm 1.5cm 2.5cm 1.2cm}, clip, width=.2\textwidth]{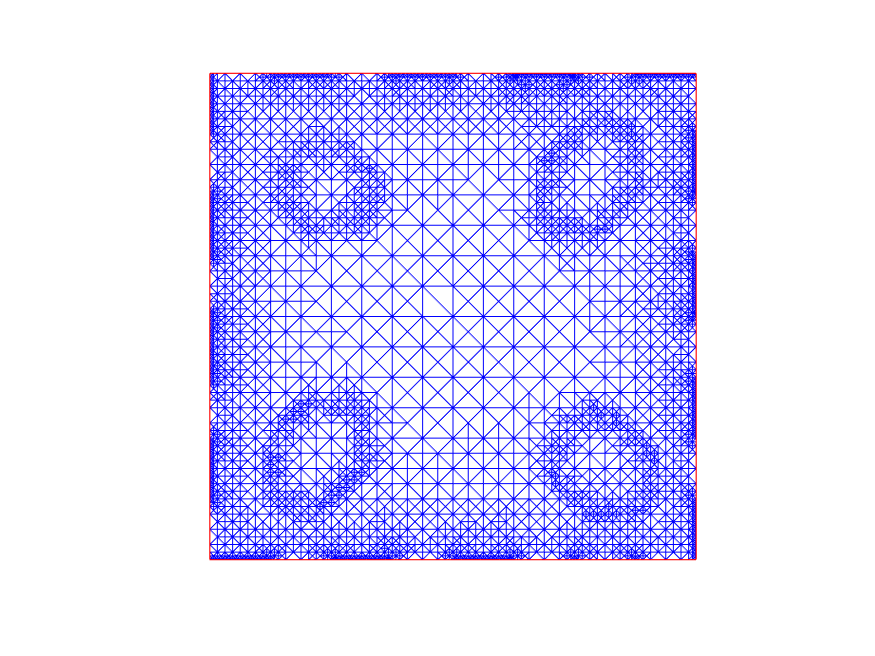}&
 \includegraphics[trim = {2.5cm 1.5cm 2.5cm 1.2cm}, clip, width=.2\textwidth]{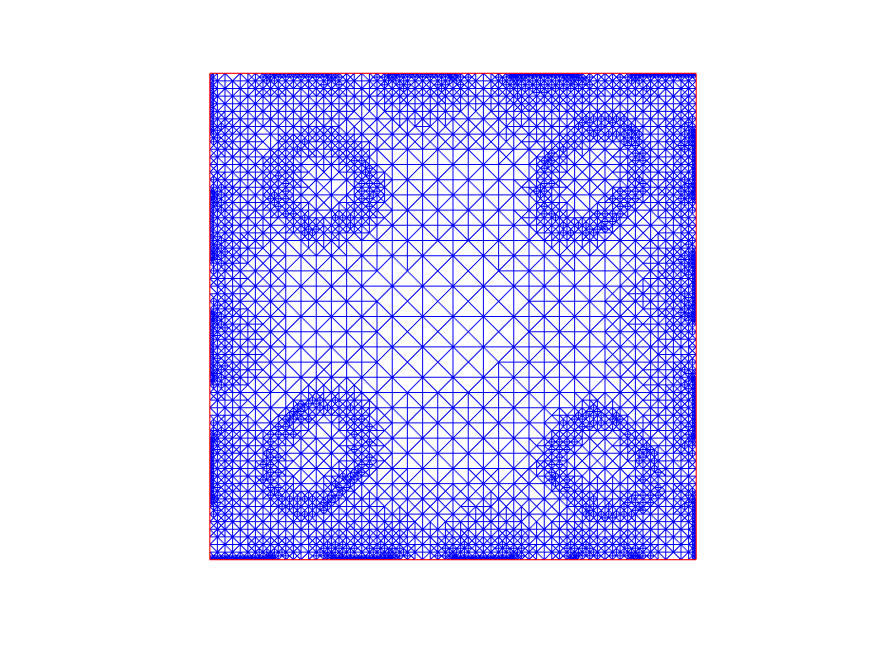}&
 \includegraphics[trim = {2.5cm 1.5cm 2.5cm 1.2cm}, clip, width=.2\textwidth]{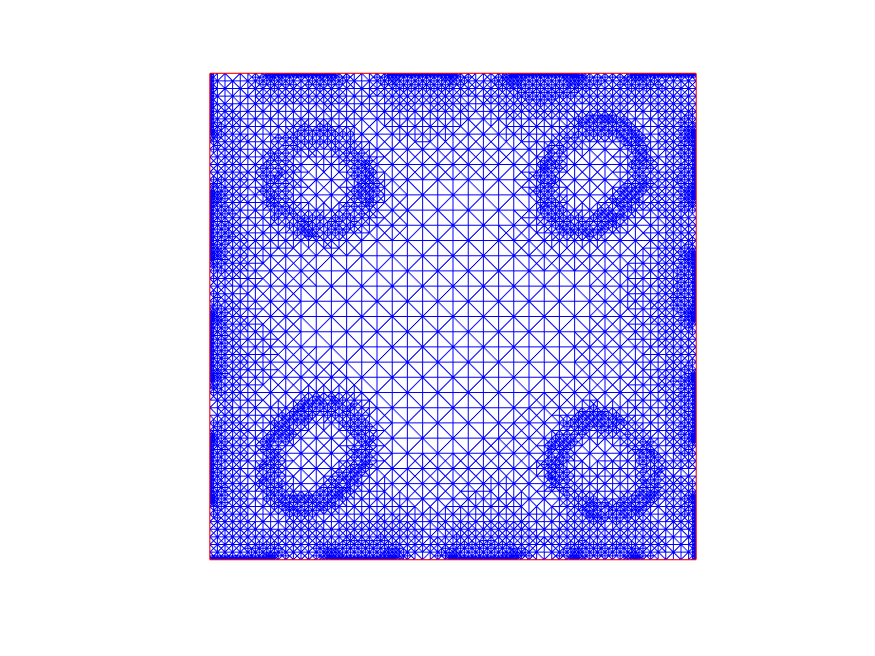}&
 \includegraphics[trim = {2.5cm 1.5cm 2.5cm 1.2cm}, clip, width=.2\textwidth]{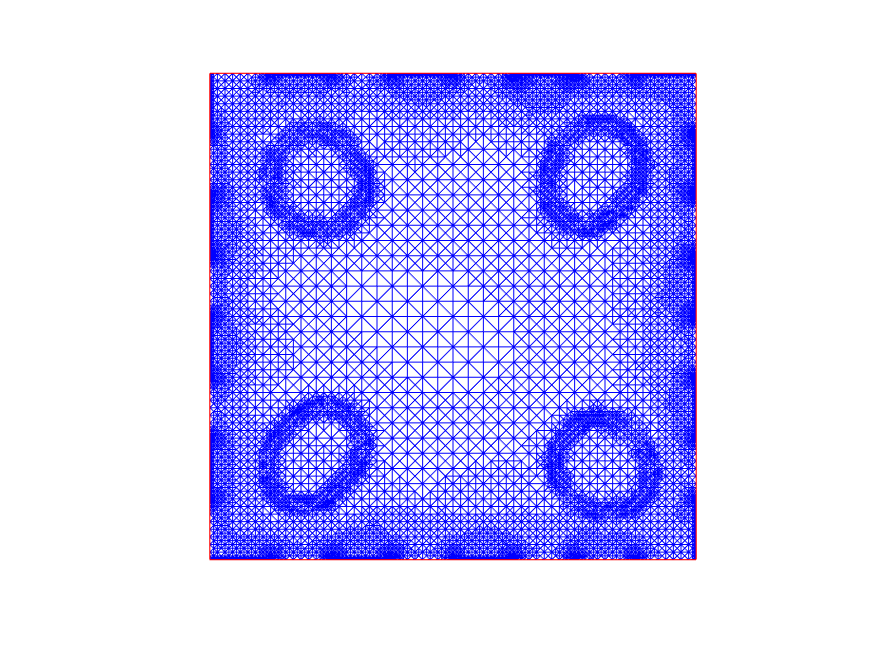}&
 \includegraphics[trim = {2.5cm 1.5cm 2.5cm 1.2cm}, clip, width=.2\textwidth]{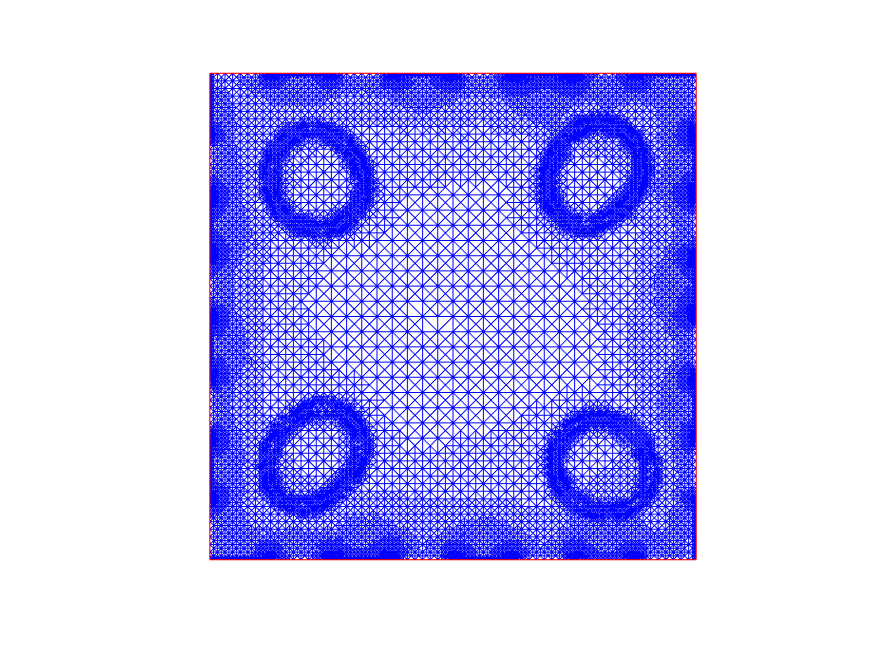}\\
 \includegraphics[trim = 1cm 0.5cm 0.5cm 0cm, width=.2\textwidth]{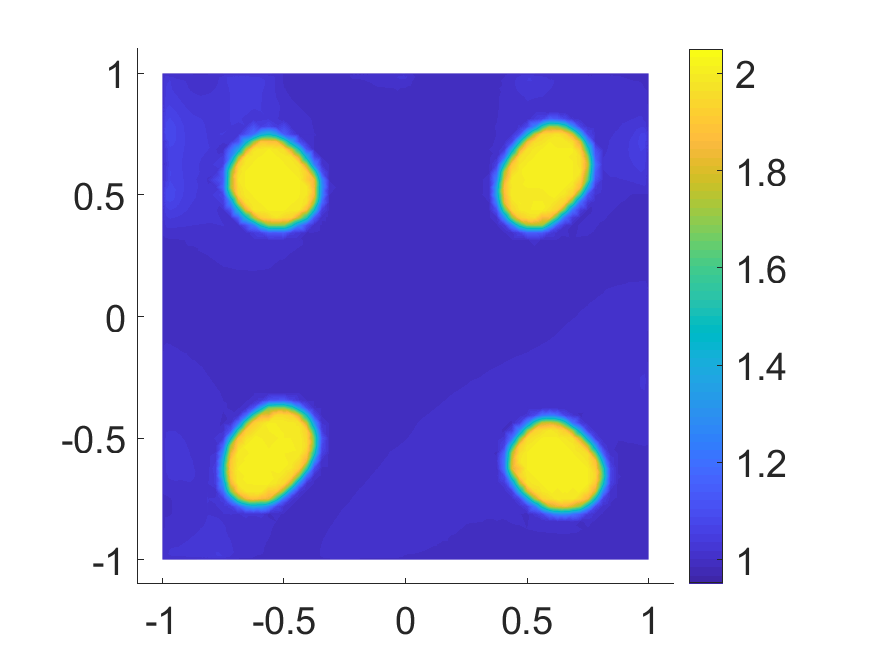}&
 \includegraphics[trim = 1cm 0.5cm 0.5cm 0cm, width=.2\textwidth]{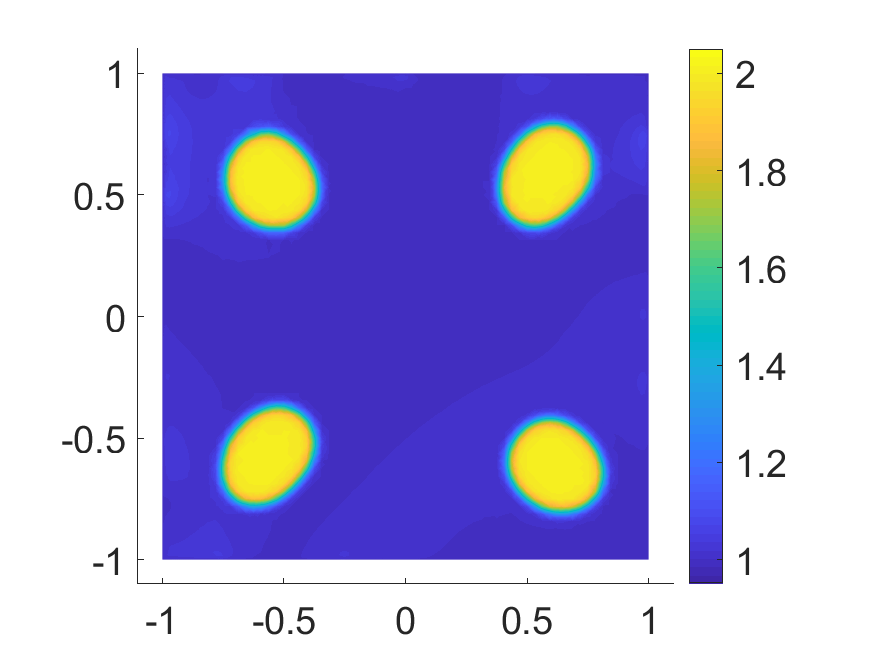}&
 \includegraphics[trim = 1cm 0.5cm 0.5cm 0cm, width=.2\textwidth]{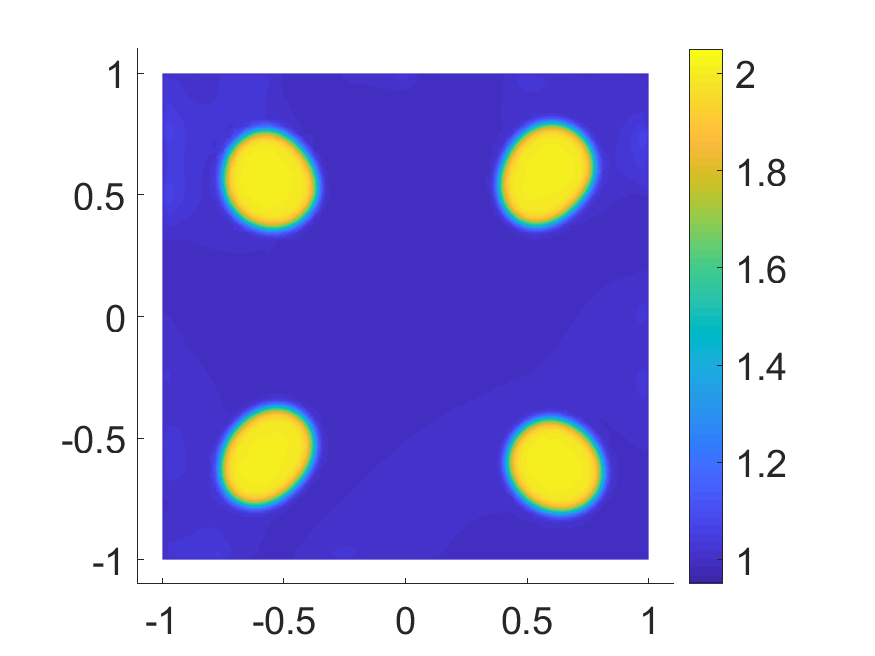}&
 \includegraphics[trim = 1cm 0.5cm 0.5cm 0cm, width=.2\textwidth]{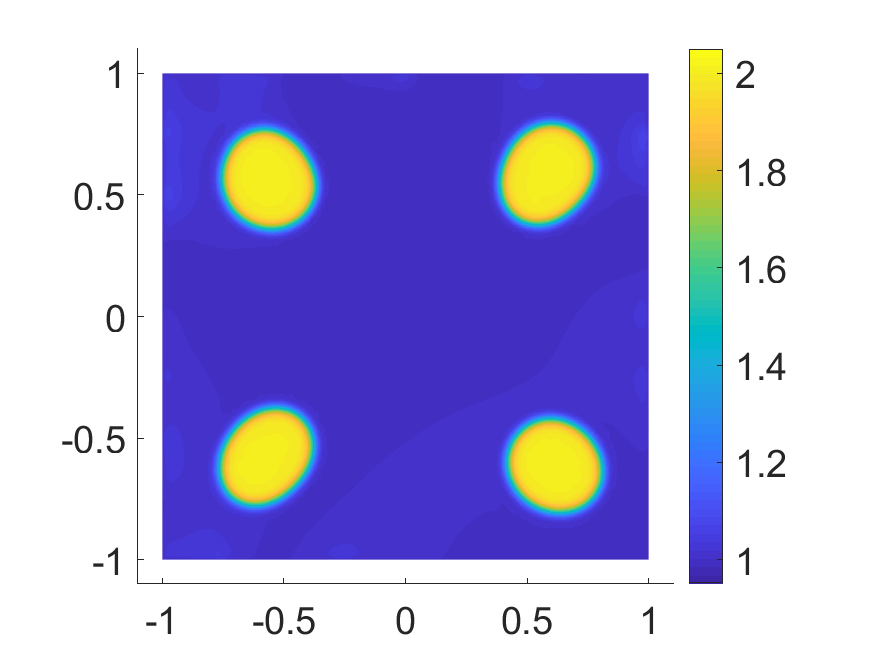}&
 \includegraphics[trim = 1cm 0.5cm 0.5cm 0cm, width=.2\textwidth]{ex3_1e3ad_it15}\\
    3999  &  6133  &  8582  & 12883 & 18008
 \end{tabular}
 \caption{The meshes $\mathcal{T}_k$ and recovered conductivities $\sigma_k$ during the adaptive refinement, for Example \ref{exam3} with
 $\epsilon=\text{1e-3}$ and  $\tilde\alpha=\text{2e-2}$.
 The number under each figure refers to d.o.f.} \label{fig:exam3-recon-iter}
\end{figure}

\begin{figure}[hbt!]
  \centering
  \begin{tabular}{cc}
    \includegraphics[trim = .5cm 0cm 0cm 0cm, clip=true,width=.23\textwidth]{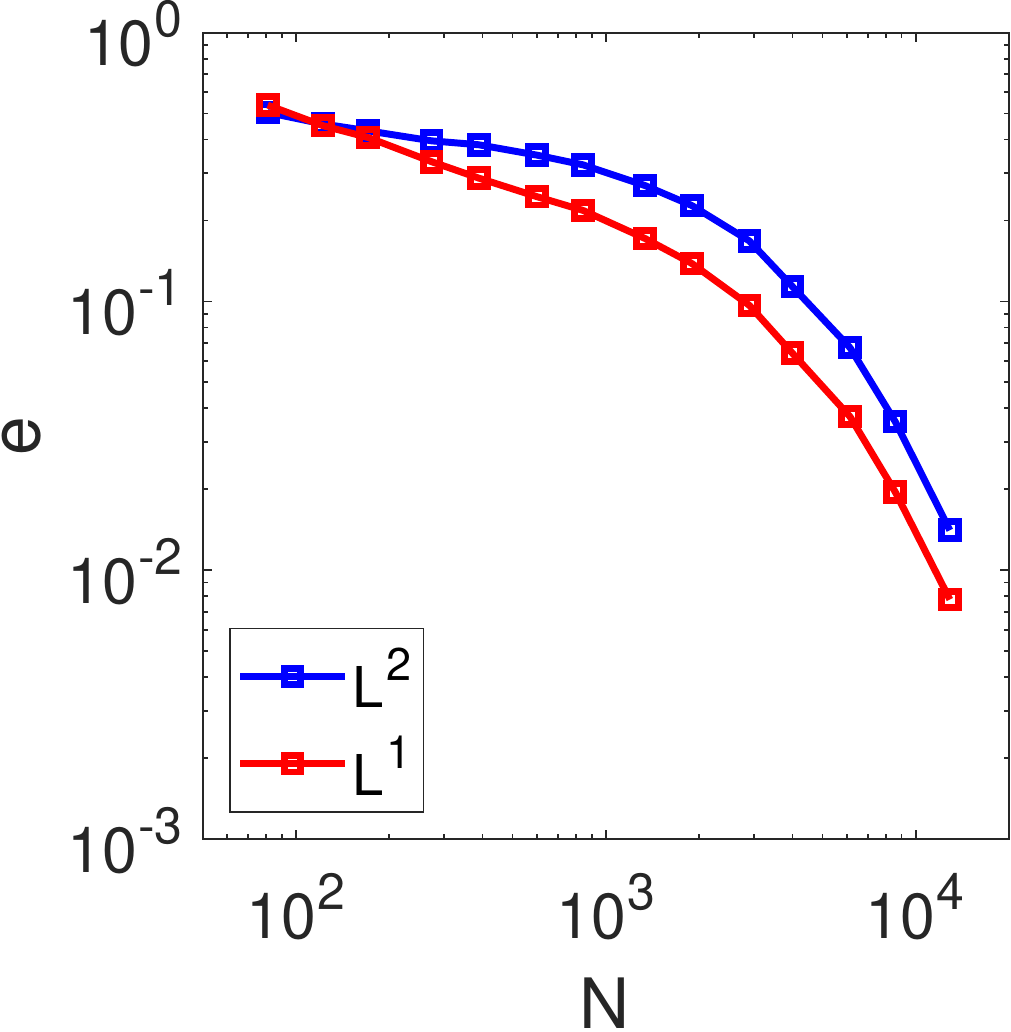}
     \includegraphics[trim = .5cm 0cm 0cm 0cm, clip=true,width=.23\textwidth]{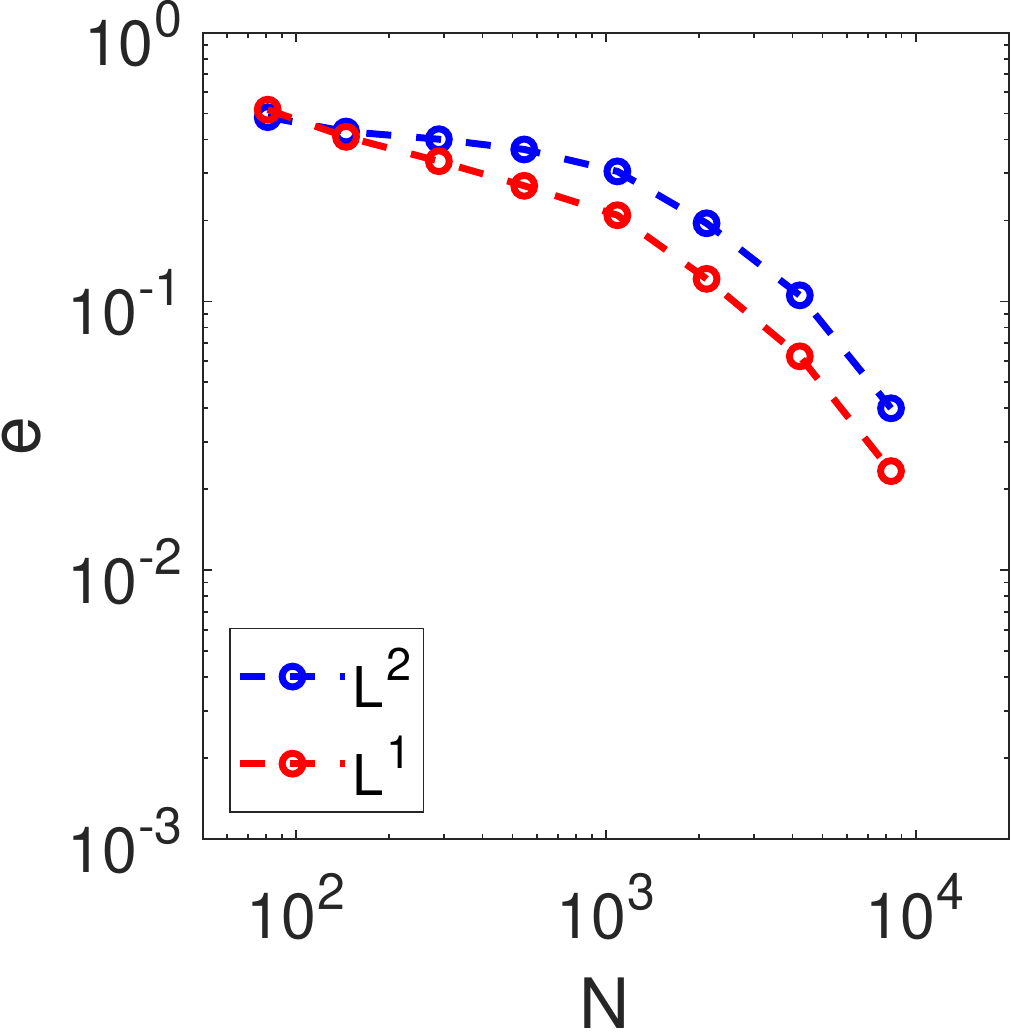}&
    \includegraphics[trim = .5cm 0cm 0cm 0cm, clip=true,width=.23\textwidth]{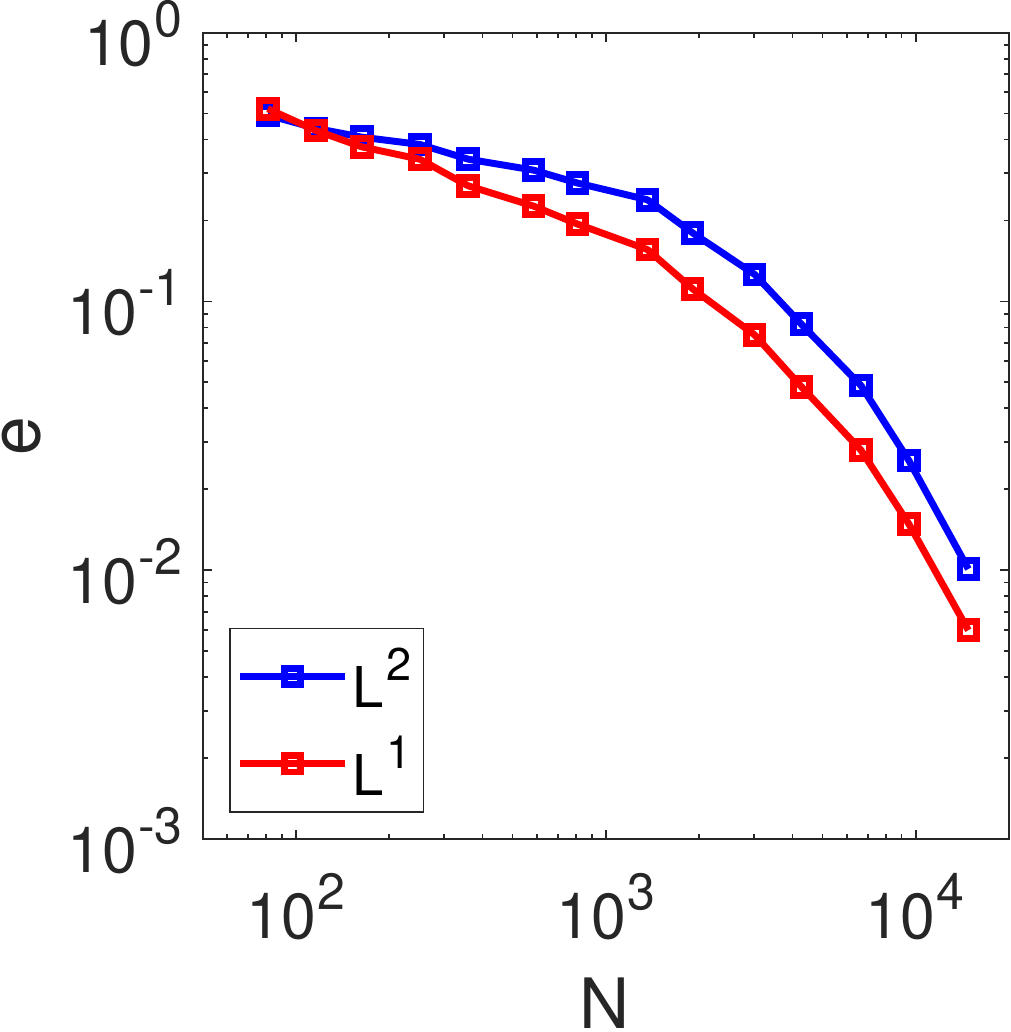}
     \includegraphics[trim = .5cm 0cm 0cm 0cm, clip=true,width=.23\textwidth]{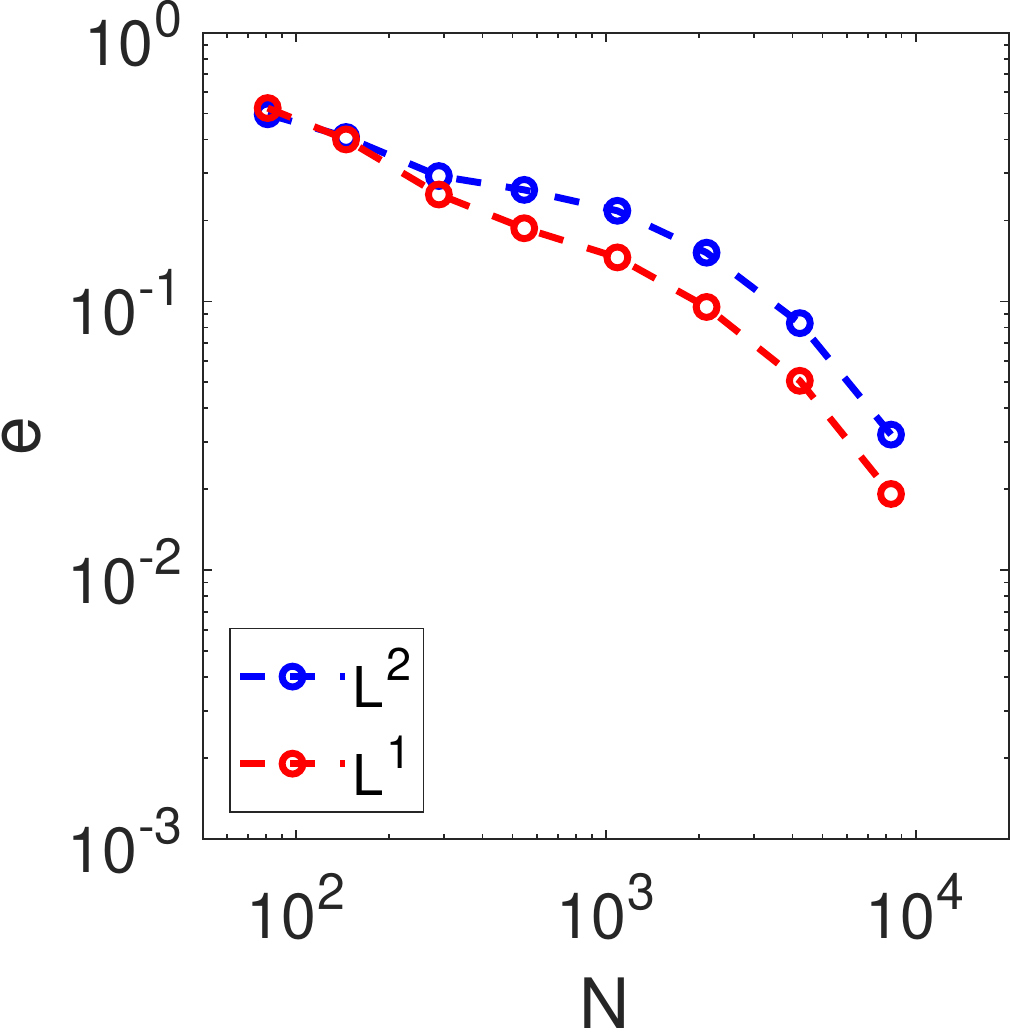}\\
    (a) $\epsilon=\text{1e-3}$, $\tilde\alpha=\text{2e-2}$ & (b) $\epsilon=\text{1e-2}$, $\tilde\alpha=\text{3e-2}$
  \end{tabular}
  \caption{The $L^2(\Omega)$ and $L^1(\Omega)$ errors versus the degree of freedom $N$ of the mesh, for Example \ref{exam3},
  using the adaptive (solid) and uniform (dashed) refinement.}\label{fig:exam3-efficiency}
\end{figure}

\section{Proof of Theorem \ref{thm:conv_alg}}\label{sect:conv}

The lengthy and technical proof is divided into two steps: Step 1 shows the convergence to an auxiliary
minimization problem over a limiting admissible set in Section \ref{subsect:conv1}, and Step 2 shows that
the solution of the auxiliary problem satisfies the necessary optimality system \eqref{eqn:cem-optsys}
in section \ref{subsect:conv2}. The overall proof strategy is similar to \cite{JinXuZou:2016}, and hence
we omit relevant arguments.

\subsection{Auxiliary convergence}\label{subsect:conv1}

Since the two sequences $\{\mathbb{H}_k\}_{k\geq0}$ and $\{\Atilde_k\}_{k\geq0}$ generated by Algorithm \ref{alg_afem_eit} are nested, we may define
\begin{equation*}
    \mathbb{H}_{\infty}:=\overline{\bigcup_{k\geq 0}\mathbb{H}_{k}}~ (\mbox{in}~\mathbb{H}\mbox{-norm})\quad\mbox{and}\quad
    \Atilde_{\infty}:=\overline{\bigcup_{k\geq 0}\Atilde_{k}}~(\mbox{in}~H^{1}(\Omega)\mbox{-norm}).
\end{equation*}

Clearly $\mathbb{H}_{\infty}$ is a closed subspace of $\mathbb{H}$. For the set $\Atilde_{\infty}$, we have
the following result \cite[Lemma 4.1]{JinXuZou:2016}.

\begin{lemma}\label{lem:convexclosed}
    $\Atilde_{\infty}$ is a closed convex subset of $\Atilde$.
\end{lemma}

Over the limiting set $\Atilde_{\infty}$, we define an auxiliary limiting minimization problem:
\begin{equation}\label{eqn:medopt}
  \min_{\sigma_{\infty}\in\Atilde_{\infty}}\left\{\J_{\varepsilon,\infty}(\sigma_{\infty}) = \tfrac{1}{2}\|U_{\infty}(\sigma_{\infty})-U^\delta\|^2 + \tfrac{\widetilde{\alpha}}{2}\mathcal{F}_\varepsilon(\sigma_\infty)\right\},
\end{equation}
where $(u_{\infty},U_{\infty})\in \mathbb{H}_{\infty}$ satisfies
\begin{equation}\label{eqn:medpoly}
  a(\sigma_\infty,(u_\infty,U_\infty),(v,V)) = \langle I,V\rangle \quad \forall (v,V)\in\mathbb{H}_\infty.
\end{equation}
By Lemma \ref{lem:normequiv} and Lax-Milgram theorem, problem \eqref{eqn:medpoly} is well-posed
for any fixed $\sigma_\infty\in\mathcal{A}_\infty$. The next result gives the existence
of a minimizer to \eqref{eqn:medopt}--\eqref{eqn:medpoly}.

\begin{theorem}\label{thm:existmedmin}
    There exists at least one minimizer to problem \eqref{eqn:medopt}--\eqref{eqn:medpoly}.
\end{theorem}
\begin{proof}
Let $\{\sigma_k^\ast, (u_k^\ast, U_k^\ast)\}_{k\geq 0}$ be the sequence of discrete solutions given by Algorithm
\ref{alg_afem_eit}. Since $c_1\in \Atilde_k$ for all $k$, by \eqref{stab-discpolyadj}, $\J_{\eps,k}
(\sigma_k^\ast)\leq \J_{\eps,k}(c_1)\leq c$, and thus $\{\sigma_{k}^\ast\}_{k\geq 0}$ is uniformly bounded in $H^1(\Omega)$. By Lemma \ref{lem:convexclosed}
and Sobolev embedding, there exist a subsequence, denoted by $\{\sigma_{k_j}^\ast\}_{j\geq 0}$, and some $\sigma^\ast\in
\Atilde_\infty$ such that
\begin{equation}\label{pf:medmin_01}
        \sigma_{k_j}^\ast \rightharpoonup \sigma^\ast\quad \mbox{weakly in}~H^1(\Om),\quad
        \sigma_{k_j}^\ast \to \sigma^\ast\quad \mbox{in}~L^2(\Om),\quad
        \sigma_{k_j}^\ast \to \sigma^\ast\quad \mbox{a.e. in}~\Om,
    \end{equation}
Next we introduce a discrete analogue of problem \eqref{eqn:medpoly} with $\sigma_\infty=\sigma^\ast$: find
$(u_{k_j},U_{k_j})\in \mathbb{H}_{k_j}$ such that
\begin{equation}\label{pf:medmin_02}
    a(\sigma^\ast,(u_{k_j},U_{k_j}),(v,V)) = \langle I,V\rangle \quad \forall (v,V)\in\mathbb{H}_{k_j}.
\end{equation}
By Lemma \ref{lem:normequiv}, Cea's lemma and the construction of the space $\mathbb{H}_{\infty}$, the solution
$(u^\ast_\infty,U_\infty^\ast)\in\mathbb{H}_\infty$ of \eqref{eqn:medpoly} with $\sigma_\infty=\sigma^\ast$ satisfies
\begin{equation}\label{pf:medmin_03}
    \|(u^\ast_\infty-u^\ast_{k_j},U^\ast_\infty-U^\ast_{k_j})\|_{\mathbb{H}}\leq c \inf_{(v,V)\in \mathbb{H}_{k_j}}\|(u^\ast_\infty-v,U^\ast_\infty-V)\|_{\mathbb{H}} \to 0.
\end{equation}
Taking the test function $(v,V)=(u_{k_j}-u_{k_j}^\ast,U_{k_j}-U_{k_j}^\ast)\in\mathbb{H}_{k_j}$ in the first line of \eqref{eqn:cem-discoptsys}
and \eqref{pf:medmin_02} and then applying the Cauchy-Schwarz inequality lead to
\begin{align*}
    &\quad a(\sigma_{k_j}^\ast,(u_{k_j}-u_{k_j}^\ast,U_{k_j}-U_{k_j}^\ast),(u_{k_j}-u_{k_j}^\ast,U_{k_j}-U_{k_j}^\ast))\\
    & =((\sigma_{k_j}^\ast-\sigma^\ast)\nabla(u_{k_j}-u^\ast_{\infty}),\nabla(u_{k_j}-u_{k_j}^\ast))
      +((\sigma_{k_j}^\ast-\sigma^\ast)\nabla u_{\infty}^\ast,\nabla(u_{k_j}-u_{k_j}^\ast))\\
      &\leq (\|(\sigma_{k_j}^\ast-\sigma^\ast)\nabla(u_{k_j}-u^\ast_{\infty})\|_{L^2(\Omega)} + \|(\sigma_{k_j}^\ast-\sigma^\ast)\nabla u_{\infty}^\ast\|_{L^2(\Omega)} ) \|\nabla(u_{k_j}-u_{k_j}^\ast)\|_{L^2(\Omega)}.
\end{align*}
In view of \eqref{pf:medmin_03}, pointwise convergence in  \eqref{pf:medmin_01} and  Lebesgue's dominated convergence theorem,
\[
    \|(\sigma_{k_j}^\ast-\sigma^\ast)\nabla(u_{k_j}-u^\ast_{\infty})\|_{L^2(\Omega)}\leq c_1 \|\nabla(u_{k_j}-u^\ast_{\infty})\|_{L^2(\Omega)} \to 0,\quad
    \|(\sigma_{k_j}^\ast-\sigma^\ast)\nabla u_{\infty}^\ast\|_{L^2(\Omega)} \to 0,
\]
This and Lemma \ref{lem:normequiv} imply $\|(u_{k_j}-u_{k_j}^\ast,U_{k_j}-U_{k_j}^\ast)\|_{\mathbb{H}} \to 0.$
Then, \eqref{pf:medmin_03} and the triangle inequality imply
\begin{equation}\label{pf:medmin_04}
    \|(u_{k_j}^\ast-u_{\infty}^\ast,U_{k_j}^\ast-U_{\infty}^\ast)\|_{\mathbb{H}} \to 0.
\end{equation}
Meanwhile, repeating the argument of Theorem \ref{thm:tikh-MM} gives
    \begin{equation}\label{pf:medmin_05}
        \int_\Om W(\sigma_{k_j}^\ast)\dx\to
        \int_\Om W(\sigma^\ast)\dx.
    \end{equation}
next we apply a density argument.
For any $\sigma_{\infty}\in \Atilde_\infty$, by the construction of the space $\mathbb{H}_\infty$, there exists a
sequence $\{\sigma_k\}_{k\geq 0}\subset \bigcup_{k\geq0}\Atilde_{k}$ such that $\sigma_k\to\sigma_\infty$ in $H^1(\Om)$.
Repeating the preceding argument gives $\|U(\sigma_{k})-U^\delta\|^2\to\|U(\sigma_\infty)-U^\delta\|^2$ and $\int_{\Om}
W(\sigma_k)\dx\to \int_\Om W(\sigma_{\infty})\dx$.
Now \eqref{pf:medmin_04}, the weak lower semicontinuity of the $H^1(\Omega)$-norm, \eqref{pf:medmin_05} and
the minimizing property of $\sigma_{k}^\ast$ to $\J_{\eps,k}$ over the set $\Atilde_k$ imply
    \begin{align}\label{eqn:lsc-med}
        \J_{\eps,\infty}(\sigma^\ast) &\leq \liminf_{j\to\infty}\J_{\eps,k_j}(\sigma^\ast_{k_j})\leq
        \limsup_{j\to\infty}\J_{\eps,k_j}(\sigma^\ast_{k_j})\nonumber\\
        &\leq\limsup_{k\to\infty}\J_{\eps,k}(\sigma^\ast_{k})\leq
        \limsup_{k\to\infty}\J_{\eps,k}(\sigma_{k})=\J_{\eps,\infty}(\sigma_\infty)\quad\forall\sigma_\infty\in \Atilde_{\infty}.
    \end{align}
Since $\sigma^\ast\in\Atilde_{\infty}$, $\sigma_\infty^\ast:=\sigma^\ast$ is a minimizer of $\J_{\eps,\infty}$ over $\Atilde_{\infty}$.
\end{proof}

Further, we have the following auxiliary convergence.

\begin{theorem}\label{thm:conv_medmin}
The sequence of discrete solutions
$\{\sigma_{k}^{\ast},(u^{\ast}_{k},U_{k}^{\ast})\}_{k\geq0}$ to problem \eqref{eqn:discopt} contains a subsequence
$\{\sigma_{k_{j}}^{\ast},(u_{k_{j}}^{\ast},U_{k_{j}}^\ast)\}_{j\geq 0}$ convergent to a minimizer
$(\sigma_{\infty}^{\ast},(u_{\infty}^{\ast},U_{\infty}^\ast))$ to problem \eqref{eqn:medopt}--\eqref{eqn:medpoly}:
\begin{equation*}
     \sigma_{k_{j}}^\ast\rightarrow\sigma_{\infty}^\ast\quad\mbox{ in}~H^1(\Omega),
     \quad \sigma_{k_{j}}^\ast\rightarrow\sigma_{\infty}^\ast\quad\mbox{a.e. in}~\Omega,
    \quad  (u_{k_{j}}^\ast,U_{k_j}^*)\rightarrow (u_{\infty}^\ast,U_\infty^*)\quad\mbox{ in}~\mathbb{H}.
\end{equation*}
\end{theorem}
\begin{proof}
The convergence of $(u_{k_j}^*,U_{k_j}^*)$ was already proved in Theorem \ref{thm:existmedmin}.
Taking $\sigma_\infty=\sigma_\infty^\ast$ in \eqref{eqn:lsc-med} gives $\lim_{j\to\infty}\J_{\eps,k_j}(\sigma_{k_j}^\ast)
=\J_{\eps,\infty}(\sigma^\ast_\infty)$. By \eqref{pf:medmin_04} and \eqref{pf:medmin_05}, we have $\|\nabla\sigma_{k_j}^\ast
\|^2_{L^2(\Om)}\to \|\nabla\sigma_{\infty}^\ast\|^2_{L^2(\Om)}$. Thus, the sequence $\{\sigma_{k_j}^\ast\}_{j\geq 0}$
converges to $\sigma_{\infty}^\ast$ in $H^1(\Om)$. 
\end{proof}

Next we consider the convergence of the sequence $\{(p_k^\ast,P_k^\ast)\}_{k\geq 0}$. With a minimizer
$(\sigma_\infty^\ast,(u^\ast_\infty, U^\ast_\infty))$ to problem \eqref{eqn:medopt}, we define a limiting adjoint problem:
find $(p^\ast_{\infty},P^\ast_{\infty})\in\mathbb{H}_\infty$ such that
\begin{equation}\label{eqn:medadj}
    a(\sigma_\infty^*,(p^*_\infty,P_\infty^*),(v,V)) = \langle U_\infty^*-U^\delta,V\rangle\quad \forall (v,V)\in\mathbb{H}_\infty.
\end{equation}
By Lemma \ref{lem:normequiv} and Lax-Milgram theorem, \eqref{eqn:medadj} is uniquely solvable. We have the following
convergence result for $(p_\infty^\ast, P_\infty^\ast)$. The proof is identical with \cite[Theorem 4.5]{JinXuZou:2016},
and hence omitted.

\begin{theorem}\label{thm:conv_medadj}
    Under the condition of Theorem \ref{thm:conv_medmin}, the subsequence of
adjoint solutions $\{(p_{k_{j}}^\ast,P_{k_{j}}^\ast)\}_{j\geq 0}$
generated by Algorithm \ref{alg_afem_eit} converges to the solution $(p_{\infty}^\ast,
P_{\infty}^\ast)$ of problem \eqref{eqn:medadj}:
\begin{equation*}
   \lim_{j\rightarrow\infty}\|(p_{k_j}^\ast-p_{\infty}^\ast,P_{k_j}^\ast-P_{\infty}^\ast)\|_{\mathbb{H}}=0.
\end{equation*}
\end{theorem}

\subsection{Proof of Theorem \ref{thm:conv_alg}}\label{subsect:conv2}

Theorem \ref{thm:conv_alg} follows directly by combining Theorems \ref{thm:conv_medmin}-\ref{thm:conv_medadj}
in Section \ref{subsect:conv1} and Theorems \ref{thm:vp_mc}-\ref{thm:gat_mc} below. The proof in this part
relies on the marking condition \eqref{eqn:marking}. First, we
show that the limit $(\sigma^\ast_{\infty},(u^\ast_{\infty},U^\ast_{\infty}),(p^\ast_{\infty},P^\ast_{\infty}))$
solves the variational equations in \eqref{eqn:cem-optsys}.
\begin{theorem}\label{thm:vp_mc}
The solutions $(\sigma_\infty^*,u_\infty^*,U_\infty^*)$ and $(p_\infty^*,P_\infty^*)$ to problems
\eqref{eqn:medopt}-\eqref{eqn:medpoly} and \eqref{eqn:medadj} satisfy
\begin{equation*}
  \begin{aligned}
   a(\sigma^*_\infty,(u^\ast_\infty,U_\infty^\ast),(v,V)) &= \langle I, V\rangle\quad \forall (v,V)\in\mathbb{H},\\
   a(\sigma^*_\infty,(p^\ast_\infty,P_\infty^\ast),(v,V)) &= \langle U_\infty^*-U^\delta, V\rangle\quad \forall (v,V)\in\mathbb{H}.\\
  \end{aligned}
\end{equation*}
\end{theorem}
\begin{proof}
The proof is identical with \cite[Lemma 4.8]{JinXuZou:2016}, using Theorems \ref{thm:conv_medmin}-\ref{thm:conv_medadj},
and hence we only give a brief sketch. By \cite[Lemma 3.5]{JinXuZou:2016}, for each $T\in\cT_k$ with its face $F$ (intersecting
with $e_l$), there hold
\begin{align*}
  \eta_{k,1}^2(\sigma_{k}^\ast,u^\ast_{k},U^\ast_{k},T) &\leq c(\|\nabla u^\ast_{k}\|^2_{L^2(D_{T})}+h_{F}\|u^\ast_{k}-U^{\ast}_{k,l}\|^2_{L^2(F\cap e_{l})}),\\
  \eta_{k,2}^2(\sigma_{k}^\ast,p^\ast_{k},P^\ast_{k},T)&\leq c(\|\nabla p^\ast_{k}\|^2_{L^2(D_{T})}+h_{F}\|p_k^\ast-P^{\ast}_{k,l}\|^2_{L^2(F\cap e_{l})}),
\end{align*}
where the notation $D_T$ is defined below. Then by the marking condition \eqref{eqn:marking}, \cite[Lemma 4.6]{JinXuZou:2016}
implies that for each convergent subsequence $\{\sigma^{\ast}_{k_j},(u^{\ast}_{k_j},U^{\ast}_{k_j}),(p^{\ast}_{k_j},
P^{\ast}_{k_j})\}_{j\geq0}$ from Theorems \ref{thm:conv_medmin} and \ref{thm:conv_medadj}, there hold
\begin{align*}
   \lim_{j\rightarrow\infty}\max_{T\in\mathcal{M}_{k_j}^1}\eta_{k_{j},1}(\sigma_{k_j}^*,u_{k_j}^*,U_{k_j}^*,T)=0 \quad\mbox{and}\quad
   \lim_{j\rightarrow\infty}\max_{T\in\mathcal{M}_{k_j}^2}\eta_{k_{j},2}(\sigma_{k_j}^*,p_{k_j}^*,P_{k_j}^*,T)=0.
\end{align*}
Last, by \cite[Lemma 4.7]{JinXuZou:2016} and Theorems \ref{thm:conv_medmin}-\ref{thm:conv_medadj}, the argument
of \cite[Lemma 4.8]{JinXuZou:2016} completes the proof.
\end{proof}

\begin{remark}\label{rem:residual_weakconv}
The argument of Theorem \ref{thm:vp_mc} dates back to \cite{Siebert:2011}, and the main tools
include the Galerkin orthogonality of the residual operator, the Lagrange and the Scott-Zhang interpolation
operators \cite{Ciarlet:2002,ScottZhang:1990}, the marking condition \eqref{eqn:marking} and a density
argument. Further, the error estimators $\eta_{k,1}(\sigma_k^*,u_k^*,U_k^*)$ and $\eta_{k,2}(\sigma_k^*,p_k^*,P_k^*)$ emerge
in the proof and are then employed in the module \emph{\texttt{ESTIMATE}} of Algorithm \ref{alg_afem_eit}.
\end{remark}

Next we prove that the limit $(\sigma^\ast_{\infty},(u^\ast_{\infty},U^\ast_{\infty}), (p^\ast_{\infty},
P^\ast_{\infty}))$ satisfies the variational inequality in \eqref{eqn:cem-optsys}. The proof relies
crucially on a constraint preserving interpolation operator. We denote by $D_{T}$ the union of elements 
in $\cT$ with a non-empty intersection with an element $T\in\cT$, and by $\omega_F$ the union of elements 
in $\cT$ sharing a common face/edge with $F\in\mathcal{F}_\cT$. Let
\begin{equation*}
    \cT_{k}^{+}:=\bigcap_{l\geq k}\cT_{l},\quad
    \cT_{k}^{0}:=\cT_{k}\setminus\cT_{k}^{+},\quad
    \Omega_{k}^{+}:=\bigcup_{T\in\cT^{+}_{k}}D_{T},\quad
    \Omega_{k}^{0}:=\bigcup_{T\in\cT^{0}_{k}}D_{T}.
\end{equation*}
The set $\cT_{k}^{+}$ consists of all elements not refined after the $k$-th iteration, and
all elements in $\cT_{k}^{0}$ are refined at least once after the $k$-th iteration. Clearly, $\cT_{l}^{+}
\subset\cT_{k}^{+}$ for $l<k$. We also define a mesh-size function $h_{k}:\overline{\Omega}\rightarrow
\mathbb{R}^{+}$ almost everywhere
\begin{equation*}
  h_k(x) = \left\{\begin{array}{ll}
     h_T, & \quad x\in T^i,\\
     h_F , & \quad x\in F^i,
  \end{array}\right.
\end{equation*}
where ${T}^i$ denotes the interior of an element $T\in\cT_{k}$, and
${F}^i$ the relative interior of an edge $F\in\mathcal{F}_{k}$. It has the following
property \cite[Corollary 3.3]{Siebert:2011}:
\begin{equation}\label{eqn:conv_zero_mesh}
    \lim_{k\rightarrow\infty}\|h_{k}\chi_{\Omega_k^0}\|_{L^\infty(\Omega)}=0.
\end{equation}

The next result gives the limiting behaviour of the maximal error indicator $\eta_{k,3}$.
\begin{lemma}\label{lem:estmarked}
Let $\{(\sigma^{\ast}_k,(u^{\ast}_k,U^{\ast}_k),(p^{\ast}_k,P^{\ast}_k))\}_{k\geq 0}$
be the sequence of discrete solutions generated by Algorithm
\ref{alg_afem_eit}. Then for each convergent subsequence
$\{\sigma^{\ast}_{k_j},(u^{\ast}_{k_j},U^{\ast}_{k_j}),(p^{\ast}_{k_j},P^{\ast}_{k_j})\}_{j\geq0}$, there holds
\begin{equation*}
   \lim_{j\rightarrow\infty}\max_{T\in\mathcal{M}_{k_j}^3}\eta_{k_{j},3}(\sigma_{k_j}^*,u_{k_j}^*,p_{k_j}^*,T)=0.
\end{equation*}
\end{lemma}
\begin{proof}
The inverse estimate and  scaled trace theorem imply that for each $T\in\cT_k$ (with its face $F$)
\begin{align*}
    h_{T}^q\|\tfrac{\tilde\alpha}{2\varepsilon}W'(\sigma_k^*) - \nabla u^*_k\cdot\nabla p^*_k\|^q_{L^q(T)} & \leq ch_T^q\|\nabla u^*_k\cdot\nabla p^*_k\|_{L^q(T)}^q+ ch_T^q\|W'(\sigma_k^*)\|_{L^q(T)}^q \\
    &\leq ch_T^{q}h_T^{d-dq}\|\nabla u^*_k\cdot\nabla p^*_k\|^q_{L^1(T)}+ch_T^q\|W'(\sigma_k^*)\|_{L^q(T)}^q,\\
    \sum_{F\subset\partial T}h_F\|J_{F,2}(\sigma_k^*)\|_{L^q(F)}^q&\leq c\sum_{F\subset\partial T}h_Fh_{F}^{-1}\|\nabla\sigma_{k}^*\|^q_{L^q(\omega_F)}.
\end{align*}
With the choice $q=d/(d-1)$, combining these two estimates gives
\begin{align}\label{eqn:stab-indicator}
 \eta_{k,3}^q(\sigma_{k}^\ast,u^\ast_{k},p^\ast_{k},T)&\leq
    c(\|\nabla u^*_k\cdot\nabla p^*_k\|^q_{L^1(T)}+h_T^q\|W'(\sigma_k^*)\|_{L^q(T)}^q+\|\nabla\sigma_{k}^\ast\|^q_{L^q(D_{T})}),
\end{align}
where $c$ depends on $\widetilde{\alpha}$ and $\eps$ in $\F_\eps$. Next,
for the subsequence $\{\sigma^{\ast}_{k_j},(u^{\ast}_{k_j},U^{\ast}_{k_j}),(p^{\ast}_{k_j},P^{\ast}_{k_j})\}_{j\geq0}$,
let $\widetilde{T}_{j}^3 \in \mathcal{M}_{k_{j}}^3$ be the element with the largest error indicator
$\eta_{k_j,3}(\sigma_{k_j}^*,u_{k_j}^*,p_{k_j}^*,T)$. Since $D_{\widetilde{T}_j^i}\subset\Omega_{k_j}^0$,
\eqref{eqn:conv_zero_mesh} implies
\begin{equation}\label{eqn:estmarked}
   |D_{\widetilde{T}^i_j}|\leq c\|h_{k_{j}}\|^d_{L^{\infty}(\Omega_{k_j}^0)}\rightarrow 0 \quad\mbox{as}~j\rightarrow\infty.
\end{equation}
By \eqref{eqn:stab-indicator}, Cauchy-Schwarz inequality and triangle inequality,
there holds
\begin{align*}
  \eta_{k_j,3}^q(\sigma^{\ast}_{k_j},&u^\ast_{k_j},p^\ast_{k_j},\widetilde{T}_{j}^3)
   \leq c(\|\nabla u^\ast_{k_j}\|_{L^2(\widetilde{T}_{j}^3)}^q\|\nabla p^\ast_{k_j}\|_{L^2(\widetilde{T}_{j}^3)}^q
   +h_{\widetilde{T}_{j}^3}^q\|W'(\sigma_k^*)\|_{L^q(\widetilde{T}_{j}^3)}^q+\|\nabla\sigma_{k_j}^\ast\|^q_{L^q(D_{\widetilde{T}_{j}^3})})\\
   &\leq c\big((\|\nabla (u^\ast_{k_j}-u^\ast_{\infty})\|_{L^2(\Om)}^q+\|\nabla u^\ast_{\infty}\|_{L^2(\widetilde{T}_{j}^3)}^q)(\|\nabla
   (p^\ast_{k_j}-p^\ast_{\infty})\|_{L^2(\Om)}^q+\|\nabla p^\ast_{\infty}\|_{L^2(\widetilde{T}_{j}^3)}^q)\\
   &\quad+h_{\widetilde{T}_j^3}^q(\|W'(\sigma_{k_j}^\ast)-W'(\sigma_\infty^\ast)\|^q_{L^q(\Om)}+\|W'(\sigma_\infty^\ast)\|^q_{L^q(\widetilde{T}_{j}^3)})\\
   &\quad +(\|\nabla(\sigma_{k_j}^\ast-\sigma_{\infty}^\ast)\|^q_{L^q(\Om)}
   +\|\nabla\sigma_{\infty}^\ast\|^q_{L^q(D_{\widetilde{T}_{j}^3})})\big).
\end{align*}
By Theorems \ref{thm:conv_medmin} and \ref{thm:conv_medadj},
Lebesgue's dominated convergence theorem, the choice $q=d/(d-1)\leq 2$ and H\"{o}lder inequality, we obtain
$\|W'(\sigma_{k_j}^\ast)-W'(\sigma_\infty^\ast)\|^q_{L^q(\Om)}\to 0$ and $\|\nabla(\sigma_{k_j}^\ast-
\sigma_{\infty}^\ast)\|^q_{L^q(\Om)}\to 0$. Then the absolute continuity of the norm $\|\cdot\|_{L^q(\Omega)}$
with respect to Lebesgue measure and \eqref{eqn:estmarked} complete the proof.
\end{proof}

Due to a lack of Galerkin orthogonality for variational inequalities, we employ a local $L^r$-stable interpolation
operator of Cl\'{e}ment/Chen-Nochetto type.
Let $ \mathcal{N}_k$ be the set of all interior nodes of $\cT_k$ and $\{\phi_{x}\}_{x\in\mathcal{N}_k} $ be the nodal basis
functions in $V_k$. For each $x\in\mathcal{N}_k$, the support of $\phi_x$ is denoted by $\omega_x$, i.e., the union of
all elements in $\cT_k$ with a non-empty intersection with $x$. Then we define $\Pi_k:L^1(\Om)\to V_k$ by
\begin{equation}\label{eqn:cn_int_def}
    \Pi_k v := \sum_{x\in\mathcal{N}_k} \frac{1}{|\omega_x|}\int_{\omega_x}v \dx \phi_x.
\end{equation}
Clearly, $\Pi_k v\in \Atilde_k$ if $c_0 \leq v \leq c_1$ a.e. $x\in\Om$. The definition is
adapted from \cite{ChenNochetto:2000} (for elliptic obstacle problems) by replacing the
maximal ball $\Delta_x\subset\omega_x$ centered at an interior node $x$ by $\omega_x$.
$\Pi_k v$ satisfies following properties; see Appendix \ref{app:int-oper} for a proof.
\begin{lemma}\label{lem:cn-int-oper}
For any $v\in W^{1,r}(\Omega)$, there hold for all $r\in[1, +\infty]$, any $T\in \cT_k$ and any $F\subset \partial T$,
\begin{align*}
   \|\Pi_k v\|_{L^r(T)} \leq c \|v\|_{L^r(D_T)}, \quad \|\nabla\Pi_k v\|_{L^r(T)} \leq c \|\nabla v\|_{L^r(D_T)},\\
  \| v - \Pi_k v \|_{L^r(T)} \leq c h_T \| \nabla v \|_{L^r(D_T)}, \quad \| v - \Pi_k v\|_{L^r(F)} \leq c h_F^{1-1/r} \|\nabla v\|_{L^r(D_T)}.
\end{align*}
\end{lemma}

Last we show that the limit $(\sigma^\ast_{\infty},(u^\ast_{\infty},U^\ast_{\infty}),(p^\ast_{\infty},P^\ast_{\infty}))$
satisfies the variational inequality in \eqref{eqn:cem-optsys}.
\begin{theorem}\label{thm:gat_mc}
The solutions $(\sigma_\infty^*,u_\infty^*,U_\infty^*)$ and $(p_\infty^*,P_\infty^*)$ to problems \eqref{eqn:medopt}-\eqref{eqn:medpoly} and \eqref{eqn:medadj} satisfy
\begin{equation*}
  \widetilde{\alpha}\eps(\nabla\sigma^\ast_\infty,\nabla(\mu-\sigma^\ast_\infty)) +
  \tfrac{\widetilde{\alpha}}{2\eps}(W'(\sigma^\ast_\infty),\mu-\sigma^\ast_\infty) - (\nabla u_\infty^\ast,\nabla p_{\infty}^\ast(\mu-\sigma_{\infty}^\ast))\geq0\quad\forall\mu\in\Atilde.
\end{equation*}
\end{theorem}
\begin{proof}
The proof is lengthy, and we break it into five steps.

\noindent{\bf Step i.} Derive a preliminary variational inequality. We relabel the subsequence $\{\sigma_{k_j}^\ast,(u_{k_j}^\ast,U_{k_j}^\ast),(p_{k_j}^\ast,P_{k_j}^\ast)\}_{j\geq0}$  in
Theorems \ref{thm:conv_medmin} and \ref{thm:conv_medadj} as $\{\sigma_{k}^\ast,(u_{k}^\ast,U_{k}^\ast),(p_{k}^\ast,P_{k}^\ast)\}_{k\geq0}$.
Let $I_k$ be the Lagrange interpolation operator on $V_k$, and let $\alpha' = \widetilde\alpha\varepsilon $ and
$\alpha''=\frac{\widetilde\alpha}{2\varepsilon}$. For any $\mu\in\widetilde{\mathcal{A}}\cap
C^\infty(\overline{\Omega})$, $I_{k}\mu\in\widetilde{\mathcal{A}}_{k}$ and let $\nu=\mu-I_k\mu$.
Direct computation gives
\begin{align}
\alpha'&(\nabla\sigma^\ast_k,\nabla(\mu-\sigma^\ast_k))
+\alpha''(W'(\sigma^\ast_k),\mu-\sigma^\ast_k)
    -((\mu-\sigma_{k}^\ast)\nabla u_k^\ast,\nabla p_{k}^\ast )\nonumber\\
   &=\alpha'(\nabla\sigma^\ast_k,\nabla(\mu-I_k\mu))
   +\alpha''(W'(\sigma^\ast_k),\mu-I_k\mu)
   -((\mu-I_{k}\mu)\nabla u_k^\ast,\nabla p_{k}^\ast)\nonumber\\
   &\qquad+\alpha'(\nabla\sigma^\ast_k,\nabla(I_k\mu-\sigma_k^\ast))
   +\alpha''(W'(\sigma^\ast_k),I_k\mu-\sigma^\ast_k)
   -((I_{k}\mu-\sigma_k^\ast)\nabla u_k^\ast,\nabla p_{k}^\ast)\nonumber\\
   & =\alpha'(\nabla\sigma^\ast_k,\nabla(\nu-\Pi_k\nu))
   +\alpha''(W'(\sigma^\ast_k),\nu-\Pi_k\nu)
   -((\nu-\Pi_{k}\nu)\nabla u_k^\ast,\nabla p_{k}^\ast)
   \nonumber\\
   &\qquad+\alpha'(\nabla\sigma^\ast_k,\nabla\Pi_{k}\nu)
   +\alpha''(W'(\sigma^\ast_k),\Pi_{k}\nu)
   -(\Pi_k\nu\nabla u_k^\ast,\nabla p_{k}^\ast)\nonumber\\
   &\qquad+\alpha'(\nabla\sigma^\ast_k,\nabla(I_k\mu-\sigma_k^\ast))
   +\alpha''(W'(\sigma^\ast_k),I_k\mu-\sigma^\ast_k)
   -((I_{k}\mu-\sigma_k^\ast)\nabla u_k^\ast,\nabla p_{k}^\ast)\nonumber\\
   &\geq\left[\alpha'(\nabla\sigma^\ast_k,\nabla(\nu-\Pi_k\nu))
   +\alpha''(W'(\sigma^\ast_k),\nu-\Pi_k\nu)
   -((\nu-\Pi_{k}\nu)\nabla u_k^\ast,\nabla p_{k}^\ast )\right]\nonumber\\
   &\quad+\left[\alpha'(\nabla\sigma^\ast_k,\nabla\Pi_{k}\nu)
   +\alpha''(W'(\sigma^\ast_k),\Pi_{k}\nu)
   -(\Pi_k\nu\nabla u_k^\ast,\nabla p_{k}^\ast)\right]:={\rm I}+ {\rm II},\label{eqn:gat_mc_01}
\end{align}
where the last inequality is due to the variational inequality in \eqref{eqn:cem-discoptsys} with $\mu_k=I_k\mu$.

\noindent\textbf{Step ii.} Bound the ${\rm I}$.
By elementwise integration by parts, H\"{o}lder inequality, the definition of the estimator $\eta_{k,3}$ and
Lemma \ref{lem:cn-int-oper} with $r=q'$ (with $q'$ being the conjugate exponent of $q$),
\begin{align*}
    |{\rm I}| &=  \left| \sum_{T \in \cT_k} \int_T R_{T,2}(\sigma_{k}^\ast, u_k^\ast, p_k^\ast)(\nu-\Pi_k\nu)\dx + \sum_{F\in\mathcal{F}_k}\int_F J_{F,2}(\sigma_k^\ast)(\nu-\Pi_k\nu) \mathrm{d}s \right|    \\
    &\leq    \sum_{T \in \cT_k}\Big(\|R_{T,2}(\sigma_{k}^\ast, u_k^\ast, p_k^\ast)\|_{L^q(T)}\|\nu-\Pi_k\nu\|_{L^{q'}(T)}+\sum_{F\subset\partial T}\| J_{F,2}(\sigma_k^\ast)\|_{L^{q}(F)}\|\nu-\Pi_k\nu\|_{L^{q'}(F)}\Big)\\
    &\leq    c\sum_{T \in \cT_k}\Big(h_T\|R_{T,2}(\sigma_{k}^\ast, u_k^\ast, p_k^\ast)\|_{L^q(T)}+\sum_{F\subset\partial T}h_F^{1/q}\| J_{F,2}(\sigma_k^\ast)\|_{L^{q}(F)}\Big)\|\nabla\nu\|_{L^{q'}(D_T)}\\
    &\leq c\sum_{T\in\cT_{k}}\eta_{k,3}(\sigma_k^\ast,u_k^\ast,p_k^\ast,T)\|\nabla\nu\|_{L^{q'}(D_T)}.
\end{align*}
Thus, for any $k>l$, by (discrete) H\"{o}lder's inequality and the finite overlapping property of the patches $D_T$, due 
to uniform shape regularity of the meshes $\mathcal{T}_k\in\mathbb{T}$, there holds
\begin{align*}
   |{\rm I}| &\leq c\big(\sum_{T\in\cT_{k}\setminus\cT_l^+}\eta_{k,3}(\sigma_k^\ast,u_k^\ast,p_k^\ast,T)\|\nabla\nu\|_{L^{q'}(D_T)}
    +\sum_{T\in\cT_l^+}\eta_{k,3}(\sigma_k^\ast,u_k^\ast,p_k^\ast,T)\|\nabla\nu\|_{L^{q'}(D_T)}\big)\\
    &\leq c\Big(\big(\sum_{T\in\cT_{k}\setminus\cT^+_l}\eta^q_{k,3}(\sigma_k^\ast,u_k^\ast,p_k^\ast,T)\big)^{1/q}\|\nabla(\mu-I_k\mu)\|_{L^{q'}(\Omega_l^0)}\\
    &\qquad    +\big(\sum_{T\in\cT_{l}^+}\eta^q_{k,3}(\sigma_k^\ast,u_k^\ast,p_k^\ast,T)\big)^{1/q}\|\nabla(\mu-I_k\mu)\|_{L^{q'}(\Omega_l^+)}\Big).
\end{align*}
Since $W'(s)\in C^1[c_0,c_1]$, by the pointwise convergence of $\{\sigma_{k}^\ast\}_{k\geq0}$ in Theorem \ref{thm:conv_medmin}
and Lebesgue's dominated convergence theorem, we deduce
\begin{equation}\label{eqn:gat_mc_aux}
    W'(\sigma_{k}^\ast)\to W'(\sigma_\infty^\ast)\quad\mbox{in}~L^2(\Om).
\end{equation}
Since $q=d/{(d-1)}\leq 2$, the sequence $\{W'(\sigma_k^\ast)\}_{k\geq0}$ is uniformly bounded in $L^q(\Omega)$.
By Theorems \ref{thm:conv_medmin} and \ref{thm:conv_medadj}, the sequences $\{\sigma_{k}^\ast\}_{k\geq0}$, $\{u_{k}^\ast\}_{k\geq0}$
and $\{p_{k}^\ast\}_{k\geq0}$ are uniformly bounded in $H^1(\Omega)$. Thus, \eqref{eqn:stab-indicator} and \eqref{eqn:conv_zero_mesh},
and H\"{o}lder inequality give
\begin{align}
     &\sum_{T\in\cT_{k}\setminus\cT^+_l}\eta^q_{k,3}(\sigma_k^\ast,u_k^\ast,p_k^\ast,T)\nonumber\\
     \leq&
     c\Big(\|\nabla u_k^\ast\cdot\nabla p_k^\ast\|_{L^1(\Omega)}^{q-1}\sum_{T\in\cT_{k}\setminus\cT^+_l}\|\nabla u_k^\ast\cdot\nabla p_k^\ast\|_{L^1(T)}
     +\|h_{l}\|^{q}_{L^\infty(\Omega^{0}_l)}\|W'(\sigma_k^\ast)\|^q_{L^q(\Om)}
     +\|\nabla\sigma_{k}^\ast\|^q_{L^q(\Omega)}\Big)\nonumber\\
    \leq & c\Big(\|\nabla u_{k}^\ast\|^q_{L^2(\Omega)}\|\nabla p_k^\ast\|^{q}_{L^2(\Omega)}+\|h_{l}\chi_{\Omega_l^0}\|^{q}_{L^\infty(\Omega)}\|W'(\sigma_k^\ast)\|^q_{L^q(\Om)}
    +\|\nabla\sigma_{k}^\ast\|^q_{L^2(\Omega)}\Big)\leq c.\label{eqn:bdd-eta}
\end{align}
Then by the error estimate of $I_k$ \cite{Ciarlet:2002},
\[
    |{\rm I}|\leq c\|h_{l}\chi_{\Omega^0_l}\|_{L^\infty(\Omega)}\|\mu\|_{W^{2,q'}(\Om)}+c\big(\sum_{T\in\cT_{l}^+}\eta^q_{k,3}(\sigma_k^\ast,u_k^\ast,p_k^\ast,T)\big)^{1/q}\|\mu\|_{W^{2,q'}(\Om)}.
\]
By \eqref{eqn:conv_zero_mesh}, $c\|h_{l}\chi_{\Omega^0_l}\|_{L^\infty(\Omega)}\|\mu\|_{W^{2,q'}(\Om)}\to 0$ as $l\to\infty$.
Since $\cT_l^+\subset\cT_{k}$ for $k>l$, \eqref{eqn:marking} implies
\begin{align*}
  \big(\sum_{T\in\cT^+_l}\eta^q_{k,3}(\sigma_k^\ast,u_k^\ast,p_k^\ast,T)\big)^{1/q}
    &\leq{|\cT_{l}^+|}^{1/q}\max_{T\in\cT_l^+}\eta_{k,3}(\sigma_{k}^\ast,u_{k}^\ast,p_{k}^\ast,T)\leq{|\cT^+_l|}^{1/q}\max_{T\in\mathcal{M}_k^3}\eta_{k,3}(\sigma_{k}^\ast,u_{k}^\ast,p_{k}^\ast,T).
\end{align*}
By Lemma \ref{lem:estmarked}, for any small $\varepsilon>0$, we can choose $k_1>l_1$ for some large
fixed $l_1$ such that whenever $k>k_{1}$,
\begin{equation*}
  c({\sum_{T\in\cT_{l}^+}\eta^q_{k,3}}(\sigma_k^\ast,u_k^\ast,p_k^\ast,T))^{1/q}\|\mu\|_{W^{2,q'}(\Omega)}<\varepsilon.
\end{equation*}
Consequently,
\begin{equation}\label{eqn:gat_mc_02}
   {\rm I} \rightarrow 0\quad\forall\mu\in\widetilde{\mathcal{A}}\cap C^\infty(\overline{\Om}).
\end{equation}

\noindent\textbf{Step iii.} Bound the term ${\rm II}$.
For the term $\rm II$, elementwise integration and H\"{o}lder inequality yield
    \begin{align*}
        |{\rm II}|&= \left|  \sum_{T \in \cT_k} \int_T R_{T,2}(\sigma_{k}^\ast, u_k^\ast, p_k^\ast)\Pi_k\nu\dx + \sum_{F\in\mathcal{F}_k}\int_F J_{F,2}(\sigma_k^\ast)\Pi_k\nu \mathrm{d}s \right|\\
        &\leq
    \sum_{T \in \cT_k}\Big(\|R_{T,2}(\sigma_{k}^\ast, u_k^\ast, p_k^\ast)\|_{L^q(T)}\|\Pi_k\nu\|_{L^{q'}(T)}+\sum_{F\subset\partial T}\| J_{F,2}(\sigma_k^\ast)\|_{L^{q}(F)}\|\Pi_k\nu\|_{L^{q'}(F)}\Big)
    \end{align*}
By the scaled trace theorem, local inverse estimate, $L^{q'}$-stability of $\Pi_k$ in Lemma \ref{lem:cn-int-oper}, local quasi-uniformity
and interpolation error estimate for $I_k$ \cite{Ciarlet:2002}, we deduce that for $k>l$
\begin{align*}
    |{\rm II}|& \leq c\sum_{T \in \cT_k}\Big(h_T\|R_{T,2}(\sigma_{k}^\ast, u_k^\ast, p_k^\ast)\|_{L^q(T)}h_T^{-1}\|\Pi_k\nu\|_{L^{q'}(T)}+\sum_{F\subset\partial T}h_F^{1/q}\| J_{F,2}(\sigma_k^\ast)\|_{L^{q}(F)}h_F^{-1/q-1/q'}\|\Pi_k\nu\|_{L^{q'}(T)}\Big) \\
    &\leq c\sum_{T \in \cT_k}\Big(h_T\|R_{T,2}(\sigma_{k}^\ast, u_k^\ast, p_k^\ast)\|_{L^q(T)}+\sum_{F\subset\partial T}h_F^{1/q}\| J_{F,2}(\sigma_k^\ast)\|_{L^{q}(F)}\Big)h_T^{-1}\|\nu\|_{L^{q'}(D_T)}\\
    &\leq c\sum_{T\in\cT_{k}}\eta_{k,3}(\sigma_k^\ast,u_k^\ast,p_k^\ast,T)h_T^{-1}\|\mu-I_k\mu\|_{L^{q'}(D_T)}\\
    &= c\big(\sum_{T\in\cT_{k}\setminus\cT_l^+}\eta_{k,3}(\sigma_k^\ast,u_k^\ast,p_k^\ast,T)h_T\|\mu\|_{W^{2,q'}(D_T)}+\sum_{T\in\cT_l^+}\eta_{k,3}(\sigma_k^\ast,u_k^\ast,p_k^\ast,T)h_T\|\mu\|_{W^{2,q'}(D_T)}\big)\\
    &\leq c\|h_{l}\chi_{\Omega_l^0}\|_{L^\infty(\Omega)}\big(\sum_{T\in\cT_k\setminus\cT_{l}^+}\eta^q_{k,3}(\sigma_k^\ast,u_k^\ast,p_k^\ast,T)\big)^{1/q}\|\mu\|_{W^{2,q'}(\Om)}
    +c\big(\sum_{T\in\cT_{l}^+}\eta^q_{k,3}(\sigma_k^\ast,u_k^\ast,p_k^\ast,T)\big)^{1/q}\|\mu\|_{W^{2,q'}(\Om)}.
\end{align*}
Since  $\big(\sum_{T\in\cT_k\setminus\cT_{l}^+}\eta^q_{k,3}(\sigma_k^\ast,u_k^\ast,p_k^\ast,T)\big)^{1/q}\leq c$, cf. \eqref{eqn:bdd-eta}, there holds
\[
    |{\rm II}|\leq c\|h_{l}\chi_{\Omega_l^0}\|_{L^\infty(\Omega)}\|\mu\|_{W^{2,q'}(\Om)}+c\big(\sum_{T\in\cT_{l}^+}\eta^q_{k,3}(\sigma_k^\ast,u_k^\ast,p_k^\ast,T)\big)^{1/q}\|\mu\|_{W^{2,q'}(\Om)}.
\]
Now by repeating the argument for the term $\rm I$, we obtain
\begin{equation}\label{eqn:gat_mc_aux2}
   {\rm II} \rightarrow 0\quad\forall\mu\in\widetilde{\mathcal{A}}\cap C^\infty(\overline{\Om}).
\end{equation}
\noindent\textbf{Step iv.} Take limit in preliminary variational inequality.
Using \eqref{eqn:gat_mc_aux} and the $H^1(\Omega)$-convergence of $\{\sigma_{k}^\ast\}_{k\geq0}$ in Theorem \ref{thm:conv_medmin}, we have for each $\mu\in\widetilde{\mathcal{A}}\cap C^\infty(\overline{\Om})$
\begin{equation}\label{eqn:gat_mc_03}
        \alpha'(\nabla\sigma^\ast_k,\nabla(\mu-\sigma^\ast_k))
        +\alpha''(W'(\sigma_{k}^\ast),\mu-\sigma^\ast_k)
        \rightarrow \alpha'(\nabla\sigma^\ast_\infty,\nabla(\mu-\sigma^\ast_\infty))
        +\alpha''(W'(\sigma_{\infty}^\ast),\mu-\sigma^\ast_\infty).
\end{equation}
Further, the uniform boundedness on $\{u_k^\ast\}_{k\geq0}$ in $H^1(\Omega)$ and the convergence of $\{p^\ast_{k}\}_{k\geq0}$ to $p_\infty^\ast$
in $H^1(\Omega)$ in Theorem \ref{thm:conv_medadj} yield
\begin{equation*}
   |(\mu\nabla u_k^\ast,\nabla (p_{k}^\ast-p_\infty^\ast))|\leq c\|\nabla(p_{k}^\ast-p_\infty^\ast)\|_{L^2(\Omega)}\rightarrow 0.
\end{equation*}
This and Theorem \ref{thm:conv_medmin} imply
\begin{equation}\label{eqn:gat_mc_04}
    (\mu\nabla u_k^\ast,\nabla p_{k}^\ast )=(\mu \nabla u_k^\ast,\nabla (p_{k}^\ast-p_\infty^\ast))+
    (\mu\nabla u_k^\ast,\nabla p_\infty^\ast)
    \rightarrow   (\mu\nabla u^\ast_\infty,\nabla p_\infty^\ast )\quad\forall\mu\in\widetilde{\mathcal{A}}\cap C^\infty(\overline{\Om}).
\end{equation}
In the splitting
\begin{align*}
   (\sigma_k^\ast\nabla u_{k}^\ast, \nabla p_k^\ast)-(\sigma_\infty^\ast\nabla u_{\infty}^\ast, \nabla p_\infty^\ast )&=
   (\sigma_k^\ast\nabla u_{k}^\ast, \nabla(p_k^\ast-p_\infty^\ast))
   +((\sigma_k^\ast-\sigma_{\infty}^\ast)\nabla u_{k}^\ast, \nabla p_\infty^\ast)\\
   &\quad+(\sigma_{\infty}^\ast\nabla(u_{k}^\ast-u_\infty^\ast),\nabla p_\infty^\ast),
\end{align*}
the arguments for \eqref{eqn:gat_mc_04} directly yields
\begin{equation*}
  |(\sigma_k^\ast\nabla u_{k}^\ast, \nabla(p_k^\ast-p_\infty^\ast) )|\rightarrow 0\quad\mbox{and}\quad
  |(\sigma_{\infty}^\ast\nabla(u_{k}^\ast-u_\infty^\ast),\nabla p_\infty^\ast)|\rightarrow 0.
\end{equation*}
The boundedness on $\{u_k^\ast\}_{k\geq0}$ in $H^1(\Omega)$, pointwise convergence of $\{\sigma_{k}^\ast\}_{k\geq0}$ of
Theorem \ref{thm:conv_medmin} and Lebesgue's dominated convergence theorem imply
\begin{equation*}
    |((\sigma_k^\ast-\sigma_{\infty}^\ast)\nabla u_{k}^\ast, \nabla p_\infty^\ast)|\leq c\|(\sigma_k^\ast-\sigma_{\infty}^\ast)\nabla p_\infty^\ast \|_{L^2(\Omega)}\rightarrow 0.
\end{equation*}
Hence, there holds
\begin{equation}\label{eqn:gat_mc_05}
   (\sigma_k^\ast\nabla u_{k}^\ast, \nabla p_k^\ast)
    \rightarrow (\sigma_\infty^\ast\nabla u_{\infty}^\ast, \nabla
    p_\infty^\ast).
\end{equation}
Now by passing both sides of \eqref{eqn:gat_mc_01} to the limit and combining the estimates \eqref{eqn:gat_mc_02}-\eqref{eqn:gat_mc_05}, we obtain
\begin{equation*}
    \alpha'(\nabla\sigma^\ast_\infty,\nabla(\mu-\sigma^\ast_\infty)) + \alpha''(W'(\sigma^\ast_\infty),\mu-\sigma^\ast_\infty)     -(\nabla u^\ast_\infty,\nabla p_\infty^\ast(\mu-\sigma_\infty^\ast))_{L^2(\Omega)}\geq 0\quad\forall\mu\in\widetilde{\mathcal{A}}\cap C^\infty(\overline{\Om}).
\end{equation*}
\noindent\textbf{Step v.} Density argument.
By the density of $C^\infty(\overline{\Omega})$ in $H^1(\Omega)$ and the construction via a
standard mollifier \cite{Evans:2015}, for any $\mu\in\widetilde{\mathcal{A}}$ there exists a sequence $\{\mu_n\}
\subset\widetilde{\mathcal{A}}\cap C^\infty(\overline{\Om})$ such that $\|\mu_n-\mu\|_{H^1(\Omega)}\rightarrow 0$ as $n\rightarrow\infty$.
Thus, $(\nabla\sigma^\ast_\infty,\nabla \mu_n )\rightarrow(\nabla\sigma^\ast_\infty,\nabla\mu )$,
$(W'(\sigma_\infty^\ast),\mu_n)\to (W'(\sigma_\infty^\ast),\mu)$, and
$(\mu_n\nabla u^\ast_\infty,\nabla p_\infty^\ast )\rightarrow(\mu\nabla u^\ast_\infty,\nabla p_\infty^\ast)$,
after possibly passing to a subsequence. The desired result follows from the preceding two estimates.
\end{proof}

\begin{remark}\label{rmk:gat_mc}
The computable quantity $\eta_{k,3}(\sigma_k^*,u_k^*,p_k^*,T)$ emerges naturally from the proof, i.e., the upper bounds on ${\rm I}$ and ${\rm II}$,
which motivates its use as the \textit{a posteriori} error estimator in Algorithm \ref{alg_afem_eit}.
\end{remark}

\section*{Acknowledgements}
The authors are grateful to an anonymous referee and the boarder member for the constructive comments, which have
significantly improved the presentation of the paper. The work
of Y. Xu was partially supported by National Natural Science Foundation of China (11201307), Ministry of
Education of China through Specialized Research Fund for the Doctoral Program of Higher Education
(20123127120001) and Natural Science Foundation of Shanghai (17ZR1420800).

\appendix

\section{The solution of the variational inequality}\label{app:newton}
Now we describe an iterative method for 
minimizing the energy functional
\begin{equation*}
  \frac{\tilde\alpha \varepsilon  }{2}\|\nabla \sigma\|_{L^2(\Omega)}^2 + \frac{\tilde\alpha}{2\varepsilon}\int_\Omega W(\sigma){\rm d}x + \frac12\|U(\sigma)-U^\delta\|^2.
\end{equation*}
Let $p(z)=(z-c_0)(z-c_1)$. Then one linearized approximation $p_L(z,z_k)$ reads (with $\delta z=z-z_k$)
\begin{align*}
  p_L(z,z_k) &= p(z_k) + p'(z_k)(z-z_k)\\
    &= (z_k^2-(c_0+c_1)z_k+c_0c_1) + (2z_k-c_0-c_1)\delta z.
\end{align*}
Upon substituting the approximation $p_L(z,z_k)$ for $p(z)$ and linearizing the forward map $U(\sigma)$, we obtain the
following surrogate energy functional (with $\delta\sigma=\sigma-\sigma_k$ being the increment and $\delta U= U^\delta-U(\sigma_k)$)
\begin{align}\label{eqn:surrogate}
  \tfrac{\tilde\alpha \varepsilon  }{2}\|\nabla(\sigma_k+\delta\sigma)\|_{L^2(\Omega)}^2 + \tfrac{\tilde\alpha}{2\varepsilon}\|p(\sigma_k)+p'(\sigma_k)\delta\sigma\|_{L^2(\Omega)}^2
    + \tfrac12\|U'(\sigma_k)\delta\sigma-\delta U\|^2.
\end{align}
The treatment of the double well potential term $\int_\Omega W(\sigma){\rm d}x$ is in the spirit of the classical
majorization-minimization algorithm in the following sense (see \cite{ZhangChenXu:2018} for a detailed derivation)
\begin{align*}
  \int_\Omega W(\sigma_k) {\rm d}x &= \int_\Omega p_L(\sigma_k,\sigma_k)^2{\rm d}x, \quad \nabla \int_\Omega W(\sigma_k) {\rm d}x = \nabla \int_\Omega p_L(\sigma_k,\sigma_k)^2{\rm d}x, \\
  \quad\mbox{and} &\quad \nabla^2 \int_\Omega W(\sigma_k) {\rm d}x \leq \nabla^2 \int_\Omega p_L(\sigma_k,\sigma_k)^2{\rm d}x.
\end{align*}
This algorithm is known to have excellent numerical stability. Upon ignoring the box
constraint on the conductivity $\sigma$, problem \eqref{eqn:surrogate} is to find $\delta\sigma \in H^1(\Omega)$ such that
\begin{align*}
  (U'(\sigma_k)^*U'(\sigma_k)\delta\sigma,\phi) + \tilde\alpha \varepsilon& (\nabla \delta\sigma,\nabla \phi) + \tfrac{\tilde\alpha}{\varepsilon}
  (p'(\sigma_k)^2\delta\sigma,\phi) \\ & = (U'(\sigma_k)^*\delta U, \phi)-\tfrac{\tilde\alpha}{\varepsilon}(p(\sigma_k)p'(\sigma_k),\phi)-\tilde\alpha\varepsilon(\nabla\sigma_k,\nabla\phi),\quad \forall \phi\in H^1(\Omega).
\end{align*}
This equation can be solved by an iterative method for the update $\delta\sigma$ (with the box constraint treated by
a projection step). Note that $U'(\sigma_k)$ and $U'(\sigma_k)^*$ can be implemented in matrix-free manner using the
standard adjoint technique. In our experiment, we employ the conjugate gradient method to solve the resulting
linear systems, preconditioned by the sparse matrix corresponding to $\tilde\alpha \varepsilon (\nabla
\delta\sigma,\nabla \phi) + \frac{\tilde\alpha}{\varepsilon}(p'(\sigma_k)^2\delta\sigma,\phi)$.

\section{Proof of Lemma \ref{lem:cn-int-oper}}\label{app:int-oper}
The proof follows that in \cite{ChenNochetto:2000,HildNicaise:2005}.
By H\"{o}lder inequality and $h_T^d \leq |\omega_x|$ for each node $x\in T$,
\[
  \left|\frac{1}{|\omega_x|}\int_{\omega_x}v\dx\right| \leq |\omega_x|^{-1/r}\|v\|_{L^r(\omega_x)}\leq h_T^{-d/r}\|v\|_{L^r(\omega_x)}.
\]
The desired $L^r$-stability follows from the estimate $\|\phi_x\|_{L^r(T)}\leq c h_T^{d/r}$, by the local quasi-uniformity of
the mesh. In view of the definition \eqref{eqn:cn_int_def}, $\Pi_k \zeta = \zeta$ for any $\zeta\in\mathbb{R}$. By local inverse
estimate, the $L^r$-stability of $\Pi_k$, standard interpolation error estimate \cite{Ciarlet:2002} and local quasi-uniformity,
\begin{align}
   \|\nabla \Pi_k v\|_{L^r(T)}&=\inf_{\zeta\in \mathbb{R}} \|\nabla \Pi_k (v - \zeta)\|_{L^r(T)} \leq c h_T^{-1}\inf_{\zeta\in \mathbb{R}} \| \Pi_k (v - \zeta) \|_{L^r(T)} \nonumber\\
   &\leq c h_T^{-1}\inf_{\zeta\in \mathbb{R}} \| v - \zeta \|_{L^r(D_T)} \leq c h_T^{-1} \| v - \frac{1}{|D_T|}\int_{D_T} v \dx \|_{L^r(D_T)} \leq c \|\nabla v \|_{L^r(D_T)}.\label{pf:int_err1}
\end{align}
Similarly,
\begin{equation}\label{pf:int_err2}
\begin{aligned}
  \|v - \Pi_k v\|_{L^r(T)}&= \|v - \zeta - \Pi_k ( v - \zeta )\|_{L^r(T)} \\
    &\leq c \inf_{\zeta \in \mathbb{R}} \| v - \zeta\|_{L^r(D_T)} \leq c h_T \|\nabla v \|_{L^r(D_T)}.
\end{aligned}
\end{equation}
By the scaled trace theorem, for any $F\subset \partial T$, there holds
    \[
        \|v - \Pi_k v \|_{L^r(F)} \leq c (h_F^{-1/r} \| v - \Pi_k v\|_{L^r(T)} + h_F^{1-1/r} \|\nabla ( v - \Pi_k v) \|_{L^r(T)}).
    \]
Then \eqref{pf:int_err1} and \eqref{pf:int_err2} complete the proof of the lemma.

\bibliographystyle{abbrv}
\bibliography{eit}

\begin{thebibliography}{10}

\bibitem{AinsworthOden:2000}
M.~Ainsworth and J.~T. Oden.
\newblock {\em {A} {P}osteriori {E}rror {E}stimation in {F}inite {E}lement
  {A}nalysis}.
\newblock Wiley-Interscience, New York, 2000.

\bibitem{Alberti:2000}
G.~Alberti.
\newblock Variational models for phase transitions¡ªan approach via
  {$\Gamma$}-convergence.
\newblock In L.~Ambrosio, N.~Dancer, G.~Buttazzo, A.~Marino, and M.~K.~V.
  Murthy, editors, {\em Calculus of Variations and Partial Differential
  Equations}, pages 95--114. Springer, New York, 2000.

\bibitem{AlbertiAmmari:2016}
G.~S. Alberti, H.~Ammari, B.~Jin, J.-K. Seo, and W.~Zhang.
\newblock The linearized inverse problem in multifrequency electrical impedance
  tomography.
\newblock {\em SIAM J. Imaging Sci.}, 9(4):1525--1551, 2016.

\bibitem{AttouchButtazzoMichaille:2006}
H.~Attouch, G.~Buttazzo, and G.~Michaille.
\newblock {\em Variational {A}nalysis in {S}obolev and {BV} spaces}.
\newblock SIAM, Philadelphia, PA, 2006.

\bibitem{Bartels:2015}
S.~Bartels.
\newblock Error control and adaptivity for a variational model problem defined
  on functions of bounded variation.
\newblock {\em Math. Comp.}, 84(293):1217--1240, 2015.

\bibitem{BeilinaClason:2006}
L.~Beilina and C.~Clason.
\newblock An adaptive hybrid {FEM}/{FDM} method for an inverse scattering
  problem in scanning acoustic microscopy.
\newblock {\em SIAM J. Sci. Comput.}, 28(1):382--402, 2006.

\bibitem{BeilinaKlibanov:2010a}
L.~Beilina and M.~V. Klibanov.
\newblock {A} posteriori error estimates for the adaptivity technique for the
  tikhonov functional and global convergence for a coefficient inverse problem.
\newblock {\em Inverse Problems}, 26(4):045012, 27pp, 2010.

\bibitem{BeilinaKlibanov:2010b}
L.~Beilina and M.~V. Klibanov.
\newblock {R}econstruction of dielectrics from experimental data via a hybrid
  globally convergent/adaptive algorithm.
\newblock {\em Inverse Problems}, 26(12):125009, 30 pp, 2010.

\bibitem{BeilinaKlibanovKokurin:2010}
L.~Beilina, M.~V. Klibanov, and M.~Y. Kokurin.
\newblock {A}daptivity with relaxation for ill-posed problems and global
  convergence for a coefficient inverse problem.
\newblock {\em J. Math. Sci.}, 167(3):279--325, 2010.

\bibitem{Braides:2002}
A.~Braides.
\newblock {\em $\Gamma$-convergenc for Beginners}.
\newblock Oxford University Press, Oxford, UK, 2002.

\bibitem{Cahn:1958}
J.~Cahn and J.~Hilliard.
\newblock Free energy of a non-uniform system {I} -- {I}nterfacial free energy.
\newblock {\em J. Chem. Phys.}, 28:258--267, 1958.

\bibitem{CFPP:2014}
C.~Carstensen, M.~Feischl, M.~Page, and D.~Praetorius.
\newblock {A}xioms of adaptivity.
\newblock {\em Comput. Math. Appl.}, 67(6):1195--1253, 2014.

\bibitem{ChenNochetto:2000}
Z.~Chen and R.~Nochetto.
\newblock Residual type a posteriori error estimates for elliptic obstacle
  problems.
\newblock {\em Numer. Math.}, 84(4):527--548, 2000.

\bibitem{ChengIsaacsonNewellGisser:1989}
K.-S. Cheng, D.~Isaacson, J.~C. Newell, and D.~G. Gisser.
\newblock {E}lectrode models for electric current computed tomography.
\newblock {\em IEEE Trans. Biomed. Eng.}, 36(9):918--924, 1989.

\bibitem{ChowItoZou:2014}
Y.~T. Chow, K.~Ito, and J.~Zou.
\newblock A direct sampling method for electrical impedance tomography.
\newblock {\em Inverse Problems}, 30(9):095003, 25 pp., 2014.

\bibitem{Ciarlet:2002}
P.~G. Ciarlet.
\newblock {\em {T}he {F}inite {E}lement {M}ethod for {E}lliptic {P}roblems}.
\newblock SIAM, Philadelphia, PA, 2002.

\bibitem{ClasonKaltenbacher:2016}
C.~Clason, B.~Kaltenbacher, and D.~Wachsmuth.
\newblock Functional error estimators for the adaptive discretization of
  inverse problems.
\newblock {\em Inverse Problems}, 32(10):104004, 25 pp, 2016.

\bibitem{DunlopStuart:2015}
M.~M. Dunlop and A.~M. Stuart.
\newblock The {B}ayesian formulation of {EIT}: analysis and algorithms.
\newblock {\em Inverse Probl. Imaging}, 10(4):1007--1036, 2016.

\bibitem{Evans:2015}
L.~C. Evans and R.~F. Gariepy.
\newblock {\em {M}easure {T}heory and {F}ine {P}roperties of {F}unctions}.
\newblock CRC Press, Boca Raton, FL, 2015.
\newblock revised edition.

\bibitem{FengYanLiu:2008}
T.~Feng, Y.~Yan, and W.~Liu.
\newblock {A}daptive finite element methods for the identification of
  distributed parameters in elliptic equation.
\newblock {\em Adv. Comput. Math.}, 29(1):27--53, 2008.

\bibitem{GehreJin:2014}
M.~Gehre and B.~Jin.
\newblock Expectation propagation for nonlinear inverse problems with an
  application to electrical impedance tomography.
\newblock {\em J. Comput. Phys.}, 259:513--535, 2014.

\bibitem{GehreJinLu:2014}
M.~Gehre, B.~Jin, and X.~Lu.
\newblock An analysis of finite element approximation of electrical impedance
  tomography.
\newblock {\em Inverse Problems}, 30(4):045013, 24 pp., 2014.

\bibitem{Grisvard:1985}
P.~Grisvard.
\newblock {\em Elliptic {P}roblems in {N}onsmooth {D}omains}.
\newblock Pitman, Boston, MA, 1985.

\bibitem{HarrachUllrich:2013}
B.~Harrach and M.~Ullrich.
\newblock Monotonicity-based shape reconstruction in electrical impedance
  tomography.
\newblock {\em SIAM J. Math. Anal.}, 45(6):3382--3403, 2013.

\bibitem{HildNicaise:2005}
P.~Hild and S.~Nicaise.
\newblock A posteriori error estimations of residual type for {S}ignorini's
  problem.
\newblock {\em Numer. Math.}, 101(3):523--549, 2005.

\bibitem{HintermullerLaurain2008}
M.~Hinterm{\"u}ller and A.~Laurain.
\newblock Electrical impedance tomography: from topology to shape.
\newblock {\em Control \& Cybernetics}, 37:913--933, 2008.

\bibitem{Hinze:2018}
M.~Hinze, B.~Kaltenbacher, and T.~Quyen.
\newblock Identifying conductivity in electrical impedance tomography with
  total variation regularization.
\newblock {\em Numer. Math.}, 138:723--765, 2018.

\bibitem{HyvonenMustonen:2018}
N.~Hyv\"{o}nen and L.~Mustonen.
\newblock Generalized linearization techniques in electrical impedance
  tomography.
\newblock {\em Numer. Math.}, 140(1):95--120, 2018.

\bibitem{ItoJin:2014}
K.~Ito and B.~Jin.
\newblock {\em {Inverse Problems: Tikhonov Theory and Algorithms}}.
\newblock World Scientific, Singapore, 2014.

\bibitem{JinKhanMaass:2012}
B.~Jin, T.~Khan, and P.~Maass.
\newblock A reconstruction algorithm for electrical impedance tomography based
  on sparsity regularization.
\newblock {\em Internat. J. Numer. Methods Engrg.}, 89(3):337--353, 2012.

\bibitem{JinMaass:2010}
B.~Jin and P.~Maass.
\newblock {A}n analysis of electrical impedance tomography with applications to
  {T}ikhonov regularization.
\newblock {\em ESAIM: Control, Optim. Calc. Var.}, 18(4):1027--1048, 2012.

\bibitem{JinXuZou:2016}
B.~Jin, Y.~Xu, and J.~Zou.
\newblock An adaptive finite element method for electrical impedance
  tomography.
\newblock {\em IMA J. Numer. Anal.}, 37(3):1520--1550, 2017.

\bibitem{KlibanovLiZhang:2019}
M.~Klibanov, J.~Li, and W.~Zhang.
\newblock Convexification of electrical impedance tomography with restricted
  dirichlet-to-neumann map data.
\newblock {\em Inverse Problems}, 35(3):035005, 33 pp., 2019.

\bibitem{Klibanov:2017}
M.~V. Klibanov.
\newblock Convexification of restricted {D}irichlet-to-{N}eumann map.
\newblock {\em J. Inverse Ill-Posed Probl.}, 25(5):669--685, 2017.

\bibitem{KnudsenLassasMueller:2009}
K.~Knudsen, M.~Lassas, J.~L. Mueller, and S.~Siltanen.
\newblock Regularized {D}-bar method for the inverse conductivity problem.
\newblock {\em Inverse Probl. Imaging}, 3(4):599--624, 2009.

\bibitem{LechleiterHyvonen:2008}
A.~Lechleiter, N.~Hyv{\"o}nen, and H.~Hakula.
\newblock The factorization method applied to the complete electrode model of
  impedance tomography.
\newblock {\em SIAM J. Appl. Math.}, 68(4):1097--1121, 2008.

\bibitem{LechleiterRieder:2006}
A.~Lechleiter and A.~Rieder.
\newblock Newton regularizations for impedance tomography: a numerical study.
\newblock {\em Inverse Problems}, 22(6):1967--1987, 2006.

\bibitem{LiXieZou:2011}
J.~Li, J.~Xie, and J.~Zou.
\newblock {A}n adaptive finite element reconstruction of distributed fluxes.
\newblock {\em Inverse Problems}, 27(7):075009, 25pp, 2011.

\bibitem{LiuKolehmainen:2015}
D.~Liu, V.~Kolehmainen, S.~Siltanen, and A.~Sepp\"{a}nen.
\newblock A nonlinear approach to difference imaging in {EIT}; assessment of
  the robustness in the presence of modelling errors.
\newblock {\em Inverse Problems}, 31(3):035012, 25 pp., 2015.

\bibitem{MalonedosSantosHolder:2014}
E.~Malone, G.~S. {dos Santos}, D.~Holder, and S.~Arridge.
\newblock Multifrequency electrical impedance tomography using spectral
  constraints.
\newblock {\em IEEE Trans. Med. Imag.}, 33(2):340--350, 2014.

\bibitem{Mitchell:1989}
W.~F. Mitchell.
\newblock A comparison of adaptive refinement techniques for elliptic problems.
\newblock {\em ACM Trans. Math. Software}, 15(4):326--347 (1990), 1989.

\bibitem{Modica:1987}
L.~Modica.
\newblock The gradient theory of phase transitions and the minimal interface
  criterion.
\newblock {\em Arch. Rational Mech. Anal.}, 98:123--142, 1987.

\bibitem{Modica:1977}
L.~Modica and S.~Mortola.
\newblock Un esempio di {$\Gamma^-$}-convergenza.
\newblock {\em Boll. Un. Mat. Ital.}, 14-B:285--299, 1977.

\bibitem{NSV:2009}
R.~H. Nochetto, K.~G. Siebert, and A.~Veeser.
\newblock {T}heory of adaptive finite element methods: an introduction.
\newblock In R.~A. DeVore and A.~Kunoth, editors, {\em Multiscale, Nonlinear
  and Adaptive Approximation}, pages 409--542. Springer, New York, 2009.

\bibitem{Pidcock:1995}
M.~K. Pidcock, S.~Ciulli, and S.~Ispas.
\newblock Singularities of mixed boundary value problems in electrical
  impedance tomography.
\newblock {\em Physiol. Meas.}, 16(3):A213--A218, 1995.

\bibitem{Rondi:2008}
L.~Rondi.
\newblock {O}n the regularization of the inverse conductivity problem with
  discontinuous conductivities.
\newblock {\em Inverse Probl. Imaging}, 2(3):397--409, 2008.

\bibitem{Rondi:2016}
L.~Rondi.
\newblock Discrete approximation and regularisation for the inverse
  conductivity problem.
\newblock {\em Rend. Istit. Mat. Univ. Trieste}, 48:315--352, 2016.

\bibitem{ScottZhang:1990}
L.~R. Scott and S.~Zhang.
\newblock Finite element interpolation of nonsmooth functions satisfying
  boundary conditions.
\newblock {\em Math. Comp.}, 54(190):483--493, 1990.

\bibitem{Siebert:2011}
K.~G. Siebert.
\newblock {A} convergence proof for adaptive finite elements without lower
  bounds.
\newblock {\em IMA J. Num. Anal.}, 31(3):947--970, 2011.

\bibitem{SomersaloCheneyIsaacson:1992}
E.~Somersalo, M.~Cheney, and D.~Isaacson.
\newblock {E}xistence and uniqueness for electrode models for electric current
  computed tomography.
\newblock {\em SIAM J. Appl. Math.}, 52(4):1023--1040, 1992.

\bibitem{TanLvDong:2019}
C.~Tan, S.~Lv, F.~Dong, and M.~Takei.
\newblock Image reconstruction based on convolutional neural network for
  electrical resistance tomography.
\newblock {\em IEEE Sensors J.}, 19(1):196--204, 2019.

\bibitem{Traxler:1997}
C.~Traxler.
\newblock {A}n algorithm for adaptive mesh refinement in $n$ dimensions.
\newblock {\em Computing}, 59:115--137, 1997.

\bibitem{Ver:2013}
R.~Verf\"{u}rth.
\newblock {\em {A} {P}osteriori {E}stimation Techniques for Finite Element
  Methods}.
\newblock Oxford University Press, Oxford, 2013.

\bibitem{XiaoLiuZhao:2018}
J.~Xiao, Z.~Liu, P.~Zhao, Y.~Li, and J.~Huo.
\newblock Deep learning image reconstruction simulation for electromagnetic
  tomography.
\newblock {\em IEEE Sensors J.}, 18(8):3290--3298, 2018.

\bibitem{XuZou:2015a}
Y.~Xu and J.~Zou.
\newblock {C}onvergence of an adaptive finite element method for distributed
  flux reconstruction.
\newblock {\em Math. Comp.}, 84(296):2645--2663, 2015.

\bibitem{ZhangChenXu:2018}
Q.~Zhang, L.~Chen, and Y.~Xu.
\newblock A minimization method for the double-well energy functional.
\newblock Preprint, arXiv:1809.01839, 2018.

\bibitem{ZhouHarrachSeo:2018}
L.~Zhou, B.~Harrach, and J.~K. Seo.
\newblock Monotonicity-based electrical impedance tomography for lung imaging.
\newblock {\em Inverse Problems}, 34(4):045005, 25, 2018.

\end{thebibliography}
\end{document}